\documentclass[a4paper,twoside]{report}
\usepackage[utf8x]{inputenc}
\usepackage[english,french]{babel}
\usepackage{aeguill}
 \usepackage[final]{graphicx}
\usepackage{amsmath}		
\usepackage{makeidx}
\usepackage{amssymb}		
\usepackage{euscript}
\usepackage{mathrsfs}
\usepackage{bbm}
\usepackage[amsmath,thmmarks]{ntheorem}	
\usepackage[all]{xy}
\usepackage[colorlinks=true]{hyperref}
\usepackage{stmaryrd}
\usepackage{bussproofs}
\usepackage{array}
\usepackage{enumitem}
\usepackage{geometry}
\usepackage{titlesec}

\def\xyunfrenchcatcodes{\catcode 33 12 \catcode 58 12 \catcode 59 12 \catcode 63 12 }
\EnableBpAbbreviations

\newcommand{\mcf}{\mathcal{F}}

\newcommand{\mcm}{\mathcal{M}}

\newcommand{\catset}{{\mathbf{Set}}}

\newcommand{\catiord}{\mathbf{IOrd}}

\newcommand{\catord}{\mathbf{Ord}}

\newcommand{\catpfib}{\mathbf{PFib}}
\newcommand{\pfib}{\mathbf{PFib}}
\newcommand{\epfib}{\exists\text{-}\mathbf{PFib}}
\newcommand{\mpfib}{\wedge\text{-}\mathbf{PFib}}
\newcommand{\ffrm}{\mathbf{FFrm}}

\newcommand{\catpdist}{\mathbf{PDist}}
\newcommand{\catcstruct}{\mathcal{CSTRUCT}}

\newcommand{\catcat}{{\mathbf{Cat}}}

\newcommand{\catudist}{\mathbf{UDist}}
\newcommand{\catidist}{\mathbf{IDist}}
\newcommand{\catfib}{{\mathbf{Fib}}}

\newcommand{\tope}{{\mathcal{E}}}

\newcommand{\tops}{{\mathcal{S}}}

\newcommand{\cata}{\mathbb{A}}
\newcommand{\catb}{\mathbb{B}}
\newcommand{\catc}{\mathbb{C}}
\newcommand{\catd}{\mathbb{D}}

\newcommand{\catl}{\mathbb{L}}
\newcommand{\catp}{\mathbb{P}}
\newcommand{\catq}{\mathbb{Q}}
\newcommand{\catr}{\mathbb{R}}
\newcommand{\cats}{\mathbb{S}}

\newcommand{\catx}{\mathbb{X}}

\newcommand{\hyph}{{\EuScript{H}}}

\newcommand{\pcaa}{\mathcal{A}}
\newcommand{\pcaia}{(I,\mathcal{A})}
\newcommand{\pcaai}{\mathcal{A}_i}
\newcommand{\pcaaii}{(\mathcal{A}_i)_{i\in I}}
\newcommand{\pcaas}{{\mathcal{A}_\#}}
\newcommand{\pcaasi}{\mathcal{A}_{\#,i}}
\newcommand{\pcaass}[1]{\mathcal{A}_{\#,#1}}

\newcommand{\pcab}{\mathcal{B}}

\newcommand{\ipcaasa}{\mathcal{A}_\#\subseteq\mathcal{A}}
\newcommand{\iipcaasa}{(I,\mathcal{A}_\#\subseteq\mathcal{A})}

\newcommand{\ufpcaa}{(\pcaa,R(\pcaa))}

\newcommand{\ufpcaia}{(I,\pcaa,R(\pcaa))}
\newcommand{\ufpcaias}{(I,\pcaa,R(\pcaas))}

\newcommand{\trip}{{\EuScript{P}}}
\newcommand{\triq}{{\EuScript{Q}}}

\newcommand{\rtr}[1]{\mathbf{rt}(#1)}

\newcommand{\catrt}{\mathbf{RT}}
\newcommand{\catrc}{\mathbf{RC}}

\newcommand{\comk}{\mathsf{k}}				
\newcommand{\coms}{\mathsf{s}}				
\newcommand{\compair}{\mathsf{pair}}				
\newcommand{\comfst}{\mathsf{fst}}				
\newcommand{\comsnd}{\mathsf{snd}}				

\newcommand{\vtp}{\mathclose:}				

\newcommand{\csep}{\mathrel|\,}		
\newcommand{\msep}{\mathrel|\,}		

\newcommand{\op}{\mathsf{op}}				
\newcommand{\co}{\mathsf{co}}				%
\newcommand{\prop}{{\mathsf{Prop}}}		
\newcommand{\triptr}{{\mathsf{tr}}}		

\newcommand{\id}{\mathrm{id}}				
\newcommand{\defequi}{:\Leftrightarrow}			

\newcommand{\appca}{\mathclose\cdot}			
\newcommand{\qdot}{\,.\hspace{.3mm}}				
\newcommand{\sheaf}[1]{{\mathbf{Sh}(#1)}}		

\newcommand{\setof}[2]{\{#1\msep #2\}}

\newcommand{\asm}{\mathbf{Asm}}
\newcommand{\basm}{\mathbf{B}\textrm{-}\mathbf{Asm}}
\newcommand{\commacat}[2]{#1\mathopen\downarrow#2}
\newcommand{\pscommacat}[2]{#1\mathopen\Downarrow#2}

\newcommand{\olaco}[2]{{(#1\mathopen\swarrow #2)}}

\newcommand{\cro}{{(C,\rho)}}

\newcommand{\dsi}{{(D,\sigma)}}

\newcommand{\predeq}{\mathclose=}

\newcommand{\N}{\mathbb{N}}

\newdir_{ (}{{ }*!/-.5em/@_{(}}%
\newdir{ (}{{ }*!/-.5em/@^{(}}%
\newdir{ >}{{}*!/-10pt/@{>}}

\newdir{|>}{{}*!/+8pt/@{|}*!/+3pt/:(1,-.2)@^{>}*!/+3pt/:(1,+.2)@_{>}}

\newcommand{\reg}{\ar@{|->}}
\newcommand{\mono}{\ar@{ >->}}
\newcommand{\cocov}{\ar@{ |>->}}
\newcommand{\depi}{\ar@{->>}}
\newcommand{\cov}{\ar@{-|>}}
\newcommand{\monepi}{\ar@{ >->>}}
\newcommand{\ccov}{\ar@{->>}|-{\mathrm{c}}}
\newcommand{\vmono}{\ar@{ >->}|-{\mathrm{v}}}
\newcommand{\vepi}{\ar@{->>}|-{\mathrm{v}}}

\setbox0=\hbox{$A$}
 \newdimen\axis \axis=\fontdimen22\textfont2%
 \newbox\bullbox \setbox\bullbox=\hbox{.}
 \wd\bullbox 0pt
 \ht\bullbox 2.0pt
 \dp\bullbox 1.0pt
\newcommand{\dashed}{\ar@{-->}}

 \def\bull{\kern 1.9pt \usebox\bullbox\kern 1.1pt}
\newdir{..}{\object{\bull}}
\newcommand{\dotted}{\ar@{..>}}
\newcommand{\dline}{\ar@{..}}
\newcommand{\dar}{\ar@{..>}}
\newcommand{\dreg}{\ar@{|..>}}
\newcommand{\dmono}{\ar@{ >..>}}
\newcommand{\dcocov}{\ar@{ |>..>}}
\newcommand{\ddepi}{\ar@{..>>}}
\newcommand{\dcov}{\ar@{..|>}}
\newcommand{\dmonepi}{\ar@{ >..>>}}

\newcommand{\dist}{\ar|-*=0@{|}}
\newcommand{\edist}{\mathbin{\mathop{\longrightarrow\hspace{-.75em}\mapsfromchar\hspace{.75em}}}}
\newcommand{\dashprof}{\ar@{-->}|-*=0@{|}}
\newdir{+>}{{}*!/-4.5pt/@{>}}
\newcommand{\mar}{\ar@{-}|-*=0@{+>}}
\newcommand{\ppair}[3]{{\ar@<3pt>[#1]^-{#2}\ar@<-3pt>[#1]_-{#3}}}
\newcommand{\nohead}{\ar@{-}}
\newcommand{\incl}{\ar@{^{ (}->}}
\newcommand{\emar}{\ar@{}}

\newcommand{\cart}{\ar@{~>}}
\newcommand{\epicart}{\ar@{~|>}}
\newdir{||}{{}*!/-5pt/@{}}
\newdir{ >|}{{}*!/-8pt/@{>}*!/-16pt/@{}}
\newcommand{\coca}{\ar@{||{+} >|}}
\newcommand{\arid}{\ar@{=}}

\renewcommand{\to}{\rightarrow}
\newcommand{\emono}{\rightarrowtail}
\newcommand{\eepi}{\twoheadrightarrow}
\newcommand{\eiso}{\stackrel{\cong}{\longrightarrow}}

\newcommand{\erel}{\looparrowright}
\newcommand{\etrel}{{\looparrowright\hspace{-8.6pt}\rightarrowtriangle}}
\newcommand{\eepicart}{\xymatrix@1@-3mm{{}\epicart[r]&{}}}
\newcommand{\ecovercart}{\xymatrix@1@-3mm{{}\epicart[r]&{}}}

\newcommand{\abs}[1]{\lvert#1\rvert}

\newcommand{\ve}{\varepsilon}

\newcommand{\qtext}[1]{\quad\text{#1}\quad}
\newcommand{\stext}[1]{\;\text{#1}\;}

\newcommand{\sub}{\operatorname{sub}}
\newcommand{\Sub}{\operatorname{Sub}}
\newcommand{\siev}{\operatorname{siev}}
\newcommand{\Siev}{\operatorname{Siev}}

\newcommand{\imp}{\Rightarrow}

\newcommand{\cxymatrix}[1]{\vcenter{\xymatrix{#1}}}

\newcommand{\bbtwo}{\mathbbm{2}}

\newcommand{\sep}{\mathbf{Sep}}
\newcommand{\sh}{\mathbf{Sh}}

\newcommand{\dom}{\mathrm{dom}}

\newcommand{\famf}{\mathrm{fam}}
\newcommand{\Famf}{\mathrm{Fam}}
\newcommand{\ufam}{\mathrm{ufam}}
\newcommand{\Ufam}{\mathrm{UFam}}
\newcommand{\subf}{\mathrm{sub}}

\newcommand{\pullbackcorner}[1][dr]{\save*!/#1-1.2pc/#1:(-1,1)@^{|-}\restore}

\newcommand{\gl}{\mathrm{gl}}
\newcommand{\Gl}{\mathrm{Gl}}

\newcommand{\brprod}{\Join}

\newcommand{\adj}{\dashv}
\newcommand{\ent}{\vdash}

\newcommand{\delsep}{\delta:\sub(\tope)\to\trip}

\newcommand{\pto}{\rightharpoonup}

\newcommand{\catlex}{\mathbf{Lex}}
\newcommand{\catreg}{\mathbf{Reg}}
\newcommand{\catex}{\mathbf{Ex}}

\newcommand{\catpretop}{\mathbf{Pretop}}
\newcommand{\catlxv}{\mathbf{Lxv}}
\newcommand{\catgeo}{\mathbf{Geo}}
\newcommand{\catpos}{\mathbf{Pos}}

\newcommand{\adjr}[2]{
\ar@/_6pt/[r]_{#1}
\ar@{}[r]|\top
\ar@{<-}@/^6pt/[r]^{#2}
}
\newcommand{\adjrr}[2]{
\ar@/_6pt/[rr]_{#1}
\ar@{}[rr]|\top
\ar@{<-}@/^6pt/[rr]^{#2}
}
\newcommand{\adjl}[2]{
\ar@/^6pt/[l]^{#1}
\ar@{}[l]|\top
\ar@{<-}@/_6pt/[l]_{#2}
}
\newcommand{\adjd}[2]{
\ar@/_6pt/[d]_{#1}
\ar@{}[d]|\dashv
\ar@{<-}@/^6pt/[d]^{#2}
}
\newcommand{\adju}[2]{
\ar@/^6pt/[u]^{#1}
\ar@{}[u]|\dashv
\ar@{<-}@/_6pt/[u]_{#2}
}
\newcommand{\dadj}[4]{ 
\ar@/^6pt/[#3]^{#1}
\ar@{}[#3]|{#4}
\ar@{<-}@/_6pt/[#3]_{#2}
}

\newcommand{\bco}{\mathbf{BCO}} 

\newcommand{\abracks}[1]{(#1)} 
\newcommand{\ebracks}[1]{[#1]} 
\newcommand{\pil}{\pi_l} 
\newcommand{\pir}{\pi_r} 
\newcommand{\ldist}[1]{{#1}^*} 
\newcommand{\rdist}[1]{{#1}_*} 

\newcommand{\brar}{{(A,R)}}

\newcommand{\upiar}{(I,A,R)}
\newcommand{\upjbs}{(J,B,S)}
\newcommand{\upkct}{(K,C,T)}
\newcommand{\upldu}{(L,D,U)}
\newcommand{\upa}{\mathfrak{A}}
\newcommand{\upb}{\mathfrak{B}}
\newcommand{\upc}{\mathfrak{C}}

\newcommand{\pplus}{P_+}
\newcommand{\dplus}{D_+}

\newcommand{\catufp}{\mathbf{UOrd}}
\newcommand{\catuord}{\mathbf{UOrd}}
\newcommand{\cateuord}{\exists\text{-}\mathbf{UOrd}}
\newcommand{\catmuord}{\wedge\text{-}\mathbf{UOrd}}
\newcommand{\catufrm}{\mathbf{UFrm}}

\newcommand{\fibc}{\mathscr{C}}
\newcommand{\fibd}{\mathscr{D}}
\newcommand{\fibe}{\mathscr{E}}
\newcommand{\fibf}{\mathscr{F}}
\newcommand{\fibp}{\mathscr{P}}
\newcommand{\fibq}{\mathscr{Q}}

\newcommand{\fibs}{\mathscr{S}}

\newcommand{\fibx}{\mathscr{X}}

\newcommand{\ffrma}{\EuScript{A}}
\newcommand{\ffrmb}{\EuScript{B}}

\newcommand{\fifa}{\EuScript{A}}
\newcommand{\fifb}{\EuScript{B}}
\newcommand{\fifc}{\EuScript{C}}

\newcommand{\fifx}{\EuScript{X}}

\newcommand{\fpa}{\EuScript{A}}

\newcommand{\tot}[1]{\abs{#1}}

\newcommand{\whfc}{\widehat{\fibc}}

\newcommand{\adcl}{\mathopen\downarrow}
\newcommand{\uf}{(u,f)}
\newcommand{\vg}{(v,g)}

\newcommand{\srel}[2]{{#1\{#2\}}}
\newcommand{\defined}{\mathclose\downarrow}

\newcommand{\srelrc}{\srel{\catr}{\fibc}}
\newcommand{\geostack}{geometric pre-stack}

\newcommand{\trel}{\operatorname{tRel}}
\newcommand{\prim}{\mathbf{Prim}}

\newcommand{\les}{\preceq}
\newcommand{\ges}{\succeq}
\newcommand{\ptype}{*}

\newcommand{\pslice}{\!\sslash\!}
\newcommand{\geopos}[1]{\mathfrak{P}(#1)}
\newcommand{\flpos}[1]{\mathfrak{A}(#1)}
\newcommand{\ilbracks}[1]{\bigl( #1 \bigr)}
\newcommand{\fund}[1]{\mathrm{cod}(#1)}

\newcommand{\lfund}[1]{\mathrm{cod}(#1):\commacat{#1}{#1}\to #1}
\newcommand{\dindex}[1]{\index{#1}}
\newcommand{\umeet}{(\ptype,\wedge)}
\newcommand{\uterm}{(1,\top)}
\newcommand{\lfisufi}[1]{\sub(#1):\Sub(#1)\to\tot{#1}}
\newcommand{\sall}{\;\forall}
\newcommand{\sists}{\;\exists}
\newcommand{\per}{\mathrm{PER}}
\newcommand{\iemph}[1]{\emph{#1}\index{#1}}
\newcommand{\twif}{\top,\wedge,\imp,\forall}

\newcommand{\pp}{\mathscr{P}}
\newcommand{\scs}{\commacat{\tops}{\tops}}
\newcommand{\sierp}{\widehat{\bbtwo}}
\newcommand{\pcakone}{\mathcal{K}_1}

\newcommand{\ddesc}{\mathbf{Desc}}
\newcommand{\simp}{\!\imp\!}

\newcommand{\conflict}{\ar@{~}}

\def\signed #1{{\leavevmode\unskip\nobreak\hfil\penalty50\hskip2em
  \hbox{}\nobreak\hfil(#1)%
  \parfillskip=0pt \finalhyphendemerits=0 \endgraf}}

\newsavebox\mybox

\newcommand{\comment}[1]{}
\theoremnumbering{arabic}
\theoremstyle{plain}
\theoremsymbol{}
\theorembodyfont{\itshape}
\theoremheaderfont{\normalfont\bfseries}
\theoremseparator{}
\newtheorem{theorem}{Theorem}[section]

\newtheorem{lemma}[theorem]{Lemma}

\newtheorem{corollary}[theorem]{Corollary}

\theorembodyfont{\upshape}
\theoremsymbol{\ensuremath{\diamondsuit}}

\newtheorem{definition}[theorem]{Definition}

\theorembodyfont{\normalfont}

\newtheorem{remark}[theorem]{Remark}

\newtheorem{remarks}[theorem]{Remarks}
\newtheorem{example}[theorem]{Example}
\newtheorem{examples}[theorem]{Examples}

\newtheorem{convention}[theorem]{Convention}

\theoremsymbol{\ensuremath{_\blacksquare}}
\theoremstyle{nonumberplain}
\theoremheaderfont{\itshape}
\newtheorem{proof}{Proof.}
\qedsymbol{\ensuremath{_\blacksquare}}

\numberwithin{equation}{section}

\makeindex
\setlist{itemsep=0pt}
\CompileMatrices

\begin{document}
\xyunfrenchcatcodes
\selectlanguage{english}
\pagestyle{empty}
{\pagestyle{empty}
\begin{center}
{\large\sc Université Paris Diderot -- Paris 7\\[2mm] Laboratoire PPS}
\end{center}
\vspace{2cm}
\begin{center}
{\huge\bfseries\sc 
A fibrational study of\\[2mm] 
realizability toposes}
\end{center}
\vspace{9mm}
\begin{center}
{\Large Jonas {\sc Frey}}
\end{center}
\vspace{2cm}
\begin{center}
Defended on 20 June 2013\\
Minor modifications 14 March 2014
\end{center}
}

\clearpage
\pagestyle{plain}
\chapter*{Thanks}

Science is a community effort, and this thesis wouldn't exist without
the help of many.

\bigskip

First and foremost, I want to thank my supervisor Paul-André Melliès for sharing
his inspiration and enthusiasm, and for a lot of support and encouragement.

Thanks to Thomas Streicher for teaching me categorical logic, and for
many discussions, email exchanges and pointers to literature which were crucial
to my work.

Thanks to Martin Hyland, Pino Rosolini, Alex Simpson, Jaap van Oosten, Benno van
den Berg, Naohiko Hoshino, Sam Staton and Wouter Stekelenburg for inspiring and
instructive discussions on categorical realizability, of which this work has greatly profited.

Thanks to Peter Johnstone, Jaap van Oosten and Thomas Streicher for 
agreeing to be reviewers and devoting their time to read the manuscript.

\bigskip

The four years that I spent at Paris 7 University have been an enriching time, 
and while it would be hard to mention all the people who contributed to this, 
I want to thank at least a few of them.

Thanks to Stéphane, Thibaut, Gabriel, Antoine, Fabien,
Pierre, Boris, Samuel, Nicolas and Mehdi of 6C10-fame for never being above a 
joke, and for patiently helping me out with the subtleties of French language, 
life, and bureaucracy. 

Thanks to Albert Burroni, Georges Maltsiniotis, Pierre-Louis Curien, François
Métayer, Dimitri, Samuel, Jonathan, and all the other
participants for the homotopy seminar, which was a source of inspiration and a
fixed point of my life in Paris.

Thanks to Christine, Dimitri, Nicolas and Samuel for taking time to discuss 
category theory.

Thanks to Beniamino, Guillaume, Noam and Stéphane 
for sharing with me their visions of polarization and focalization.

Thanks to Alexis, Shahin, Kuba, Flavien, Ionna, Sara, Sarah, Matthieu, Pierre B, 
Charles, Guilio, Serguei, Danko, Matthias, Stéphane G, Séverine, Mathias, Jaap, Alberto, Jim, Daisy, Juliusz and 
all that I have forgotten for
discussions, conversations, inspirational coffee breaks, picnics, and generally a great time.

Thanks to all of PPS for your open and friendly atmosphere.

\tableofcontents
\chapter{Introduction}

The present thesis is meant to be a contribution to the theory of
`realizability toposes', and more generally to `categorical realizability' and
`categorical logic'.

\section{Context}

\subsection{Realizability}

Realizability is a technique originating from \emph{proof theory} and was
originally devised by Kleene~\cite{kleene1945interpretation} to reason about
metamathematical properties of formal
systems. This is done by associating `realizers' to logical formulas, which are
usually finitary objects (integers, terms) viewed as approximations of
proofs, using ideas from constructive mathematics. 

In \emph{computer science}, realizability is strongly related to typed lambda
calculi via the proofs-as-programs correspondence. 
 In this context, it is very instructive to compare
realizability to typing à la Curry: typing constructs a binary relation between
types and terms by induction on the term structure, whereas realizability
constructs such a relation by induction on the structure of types. The
realizability relation is generally bigger than the typing relation, and
undecidable. This point of view on realizability is strongly related to
normalization proofs of typed lambda calculi via `reducibility candidates'.
Besides normalization, realizability can be used to reason about operational
semantics.

\emph{Semantically}, realizability 
can be viewed as a model construction for predicate logic. The
basic idea here is that the set of realizers of a closed formula is viewed
as its `truth
value'; in the simplest case a truth value is `true', if it is inhabited. This
seems rather restricted as it admits (up to equivalence) only two truth values
for closed formulas, but the structure diverges from the classical model as soon
as we consider open formulas: these are interpreted by `predicates', which are
families of truth values, and a predicate is considered true whenever it has a
\emph{uniform} realizer, i.e.\ the intersection of all its truth values is
inhabited.

\subsection{The categorical approach}

At the end of the 70ies, it was observed that the semantic aspect of
realizability fits into a formalism described by
Lawvere~\cite{lawvere1970equality}: the above
mentioned semantic predicates can be arranged into a \emph{hyperdoctrine},
whereby it becomes apparent that the constructions thought out by proof
theorists to interpret the logical connectives are characterized by universal
properties. It was then realized that if the hyperdoctrine has enough
structure, the model given by the hyperdoctrine can be `internalized' into a
topos by performing a construction analogous to the construction of sheaves on a
locale, providing an abstract way of turning realizability interpretations into
models of \emph{higher order logic}. The hyperdoctrines for which this
construction is possible and results in a topos were called \emph{triposes} by
Hyland, Johnstone and Pitts~\cite{hjp80}\footnote{Actually the definition of
tripos is slightly stronger
than what is necessary to construct a topos (see~\cite{pitts2002}), but the
additional strength can naturally be viewed as a \emph{smallness} condition.}.

The most well known and well studied of the toposes obtained this way is
Hyland's \emph{effective topos}~\cite{hyland82}, which is the categorical
incarnation of Kleene realizability. Kleene realizability uses natural numbers
as realizers, but --
as already noted in~\cite{hjp80} --, the same construction works just as well
when we
replace the natural numbers by elements of an arbitrary \emph{partial
combinatory algebra} (pca) $\pcaa$, giving rise to the \emph{realizability
topos} $\catrt(\pcaa)$.

\section{Motivation}

The present work is not concerned with realizability in its proof theoretic
sense at all, but only with its categorical abstraction in the form of partial
combinatory algebras, triposes and toposes.

Broadly speaking, the intention of this work is to get a more abstract
understanding
of realizability toposes.

\subsection{Questions}

The ideas and constructions presented in this thesis can be motivated  
by  the following two questions/challenges.
\begin{enumerate}
 \item 
Johnstone compared the state of the study of
realizability toposes to `stamp-collecting', calling for a
Giraud~\cite{giraud1963analysis} style
`extensional' characterization of realizability toposes. 
\item
The construction of realizability toposes is motivated by the analogy to 
Grothendieck toposes of sheaves on a locale, but realizability toposes
themselves are not Grothendieck toposes. Nevertheless, we want to push the
analogy further and try to adapt techniques from Grothendieck toposes to
realizability. A more ambitious goal is to find an axiomatic framework
generalizing Grothendieck toposes in such a way that it contains realizability
toposes.
\end{enumerate}

\subsection{Approach}

\subsubsection*{Johnstone's question}

Johnstone's question for a Giraud style characterization is difficult to answer
since it is not clear what precisely we want to call realizability. Clearly, the
toposes constructed from
pcas are realizability toposes, but we might also want to include toposes
arising from modified realizability, relative realizability, and the Dialectica
interpretation, to name just a few. Krivine's \cite{krivine2011realizability}
notion of \emph{realizability structure} subsumes set-theoretical
\emph{forcing}, and if we want to adopt this liberal point of view we
should also include a lot of Boolean valued models, and if we do not insist on
classical logic, Heyting valued models as well. On the level of
hyperdoctrines, a class of structures that contains all these examples and
furthermore has good closure properties is the class of triposes, and in my
opinion this is a reasonable candidate for an abstract framework for
realizability\footnote{At least as long as we are interested in realizability
\emph{toposes} -- there are notions of realizability which do not give rise to
impredicative models, and we will later consider corresponding types of
hyperdoctrines. However, in Section~\ref{sec:arbitrary-bases} we will see that
those predicative classes of models are more difficult to handle on base
categories other than $\catset$, and in particular iteration seems more
difficult if possible at all, so there are good reasons to consider primarily
the impredicative case.}.

We do not know how to give an abstract characterization of
toposes arising from triposes, but such a characterization can be given (as has
been done by Pitts~\cite{pitts81}) if we admit as additional structuring data
the `constant objects functor' $\Delta:\catset\to\catset[\trip]$ which embeds
the category of sets into the topos.

\subsubsection*{The constant objects functor}

The inclusion of the constant objects functor in the data can be justified by
taking the point of view of `relative (Grothendieck) topos theory': whenever
we have a bounded
geometric morphism $(\Delta\adj \Gamma:\tope\to\tops)$, $\tope$ is equivalent to
a topos of sheaves on an internal site in $\tops$, whence bounded geometric
morphisms into $\tops$ can be viewed as Grothendieck toposes relative to
$\tops$. In the same way, functors $\Delta:\tops\to\tope$ satisfying Pitts'
condition correspond to tripos-induced toposes relative to $\tops$ (the analogy
makes sense since  for triposes coming from complete Heyting algebras,
the constant objects functor coincides with the inverse image functor of the
geometric morphism). 
Now in the case of Grothendieck toposes relative to $\catset$, the geometric
morphism can be constructed from the topos alone, but this is not the case for
toposes coming from triposes, whence we have to include the functor in the data.

\subsubsection*{The fibrational point of view}

A way to understand the relevance of the constant objects and direct
image functors is via the gluing fibration: any regular functor
$\Delta:\tops\to\tope$ between toposes allows to view $\tope$ as a fibration (actually a
stack) over $\tops$ by gluing (i.e.\ taking the pullback of the fundamental fibration
of $\tope$ along $\Delta$), and one can argue that this is the real object of
interest. Relative to $\catset$, this fibration is just the family fibration for
Grothendieck toposes, while we obtain non-standard fibrations in case of
tripos-induced toposes.

The fibrational point of view opens up an interesting new perspective: while
from the non-fibered viewpoint the main structural difference between
realizability toposes and Grothendieck toposes is that the former are not
cocomplete, the gluing fibration of a realizability topos is fibrationally
cocomplete, but not \emph{locally small}, which implies for example that we
can't use Freyd's adjoint functor theorems.

\subsubsection*{Unifying realizability and Grothendieck toposes}

The previous deliberations suggest that a common framework for realizability and
Grothendieck toposes is given by fibrations of toposes arising from gluing
along regular functors $\Delta:\catset\to\tope$ into toposes. To
restrict this very general class of structures, it seems sensible to impose a
boundedness condition on the corresponding constant objects functors, which
should correspond to the fact that the fibration can be constructed from a
small, `site-like' structure.

We do not study this question in detail or try to give an axiomatics, but some
further speculations in this direction can be found
Sections~\ref{sec:arbitrary-bases} and \ref{sec:non-posetal}. 
In particular, in Section~\ref{sec:non-posetal} we sketch a definition of the
alluded site like structures which we call `uniform categories', and the
considerations of Section~\ref{sec:arbitrary-bases} suggest that
on base toposes $\tops$ other than $\catset$ it seems to be necessary to
postulate, in addition to the constant objects functor $\Delta:\tops\to\tope$,
a kind of `global sections functor' $\Gamma:\tope\to\tops$ (which is however not
required to be
adjoint to $\Delta$ in general), in order to be able to reconstruct the
generalized site from the
fibration and the bound.

\subsubsection*{In this work}

As explained above, one might argue that when taking the fibrational point
of view, and adopting a liberal notion of realizability, Johnstone's question
has already been answered by Pitts.

The main result of the present work is an analogous characterization
for the \emph{smallest} reasonable class of realizability toposes -- those
arising from pcas. 

In order to achieve this, we develop a framework of `fibrational
cocompletions', which is manifested as a chain of biadjunctions between
2-categories of pre-stacks on a regular category $\catr$. These
biadjunctions are viewed as a fibrational generalization of the transition from
a small finite-limit category $\catc$ to the presheaf category
$\widehat{\catc}$, and several intermediate steps.

The analogy being that finite-limit fibrations are generalized finite-limit
categories on which we construct generalized presheaf categories, we will
then concentrate on the posetal case, in particular on a class of posetal
fibrations on $\catset$	admitting a small representation -- the
\emph{uniform preorders}. Uniform preorders provide an adequate framework for
the analysis of realizability over pcas since they contain all triposes (in
the presence of choice), as
well as representations of pcas.

\section{Overview}

\subsubsection*{Chapter~\ref{chap:fibrations}}

In Chapter~\ref{chap:fibrations}, we introduce the necessary parts of fibered
category theory, including \emph{Moens' theorem} which clarifies the
relationship between constant object functors and their gluing fibrations, and
(pre-)stacks for the regular topology. In particular,
we present the chain
\begin{equation}\label{eq:chain-intro}
\catlex(\catr)
\hookleftarrow\catgeo(\catr) 
\hookleftarrow\catpos(\catr)
\hookleftarrow\catpretop(\catr)
\end{equation}
of 2-categories of \emph{fibered pretoposes, positive pre-stacks, geometric
pre-stacks and finite-limit pre-stacks}
on a regular category $\catr$ which form the basis of the developments in
Chapter~\ref{chap:fib-cocompletions}.

\subsubsection*{Chapter~\ref{chap:fib-cocompletions}}

Chapter~\ref{chap:fib-cocompletions} is about \emph{fibrational cocompletions}
-- more precisely we construct left biadjoints to the chain of inclusion
functors above. The fibrational cocompletions are meant to provide
a common framework for presheaf-constructions and realizability constructions,
the motivating examples are the following:
\begin{itemize}
 \item Given a small finite-limit category $\catc$, its family fibration
$\famf(\catc)$ is a finite-limit pre-stack. The fibered pretopos cocompletion
of $\famf(\catc)$ is $\famf(\widehat{\catc})$ (the family fibration of the
category of presheaves), and the geometric and positive prestack cocompletions
are the subfibrations on families of subrepresentable presheaves, and coproducts
of subrepresentable presheaves, respectively.
\item
For a pca $\pcaa$, the posetal fibration of singleton valued predicates in
$\pcaa$ is a finite-limit pre-stack. Its geometric cocompletion is the
realizability tripos, and its positive and pretopos cocompletions are the gluing
fibrations of the category of assemblies and the realizability topos,
respectively.
\end{itemize}
The reason for the use of pre-stacks is explained at the beginning of
Section~\ref{sec:prestacks}.

In Sections~\ref{sec:fibered-presheaves},
\ref{sec:fibered-sheaf-construction}, and \ref{sec:fibered-pretop-pos}, we
present the constructions of the biadjoints in detail, and in
Section~\ref{sec:preordered-case} we take a second look on the cocompletions
of finite-limit pre-stacks and geometric pre-stacks in the special case
of \emph{posetal} fibrations. This restriction simplifies the constructions
considerably, and is sufficiently general for the treatment of realizability.

In Section~\ref{sec:assemblies}, we slightly deviate from our main line of
thought, to treat \emph{assemblies} -- as emphasized by Johnstone,
realizability toposes constructed from pcas are much easier to work in than
more general toposes constructed from triposes, the reason being that they 
can be presented as 
ex/reg completions of `concrete' categories of assemblies. However,
among the toposes constructed from triposes, those coming from pcas are not the
only ones admitting such a presentation -- other examples are given by
\emph{relative realizability}, and presheaves on meet-semilattices. This raises
the question for a general criterion for when the construction via assemblies
is possible. In Section~\ref{sec:assemblies}, we show that for a topos
$\catset[\trip]$ constructed from a tripos $\trip$ on $\catset$ to be the
ex/reg completion of a concrete subcategory of assemblies, it is sufficient
that the embedding $\delta:\sub(\catset)\to\trip$ of classical predicates into
the tripos has a finite meet preserving left adjoint. Using classical logic, we
can moreover show that in this case, the assemblies coincide with the
$\neg\neg$-sheaves
in $\catset[\trip]$.

\subsubsection*{Chapter~\ref{chap:ufp}}

Although fibrational cocompletions work well, the framework that we
present in Chapter~\ref{chap:fib-cocompletions} is a bit too general for
the purpose of analyzing realizability -- as hinted earlier, we want to impose
boundedness conditions on the right side of the chain~\eqref{eq:chain-intro}
(for positive pre-stacks and
fibered pretoposes), and this should correspond to smallness conditions on the
left side (for finite-limit pre-stacks and geometric pre-stacks). 

To make this work, we need more structure in the base than that of a regular
category -- for example we want to express the transition from finite-limit
pre-stacks to geometric pre-stacks entirely on the level of internal data, and
to internalize the necessary constructions it is convenient to demand the
base to be at least a topos. It turns out that not even that is enough --
the only base category on which we can endow the chain of biadjunctions with
size data in a straightforward manner is $\catset$ (we can make it work on
other base categories if we introduce an additional layer in the fibrations, as
suggested in Section~\ref{sec:arbitrary-bases}, but this is not worked out in
detail in this thesis). As a further restriction we demand the fibrations on the
small side of the scale (finite-limit pre-stacks and geometric pre-stacks) to be
posetal from now on -- as already pointed out, this is sufficient for the
treatment of
realizability, and leads us to the concept of \emph{uniform preorder}.

Uniform preorders are representations of certain posetal fibrations on
$\catset$ conforming to a smallness condition -- in the presence of choice the
locally ordered category $\catuord$ of uniform preorders is equivalent to the
full subcategory of posetal fibrations on $\catset$ on those posetal fibrations
which have a \emph{generic family of predicates}. $\catuord$ is quite similar
in structure to the category of preorders and has very good closure properties
(in particular it is bicartesian closed and comes with a notion
of distributor that makes it a locally posetal cartesian bicategory with
duals~\cite{carboni1987cartesian}). 
This and the fact that it accommodates triposes and partial combinatory algebras
in a natural way\footnote{Following Hofstra~\cite{hofstra2006all} we do
not identify a pca with the corresponding realizability tripos, but with its
subfibration on singleton valued predicates.} makes $\catuord$ an ideal
framework to analyze realizability constructions, and enables us to prove our
main results near the end of the chapter. These are:
\begin{itemize}
 \item The identification of inclusions of typed pcas (corresponding to typed 
relative realizability) with \emph{relationally complete functional uniform
preorders} in Lemma~\ref{lem:rel-compl-func-tpca}. This identification can be
specialized to characterizations of the untyped and non-relative cases by adding
conditions.
\item The characterization of realizability triposes and hyperdoctrines arising
from (typed)
(inclusions of) pcas in Theorem~\ref{theo:character-relrealhyper}, using the
previous result and a concept of `$\exists$-primality'
(Definition~\ref{def:exprime}) generalizing the notion of `completely join
prime element' in complete lattices.
\item The characterization of the fibered (pre)toposes arising from (typed)
(inclusions of) pcas by the fibered presheaf construction in
Theorem~\ref{theo:char-cat}, using the characterization of pcas and a
fibrational generalization of the characterization of presheaf toposes in terms
of indecomposable projectives.

In the non-relative case, our characterization gives rise to a characterization
of the non-fibered realizability categories/toposes, since the constant objects
functor is right adjoint to the global sections functor in this case, and thus
doesn't give additional information.
\end{itemize}
In Section~\ref{sec:arbitrary-bases}, we describe how the correspondence
between uniform preorders and posetal fibrations can be expressed on base
toposes other than $\catset$, which gives an approach of how to generalize the
the listed results to arbitrary base toposes.

\subsubsection*{Appendix}

In Section~\ref{sec:pcas} we recall standard definitions from categorical
realizability.

In Section~\ref{sec:decompo}, we describe a decomposition result which is
inverse to Pitts' iteration theorem and is inspired by the idea that constant
objects functors are `generalized geometric morphisms' (since for geometric
morphisms there are several such decompositions known).

Finally, Section~\ref{sec:bits} contains outlines of unfinished work and ideas
for future investigations.

\section{Related work}

The present work is based on a large body of work in categorical realizability
that has been carried out throughout the last 20 years. This work can loosely
be divided into two themes -- \emph{exact completion} and \emph{combinatory
structures}: 
\begin{itemize}
 \item The connection between realizability toposes and exact completion was
established by Robinson and Rosolini~\cite{robinson1990colimit}, who observed
that realizability toposes are exact completions of their subcategories of
\emph{partitioned assemblies}. This inspired subsequent work by Birkedal,
Carboni,
Hofstra, Menni, Rosolini and Scott~\cite{carboni1995some, birkedal1998type,
menni2000exact, carboni2000locally, menni2003characterization,
hofstra2004relative}, of which \cite{carboni2000locally} is of particular
importance for this thesis.
\item The study of `combinatory structures' generalizing pcas started with
van Oosten's~\cite{vanoosten1997extensional} definition of ordered
pcas\footnote{Longley's thesis~\cite{longley1995realizability} should also be
mentioned as it introduced the idea of organizing 	pcas into a
\emph{locally ordered category}, which was subsequently adopted for more general
classes of combinatory structures.},
and ordered pcas were further developed by Hofstra and
van~Oosten~\cite{hofstra2003completions, hofstra2003ordered}. Longley
generalized pcas by adding types~\cite{longley1999unifying}, and Hofstra
further generalized ordered pcas into \emph{basic combinatory
objects} (BCOs)~\cite{hofstra2006all, hofstra2008iterated}. Recently,
Longley~\cite{longley2011computability} presented a vast generalization of his
ordered pcas.

Of these works, \cite{hofstra2006all} has been crucial to the
development of this thesis -- a lot of results about uniform preorders are just
adoptions of Hofstra's results about BCOs (we will
indicate this in the text). However, BCOs correspond just to single-sorted
uniform preorders. The idea to consider the many-sorted version was triggered
by Streicher's remark that the definition of uniform preorder
(in its first, one-sorted version) resembled the definition of Longley's
C-structures~\cite{longley2011computability}. Remarkably, some of our concepts
that are based on intuitions about fibered preorders are similar to notions
that Longley devised without these intuitions -- most notably the definition of
`relationally complete uniform preorder' (Definition~\ref{def:rel-compl-ufp}) is
similar to Longley's~\emph{higher order C-structure}.

Krivine's realizability structures~\cite{krivine2011realizability} can also be
seen as generalizations of pcas, but their link to the structures
mentioned above is not explored here.
\end{itemize}
The idea to view realizability toposes as \emph{presheaf} toposes can be traced
to Hofstra's~\cite{hofstra2006all} observation that realizability triposes can
be generated by freely adding existential quantification to the singleton
fibrations (analogous to the fact that sheaves on a downset lattice are
equivalent to presheaves on the generating preorder), the idea that this leads
to an alternative reading of the `exact completion description' of
realizability toposes materialized during a discussion with Rosolini. This lead
to a new approach to the question of finding a choice-free presentation, by
taking inspiration from Bunge's~\cite{bunge1977internal} characterization of
preshaf toposes over arbitrary bases. When adapting this viewpoint to
realizability, the `partitioned assemblies' which are important for exact
completion become \emph{families of subterminal indecomposable projectives} --
the corresponding assembly is the internal sum of this family, and is only
projective if the indexing set is, which explains the role of choice in a
sense.

\smallskip

In November 2011, I learned from Naohiko Hoshino that he developed a framework
for combinatory objects that somewhat resembled our \emph{uniform preorders}.
Hoshino defined his version of combinatory objects as monads in a more primitive
locally ordered bicategory, a point of view that was influential for the ideas
sketched in Appendix~\ref{sec:uniform-as-internal}.

\smallskip

Wouter Stekelenburg's thesis~\cite{stekelenburg2013realizability} on realizability 
categories is close in spirit to the present work in that it also aims towards a 
more abstract understanding of categorical realizability constructions. In 
comparison, Stekelenburg's approach seems to be more `logical', as opposed to the 
`geometric' flavor of the present work which arises from emphasizing the analogy
to Grothendieck toposes. Further analysis may be needed.

\section{Conventions}\label{sec:conventions}

\subsection{Logical notation}

In the present text, we make extensive use of notation in predicate logic, both
informally and as `internal language' in categories and posetal fibrations. We
use the propositional connectives $\bot,\top,\wedge,\vee,\imp$ with the
convention that $\wedge$ binds strongest of the binary connectives, and $\imp$
binds weakest, i.e.\ $\varphi\vee\psi\wedge\gamma\imp\psi$ has to be read
as $(\varphi\vee(\psi\wedge\gamma))\imp\psi$. Quantifiers $\forall,\exists$ have
lowest
precedence, in other words their scope stretches as far to the right as
possible. For example, $\forall x\qdot \varphi\imp\psi$ means $\forall x\qdot(
\varphi\imp\psi)$.

\subsection{Internal logic}

For an introduction to categorical logic we refer to
\cite{jacobs2001categorical} and \cite[Part~D]{elephant2}, here we only review
notational conventions and
shorthands.

When reasoning in the internal logic of a category or a fibration we use the
same symbols as in informal reasoning, but a bit more rigorously. We have three
syntactic classes -- terms\index{term}, formulas\index{formula}, and
judgments\index{judgment}, which we write as follows:
\begin{align*}
 &\text{term\index{term}:} & x_1\vtp A_1\dots x_n\vtp A_n &\csep t\\
&\text{formula\index{formula}:} & x_1\vtp A_1\dots x_n\vtp A_n &\csep \varphi\\
&\text{judgment\index{judgment}:} & x_1\vtp A_1\dots x_n\vtp A_n
&\csep\varphi_1\dots\varphi_n\ent\psi
\end{align*}
Note that all come with explicit contexts $x_1\vtp A_1\dots x_n\vtp
A_n$ of \emph{typed} (or rather `sorted') variables, without which the
expressions are meaningless. Similarly, bound variables always have types.
Since the notation in this style is rather heavy, we will often omit types or
even entire variable contexts -- the types are normally clear from the
context (in the non-technical sense of the word), and if no variable context is
given, we mean by convention simply all free variables that occur in the
expression.

Expressions in the internal language denote semantic entities -- terms denote
morphisms, formulas denote `predicates'\index{predicate}, by which we mean
either monomorphisms (in categories) or objects in a fibration over the
denotation of the context. 
If we assert the validity of a judgment, we mean that a certain
inequality holds in a subobject lattice or the fiber of a fibration (we will
use predicate logic only in posetal fibrations). 

In general, we do not
distinguish between a syntactic expression in the internal language and its
denotation since it is either clear from the context what we mean (if we refer
to a formula as a predicate, then we mean its denotation), or it doesn't matter.

\subsection{Reasoning with partial terms}\label{sec:partial-terms}

 When reasoning with partial terms (for example in pcas), we interpret
function symbols and primitive predicate symbols in the \emph{strict} sense --
for example if we
assert  $s=t$  or $s\in M$ where $s$ and $t$ are partial terms, then this in
particular means that $s$ and $t$ and all of their subterms have to be
defined\footnote{Be aware, however, that we can not deduce that $t$ is defined
from
$\varphi(t)$ for general non-atomic formulas $\varphi$ -- this goes wrong
already for
$\neg (t=t)$.}.
This  spares us from inserting `... then $t$ is defined and ...' in many
definitions and arguments.

We write $t\defined$ for the proposition that a possibly partial term $t$ is
defined; in accordance with our strictness convention this is equivalent to
$t=t$.

Formally, when doing first order logic with partial terms, we use $E^+$-logic
(see~\cite[Chapter~2.2]{troelstra1988constructivism}). 

We use the notations $s\preceq t$ for $s\defined\imp s=t$, and $s\simeq t$ for
$(s\defined\vee t\defined)\imp s=t$; the second relation is also called `strong
equality'.

\subsection{The axiom of choice}\label{sec:choice}

The proofs and developments in this thesis do not rely on the axiom of choice,
unless explicitly said otherwise. The abandonment of
choice leads to a certain proliferation of mathematical concepts, in particular
in category theory -- in the case of finite limit categories, for example, it
makes a difference whether we demand the \emph{existence} of limiting cones for
every finite diagram, or whether we ask each such diagram to come with an
explicit choice of such cone. Similarly, it is sometimes desirable to have an
explicit choice of cartesian liftings in fibrations. As is common practice, we
will generally assume that our categorical structures come with an explicit
choice of whichever structure we postulate.

\chapter{Fibrations}\label{chap:fibrations}
\section{Basic theory}

In the present work, fibrations form the central tool and formalism. We refer
the reader to \cite{streicherfib} for a general introduction, and to
\cite{jacobs2001categorical} and~\cite[Section~B1.3]{elephant1}
for fibrations in categorical logic and topos theory.
Bénabou's original paper~\cite{benabou1985fibered} gives a more
philosophical account.

In the following, we recall some basic theory, mostly without proofs.
\begin{definition}
 Let $\fibc:\catc\to\catb$ be a functor between categories $\catb,\catc$.
\begin{enumerate}
 \item Let $u:J\to I$ in $\catb$ and $f:B\to A$ in $\catc$ such that
$\fibc(f)=u$. We call $f$ 
 \emph{cartesian}\index{cartesian morphism}\index{morphism!cartesian} (with respect to $\fibc$), if for
any $v:K\to J$ in $\catb$
and $g:C\to A$ in $\catc$ such that $\fibc(g)=uv$ there exists a unique
$h:C\to B$ such that $\fibc(h)=v$ and $fh=g$.
\begin{equation}\label{eq:diag-cart-ar}
\vcenter{\xymatrix@R-5mm@C-3mm{
C\dashed[rd]_h\ar[rrrd]^g\\
& B\cart[rr]_f && A\\
K\ar[rd]_v\\
& J\ar[rr]^u && I\\
}}
\end{equation}
\item $\fibc$ is a \emph{fibration}\index{fibration}, if for every $A\in\catc$
and $u:J\to I$ in
$\catb$ such that $\fibc(A)=I$ there exists a \emph{cartesian lifting of $A$
along $u$}\index{cartesian lifting}, by which we mean a cartesian arrow $f:B\to A$ such that
$\fibc(f)=u$.
\end{enumerate}
\end{definition}
The domain of a fibration is called its \emph{total category}\index{total
category}, and the codomain
its \emph{base category}\index{base category}. To avoid having to come up with a
new letter, we often
denote the total category of a fibration $\fibc$ by $\tot{\fibc}$. Given a
fibration $\fibc:\tot{\fibc}\to\catb$, we also say that $\fibc$ is a fibration
\emph{on} $\catb$.
If
$\fibc(A)=I$, or $\fibc(f)=u$, we will say that $A$ is \emph{over} $I$, and that
$f$ is \emph{over} $u$, respectively. Morphisms over identity morphisms in the
base
are called \emph{vertical}\index{vertical morphism}\index{morphism!vertical}.
In diagrams
we represent the `over' relation by vertical alignment, as we already did in
Diagram~\eqref{eq:diag-cart-ar}. Given $I\in\catb$, $\fibc_I$ is
the \emph{fiber of $\fibc$ over $I$}\index{fiber!of fibration}, which is the
subcategory of $\tot{\fibc}$
on objects over $I$ and morphisms over $\id_I$. We
use the arrow symbol $\xymatrix@1@C-2mm{\cart[r]&}$ for cartesian arrows. 

A \emph{posetal}\index{posetal fibration}\index{fibration!posetal} fibration is
a fibration $\fifa:\tot{\fifa}\to\catb$
which is faithful as a functor, which is equivalent to the fact that all fibers
$\fifa_I$ for $I\in\catb$ are preorders. Since posetal fibrations can serve as
models of first order logic, we refer to the objects in $\tot{\fifa}$ as
\emph{predicates}\index{predicate} in this case. Specifically, given
$\varphi\in\fifa_I$, we
call $\varphi$ a \emph{predicate on $I$}.

\begin{definition}\label{def:catfib}
 Let $\fibc:\tot{\fibc}\to\catb$, $\fibd:\tot{\fibd}\to\catb$ be fibrations on
a category $\catb$.
\begin{enumerate}
 \item A \emph{fibered functor}\index{fibered!functor}\index{functor!fibered} is
a functor
$F:\tot{\fibc}\to\tot{\fibd}$ which maps cartesian arrows in $\fibc$ to
cartesian arrows in $\catd$, and such that $\fibd\circ F = \fibc$.
\item Given two fibered functors $F,G:\fibc\to\fibd$, a \emph{fibered natural
transformation}\index{fibered!natural transformation}\index{natural
transformation!fibered} is a natural transformation $\eta:F\to G$ such that all
components $\eta_C$ for $C\in\tot{\fibc}$ are vertical (or equivalently
$\fibd\eta=\id_\fibc$).
\item Fibrations, fibered functors, and fibered natural transformations on
$\catb$ form a 2-category which we denote by $\catfib(\catb)$.
\end{enumerate}
\end{definition}
Since we don't rely on the axiom of choice, it makes a difference whether we
assume mere existence of cartesian liftings or an actual assignment of a
lifting to each pair $(A\in\fibc_I,u:J\to I)$ in a fibration
$\fibc:\tot{\fibc}\to\catb$. Such a choice of cartesian liftings is called a
\emph{cleavage}\index{cleavage}. Unless otherwise specified, we \emph{always}
assume that our fibrations are equipped with cleavages, which we leave implicit
(we do, however, not require fibered functors to preserve chosen cartesian
arrows).
Using cleavages, we can employ a functorial notation for cartesian
liftings: Given $u:J\to I$ and $A\in \fibc_I$, we denote the domain of the
designated cartesian lifting of $A$ along $u$ by $u^*A$, and using the universal
property of cartesian liftings, we can transport vertical maps $(f:A\to
B)\in\fibc_I$ to vertical maps $(u^*f:u^*A\to u^*B)\in\fibc_J$.
\begin{equation}\label{eq:lift-morphism}
 \vcenter{\xymatrix@R-3mm{
u^* A \cart[r]\ar[d]_{u^*f} & A\ar[d]^f\\
u^* B \cart[r] & B\\
J\ar[r]^u & I
}}
\end{equation}
From the universal property of cartesian liftings we can deduce that
\begin{itemize}
 \item squares of the form \eqref{eq:lift-morphism} are always pullbacks,
\item the construction gives rise to a functor $u^*(-):\fibc_I\to
\fibc_J$, and
\item the assignment $u\mapsto u^*(-)$
is functorial up to isomorphism and thus gives rise to a pseudofunctor of type
\[
\catb^\op\to \catcat
\]
which we call the \emph{indexed category}\index{indexed category}\index{category!indexed} associated to
$\fibc$.
\end{itemize}
We can show the
following.
\begin{lemma}\label{lem:indexed-fibered-equivalence}
The transition from a fibration to the associated indexed category gives a
biequivalence of 2-categories
\[
 \catfib(\catb)\simeq [\catb^\op,\catcat],
\]
where $[\catb^\op,\catcat]$ is the 2-category \emph{indexed
categories} on $\catb$, i.e., the 2-category of
pseudofunctors, pseudo-natural transformations, and modifications.
\qed
\end{lemma}
To make this statement precise, we have to say something about relative
sizes of the involved entities. It is easiest to assume that $\catb$ and the
fibrations in $\catfib(\catb)$ are small
relative to some universe, whereas $\catcat$ is the large 2-category of small
categories relative to the same universe. 

The construction in the inverse direction from the one sketched above, i.e.\
the transition from an indexed categoy to a fibration, 
is known as the \emph{Grothendieck
construction}\index{Grothendieck construction} (see
\cite[Definition~B1.3.1]{elephant1}). 
When defining fibrations, in particular
posetal ones, we will often make implicit use of the Grothendieck construction,
and only define the ordering in the fibers and the cartesian liftings. 

\medskip

In the spirit of Bénabou's work~\cite{benabou1985fibered}, we view
fibrations as generalized categories. To justify this point of view, we explain
now how ordinary categories can be viewed as fibrations.
\begin{definition}\label{def:fam-fibration}
 Let $\cata$ be a category.
\begin{enumerate}
 \item $\Famf(\cata)$ is the \emph{category of families}\index{category of families}
 in $\cata$. Its objects
are families $(A_i)_{i\in I}$, where $I$ is a set and $A_i\in\cata$ for $i\in
I$. A morphism from $(A_i)_{i\in I}$ to $(B_j)_{j\in J}$ is a pair
$(u,(f_i)_{i\in I})$ of a function $f:I\to J$ and a family of morphisms
$f_i:A_i\to B_{ui}$.
\item The \emph{family fibration}\index{familiy fibration}\index{fibration!family} of $\cata$ is the
functor
$\famf(\cata):\Famf(\cata)\to\catset$ which sends $(A_i)_{i\in I}$ to $I$, and
$(u,(f_i)_{i\in I})$ to $u$.
\end{enumerate}
\end{definition}
\begin{lemma}
 The assignment $\cata\mapsto\famf(\cata)$ gives rise to a 2-functor
\[
 \famf:\catcat\to\catfib(\catset),
\]
which is a local equivalence.
\end{lemma}
A central theme in the `fibered category' approach to fibrations is to take a
possibly \emph{non-elementary} property of categories, such as small
completeness or local smallness, and to try to express it as an
\emph{elementary}\footnote{`Elementary' here means roughly `first order
axiomatizable', in particular without references to set theoretic `size
conditions'.}
 property of the corresponding family fibration. A concept
that fits in this pattern is `having internal sums'
in Definition~\ref{def:internal-sums}-\ref{def:internal-sums-sums}; for
systematic treatments of this point of view see~\cite{benabou1985fibered,
streicherfib}.

\subsection{Fibrations from (typed) pcas}

To give examples of fibrations which are not given by they family
construction, and since they are of central interest for this work, let us 
explain how to obtain posetal fibrations from (typed) pcas. 
The fibrations that we will now introduce are \emph{not} the realizability
triposes known from \cite{hjp80} and many subsequent works (we will present
those in Definition~\ref{def:realizability-tripos}), but come from a more
primitive construction whose importance was apparently first realized by
Hofstra~\cite{hofstra2006all}.

Even though we are
not dealing with family fibrations in the previously defined sense, we are using
similar notation and terminology. We do this in the hope that it will lead the
intuition of the reader in the intended direction, which is to view pcas (or
rather the associated uniform preorders -- see
Example~\ref{ex:uords}-\ref{ex:uords-pca}) as \emph{generalized
preorders}. 
\begin{definition}\label{def:fam-pca}
 Let $\pcaia$ be a typed pca (Definition~\ref{def:typed-pca}). The
\emph{uniform family fibration}\index{uniform family fibration!of a (typed) pca}\index{fibration!uniform family}
\[
\ufam\pcaia:\Ufam\pcaia\to\catset 
\]
of $\pcaia$ is the posetal fibration defined as follows.
\begin{itemize}
 \item Predicates on $M\in \catset$ are pairs $(i\in I,\varphi:M\to \pcaa_i)$.
\item The ordering on $\ufam\pcaia_M$ is defined by 
\[(i,\varphi)\leq(j,\psi)\qtext{iff} \exists e\vtp \pcaa_{i\imp
j}\;\forall m\qdot \;e\appca \varphi(m)=\psi(m).
\]
\item Reindexing is given by precomposition.
 \end{itemize}
\end{definition}
We see that on a
literal level the analogy to family fibrations makes sense, since the
predicates in $\ufam\pcaia$ really \emph{are} families.

There is an obvious untyped analogue of the previous definition, which
associates to each pca $\pcaa$ a posetal fibration
\[
 \ufam(\pcaa):\Ufam(\pcaa)\to\catset.
\]

\subsection{Finite-limit fibrations}\label{sec:fl-fibs}

Let us recall some basic facts about fibrations of finite limit categories
from~\cite[Section~8]{streicherfib}.
\begin{definition}\label{def:fl-fibs}
 A \emph{fibration of finite limit categories}, or \emph{finite limit
fibration}\index{finite limit fibration}\index{fibration!finite limit} on a base category $\catb$ is a
fibration
$\fibc:\tot{\fibc}\to\catb$ where
\begin{itemize}
 \item all fibers $\fibc_I$ for $I\in \catb$ have finite limits, and
\item reindexing preserves finite limits.
\end{itemize}
\end{definition}
The following fact about finite limit fibrations is of central importance.
\begin{lemma}
Let $\fibc:\tot{\fibc}\to\catc$ be a fibration on a category $\catc$ with
finite limits. Then $\fibc$ is a finite limit fibration iff $\tot{\fibc}$ has
finite limits and $\fibc$ preserves them.
\end{lemma}
\begin{proof}
See \cite[Theorem~8.5]{streicherfib}.
\end{proof}
This lemma highlights that we have to distinguish two kinds of limits in a
finite limit fibration (as long as the base has finite limits, what we will
always assume from now on) -- the limits in the fibers (which we
sometimes call `fiberwise'), and the `global' limits in
the total category. Furthermore, the lemma explains how they are connected:
\begin{itemize}
 \item Since $\fibc$ maps global limits to limits in the base, global limiting
cones on vertical connected diagrams can be chosen vertical as well. This implies 
in particular that fiberwise connected limits are also limits in the total category.
\item The fact that fiberwise connected limits are global limits implies that 
monomorphisms in the fibers are also monic in the total category, since monos
can be characterized in terms of pullbacks.
\item In general, we can express global limits in terms of limits in the fibers
and in the base. For example, given a cospan $B\xrightarrow{f}A\xleftarrow{g}C$
in $\tot{\fibc}$ over a cospan $J\xrightarrow{u}I\xleftarrow{v}K$ in the base,
we can take the pullback of $f$ and $g$ by first reindexing the cospan
$B\xrightarrow{f}A\xleftarrow{g}C$ into the fiber over the pullback of $u$ and
$v$, and then taking the fiberwise pullback.
\[
\xyunfrenchcatcodes
{
\vcenter{\xymatrix@R-7mm{
& B\times_A C\ar[rrdd]\ar@/^.8mm/[rddd]\\
\\
& \emar[l]|(.3){\textstyle \bullet}="1"\emar[r]|(.3){\textstyle \bullet}="2"& &
C\ar[dd]\ar[dddr]^g\\
& & B\ar[dd]\ar@/_2.5mm/[ddrr]|f & &\\
& \bullet\cart[rr]|(.57)\hole|(.8)\hole\cart[rd]& & \bullet\cart[rd]\\
& & \bullet\cart[rr] & & A\\
& J\times_I K\ar[rr]\ar[rd]& & K\ar[rd]^{v}\\
& & J\ar[rr]^{u} & & I\\
\ar"1,2";"1"*+\frm{}
\ar"1,2";"2"*+\frm{}
\ar"1"*+\frm{};"5,2"
\ar"2"*+\frm{};"5,2"|(.25)\hole
\cart"2"*+\frm{};"3,4"|(.25)\hole
\cart"1"*+\frm{};"4,3"
}}}
\]
\end{itemize}
Having both fiberwise and global limits in finite limit fibrations, we have to
be careful to distinguish them notationally. In general we use global 
limits unless saying otherwise explicitly. For example, given
$C,D\in\tot{\fibc}$, $C\times D$ means their global product. For the
fiberwise product of $C,D\in\fibc_I$, on the other hand, we write
$C\times_I D$, in analogy to the common notation for pullbacks. More generally,
given a cospan $J\xrightarrow{u}I\xleftarrow{v}K$ in $\catc$, and $B\in\fibc_J$,
$C\in\fibc_K$, we write $B\times_I C$ for the product of $B$ and $C$ relative
to $u$ and $v$, in the sense of the following diagram.
\[
\xyunfrenchcatcodes
\vcenter{\xymatrix@R-5mm{
& B\times_I C\ar[rrdd]^q\ar@/^.8mm/[rddd]^p\\
\\
& \emar[l]|(.3){\textstyle \bullet}="1"\emar[r]|(.3){\textstyle \bullet}="2"& &
C\\
& & B & &\\
& J\times_I K\ar[rr]\ar[rd]& & K\ar[rd]^{v}\\
& & J\ar[rr]^{u} & & I\\
\ar"1,2";"1"*+\frm{}
\ar"1,2";"2"*+\frm{}
\cart"2"*+\frm{};"3,4"|(.25)\hole
\cart"1"*+\frm{};"4,3"
}},
\]
Formally, this product is characterized as terminal among cones
$J\xleftarrow{p}\bullet\xrightarrow{q}K$ in $\tot{\fibc}$ such that
$u\fibc(p)=v\fibc(q)$; it can equivalently be described
as fiberproduct $B\times_{1_I}C$ in $\tot{\fibc}$.
\begin{remarks}\label{rem:flfibs}
 \begin{itemize}
  \item In posetal fibrations, we refer to fiberwise finite limits as
\emph{finite
meets}, and we use the symbols $\top$ and $\wedge$ instead of $1$ and $\times$.
\item The uniform family fibrations $\ufam(\pcaa)$ and $\ufam\pcaia$ of
(typed) pcas $\pcaa$ and $\pcaia$ have finite meets -- in the typed case, if
$\varphi:M\to \pcaa_i$ and $\psi:M\to \pcaa_j$ are predicates in $\ufam\pcaia$,
a greatest lower bound is the predicate 
$\varphi\wedge\psi:M\to \pcaa_{i\ptype j}$ given by 
$m\mapsto
\compair\appca\varphi(m)\appca\psi(m)$. In the untyped case, we can derive the
existence of a pairing combinator from \emph{functional completeness}\index{functional completeness} -- see
e.g.\ \cite[Section~1.1.1]{vanoosten2008realizability}.
 \end{itemize}
\end{remarks}

\subsection{Localization and slicing}

\begin{definition}\label{def:slice-functor}
Given a functor $F:\cata\to\catb$ and $A\in\cata$ we define the \emph{slice
functor\index{slice functor}\index{functor!slice} $F/A:\cata/A\to\catb/FA$ of $F$ over $A$} by
\[
 (f:X\to A)\mapsto (Ff:FX\to FA)
\]
and in the obvious way for morphisms.
\end{definition}
The following is easy to show.
\begin{lemma}\label{lem:slice-fibration}
Given a fibration $\fibc:\tot{\fibc}\to\catb$, and $C\in\fibc_I$,
$\fibc/C:\tot{\fibc}/C\to\catb/I$ is a fibration as well. 
If $\fibc$ is a
finite limit fibration, then so is $\fibc/C$.
\qed
\end{lemma}

There is a construction that is somewhat similar to slicing of fibrations
(see~\cite[Section~4]{streicherfib}):
\begin{definition}\label{def:localization}
 Given a fibration $\fibc:\tot{\fibc}\to\catc$ and $I\in B$, we define the
\emph{localization $\fibc/I$ of $\fibc$ to $I$}\index{localization!of a fibration} by the
pullback
\[
\vcenter{\xymatrix@R-3mm{
\tot{\fibc/I}\ar[r]\pullbackcorner\ar[d]_{\fibc/I} & \tot{\fibc}\ar[d]^\fibc\\
\catb/I\ar[r] & \catb
}}.
\]
\end{definition}
\begin{lemma}\label{lem:local-terminal-slice}
 If $\fibc:\tot{\fibc}\to\catb$ has terminal objects in the fibers which are
stable under reindexing\footnote{Such a fibration is called a \emph{fibration
of categories with terminal objects} in \cite[Definition~8.2]{streicherfib}.},
then localization is a special case of slicing. More precisely, given
$I\in\catb$, the localization $\fibc/I$ is equivalent to the slice fibration
$\fibc/1_I$. \qed
\end{lemma}

To develop an intuition, the best way to think about fibrations resulting from
slicing and localization is as `fibered fibrations' -- by `fibered fibration'
we mean a fibration on the total category of another fibration (see
also the remarks at the beginning of Section~\ref{sec:arbitrary-bases}).

Given an object $I\in\catc$ of a finite limit category, the projection
$\catc/I\to\catc$ is a fibration --
it is the externalization (see \cite[Section~4]{streicherfib}) of the discrete
internal category with $I$ as set of
objects. Now a fibration on a discrete category is the same thing as a family
of categories -- analogously we may regard a fibration on a discrete fibration
on $\catc$ as a family of fibrations on $\catc$. From this point of view, the
localization of $\fibc$ at $I$ as a \emph{fibered} 
fibration\index{fibered!fibration}
\[
 \tot{\fibc/I}\xrightarrow{\fibc/I}\catc/I\to\catc
\]
can be viewed as $I$-indexed family of fibrations on $\catc$ of value constant
$\fibc$, and for $C\in\fibc/I$, the fibered fibration
\[
 \tot{\fibc}/C\xrightarrow{\fibc/C}\catc/I\to\catc
\]
can be viewed as $I$-indexed family of fibrations whose components are the
(ordinary) slices of $\fibc$ (viewed as generalized category) over the
components of $C$ (viewed as $I$-indexed family of objects in the generalized
category).

This intuition is a helpful guideline to understand for example which
properties of fibrations are preserved by localization and slicing -- since
localization corresponds simply to taking many copies of the same fibration, we
can expect it to preserve all 
properties of fibrations that correspond to `reasonable' properties of
categories, such as having small (co)limits, local smallness, and
well-poweredness. With slicing, we have to be a bit more careful -- just as for
ordinary categories, slicing of fibrations preserves the existence of finite
limits, as well as the existence of (finite/small) coproducts, but not of
products (neither finite nor small).

It seems that all classes of fibrations that we consider in this work are stable
under localization and slicing -- we give explicit proofs whenever we actually
need this.

\section{Internal sums}

Internal sums in a fibration are an abstraction of infinite coproducts in
categories, and of infinite joins in preorders. In this section, we introduce
the general concept and then devote closer attention to two special cases --
internal sums in posetal fibrations, and \emph{extensive} sums.

We recall the definition of internal sums in fibrations from
\cite[Sections 6]{streicherfib}.
\begin{definition}\label{def:internal-sums}
Let $\fibc:\tot{\fibc}\to\catc$ be a fibration on a finite limit category
$\catc$.
\begin{enumerate}
 \item\label{def:internal-sums-sums} We say that $\fibc$ \emph{has internal
sums}\index{internal sums} if 
\begin{enumerate}
 \item in addition to being a fibration, $\fibc$ is an \emph{opfibration}\index{opfibration} in the
sense that $\fibc^\op:\tot{\fibc}^\op\to\catc^\op$ is a fibration -- in this
case we call the cartesian arrows in $\fibc^\op$ \emph{cocartesian}\index{cocartesian morphism}\index{morphism!cocartesian} in
$\fibc$, and we use the arrow symbol $\xymatrix@1{\coca[r]&}$ for them --, and
\item the \emph{Beck-Chevalley condition (BCC)}\index{Beck-Chevalley condition} holds: cocartesian maps in
$\tot{\fibc}$ are stable under pullbacks along cartesian maps\footnote{For this
statement to make sense we don't have to assume that $\fibc$ is a finite limit
fibration -- for pullbacks along cartesian maps to exist we only need pullbacks
in the base.}.
\end{enumerate}
\item\label{def:internal-sums-stable}  If $\fibc$ is a finite limit fibration
with internal sums,
then we say that 
internal sums in $\fibc$ are \emph{stable}\index{internal sums!stable}, if
cocartesian maps in $\tot{\fibc}$
are stable under pullback along \emph{arbitrary} maps.
\end{enumerate}
\end{definition}
If $\cata$ is a category, then $\famf(\cata):\Famf(\cata)\to\catset$ has
internal
sums iff $\cata$ has small coproducts. Internal sums in $\famf(\cata)$ are
stable iff the same is true for the coproducts in $\cata$.

Let us now consider the posetal case.

\subsection{Existential fibrations and partial equivalence relations}

\begin{definition}\label{def:existential-fibration}
 An \emph{existential fibration}\index{existential fibration}\index{fibration!existential} is a posetal
fibration
$\fifx:\tot{\fifx}\to\catc$ on a finite limit category $\catc$ which has finite
meets and stable internal sums.
\end{definition}
In an existential fibration, we can soundly interpret the $\top,\wedge,\exists$
fragment of first order logic, where conjunctions are interpreted by meets in
the fibers; and existential quantification is interpreted by internal sums (see
\cite[Section~4.2]{jacobs2001categorical}). Because of this correspondence, we
normally refer to internal sums in existential fibrations as
\emph{existential quantification}\footnote{In particular, we
always assume the Beck-Chevalley condition when speaking
about existential quantification.}\index{existential quantification}, and we
denote existential quantification of a predicate $\psi\in\fifx_J$ along a
morphism $u:J\to I$ in $\catr$ as $\exists_u\psi$. Stability is known as
\emph{Frobenius law}\index{Frobenius law} in the posetal context, where it is
normally expressed as
\[
 \varphi\wedge\exists_u\psi\cong\exists_u u^*\varphi\wedge\psi\qquad\text{for
$\varphi\in \fifx_I$ and $\psi\in\fifx_J$.}
\]
\begin{example}
 The subobject fibration $\sub(\catr)$ of a regular category $\catr$ is an
existential fibration.
\end{example}
A very useful construction on existential fibrations is the category of partial
equivalence relations.
\begin{definition}\label{def:catper}
Let $\fifx:\tot{\fifx}\to\catc$ be an existential fibration. The category
$\per(\fifx)$\footnote{
The construction of $\per(\fifx)$ will turn out to be the same as the
construction of $\catc[\fifx]$ that we consider in
Section~\ref{sec:fibered-sheaves}, but for bootstrapping reasons, we use a
different notation here.} is defined as follows.
\begin{itemize}
  \item Objects are pairs $\cro$ of an object $C\in\catc$ and a binary
    predicate $\rho\in\fifx_{C\times C}$ which is a \emph{partial equivalence
      relation}\index{partial equivalence relation}\index{equivalence relation!partial}, i.e.\ the judgments
\begin{align*}
&  \text{(symm)} & \rho(c,d)&\ent\rho(d,c)\\
&\text{(trans)} & \rho(c,d),\rho(d,e)&\ent\rho(c,e)
\end{align*}
hold.
\item Morphisms from $\cro$ to $\dsi$ are equivalence classes of binary
  predicates $\phi\in\fifx_{C\times D}$ which are functional and total in a
  sense relative to $\rho$ and $\sigma$ -- more precisely the judgments
\begin{align*}
&  \text{(strict)} & \phi(c,d)&\ent\rho(c)\wedge\sigma(d)\\
& \text{(cong)} & \phi(c,d),\rho(c,c'),\sigma(d,d')&\ent\phi(c',d')\\
& \text{(singval)} & \phi(c,d),\phi(c,d')&\ent\sigma(d,d')\\
& \text{(tot)} & \rho(c)&\ent\exists d\qdot\phi(c,d)
\end{align*}
hold\footnote{For a partial equivalence relation $\rho$, we often use $\rho(c)$
  as an abbreviation for $\rho(c,c)$ -- the `definedness' part of the
  relation.}. Two such predicates $\phi,\phi'\in\fifx_{C\times D}$ are
identified as morphisms in $\per(\fifx)$ if they are logically equivalent,
i.e.\ $\phi(c,d)\adj\ent\phi'(c,d)$ holds.
\item 
Composition of is given by relational
composition, i.e.\ $(\gamma\circ\phi)(c,e)\equiv\exists d\qdot
\phi(c,d)\wedge\gamma(d,e)$. 
\end{itemize}
\end{definition}
Associativity of composition follows from the Frobenius law, and it is easy to
see that an identity morphism for $\cro$ is given by $\rho$ itself, thus
$\per(\fifx)$ is
really a category.
In the following, we want to show that $\per(\fifx)$ is an exact category for
any existential fibration $\fifx$. This fact can probably be considered
folklore; we
will give a detailed proof since the occurring constructions will be important
later.

\medskip

The following lemma is easy to show.
\begin{lemma}\label{lem-inj-surj}
  Let $\fifx:\tot{\fifx}\to\catc$ be an existential fibration, and let 
  $\phi:\cro\to\dsi$ be a morphism in $\per(\fifx)$.
  \begin{itemize}
  \item If the judgment
\begin{equation}
 \sigma(d)\ent\exists
c\qdot\phi(c,d)\tag{inj}\label{eq:judg-inj}
\end{equation}
holds, then $\phi$ is a monomorphism.
  \item If the judgment 
\begin{equation}
\phi(c,d),\phi(c',d)\ent\rho(c,c')\tag{surj}\label{eq:judg-surj}
\end{equation}
holds, then $\phi$ is an cover\index{cover}, i.e.\
left
    orthogonal (Definition~\ref{def:facsys}-\ref{def:facsys-orth}) to all
monomorphisms.
  \item $\phi$ is an isomorphism iff \eqref{eq:judg-inj} and
\eqref{eq:judg-surj} both hold.
  \end{itemize}
  \qed
\end{lemma}
Given an arbitrary morphism $\phi:\cro\to\dsi$ in $\per(\fifx)$, we have a
decomposition
\begin{equation}\label{eq:decomposition}
\xymatrix@1{
\cro\depi[r]^\pi&(C,\pi)\ar[r]_\cong^{\overline{\phi}}&(D,\sigma|_\upsilon)\mono
[r]^{\sigma|_\upsilon}&\dsi
}
\end{equation}
into an cover, an isomorphism, and a monomorphism. Here,
\begin{align*}
  \pi(c,c')&\equiv \rho(c)\wedge\rho(c')\wedge\exists d\qdot
  \phi(c,d)\wedge\phi(c',d)\\
\overline{\phi}(c,d)&\equiv\phi(c,d)\\
\upsilon(d)&\equiv\exists c\qdot \phi(c,d)\\
\sigma|_\upsilon(d,d')&\equiv\sigma(d,d')\wedge\upsilon(d).
\end{align*}
The predicate $\upsilon$ in this decomposition has a special relation to
$\sigma$, which (following~\cite{vanoosten2008realizability}) we call
\emph{strictness}:
\begin{definition}\label{def:strict}
 Let $\fifx:\tot{\fifx}\to\catc$ be an existential fibration, and let
$\cro\in\per(\fifx)$. We call $\varphi\in\fifx_C$ \emph{strict}\index{strict!predicate}\index{predicate!strict} with respect to $\rho$, if the judgments $\varphi(x)\ent\rho(x)$ and
$\varphi(x),\rho(x,y)\ent\varphi(y)$ hold in $\fifx$.
\end{definition}
The exactness of $\per(\fifx)$ is now shown as follows.
\begin{lemma}\label{lem:decompo}
Let $\fifx:\tot{\fifx}\to\catc$ be an existential fibration.
  \begin{enumerate}
  \item\label{lem:decompo-monorepr} 
For $\dsi\in\per(\fifx)$, subobjects of $\dsi$ correspond to predicates in
$\fifx_D$ which are strict with respect to $\dsi$. More precisely, monomorphisms
into $\dsi$ can up to isomorphism be represented as
    $\sigma|_\upsilon:(D,\sigma|_\upsilon)\emono\dsi$ where $\upsilon\in
    \fifx_D$ is strict with respect to $\sigma$.
  \item\label{lem:decompo-flims} $\per(\fifx)$ has finite limits.
  \item\label{lem:decompo-inj} If $\phi:\cro\to\dsi$ is a monomorphism,
then \textup{(inj)} holds.
  \item\label{lem:decompo-surj} If $\phi:\cro\to\dsi$ is a regular epimorphism,
then \textup{(surj)} holds.
  \item\label{lem:decompo-reg} $\per(\fifx)$ is regular.
  \item\label{lem:decompo-coverrepr} Covers with domain $\cro$ can up to
isomorphism be represented as
    $\pi:\cro\eepi(C,\pi)$ where $\pi\in\fifx_{C\times C}$ is a partial
    equivalence relation satisfying $\rho(c,c')\ent\pi(c,c')$ and
    $\pi(c)\ent\rho(c)$.
  \item\label{lem:decompo-ex} $\per(\fifx)$ is exact.
  \end{enumerate}
\end{lemma}
\begin{proof}
  \emph{Ad \ref{lem:decompo-monorepr}.} If we apply the
decomposition~\eqref{eq:decomposition} to a monomorphism, then by
  orthogonality the cover part will be an isomorphism. It is easy to see that
  the resulting $\upsilon$ is strict with respect to $\sigma$.

  \emph{Ad \ref{lem:decompo-flims}.}  A product of $\cro,\dsi$ is given by
$(C\times
  D,\rho\brprod\sigma)$ where
  \[\rho\brprod\sigma(c,d,c',d')\equiv\rho(c,c')\wedge\sigma(d,d'),\]
 $(1,\top)$
  is a terminal object, and an equalizer of $\phi,\gamma:\cro\to\dsi$ is given
  by the subobject of $\cro$ corresponding to the predicate
  $\ilbracks{c\csep\exists d\qdot \phi(c,d)\wedge\gamma(c,d)}$.

  \emph{Ad \ref{lem:decompo-inj}.} The kernel of $\phi$ is represented by the
predicate
  $\ilbracks{c,c'\csep \exists d\qdot \phi(c,d)\wedge \phi(c',d)}$, and the
diagonal
  subobject $\delta:\cro\to\cro\times\cro$ is represented by $\rho$ itself. The
  claim follows since the kernel coincides with the diagonal for monomorphisms.

  \emph{Ad \ref{lem:decompo-surj}.} The predicate $\ilbracks{d\csep\exists
c\qdot \phi(c,d)}$ represents a
  subobject of $\dsi$ through which $\phi$ factors. If $\phi$ is a cover, then
  this has to be the maximal subobject.

  \emph{Ad \ref{lem:decompo-reg}.} We already know that $\per(\fifx)$ has finite
limits and
  cover/mono factorizations. It remains to show that covers are stable under
  pullback. Using the previous construction of finite limits and
  characterization of covers, this is easy to verify.

  \emph{Ad \ref{lem:decompo-coverrepr}.} This follows from the factorization and
orthogonality.

  \emph{Ad \ref{lem:decompo-ex}.} We have to show that equivalence relations are
effective, i.e.\
  appear as kernel pairs. This follows since the binary predicates
  $\tau\in\fifx_{C\times C}$ representing equivalence relations on $\cro$
  coincide exactly with those binary predicates representing quotients of
  $\cro$ as in~\ref{lem:decompo-coverrepr}.
\end{proof}

\subsection{Extensive fibrations and Moens' theorem}

We recall the definition of extensivity for internal sums from
\cite[Section 15]{streicherfib}.
\begin{definition}\label{def:lextensive-fibs}
\begin{enumerate}
\item\label{def:lextensive-fibs-disjoint} If $\fibc:\tot{\fibc}\to\catc$ is a
finite limit fibration with internal
sums,
then we say that internal sums in $\fibc$ are \emph{disjoint}\index{internal
sums!disjoint}\index{disjoint internal sums}, if for any cocartesian map
$s:\xymatrix@1@-1.5mm{A\coca[r] & B}$ in $\tot{\fibc}$, 
the canonical map $\delta:A\to A\times_B A$ is also cocartesian.
\item We call internal sums in a finite limit fibration
\emph{extensive}\index{internal sums!extensive}\index{extensive internal sums}, if
they are stable and disjoint. A finite limit fibration with extensive internal
sums is also called a
\emph{lextensive}\index{lextensive fibration}\index{fibration!lextensive} fibration.
\item
$\catlxv(\catc)$ is the 2-category of lextensive fibrations on $\catc$. Its
1-cells are fibered functors preserving finite limits and internal sums, and
its 2-cells are fibered natural transformations.
\end{enumerate}
\end{definition}
The following lemma describes two ways to construct lextensive fibrations.
\begin{lemma}\label{lem:fund-glu}
\begin{enumerate}
 \item \label{lem:fund-glu-fund}
 Let $\catc$ be a finite limit category. Then the functor
\[
 \lfund{\catc},
\]
which sends every morphism to its codomain, is a lextensive fibration which we
call the \emph{fundamental fibration}\index{fundamental
fibration}\index{fibration!fundamental} of $\catc$ (the fundamental fibration
is also known as \emph{codomain fibration}\index{codomain fibration}\index{fibration!codomain}, e.g.\ in
\cite{jacobs2001categorical, vanoosten2008realizability}).
\item If $\Delta :\catc\to\catd$ is a finite limit preserving functor between
finite limit categories, and $\fibc$ is a lextensive fibration on $\catd$, then
the (strict) pullback $\Delta ^*\fibc$ of $\fibc$ along $\Delta $ is a
lextensive fibration on $\catc$.
\begin{equation*}
 \vcenter{\xymatrix{
\tot{\Delta ^*\fibc}\pullbackcorner\ar[r]\ar[d]_{\Delta ^*\fibc} &
\tot{\fibc}\ar[d]^\fibc\\
\catc\ar[r]^\Delta &\catd
}}
\end{equation*}
\end{enumerate}
\qed
\end{lemma}
\begin{remark}\label{rem:subobject-fibration}
 The fundamental fibration of a finite limit category $\catc$ has an important
subfibration -- the \emph{subobject fibration}\index{subobject fibration}\index{fibration!subobject}
\[
 \sub(\catc):\Sub(\catc)\to\catc,
\]
which is defined as the full subfibration of $\fund{\catc}$ on those objects in
$\commacat{\catc}{\catc}$ which are monomorphisms in $\catc$.
\end{remark}
The combination of  pullback and
fundamental fibration is known as the \emph{gluing construction}.
\begin{definition}\label{def:gluing}
 Let $\Delta :\catc\to\catd$ be a finite limit preserving functor. 
The \emph{gluing\index{gluing construction} of $\catd$ along 
$\Delta $}, denoted by 
\[\gl_\Delta(\catd):\Gl_\Delta (\catd)\to\catc,\] is the fibration obtained by 
pulling back $\fund{\catd}$ along $\Delta$.
\begin{equation}\label{eq:gluing}
 \vcenter{\xymatrix{
\Gl_\Delta (\catd)\pullbackcorner\ar[r]\ar[d]_{\gl_\Delta (\catd)}
& \commacat{\catd}{\catd}\ar[d]^{\fund{\catd}}\\
\catc\ar[r]^\Delta &\catd
}}
\end{equation}
\end{definition}
Concretely, the total category $\Gl_\Delta (\catd)$ of the gluing fibration is
the comma category $\commacat{\catd}{\Delta }$, and the functor $\gl_\Delta
(\catd)$ is the evident projection. Streicher~\cite{streicherfib} denotes the
gluing along a functor simply by $\gl(\Delta )$, but since essentially all of
the functors
of which we take the gluing are called $\Delta$, I chose a more informative
notation.

\medskip

Moens \cite{moens1982characterization} observed that up to equivalence, all
lextensive fibrations are obtained by gluing, which can be expressed as a
biequivalence of 2-categories.
\begin{theorem}[Moens' theorem]\label{theo:moens}
  Given a finite limit category $\catc$, the assignment
\[(\Delta:\catc\to\catd)\mapsto\gl_\Delta(\catd)\]
from Diagram~\eqref{eq:gluing} gives rise to a biequivalence
\begin{equation}\label{eq:biequ-moens-strong}
\catc\pslice\catlex \simeq \catlxv(\catc),
\end{equation}
where $\catlex$ is the 2-category of finite limit categories, finite limit
preserving functors and natural transformations, and $\catc\pslice\catlex$ is
the
pseudo-co-slice\footnote{The `pseudo' here means that the triangles in the
  definition of morphism commute only up to specified isomorphism.} 2-category
of $\catlex$ under $\catc$.
\end{theorem}
\begin{proof} 
Given a lextensive fibration $\fibe:\tot{\fibe}\to\catc$, the associated
functor is given by 
\[
\Delta_\fibe:\catc\to\fibe_1,\quad I\mapsto \Sigma_I 1_I,
\]
and given a fibered functor $F:\fibe\to\fibf$ between extensive fibrations which
preserves finite limits and internal sums, it follows from the preservation of
internal sums that $F_1\circ \Delta_\fibe\cong \Delta_\fibf$.

We refer to \cite[Section~15]{streicherfib} for details.
\end{proof}
As a first application, we can deduce that the fundamental fibration
$\lfund{\catc}$ is bi-initial in $\catlxv(\catc)$ since $\id_\catc$ is
bi-initial in $\catc\pslice\catlex$. This means that for every
lextensive fibration $\fibe:\tot{\fibe}\to\catc$ there exists a unique (up to
unique equivalence) lextensive fibered functor
\begin{equation}\label{eq:def-fibered-delta}
 \Delta:\fund{\catc}\to\fibe,
 \end{equation}
which we call $\Delta$ since it is the fibered analogue of the functor
$\Delta:\catc\to\fibe_1$ defined in the proof above. Concretely, $\Delta$ is
given by 
\[
 \fund{\catc}_I\ni(f:J\to I)\;\mapsto\; \Sigma_f 1_J\in\fibe_I.
\]
\begin{remark}\label{rem:generalized-moens}
 The biequivalence~\eqref{eq:biequ-moens-strong} can be generalized to a more
general class of functors: given a fibered functor
\[
F:\fibe\to\fibf
\]
between extensive fibrations $\fibe,\fibf$ which preserves finite limits but
not necessarily internal sums, the isomorphism
$F_1\circ\Delta_\fibe\cong\Delta_\fibc$ becomes replaced by a natural
transformation of type $\Delta_\fibf\to F_1\circ\Delta_\fibe$
\begin{equation}\label{eq:lax-triangle}
\vcenter{\xymatrix{
\catc\ar[d]_{\Delta_\fibe}\ar[dr]^{\Delta_\fibf} & \emar[dl]|(.7){\Downarrow}\\
\fibe_1\ar[r]_{F_1} & \fibf_1
} }.
\end{equation}
Triangles of the form~\eqref{eq:lax-triangle} form a 2-category
$\olaco{\catc}{\catlex}$ generalizing the pseudo-co-slice 2-category
$\catc\pslice\catlex$ (which one may call `oplax co-slice 2-category'), and one
can show that $\olaco{\catc}{\catlex}$ is equivalent to the 2-category of
extensive fibrations and finite limit preserving fibered functors.
This is relevant in Section~\ref{sec:glob-secs} where we treat `global sections
functors'.
\end{remark}

\section{Regular pre-stacks and stacks}\label{sec:prestacks}

In this section, we introduce the classes of fibrations that we will use in
Chapter~\ref{chap:fib-cocompletions} to give a framework for (pre)sheaf
  constructions and realizability. In this context, we will always require our
fibrations to be pre-stacks on regular base categories. We have several reasons
for
insisting on the pre-stack condition.
\begin{itemize}
 \item In his thesis~\cite{pitts81}, 
Pitts showed how iterated tripos constructions can be `composed',
provided the corresponding constant object functors are \emph{regular}. This
regularity requirement is equivalent to the pre-stack condition for the
corresponding gluing fibrations (or triposes).
\item In Chapter~\ref{chap:ufp} we study \emph{uniform preorders}, which are
representations of fibered preorders. It turns out that a fibered preorder on
$\catset$ can be represented by a uniform preorder iff it is a pre-stack and has
a generic family of predicates (see Lemma~\ref{lem:reconstuct-ufp}).
\item Robinson and Rosolini~\cite{robinson1990colimit} and
Carboni~\cite{carboni1995some} give construction of realizability toposes and
presheaf toposes using exact completion. These constructions rely on the axiom
of choice. The \emph{fibered presheaf construction} (Section~\ref{sec:fpc})
captures these relationships without choice, giving a universal
characterization in terms of a biadjunction between 2-categories of
{pre-stacks}. 
\end{itemize}

\subsection{Definition and basic properties}

In the present work, \emph{(pre-)stack} always means `(pre-)stack for the
regular topology on a regular category'.

We refer to \cite[Definition~4.6]{vistoli2004notes} for the general definition
of what it means for a fibration to be a (pre-)stack for a Grothendieck
topology,
and to \cite[Example~A2.1.11(a)]{elephant1} for the definition of the regular
Grothendieck topology.
In the following we give the instantiated definition of (pre-)stack for the
regular topology.

\medskip

 Let $\fibc:\tot{\fibc}\to\catr$ be a fibration on a regular category $\catr$,
and let $e:J\eepi I$ be a regular epimorphism in $\catr$. We form the diagram
\begin{equation}\label{eq:truncated-complex}
\xymatrix@C+5mm{
J\times_I J\times_I
J{\ar@<4.5pt>[r]^-{\partial_{0},\partial_{1},\partial_{2}}\ar[r]\ar@<-4.5pt>[r]}
& 
J\times_I J
\ar@<3pt>[r]^-{\partial_0,\partial_1}\ar@<-3pt>[r]& 
J\depi[r]^e & I
}
\end{equation}
of products and projection maps in $\catr/I$. We employ notation
coming from simplicial sets and write $\partial_i$ for the projection which
omits the $i$-th component (see
e.g.~\cite[Chapter~I-1]{goerss2009simplicial}).
In particular, 
the \emph{simplicial
identities}~\cite[Equation~I-(1.3)]{goerss2009simplicial}\index{simplicial
identities}
\begin{equation}\label{eq:simpl-id}
 \partial_0\partial_1=\partial_0\partial_0\qquad\partial_0\partial_2=\partial_1
\partial_0\qquad\partial_1\partial_2=\partial_1\partial_1
\end{equation}
are satisfied.
\begin{definition}\label{def:cat-descent-data}
The \emph{category $\ddesc(\fibc,e)$ of descent data\index{category!of descent
data}\index{descent data} over $e:J\eepi I$} is defined as follows.
\begin{itemize}
 \item An \emph{object with descent data}\index{object!with descent data} $((A_i),(p_i),(q_i))$ is a
configuration
\[
\xymatrix@C+5mm{
A_3
{\cart@<4.5pt>[r]^-{q_{0},q_{1},q_{2}}\cart[r]
\cart@<-4.5pt>[r]}
& 
A_2
\cart@<3pt>[r]^-{p_0,p_1}\cart@<-3pt>[r]& A_1
\\
J\!\times_I\! J\!\times_I\!
J{\ar@<4.5pt>[r]^-{\partial_{0},\partial_{1},\partial_{2}}\ar[r]\ar@<-4.5pt>[r]}
& 
J\!\times_I\! J
\ar@<3pt>[r]^-{\partial_0,\partial_1}\ar@<-3pt>[r]& 
J\depi[r]^e & I
}
\]
of objects and cartesian arrows in $\tot{\fibc}$ over the truncated
complex~\eqref{eq:truncated-complex} satisfying the same simplicial identities
\begin{equation*}
 p_0q_1=p_0q_0\qquad p_0q_2=p_1
q_0\qquad p_1q_2=p_1q_1.
\end{equation*}
\item A \emph{morphism} between objects with descent data 
\[
 (f_i):((A_i),(p_i),(q_i))\to((B_i),(r_i),(s_i))
\]
is a family of \emph{vertical} maps $f_i:A_i\to B_i$
\[
\xymatrix@C+5mm{
A_3\ar[d]_{f_3}
{\cart@<4.5pt>[r]^-{q_{0},q_{1},q_{2}}\cart[r]
\cart@<-4.5pt>[r]}
& 
A_2\ar[d]_{f_2}
\cart@<3pt>[r]^-{p_0,p_1}\cart@<-3pt>[r]& A_1\ar[d]^{f_1}
\\
B_3
{\cart@<4.5pt>[r]^-{s_{0},s_{1},s_{2}}\cart[r]
\cart@<-4.5pt>[r]}
& 
B_2
\cart@<3pt>[r]^-{r_0,r_1}\cart@<-3pt>[r]& B_1
\\
J\!\times_I\! J\!\times_I\!
J{\ar@<4.5pt>[r]^-{\partial_{0},\partial_{1},\partial_{2}}\ar[r]\ar@<-4.5pt>[r]}
& 
J\!\times_I\! J
\ar@<3pt>[r]^-{\partial_0,\partial_1}\ar@<-3pt>[r]& 
J\depi[r]^e & I
}
\]
such that $f_1p_i=r_if_2$ and $f_2q_i=s_if_3$ for all appropriate $i$.
\end{itemize}
\end{definition}
The previous definition does not refer to a cleavage on $\fibc$. 
In some situations (in particular in the proof of
Lemma~\ref{lem:excat-lccc-cover}) it is useful to rephrase the definition in a
way that uses an explicit cleavage, in which case the so-called \emph{cocycle
condition}\index{cocycle condition} becomes visible:
\begin{lemma}\label{lem:desc-cleavage}
 The category $\ddesc(\fibc, e)$ of descent data is equivalent to the category
given by
\begin{itemize}
 \item \textbf{objects:} pairs
$(A\in\fibc_J,\alpha:\partial_1^*A\to\partial_0^*A)$ such that the diagram
\[
 \xymatrix@R-5mm@C-8mm{
&& (\partial_1\partial_1)^*A\ar[rr]^-\cong\arid[ld] &&
\partial_1^*\partial_1^*A\ar[rr]^{\partial_1^*\alpha} &&
\partial_1^*\partial_0^*A\ar[rd]^\cong \\
& (\partial_1\partial_2)^*A\ar[ld]_{\cong} &&&&&&
(\partial_0\partial_1)^*A\arid[rd] \\
\partial_2^*\partial_1^*A\ar[rd]_{\partial_2^*\alpha} &&&&&&&&
(\partial_0\partial_0)^*A \\
& \partial_2^*\partial_0^*A\ar[rd]_\cong  &&&&&&
\partial_0^*\partial_0^*A\ar[ur]_\cong \\
&& (\partial_0\partial_2)^*A\arid[rr] && (\partial_1\partial_0)^*A\ar[rr]_\cong
&& \partial_0^*\partial_1^*A\ar[ru]_{\partial_0^*\alpha} \\
}
\]
in $\fibc_{J\times_I J\times_I J}$ commutes (the isomorhisms are those
given by the universal property of cartesian liftings, and the equalities come
from the simplicial identities), and 
\item\textbf{morphisms} from $(A,\alpha)$ to $(B,\beta)$: arrows $(f:A\to B)\in
\fibc_J$ such that $\partial_0^*f\circ\alpha = \beta\circ\partial_1^*f$.
\end{itemize}
\end{lemma}
\begin{proof}
 \cite[Section~4.1.2]{vistoli2004notes}.
\end{proof}

There is yet another definition of $\ddesc(\fibc,e)$ in terms of \emph{sieves}\index{sieve}:
the sieve $\langle e\rangle$ generated by $e:J\eepi I$ is a subfibration of the
representable (discrete) fibration $YI=\dom:\catr/I\to\catr$, and one can
show the following lemma.
\begin{lemma}
 We have an equivalence of categories
$\ddesc(\fibc,e)\simeq\catfib(\catr)(\langle e\rangle,\fibc)$.
\end{lemma}
\begin{proof}
 \cite[Proposition~4.5]{vistoli2004notes}.
\end{proof}
This allows us to embed $\fibc_I$ into $\ddesc(\fibc,e)$ via the chain
\begin{equation}\label{eq:into-descent}
 \fibc_I\simeq\catfib(\catr)(YI,\fibc)\to \catfib(\catr)(\langle
e\rangle,\fibc)\simeq\ddesc(\fibc,e)
\end{equation}
where the first equivalence is the \emph{2-Yoneda Lemma}\index{2-Yoneda Lemma}
\cite[Section~3.6.2]{vistoli2004notes} and the arrow is given by precomposition
with the inclusion $\langle e\rangle\subseteq YI$. We can now define pre-stacks
and stacks.
\begin{definition}\label{def:pre-stack}
 Let $\fibc:\tot{\fibc}\to\catr$ be a fibration on a regular category.
\begin{itemize}
 \item $\fibc$ is a \emph{pre-stack}\index{pre-stack}, if for each regular
epimorphism $e:J\eepi
I$ in $\catr$, the embedding~\eqref{eq:into-descent} is full and faithful.
\item $\fibc$ is a \emph{stack}\index{stack}, if for each regular epimorphism
$e:J\eepi
I$ in $\catr$, the embedding~\eqref{eq:into-descent} is an equivalence of
categories.
\end{itemize}
\end{definition}
In order to effectively manipulate pre-stacks and stacks, we introduce some
more terminology.
\begin{definition}\label{def:covercart-collepi}
 Let $\fibc:\tot{\fibc}\to\catr$ be a fibration on a regular category $\catr$.
\begin{enumerate}
 \item A
\emph{cover-cartesian} morphism\index{cover-cartesian morphism}\index{morphism!cover-cartesian} in $\tot{\fibc}$
is a cartesian morphism over a regular epimorphism.

We will denote cover-cartesian morphisms by the arrow symbol
$\eepicart$.
\item
A \emph{collective epimorphism}\index{collective epimorphism}\index{epimorphism!collective} in $\tot{\fibc}$
is a map $e$ such that $fe=ge$ implies $f=g$ for vertical $f,g$.

\end{enumerate}
\end{definition}
The following lemma gives a diagrammatic criterion for a fibration to be a
pre-stack.
\begin{lemma}\label{lem:diagrammatic-pre-stack}
A fibration $\fibc:\tot{\fibc}\to\catr$ is a
pre-stack, iff for any configuration
\[
\xymatrix@C+5mm{
A_2\ar[d]_{f_2}
\cart@<3pt>[r]^-{p_0,p_1}\cart@<-3pt>[r]& A_1\ar[d]^{f_1}\epicart[r]^p &
A\dashed[d]^h
\\
B_2
\cart@<3pt>[r]^-{r_0,r_1}\cart@<-3pt>[r]& B_1\epicart[r]^r & B
\\
J\!\times_I\! J
\ar@<3pt>[r]^-{\partial_0,\partial_1}\ar@<-3pt>[r]& 
J\depi[r]^e & I
}
\]
where $e$ is a regular epi in the base with kernel pair $\partial_0,\partial_1$,
and the
other maps are cartesian or vertical above in $\tot{\fibc}$ as indicated
such that $pp_0=pp_1$ and $rr_0=rr_1$, if
$f_1p_0=r_0f_2$ and $f_1p_1=r_1f_2$ then there exists a \emph{unique} $h$ such
that
$hp=rf_1$.
\qed
\end{lemma}

\begin{lemma}\label{lem:descent-pstack}\label{lem:covercart-regepi}
  Let $\fibc:\tot{\fibc}\to\catr$ be a
  pre-stack. Cover-cartesian maps in
$\tot{\fibc}$ are regular epimorphisms.
\end{lemma}
\begin{proof}
Let $p:A_1\eepicart A$ be a cover-cartesian map. We show that $p$ is the
coequalizer of its kernel pair   
$\xymatrix@1{A_2\cart@<-4pt>[r]_{p_1}\cart@<4pt>[r]^{p_0} & A_1\epicart[r]^p
& A}$ (which is definable using only finite limits in the base). Let $f:A_1\to
B$ such that $fp_0=fp_1$. Then $\fibc(f)$ factors uniquely
through $\fibc(p)=e$ since $e$ is a regular epimorphism and $\fibc(f)$
coequalizes its kernel pair $\fibc(p_0)=\partial_0, \fibc(p_1)=\partial_1$.
Consider the
following diagram.
\begin{equation*}
  \xymatrix{
    A_2\cart@<-4pt>[r]_{p_1}\cart@<4pt>[r]^{p_0}\ar[d]_{k}
 & A_1\epicart[r]^p\ar[drr]_(.3){f}\ar[d]_{f^*} & A\dotted[d]^(.3){f'}|{\hole}
\\
C_2\cart@<-4pt>[r]\cart@<4pt>[r]  &C_1\epicart[r]_d & C\cart[r]_c &
B\\
     K\ar@<-4pt>[r]_{\partial_1}\ar@<4pt>[r]^{\partial_0} 
& 
I^*\depi[r]^{e} 
& 
I \ar[r]^u& J\\
}
\end{equation*}
Here $k$ can be understood both as the cartesian lifting of $fp_0 = fp_1$
along $ue\partial_0=ue\partial_1$ and as the reindexing of $f^*$ along either of
$\partial_0$ and $\partial_1$. Now since the reindexing of $f^*$ along both
components of
the
kernel pair of  $e$ are equal and $\fibc$ is a pre-stack by assumption,
Lemma~\ref{lem:diagrammatic-pre-stack} allows us to deduce that there
exists a unique $f':A\to C$ such that $f'p=df^*$. The map $cf'$ provides the
desired factorization of $f$ through $e$, uniqueness follows from uniqueness of
$f'$.
\end{proof}

A minimal structural requirement for pre-stacks to be of interest for us is to
have finite limits. This leads us to the following definition.
\begin{definition}\label{def:fl-pstack}
\begin{enumerate}\item
A \emph{finite-limit pre-stack}\index{finite-limit pre-stack}\index{pre-stack!of finite limit categories} is a pre-stack
$\fibc:\tot{\fibc}\to\catr$ on a regular category which is a finite limit
fibration in the sense of
Definition~\ref{def:fl-fibs}.
 \item 
$\catlex(\catr)$ is the 2-category of finite-limit pre-stacks on $\catr$.
Its 1-cells are finite limit preserving fibered functors, and its 2-cells are
fibered natural transformations.
\end{enumerate}
\end{definition}

\subsubsection{Weak equivalences}

There is a class of fibered functors between pre-stacks which are
almost, but not quite, equivalences -- the \emph{weak equivalences}. We recall
the definition from~\cite{bunge1979stacks}.
\begin{definition}\label{def:weak-equivalences}
 A fibered functor $F:\fibc\to\fibd$ between pre-stacks on a regular category
$\catr$ is called a \emph{weak equivalence}\index{weak equivalence}, if it is
full and faithful and for each $D\in\tot{\fibd}$ there exists a
$C\in\tot{\fibc}$ and a cover-cartesian map $e:FC\ecovercart D$.
\end{definition}
\begin{remarks}
 \begin{itemize}
  \item A good way to understand the relevance of weak equivalences between
pre-stacks is to note that if $F$ is an externalization
(\cite[Section~4]{streicherfib}) of an internal functor $F_0:\catc\to\catd$
between internal categories, then $F$ is a weak equivalence iff $F_0$ is a weak
 equivalence in the sense of the internal logic of $\catr$, meaning that the
statement that $F$ is full, faithful and essentially surjective holds in the
internal logic, but the essential surjectivity is not necessarily witnessed by a
choice of essential pre-image for each object of $\catd$.
\item Another intuition on weak equivalence is given by the fact that a fibered
functor $F:\fibc\to\fibd$ between pre-stacks $\fibc,\fibd$ is a weak
equivalence iff the induced functor $\tilde{F}:\tilde{\fibc}\to\tilde{\fibd}$
between the stack completions of $\fibc$ and $\fibd$ is an equivalence in the
standard sense (\cite[Corollary~2.12]{bunge1979stacks}, we do not treat stacks
and stack-completions here and refer to \cite{bunge1979stacks} and the
references therein).
\item In the same spirit as the weakening of the notion of equivalence, one can
consider a weak notion of `having finite limits' for regular pre-stacks
by demanding that each diagram in the total category can be covered (in the
sense of cover-cartesian maps) by a diagram having a limiting cone. This and
similar considerations are important in \cite{hrr90}. It seems reasonable to
assume that everything that we do with finite-limit pre-stacks in this work
also works for pre-stacks that only have `weak' finite limits (in the above
sense)\footnote{For a long time I didn't believe this to be relevant for
realizability, but Wouter Stekelenburg pointed out to me that depending on the
definition of pca, the category of partitioned assemblies over an internal pca
does not necessarily have (strong) finite limits.}.

I wrote `weak' in quotes above since there is a clash with standard
terminology -- normally a weak limit for a diagram is a cone that satisfies the
existence part, but not the uniqueness part of the universal property of a
limiting cone. If a pre-stack has weak finite limits in the fibrational sense,
then the total category has weak finite limits in the ordinary sense, but not
necessarily vice versa.
 \end{itemize}
\end{remarks}

\subsection{Geometric pre-stacks}\label{suse-geo-fib}

\begin{definition}
A \emph{pre-stack of regular categories}\index{pre-stack!of regular categories} on a regular category $\catr$ is a
pre-stack $\fibc:\tot{\fibc}\to\catr$ whose fibers are regular categories, and
whose reindexing functors are regular functors.
\end{definition}
\begin{lemma}\label{lem:descent-regular} \label{lem:descent-regular-descent}
Let $\fibc:\tot{\fibc}\to\catr$ be a pre-stack of regular
categories on a regular
category $\catr$. Vertical regular epimorphisms are
    closed under descent, i.e.\ if $e:J\eepi I$ is a regular epimorphism in
    $\catr$, and $(f:X\to Y)\in \fibc_I$ such that $e^*f$ is regular epic in
    $\fibc_J$, then $f$ is already a regular epi in $\fibc_I$.
\end{lemma}
\begin{proof}
  We show that $f$ is left orthogonal~(see
Definition~\ref{def:facsys}-\ref{def:facsys-orth}) to monos in $\fibc_I$. Any
square $
  \vcenter{\xymatrix@-3mm{
      X\ar[d]_f\ar[r] & U\mono[d]^m\\
      Y\ar[r] & A }} $ in $\fibc_I$ gives rise to a square $
  \vcenter{\xymatrix@-3mm{
      e^*X\depi[d]_{e^*f}\ar[r] & e^*U\mono[d]^{e^*m}\\
      e^*Y\ar[r] & e^*A }} $ in $\fibc_J$ with mediator $
  \vcenter{\xymatrix@-3mm{
      e^*X\depi[d]_{e^*f}\ar[r] & e^*U\mono[d]^{e^*m}\\
      e^*Y\ar[r]\dashed[ur]^{h} & e^*A }} $.  Now the reindexings of $h$
  along the two components of the kernel pair of $e$ coincide (since they
mediate the same orthogonality square), and thus (by
Lemma~\ref{lem:diagrammatic-pre-stack}) $h$
  descends to a mediator of the square in $\fibc_I$.
\end{proof}

\begin{definition}\label{def:fisufi}
Let $\fibc:\tot{\fibc}\to\catc$ be a fibration of finite limit categories on a
finite limit category $\catc$.
  \begin{enumerate}
  \item\label{def:fisufi-fisufi} The 
  \emph{fibered subobject fibration}\index{fibered!subobject
fibration}\index{subobject fibration!fibered}\index{fibration!fibered subobject}
  $\lfisufi{\fibc}$ of $\fibc$ is the posetal
fibration on $\tot{\fibc}$ whose predicates on
  $C\in\tot{\fibc}$ in $\sub(\fibc)$ are vertical monomorphisms with codomain
$C$, where
  entailment is given by inclusion of monomorphisms, and reindexing is given by
pullback.
\[
\Sub(\fibc)\xrightarrow{\sub(\fibc)}\tot{\fibc}
\xrightarrow{\;\;\fibc\;\;} \catr
\]
\item\label{def:fisufi-int-union} We say that $\fibc$ has \emph{internal
unions}\index{internal unions}, if $\sub(\fibc)$ admits left adjoints to
reindexing along cartesian morphisms in $\tot{\fibc}$ subject to the
Beck-Chevalley condition for pullbacks along cartesian morphisms (i.e. for
squares of cartesian morphism in $\tot{\fibc}$ over pullbacks in $\catc$).
\item\label{def:fisufi-stable-int-union} We say that $\fibc$ has \emph{stable
internal unions}\index{internal unions!stable}, if $\sub(\fibc)$ admits
left adjoints to
reindexing along cartesian morphisms in $\tot{\fibc}$, subject to the
Beck-Chevalley condition for pullbacks along \emph{arbitary maps} in
$\tot{\fibc}$.
\end{enumerate}
\end{definition}
We leave it as an exercise to verify that for a finite limit
category $\catc$, $\famf(\catc)$ has internal unions iff the subobject
lattices in $\catc$ have small joins, and that internal unions in $\famf(\catc)$
are stable iff small joins of subobjects in $\catc$ are stable under pullback.

\begin{definition}\label{def:geo-fib}
\begin{enumerate}
  \item A \emph{\geostack}\index{geometric pre-stack}\index{pre-stack!geometric} is a pre-stack
$\fibs:\tot{\fibs}\to\catr$ of
regular categories with stable internal unions.
  \item $\catgeo(\catr)$ is the 2-category of \geostack{}s on $\catr$
    -- its 1-cells are regular fibered functors that
    preserve internal unions (`geometric fibered
functors'\index{geometric fibered functor}\index{fibered!functor!geometric}),
and its 
    2-cells are fibered natural transformations.
  \end{enumerate}
\end{definition}
\begin{remark}\label{rem:geom-prestack}
 Johnstone~\cite[after Lemma~A1.4.18]{elephant1} defines a \emph{geometric
category}\index{geometric category}\index{category!geometric} to be a well-powered regular category with pullback-stable small joins
of subobjects. It is easy to see that a small category $\cats$ is geometric iff
its family fibration $\famf(\cats):\Famf(\cats)\to\catset$ is a geometric
pre-stack in the sense of the preceding definition. The well poweredness
condition is necessary for Johnstone to show that any geometric category is a
Heyting category (since a monotone map between \emph{small} cocomplete lattices
has a right adjoint iff it preserves arbitrary joins), but it would be too
restrictive for our purposes to make a similar assumption since we are
interested in examples which do \emph{not} have universal quantification in the
fibers (see Remark~\ref{rem:rel-compl}-\ref{rem:rel-compl-prim}). 

The nLab~\cite{nlab-geometric-category} doesn't demand well
poweredness either for geometric categories.
\end{remark}

\begin{lemma}\label{lem:gps-sub-equant}
  A pre-stack $\fibs:\tot{\fibs}\to\catr$ of finite limit categories is
  a \geostack{} iff its fibered subobject fibration $\sub(\fibs)$ has
  existential quantification along \emph{arbitrary} morphisms in $\tot{\fibs}$.
In this case, $\sub(\fibs)$ also validates the Frobenius condition and thus is
an existential fibration.
\end{lemma}
\begin{proof}
Let $\fibs$ be a finite limit pre-stack such that $\sub(\fibs)$ has existential
quantification.
Clearly, if we have $\exists$ along arbitrary maps, then we have it in
particular along vertical and cartesian maps. Furthermore, the global
Beck-Chevalley condition specializes to the same condition in the fibers, which
implies that the fibers of $\fibs$ are regular categories. To show that we are
dealing with a fibration of regular categories, it remains to show that
reindexing preserves regular epimorphisms. This follows from the global
Beck-Chevalley condition for squares of the form
\[
 \xymatrix@-3mm{
\overline{D}\cart[r]^d\ar[d]_{\overline{e}}\pullbackcorner&D\ar[d]^e\\
\overline{C}\cart[r]_c&C
}
\]
with $e$ (and thus $\overline{e}$) vertical, since $e$ is regular epic iff
$\exists_e\top\cong\top$, in which case we have
$\exists_{\overline{e}}\top\cong\exists_{\overline{e}}
d^*\top\cong c^*\exists_e\top\cong c^*\top\cong\top$, which means that
$\overline{e}$ is also a regular epi. Thus, $\fibs$ is a \geostack{}.

Conversely, assume that $\fibs$ is a \geostack{}. Then we can
existentially quantify in $\sub(\fibs)$ along vertical maps since the fibers
are regular, and along cartesian maps by assumption. Since every map in the
total category can be decomposed into vertical followed by cartesian part, we
can thus quantify along all maps. It remains to check if the global
Beck-Chevalley
condition holds in this case. To this end consider a pullback square
\[
\vcenter{
\xymatrix@-3mm{
 P\ar[r]\pullbackcorner\ar[d] & A\ar[d]\\ 
B\ar[r]& C
}}
\]
in $\tot{\fibs}$. This square can be decomposed into a diagram of cartesian and
vertical arrows as in 
\[
\vcenter{
\xymatrix@-3mm{
 P\pullbackcorner\ar[r]\ar[d] &\bullet\pullbackcorner\cart[r]\ar[d]& A\ar[d]\\ 
\bullet\pullbackcorner\ar[r]\cart[d]&\bullet\pullbackcorner\cart[r]\cart[d]
&\bullet
\cart[d] \\
B\ar[r] &\bullet\cart[r]& C
}},
\]
and to check the condition for the large square, it suffices to check it for
the small ones. The condition for the upper left square follows since it holds
in regular categories, for the two squares where we quantify along cartesian
arrows it follows by assumption, and for the remaining one it is a consequence
of the fact that reindexing preserves regular epimorphisms.

It remains to show the validity of the Frobenius condition. Consider the diagram
\[
 \vcenter{\xymatrix@-3mm{
&
\bullet\mono[d]_{\overline{\varphi}}\cart[r]^{\overline{u}}&\bullet\mono[d]
^\varphi\\
\bullet\mono[r]^\psi&\bullet\cart[r]^u&\bullet
}}.
\]
We have to show that
$\varphi\wedge\exists_u\psi\cong\exists_u\overline{\varphi}\wedge\psi$. To
show this, we can argue
\[
\varphi\wedge\exists_u\psi
\cong \exists_\phi\phi^*\exists_u\psi
\cong \exists_\phi \exists_{\overline{u}}\overline{\varphi}^*\psi
\cong \exists_u\exists_{\overline{\varphi}}\overline{\varphi}^*\psi
\cong\exists_u\overline{\varphi}\wedge\psi,
\]
which completes the proof.
\end{proof}

\begin{definition}\label{def:collective-regular-epi}
 Let $\fibs$ be a \geostack{} on $\catr$. We call a morphism $f:A\to B$
in $\tot{\fibs}$ a \emph{collective cover}\index{collective cover}\index{cover!collective} (or
\emph{collectively
covering}), if $\exists_f\top\cong\top$.
\end{definition}
When thinking in terms of families, this means that the objects in the family
$B$ are
covered by (the unions of) the images of the objects in the family $A$.

\medskip

We will now prove that collective covers and vertical monos form a \emph{stable
factorization system} on the total category of a geometric pre-stack. Let us
first recall the definition from~\cite[Section~5]{borceux1}.
\begin{definition}\label{def:facsys}
 Let $\catc$ be a category. 
\begin{enumerate}
 \item\label{def:facsys-orth}
Let $f:A\to B$, $g:X\to Y$ in $\catc$.
We say that $f$ is \emph{left orthogonal}\index{orthogonal}\index{left orthogonal} to $g$ (or that $g$ is
\emph{right orthogonal}\index{right orthogonal} to $f$) and we write $f\perp g$, if for any commuting
square
\[
\xymatrix@R-4mm{
A \ar[r]\ar[d]_f& X\ar[d]^g \\
B\dashed[ru]^h\ar[r] & Y
}\]
there exists a \emph{unique} $h:B\to X$ such that the two triangles commute.
\item
A \emph{factorization system}\index{factorization system} on $\catc$ is a pair
$(\mathcal{E}, \mathcal{M})$
of classes of morphisms of $\catc$ such that
\begin{enumerate}
 \item $\mathcal{E}$ and $\mathcal{M}$ are both closed under composition and
contain all isomorphisms,
\item for all $e\in\mathcal{E}$ and $m\in\mathcal{M}$ we have $e\perp m$, and
\item any $f:A\to B$ in $\catc$ can be factorized as $f=me$ with $m\in
\mathcal{M}$ and $e\in\mathcal{E}$.
\end{enumerate}
\item
If $\catc$ has pullbacks, we call a factorization system
$(\mathcal{E},\mathcal{M})$ on $\catc$ \emph{stable}\index{factorization
system!stable}\index{stable factorization system}, if given $f:A\to B$ and
$e:C\to B$ we have $e\in\mathcal{E}\imp f^*e\in\mathcal{E}$\footnote{
It follows from orthogonality that the class
$\mathcal{M}$ is closed under pullbacks.}.
\end{enumerate}
\end{definition}

\begin{lemma}\label{lem:geostack-ccov}
 Let $\fibs$ be a \geostack{} on $\catr$.
 \begin{enumerate}
  \item\label{lem:geostack-ccov-epicart}
 Cover-cartesian maps are collectively covering.
\item\label{lem:geostack-ccov-fsys}
Collective covers and vertical monos form
a stable factorization
system on $\tot{\fibs}$.
\item\label{lem:geostack-ccov-equant}
 The collective-cover/vertical-mono factorization system allows to express
existential quantification in
$\sub(\fibs)$. Concretely, given 
$f: B\to A$ and a vertical monomorphism $m:U\emono B$ in $\tot{\fibs}$,
$\exists_f m$ is given by the vertical mono part of the
collective-cover/vertical-mono factorization of $fm$.
\[
\xymatrix@R-2mm{
U\vmono[d]_m\ccov[r] &
\bullet\vmono[d]^{\exists_fm}\\
B\ar[r]^f & A
}
\]
\item\label{lem:geostack-ccov-cepi}
Collective covers are collectively epic, i.e.\ if $e$ is
collectively covering and $fe=ge$ where $f$ and $g$ are vertical, then $f=g$.
 \end{enumerate}
\end{lemma}
\begin{proof}
\emph{Ad \ref{lem:geostack-ccov-epicart}.}
Let $e:A\eepicart B$ be cover-cartesian in $\tot{\fibc}$. We have to show that
$e$
doesn't factor through any nontrivial vertical monomorphism. Consider the
diagram
\[
 \cxymatrix{
& \bullet\mono[d]^m \\
A\dotted[ru]^h\epicart[r]_e & B
}.
\]
The existence of $h$ is equivalent to $e^*m\cong\top$, thus we have to show
that $e^*m\cong\top\imp m\cong\top$. This follows from the fact that reindexing
along cover-cartesian maps reflects inclusion of vertical subobjects, which is a
consequence of the fact that $\fibs$ is a pre-stack.

\medskip

\emph{Ad \ref{lem:geostack-ccov-fsys}.}
It is clear that collective covers as well as vertical monos are closed under
composition and contain all isos.
To show the orthogonality property, since $\tot{\fibs}$ has
pullbacks which preserve vertical monos it is sufficient to show that
whenever a collective cover factors through a vertical mono, then this is
already an iso. As we already saw in the proof
of~\ref{lem:geostack-ccov-epicart}, this is just a rephrasing of the property
of being collectively covering. The existence of factorizations is not
difficult to see either -- we can factor any $f:A\to B$ in $\tot{\fibs}$
through the vertical mono given by $\exists_f\top$, and it is easy to see that
the left part of this factorization is collectively covering. Finally, pullback
stability of collective covers follows from the Beck-Chevalley condition in the
definition of stable internal unions.

\medskip

\emph{Ad \ref{lem:geostack-ccov-equant}.}
This follows because $\exists_fm$ as well as the vertical-mono part
of the factorization can be characterized as minimal vertical subobject of $A$
admitting a factorization of $fm$.

\medskip

\emph{Ad \ref{lem:geostack-ccov-cepi}.}
This is since on the one hand, $e$ factors through the equalizer of $f$ and $g$,
which is a vertical subobject of $B$ and on the other hand
$\top=\exists_e\top$ is the minimal subobject of $B$ admitting a factorization
of $e$.
\end{proof}

\subsection{Positive pre-stacks}

\begin{definition}\label{def:positive-pre-stack}
  Let $\catr$ be a regular category.
\begin{enumerate}
 \item 
 A \emph{positive pre-stack}\index{positive pre-stack}\index{pre-stack!positive}
 on $\catr$ is a pre-stack
$\fibp:\tot{\fibp}\to\catr$ of regular categories with extensive internal sums.
\item $\catpos(\catr)$ is the 2-category of positive pre-stacks on $\catr$. Its
1-cells are fibered functors which are fiberwise regular and preserve internal
sums (`positive fibered functors')\index{positive fibered
functor}\index{fibered!functor!positive}, and its 2-cells are arbitrary fibered
natural transformations.
\end{enumerate}
\end{definition}
\begin{remark}
 The family fibration $\famf(\catp):\Famf(\catp)\to\catset$ of a category
$\catp$ is a positive pre-stack iff $\catp$ is regular and has \emph{small}
extensive sums. This is an infinitary version of what Johnstone calls `positive
coherent category'~\cite[Section~A1.4]{elephant1} (and close to what he calls
`$\infty$-positive geometric category', though there is again the difference
about well poweredness that we pointed out in Remark~\ref{rem:geom-prestack}).

Observe that we do not demand 
any form of sums \emph{in} the fibers, only `between' the fibers. In
particular, the fibers of a positive fibration are not necessarily positive in
the sense of Johnstone.
\end{remark}

Since positive pre-stacks are in particular lextensive fibrations, Moens'
theorem applies and one may ask how the additional conditions on the fibration
can be expressed in terms of the corresponding functor. We will answer this
question after an auxiliary lemma.
\begin{lemma}\label{lem:descent-positive}
Let $\fibc:\tot{\fibc}\to\catr$ be a positive pre-stack.
For any regular epimorphism $e:J\eepi I$ in $\catr$ and $A\in\fibc_I$, the
fibered codiagonal map
    $\sigma:\Sigma_ee^*A\eepi A$ is a regular epimorphism in $\fibc_I$.
\end{lemma}
\begin{proof}
 By Lemma~\ref{lem:descent-regular}, it is sufficient to show that
$e^*\sigma$ is a regular
  epimorphism. In fact, $e^*\sigma$ is even a split epimorphism since it is a
  leg of one of the triangle equalities of the adjunction $\Sigma_e\adj e^*$.
\end{proof}

\begin{lemma}\label{lem:moens-positive}
  The gluing constrution 
gives rise
  to a biequivalence
\[
\catr\pslice\catreg \simeq \catpos(\catr),
\]
where $\catreg$ is the 2-category of regular categories, regular
functors and natural transformations, and $\catr\pslice\catreg$ is the
pseudo-co-slice 2-category
of $\catreg$ under $\catr$.
\end{lemma}
\begin{proof}
We rely on Moens' Theorem~\ref{theo:moens} and only show that additional
conditions on one side imply additional conditions on the other side and vice
versa.

 Since the fundamental fibration of a regular category is a pre-stack, and
pre-stacks are stable under pullback along regular functors,
$\gl_\Delta(\catq)$ is a positive pre-stack whenever $\Delta:\catr\to\catq$ is
a regular functor between regular categories. Conversely, assume that $\fibp$ is
a positive pre-stack on $\catr$. We
have to show that the functor
\[
 I\mapsto \sum_I 1_I : \catr\to\fibp_1
\]
preserves regular epimorphisms. Given a regular epimorphism $e:J\eepi I$ in
$\catc$, the map $\sigma:\Sigma_e 1_J\to 1_I$ in $\fibp_I$ is regular epic by
Lemma~\ref{lem:descent-positive}, and its image $\Sigma_I\sigma$ is regular
epic in $\fibp_1$ since $\Sigma:\fibp_I\to\fibp_1$ preserves regular
epimorphisms as a left adjoint. This shows the equivalence between fibrations
and functors. The correspondence on the level of 1-cells is easy to see.
\end{proof}

\begin{remark}
The following remark does not really fit into the flow of ideas, and
in particular relies on concepts that will be introduced only later. We
nevertheless present it here, since it seems to be the right place for the
informed reader.

There is a connection between $\catr\pslice\catreg$ and Longley's
$\nabla\Gamma$-categories. Recall from
\cite[Definition~1.4.2]{longley1995realizability} that a
$\nabla\Gamma$-category is a regular category $\catq$ together with regular
functors $\nabla:\catset\to\catq$ and $\Gamma:\catq\to\catset$ such that
$\Gamma\adj\nabla$ and $\Gamma\nabla\cong\id_\catset$; the archetypal example
being the category $\asm(\pcaa)$ of assemblies over a partial combinatory
algebra $\pcaa$. A $\nabla\Gamma$-functor from $(\catq,\nabla,\Gamma)$ to
$(\catq',\nabla',\Gamma')$ is a regular functor $F:\catq\to\catq'$ such that
$F\nabla\cong\nabla'$ and $\Gamma' F\cong\Gamma$, and Longley's `equivalence
theorem'~\cite[Theorem~2.2.20]{longley1995realizability} states that
\emph{applicative morphisms}\index{applicative morphism}\index{morphism!applicative} between pcas correspond to $\nabla\Gamma$-functors
between the corresponding categories of assemblies.

Now it turns out that for $\nabla\Gamma$-functors
\emph{between categories of assemblies}, the condition $\Gamma' F\cong\Gamma$ is
redundant, and thus in this case $\nabla\Gamma$-functors are the same things as
1-cells in $\catset\pslice\catreg$. Therefore, Longley's equivalence theorem
can be read as stating an equivalence between applicative morphisms between
pcas, and
positive fibered functors between the associated positive fibrations of
assemblies. 
\end{remark}

The following lemma clarifies the relation between $\catgeo(\catr)$ and
$\catpos(\catr)$.
\begin{lemma}\label{lem:positive-prestack-geometric}
\begin{enumerate}
 \item\label{lem:positive-prestack-geometric-geo} 
Positive pre-stacks are
\geostack{}s.
\item\label{lem:positive-prestack-geometric-cocacoco}
Let $\fibp:\tot{\fibp}\to\catr$ be a positive fibration. Let $f:A\to B$ be a
map in $\tot{\fibp}$ over $u:I\to J$ in $\catr$. Then $f$ is cocartesian iff
the judgments
\begin{align}
 x,y\vtp A\csep fx=fy&\ent x=y\label{eq:inj-in-sub}\\
z\vtp B\csep\hspace{39pt}&\ent\exists x\vtp A\qdot fx=z\label{eq:surj-in-sub}
\end{align}
hold in $\sub(\fibp)$.
\item\label{lem:positive-prestack-geometric-functors} A fibered functor
$F:\fibp\to\fibq$ between positive pre-stacks $\fibp$ and $\fibq$ is positive
iff it is geometric.
\end{enumerate}
\end{lemma}
\begin{proof}
\emph{Ad \ref{lem:positive-prestack-geometric-geo}.}
Given $f:A\to B$ over $u:I\to J$ and a vertical mono $m:V\emono A$, existential
quantification $\exists_fm$ of $m$ along $f$ is given by cocartesian lifting and
image factorization as in the following diagram.
\begin{equation}\label{eq:exists-from-sum}
\vcenter{\xymatrix@R-4mm@C+03mm{
 V \mono[d]_m\coca[r] & \Sigma_uV\depi[d] \\
A\ar[rd]_f & \mono[d]^{\exists_fm}\\
& B\\
I\ar[r]^u & J\\
}}
\end{equation}
Stability follows from stability of internal sums in extensive
fibrations.
 
\medskip
\emph{Ad \ref{lem:positive-prestack-geometric-cocacoco}.}
Assume that judgments~\eqref{eq:inj-in-sub} and \eqref{eq:surj-in-sub} hold for
$f:A\to B$. Let $g:A\to C$ over $u:I\to J$. We have to show that there exists a
unique vertical $h:B\to C$ such that $hf=g$. The uniqueness follows since $f$
is a collective cover and thus collectively epic
(Lemma~\ref{lem:geostack-ccov}-\ref{lem:geostack-ccov-cepi}). For existence,
observe that since $f$ and $g$ are over the same arrow $u:I\to J$, the map
$\langle f,g\rangle:A\to B\times C$ factors through $B\times_J C$ as $k:A\to
B\times_J C$. Consider the following diagram.
\[
 \xymatrix{
 & \vmono[d]^rU\\
A\ccov[ur]^e\ar[r]_k\ar[dr]_f & B\times_J C\ar[d]^p\\
& B
}
\]
$r\circ e$ is a collective-cover/vertical-mono factorization of $k$, and we
want to show that $pr$ is an isomorphism (this gives us a vertical map from
$B\to C$ via the span $B\xleftarrow{\cong}U\to C$). Since $\fibp_J$ is regular,
it suffices to show that $pr$ is a mono and a cover. The image of $pr$ can be
written as $\exists_pr$ in $\sub(\fibp)$, which is equal to
$\exists_f\top$, and thus equivalent to $\top$ by~\eqref{eq:surj-in-sub}. To see
that $pr$ is monic, observe that from $u\csep\ent\exists a\qdot ea=u$ and
$fa=fa'\ent a=a'$ we can derive $pru=pru'\ent u=u'$ (still in $\sub(\fibp)$),
which implies that $pr$ is monic since it is vertical (this is a formalization
of the argument that if $fe$ is monic and $e$ is epic, then $f$ is monic).

\smallskip

Conversely, assume that $f$ is cocartesian. Judgment~\eqref{eq:surj-in-sub} is
just a fancy way of saying that $f$ is a collective cover, which is true for
cocartesian maps as can be easily seen from the proof
of~\ref{lem:positive-prestack-geometric-geo}.
For the other judgment, let's make the denotations of both sides of the
turnstile explicit.
The predicate $(x,y\vtp A\csep fx =
fy)$ is the vertical image (in the sense of the
collective-cover/vertical-mono factorization system) of $A\times_B
A\to A\times A$, and the predicate $(x,y\vtp A\csep x=y)$ is the vertical image
of $A\to A\times A$. Let $m\circ e$ be a collective-cover/vertical-mono
factorization of $A\times_B A\to A\times A$. The canonical map $\delta:A\to
A\times_B A$ is cocartesian -- and thus collectively covering -- since internal
sums in $\fibp$ are disjoint (see
Definition~\ref{def:lextensive-fibs}-\ref{def:lextensive-fibs-disjoint}). 
\[
 \xymatrix{
A\ar[ddr]\coca[r]_-\delta & A\times_BA \pullbackcorner\coca[r]\ccov[d]_e&
B\ccov[d]\\
& \bullet\vmono[d]^m\pullbackcorner\coca[r] & \bullet\vmono[d]\\
& A\times A\coca[r]^{f\times f} & B\times B
}
\]
This implies that $m\circ
e\delta$ is a
collective-cover/vertical-mono
factorization of $A\to B\times B$, hence the two maps have the same vertical
image, and
the predicates are equivalent.

\medskip
\emph{Ad \ref{lem:positive-prestack-geometric-functors}.}
If $F$ is a positive fibered functor, then it preserves diagrams of the
form~\eqref{eq:exists-from-sum}, hence it also preserves existential
quantification. Conversely, if $F$ is geometric, then the
induced transformation from $\sub(\fibp)$ to $\sub(\fibq)$ preserves the
validity of the judgments \eqref{eq:inj-in-sub} and \eqref{eq:surj-in-sub},
whence it also preserves cocartesian maps.
\end{proof}

\subsection{Fibered pretoposes}

\begin{definition}\label{def:fibered-pretopos}
  Let $\catr$ be a regular category.
\begin{enumerate}
\item A \emph{pre-stack of exact categories}\index{pre-stack!of exact categories} on $\catr$ is a pre-stack whose
fibers are exact categories, and whose reindexing functors are regular functors.
 \item 
A \emph{fibered pretopos}\index{fibered!pretopos}\index{pretopos!fibered} on $\catr$ is a pre-stack of
exact categories
    with extensive internal sums.
\item $\catpretop(\catr)$ is the 2-category of fibered pretoposes on $\catr$.
Its
1-cells are positive fibered functors (Definition~\ref{def:positive-pre-stack}),
and its 2-cells are arbitrary fibered natural transformations.
\end{enumerate}
\end{definition}
\begin{remarks}
 \begin{itemize}
  \item The family fibration $\famf(\catx):\Famf(\catx)\to\catset$ of a
category $\catx$ is a fibered pretopos iff the category is an $\infty$-pretopos.
\item Contrary to Johnstone's definition of `$\mathcal{S}$-indexed
$\infty$-pretopos'~\cite[Definition~3.3.9]{elephant1}, we do \emph{not} assume
that the fibers of a fibered pretopos are pretoposes -- the extensive sums in
our case are `purely infinitary', i.e.\ between the fibers. 
 \end{itemize}
\end{remarks}

Again, we have a specialization of Moens' theorem:
\begin{lemma}\label{lem:moens-pretop}
  The gluing construction 
gives rise
  to a biequivalence
\[
\pscommacat{\catr}{\mathbf{U}} \simeq \catpos(\catr),
\]
where $\mathbf{U}:\catex\to\catreg$ is the forgetful 2-functor from exact to
regular categories, and $\pscommacat{\catr}{\mathbf{U}}$ is the evident pseudo
comma category.
\end{lemma}
\begin{proof}
This follows directly from Lemma~\ref{lem:moens-positive} since exact
categories are stable under slicing.
\end{proof}
As a consequence of the preceding lemma we see that a fibered pretopos is not
only a pre-stack, but even a \emph{stack} (since the fundamental fibration of 
an exact category is a stack, and stacks are stable under change of base along
regular
functors).

\chapter{Fibrational cocompletions}\label{chap:fib-cocompletions}

Given a regular category $\catr$, the 2-categories of pre-stacks on $\catr$ that
we introduced in Section~\ref{sec:prestacks} can be arranged into a sequence
\begin{equation}\label{eq:sequence-inclusions}
\catlex(\catr)
\hookleftarrow\catgeo(\catr) 
\hookleftarrow\catpos(\catr)
\hookleftarrow \catpretop(\catr)
\end{equation}
connected by forgetful functors.
In this section, we will prove the existence of, and give constructions for,
 left biadjoints for all forgetful functors in this sequence. Before diving
into the details, let us make some general remarks.
\begin{enumerate}
 \item On the right end of the sequence
($\catpretop(\catr)$ and $\catpos(\catr)$) we are dealing with extensive
fibrations and are thus in the realm of Moens' theorem.
\item\label{enum:ficoco-idem} The forgetful functors
$\catpretop(\catr)\to\catpos(\catr)$ and $\catpos(\catr)\to\catgeo(\catr)$ are
full on 1- and 2-cells and thus local equivalences (for the second inclusion
this follows from
Lemma \ref{lem:positive-prestack-geometric}-
\ref{lem:positive-prestack-geometric-functors}), hence their left biadjoints are
reflections. Concretely, this means for example that the fibered pretopos
cocompletion of a geometric pre-stack $\fibs$ doesn't add anything if $\fibs$
happens to be already a fibered pretopos.
\item The biadjunctions between finite limit pre-stacks and the three other
types of pre-stacks
are not reflections, but fulfill the
weaker `lax idempotency' condition $U\ve\adj \eta U$ between unit and counit
($U$ is the forgetful functor). This implies in particular that the induced
pseudomonad is a KZ-monad \cite[Definition~B1.1.11]{elephant1}, which is
characteristic of cocompletions.

One can recover a finite limit pre-stack (up to
weak equivalence) from its
(geometric/posi-tive/pretopos)-cocompletion\footnote{see e.g,{}
Lemma~\ref{lem:chr-psh}-\ref{lem:chr-psh-yip} for the pretopos cocompletion},
something that is impossible for the cocompletions considered in
\ref{enum:ficoco-idem} because of idempotency.
\end{enumerate}
To help the intuition of the reader we give some examples.
\begin{enumerate}
 \item Let $\catc$ be  a small finite limit category. The $\infty$-pretopos
cocompletion of $\catc$ is the presheaf topos $\widehat{\catc}$; the geometric
and $\infty$-positive
cocompletions are the subcategories of $\widehat{\catc}$ on sub-representables,
and coproducts of sub-representables, respectively. This fits into our
fibrational framework if we switch from categories to their family fibrations.

If $\catc$ is not small but only locally small, the fibrational construction
gives us \emph{small} presheaves.
\item The $\infty$-pretopos cocompletion of a small geometric category $\cats$
is the
topos of \emph{sheaves} for the canonical topology on $\cats$. Again, this fits
into the fibrational framework via the family construction.
\item Given a pca $\pcaa$, 
the uniform family fibration $\ufam(\pcaa)$ is a finite limit pre-stack.
The 
geometric cocompletion of $\ufam(\pcaa)$ is the tripos $\rtr{\pcaa}$, the
positive cocompletion is the gluing fibration of $\asm(\pcaa)$, and the fibered
pretopos cocompletion is the gluing of the realizability topos
$\catrt(\pcaa)$ (in each case along the `constant objects functor').
\end{enumerate}

The most `economic' way to obtain left biadjoints to all functors (including
compositions) in the sequence above would be to construct the left biadjoints
only for the forgetful functors from one level to the next. 

However, we will go a different way as it turns out to be most natural to
construct first the left biadjoint of
$\catpretop(\catr)\to\catlex(\catr)$, and then to define the geometric and
positive cocompletions of finite limit pre-stacks as suitable subfibrations.

In the same way, the left biadjoint of $\catpos(\catr)\to\catgeo(\catr)$ seems
to be more
fundamental than its 2-step decomposition (although in this case it is a bit
arguable) whence we present it first.

\section{Fibered presheaves}\label{sec:fpc}\label{sec:fibered-presheaves}

The `fibered presheaf construction'\index{fibered!presheaf construction} is the left biadjoint to the forgetful
functor
\begin{equation}\label{eq:forget-pretop-to-lex}
 \catpretop(\catr)\hookrightarrow\catlex(\catr).
\end{equation}
It is the starting point of our exploration of fibrational cocompletions. It is
motivated by the following two facts.
\begin{itemize}
 \item Robinson and Rosolini~\cite{robinson1990colimit} observed that
realizability toposes are exact completions of their subcategories
of `partitioned assemblies'.
\item Carboni~\cite{carboni1995some} discovered that for a small category
$\catc$ with finite limits, the presheaf category $\widehat{\catc}$ is
the exact completion of the category $\Famf(\catc)$ of families.
\end{itemize}
Both statements are only true in presence of the axiom of choice. Now the
motivation of the fibered presheaf construction is that it provides a common
framework for the two observations, while managing to avoid the axiom of
choice. The central observation is that both the category $\Famf(\catc)$ of
families of $\catc$, and the category of partitioned assemblies over a pca
$\pcaa$ are total categories of certain fibrations -- $\Famf(\catc)$ is the
total category of the family fibration $\famf(\catc):\Famf(\catc)\to\catset$,
while the category of partitioned assemblies is the total category of the
uniform family fibration $\ufam(\pcaa):\Ufam(\pcaa)\to\catset$ defined after
Definition~\ref{def:fam-pca} (this observation occurs to my knowledge first
explicitly in Hofstra's
\cite{hofstra2006all}). We now take the following point of view:
\begin{itemize}
 \item Fibrations on $\catset$ are `generalized categories/preorders'.
 \item In presence of choice, exact completion of the total categories followed
by gluing along the appropriate inclusion functor is a cocompletion operation on
fibrations which can be viewed as fibrational analogue of
$\catc\mapsto\widehat{\catc}$ for small $\catc$ (more precisely, it is left
adjoint to the forgetful functor~\eqref{eq:forget-pretop-to-lex}).
\end{itemize}
 Without choice, the exact completion of the total category doesn't do the
job anymore, and we have to replace it by a construction that takes the
structure of the fibration explicitly into account, while at the same time
being closer intuitively to ideas about presheaves.

We start out by presenting an alternative reading of the exact completion of
$\Famf(\catc)$ ($\catc$ small with finite limits), which makes the link to
the presheaf construction directly visible.

The central idea is the \emph{fibration of sieves}\footnote{I think I first heard about this fibration when Streicher explained Shulman's `stack semantics'~\cite{shulman2010stack} to me.} on $\catc$, which is defined
as the pullback of the subobject fibration of $\widehat{\catc}$ along $Y$.
\begin{equation}\label{eq:fib-siev}
\vcenter{\xymatrix@R-3mm{
{\Siev(\catc)}\ar[r]\ar[d]_{\siev(\catc)}{\pullbackcorner}
&{\Sub(\widehat{\catc})}
\ar[d]^{\sub(\widehat{\catc})}\\
\catc\ar[r]^Y&{\widehat{\catc}}
}}
\end{equation}
We can reconstruct $\widehat{\catc}$ from $\siev(\catc)$ 
in the following way.  Given $F\in\widehat{\catc}$, we can cover it by
representables $\Sigma_i Y(C_i)\eepi F$, and the kernel pair 
\[U\rightarrowtail
(\Sigma_i Y(C_i))\times (\Sigma_i Y(C_i))\cong \Sigma_{i,j}Y(C_i\times C_j)\] 
of
this cover is determined by the sieves $U_{ij}\rightarrowtail Y(C_i\times C_j)$
obtained by restricting to the summands.  
\[
\xymatrix@R-2mm@C-2mm{
U_{ij} \mono[r]\mono[d]\pullbackcorner &
U \mono[d]\\
Y(C_i\times C_j)\mono[r]&\Sigma_{i,j}Y(C_i\times C_j)
}
\]
The $U_{ij}$ together represent an
equivalence relation on $\Sigma_{i,j}Y(C_i\times C_j)$, and it turns out that
we can express this fact without actually referring to the coproduct -- we
can write the conditions indexwise in the form
\begin{align*}
  i,j,k\in I&\imp U_{jk}\circ U_{ij}\subseteq U_{ik}\\
  i\in I &\imp \id_{A_i}\subseteq U_{ii}.
\end{align*}
Following a suggestion by Alex Simpson, we call a system $(U_{ij})_{i,j\in I}$
of predicates that fulfills these conditions a \emph{heterogeneous equivalence
  relation}\index{heterogeneous equivalence relation}\index{equivalence relation!heterogeneous}
(compare~\cite[Lemma~11.6]{awodey2007relating}). 

We now get our representation of $\widehat{\catc}$ by taking families $(C_i)_i$
of objects in $\catc$ equipped with heterogeneous equivalence relations
$(U_{ij})_{ij}$ in $\siev(\catc)$ as objects, and an appropriate version of
`heterogeneous functional relations' as morphisms.  

If we want to internalize
our handling of families, we can express the heterogeneous relations as
ordinary equivalence relations and functional relations in the fibration
\begin{equation}\label{eq:fam-siev}
\Famf(\siev(\catc)):\tot{\Famf(\siev(\catc))}\to\Famf(\catc)
\end{equation}
whose predicates on a family $(C_i)_{i\in I}$ are just $I$-indexed families of
sieves.  Now the somewhat surprising observation is that in presence of choice, this fibration is
equivalent to the posetal reflection of the fundamental fibration
\[\lfund{\Famf(\catc)}\] -- a morphism
$(u,(f_j)_j):(D_j)_j\to(C_i)_i$ can be viewed as associating to each $i$ a
family of morphisms in $\catc$ with codomain $C_i$, and these morphisms
generate a sieve on $C_i$. Here, we get the link to the exact completion, which
can be described as the category of equivalence relations and functional
relations in the posetal reflection of the fundamental fibration for any
category with finite limits.

\subsection{The fibration of sieves}

\begin{definition}\label{def:fibration-of-sieves}
  Let $\fibc:\tot{\fibc}\to\catr$ be a finite limit pre-stack on a regular
  category $\catr$. The \emph{fibration of sieves}\index{fibration!of sieves} on
$\fibc$ is the fibered
preorder $\siev(\fibc)$ on
  $\tot{\fibc}$, where
  \begin{itemize}
  \item predicates on $X\in\tot{\fibc}$ are morphisms $f:Y\to X$ in
    $\tot{\fibc}$
  \item $f\leq g$ over $X$ if there exist $e,k$ with $e$ cover-cartesian such
that
    \[
    \xymatrix{
      Y'\ar[r]_k\epicart[d]_e &  Z\ar[d]^g \\
      Y\ar[r]^f & X \\
    }
    \]
    commutes.
  \end{itemize}
\end{definition}
\begin{remark}\label{rem:ord-mono}
\begin{itemize}
 \item 
 If $g$ is a monomorphism in the previous definition, then $f\leq
g$ iff $f$ factors through $g$, since cover-cartesian maps in
$\fibc$ are regular epis by Lemma~\ref{lem:covercart-regepi}, and thus left
orthogonal to monomorphisms.
\item For a small finite limit category $\catc$, the fibration
$\siev(\famf(\catc))$ is equivalent to the fibration $\Famf(\siev(\catc))$ from~\eqref{eq:fam-siev}.
\end{itemize}
\end{remark}
\begin{lemma}\label{lem:siev-fibc-existential}
For any finite limit pre-stack $\fibc$, $\siev(\fibc)$ is an existential
fibration in the sense of Definition~\ref{def:existential-fibration}.
\end{lemma}
\begin{proof}
  Conjunction is given by pullback and existential quantification
  by postcomposition.
\end{proof}
\begin{remark}\label{rem:covercert-surj}
Cover-cartesian maps are surjective from the point of view of $\siev(\fibc)$.
More precisely, if $e:A\eepicart B$ is cover-cartesian in $\tot{\fibc}$, then
$b\csep\ent\exists a\qdot ea=b$ holds in $\siev(\fibc)$. This follows from the
fact that existential quantification is given by postcomposition.
\end{remark}

The following easy lemma will be very useful later.
\begin{lemma}\label{lem:factor-covercart}
Given a diagram
\[
\vcenter{\xymatrix@-2mm{
A^*\epicart[d]_e\ar[rd]^f\\
A\dashed[r]_h & B
} }
\]
in $\tot{\fibc}$ where $\fibc$ is a finite limit pre-stack and $e$ is
cover-cartesian, there exists a mediator $h$ iff $ex=ey\ent fx=fy$ in
$\siev(\fibc)$.
\end{lemma}
\begin{proof}
 Since $e$ is a regular epimorphism by Lemma~\ref{lem:descent-pstack}, it
suffices to show that the kernel of $e$ is contained in the kernel of $f$. This
is equivalent to $ex=ey\ent fx=fy$ by Remark~\ref{rem:ord-mono}.
\end{proof}

\subsection{The fibered presheaf construction}\index{fibered!presheaf construction}

Using the fibration of sieves, we can define presheaves on $\fibc$ as
`quotients of sums of representables', or rather as formal quotients of objects
in $\tot{\fibc}$ with respect to equivalence relations in $\siev(\fibc)$.
\begin{definition}\label{def:rc}
Given a finite limit
  pre-stack $\fibc:\tot{\fibc}\to\catr$, the category $\srelrc$ is the full
subcategory of $\per(\siev(\fibc))$ on \emph{total} equivalence relations.
\end{definition}
\begin{lemma}\label{lem:whfc-exact}
For any finite limit pre-stack $\fibc:\tot{\fibc}\to\catr$, the category
$\srel{\catr}{\fibc}$ is equivalent to $\per(\siev(\fibc))$ and therefore
in particular exact.
\end{lemma}
\begin{proof}
We have to show that every partial equivalence relation in $\siev(\fibc)$ is
equivalent to a total one in the sense of $\per(\siev(\fibc))$.
Given $\cro$ in the latter category, the predicate $\ilbracks{c\csep
\rho(c,c)}\in\siev(\fibc)_C$ is a morphism $f:D\to C$. If we define
$\sigma\in\siev(\fibc)_{D\times D}$ by $\sigma(x,y)=\rho(fx,fy)$ then
$\sigma$ is a total equivalence relation, and furthermore $\dsi$ is isomorphic
to $\cro$.
\end{proof}

Thinking about presheaves as quotients of sums of representables, there is
another way to represent morphisms besides functional relations in the
fibration of sieves -- thinking non-fibered for the moment, if $F$ and $G$ are
presheaves covered by families of representables $(YA_i)_{i\in I}$ and
$(YB_j)_{j\in J}$, then since representables are indecomposable and projective,
the restriction $YA_i\to G$ of a morphism $f:F\to G$ to some $YA_i$ factors
through some $YB_j$. 
\[
 \vcenter{\xymatrix@R-3mm{
YA_i\incl[d]\dashed[r]^h& YB_j\incl[d]\\
F\ar[r] &G
}}\qquad \forall i\;\exists j,h
\]
This implies that in the presence of choice, morphisms
from $F$ to $G$ can be represented by maps $u:I\to J$ and families $(f_i:A_i\to
B_{ui})$ which are compatible with the equivalence relations. The next
lemma does this in the fibrational setting. In the absence of choice in the
base, however, while we can assert the existence of $j$ and $f_i:A_i\to B_j$ for
a given $i\in I$, we can not actually choose one\footnote{Unless $G$ is the
coproduct of the family $(YB_J)_J$, see
Lemma~\ref{lem:tracking-fam}-\ref{lem:tracking-fam-strict} below.}. This can be
taken care
of by working with spans $I \twoheadleftarrow\bullet \to J$ instead
of maps $I\to J$; which motivates the following definition.
\begin{definition}\label{def:tracking-family}
  Let $\phi:\cro\to\dsi$ in $\srel{\catr}{\fibc}$. We call a span
\[\xymatrix@1@C-1.5mm{C &\epicart[l]_eC^*\ar[r]^f & D}\]
with $e$ cover-cartesian
 a \emph{tracking family}\index{tracking!family} for $\phi$, if one of the
following
  equivalent conditions holds in $\siev(\fibc)$.
  \begin{itemize}
  \item $\phi(c,d)\adj\ent\exists x\qdot ex=c\wedge \sigma(fx,d)$
  \item $\sigma(fx,d)\ent\phi(ex,d)$
  \item $\phi(ex,d)\ent\sigma(fx,d)$
  \item $\ent\phi(ex,fx)$
  \end{itemize}
A \emph{strict}\index{strict!tracking family}\index{tracking family!strict} tracking family is a tracking family where the cover-cartesian
part is an identity.
\end{definition}
The first characterization shows how $\phi$ can be reconstructed from
$(e,f)$. Thus, for fixed $\rho$ and $\sigma$, $(e,f)$ can be the tracking
family of at most one morphism.

\begin{lemma}\label{lem:tracking-fam}
  \begin{enumerate}
  \item\label{lem:tracking-fam-is} A span 
$\xymatrix@1@C-1.9mm{C &
      \epicart[l]_eC^*\ar[r]^f & D}$ 
is a tracking family of some morphism from
    $(C,\rho)$ to $(D,\sigma)$ in $\srel{\catr}{\fibc}$ iff
$\rho(ex,ey)\ent\sigma(fx,fy)$ holds.
\item\label{lem:tracking-fam-equal} Two spans $(e,f)$, $(e',f')$ are tracking
families of the \emph{same} morphism from
    $(C,\rho)$ to $(D,\sigma)$  iff $\rho(ex,e'y)\ent\sigma(fx,f'y)$ holds.
  \item\label{lem:tracking-fam-exist} All morphisms in $\srel{\catr}{\fibc}$
have tracking families.
\item\label{lem:tracking-fam-strict}
Morphisms of type $(C,\rho)\to(D,\predeq)$ (where the equivalence
relation in the image is discrete), have \emph{strict} tracking families, i.e.,
morphisms $f:C\to D$ such that $\rho(c,c')\ent fc=fc'$. The latter
judgment is equivalent to the fact $f$ coequalizes the components
of $\rho:R\to C\times C$.
\item\label{lem:tracking-fam-strict-two}
For every morphism $\phi:\cro\to\dsi$ we can find an
isomorphism  $\iota:(C',\rho')\eiso\cro$ such that $\phi\iota$ has a strict
tracking family.
  \end{enumerate}
\end{lemma}

\begin{proof}
  \emph{Ad \ref{lem:tracking-fam-is}.} Straightforward.

\medskip

\emph{Ad \ref{lem:tracking-fam-equal}.} Assume that
$\rho(ex,e'y)\ent\sigma(fx,f'y)$ holds. We first show that $(e,f)$ (and
therefore by symmetry also $(e',f')$) \emph{is} a tracking family. Using
\ref{lem:tracking-fam-is}, we have to show that $\rho(ex,ey)\ent\sigma(fx,fy)$.
To do this, we argue informally in $\siev(\fibc)$. Assume $\rho(ex,ey)$. By
Remark~\ref{rem:covercert-surj}, there exists $z$ in the domain of $e'$ such
that $ex=e'z$. Since $\rho$ is total, we have $\rho(ex,e'z)$, which implies by
assumption that $\sigma(fx,f'z)$. Substituting $e'z$ in $\rho(ex,ey)$, we have
$\rho(e'z,ey)$, which implies (using symmetry) that $\sigma(f'z,fy)$. Together
we have $\sigma(fx,fy)$.

The functional relation associated to a span $(e,f)$ is given by
$\ilbracks{c,d\csep \exists x\qdot ex=c\wedge \sigma(fx,d)}$, therefore to show
that
$(e,f)$ and $(e',f')$ are equivalent, it remains to show that $\exists
x\qdot ex=c\wedge \sigma(fx,d)\ent \exists y\qdot e'y=c\wedge \sigma(f'y,d)$.
Again, we reason informally in $\siev(\fibc)$. Assume that $ex=c$ and
$\sigma(fx,d)$. There exists $y$ such that $ex=c=e'y$. It remains to show that
$\sigma(f'y,d)$, which by symmetry, transitivity and assumption is equivalent to
$\sigma(fx,f'y)$. This follows from $\rho(ex,e'y)$, which holds because of
totality of $\rho$.

\medskip

  \emph{Ad \ref{lem:tracking-fam-exist}.}  Given a functional relation
  $\tot{\phi}\xrightarrow{\phi}C\times D$ between $(C,\rho)$ and $(D,\sigma)$,
  the totality judgment means precisely that there exist $e$ and $h$ making
\[
\vcenter{\xymatrix{
C^*\epicart[d]_e\ar[rd]^h\\
C & \ar[l]^{\phi_1}\tot{\phi}\ar[r]_{\phi_2} & D\\
}}.
\] 
commute, and the desired span is obtained by setting $f=\phi_2h$.

\medskip

\emph{Ad \ref{lem:tracking-fam-strict}.} 
If the span $(e,f)$ tracks a morphism from $\cro$ to $(D,=)$, we have
$\rho(ex,ey)\ent fx=fy$ by \ref{lem:tracking-fam-is}, thus in particular
$ex=ey\ent fx=fy$. By Lemma~\ref{lem:factor-covercart}, there exists thus
$h:C\to D$ such that $he=f$. Using \ref{lem:tracking-fam-equal} one can verify
that the span $(\id_C,h)$ tracks the same morphism as $(e,f)$.

The fact that $\rho(c,c')\ent fc=fc'$ iff $f\rho_1=f\rho_2$ follows from
Remark~\ref{rem:ord-mono}.

\medskip

\emph{Ad \ref{lem:tracking-fam-strict-two}.} Given $\phi:\cro\to\dsi$ with
tracking family $\xymatrix@1@C-1.5mm{C &\epicart[l]_eC^*\ar[r]^f & D}$, define
an equivalence relation on $C^*$ by $\rho'(x,y)\equiv \rho(ex,ey)$. Then
$(C^*,\rho')\cong\cro$, and $f$ is a strict tracking family of $\phi$ composed
with the isomorphism.
\end{proof}

$\srel{\catr}{\fibc}$ will be the fiber of the fibration of presheaves on
$\fibc$ over the
terminal object. To get the entire fibration, we have to define a functor along
which we can glue.
\begin{definition}
  \begin{enumerate}
  \item The functor $\Delta:\catr\to\srel{\catr}{\fibc}$ sends $I\in\catr$  to
    $(1_I,\predeq)$, i.e.\ the terminal object in $\fibc_I$ equipped with the
    discrete equivalence relation.

    $u:I\to J$ in $\catr$ is mapped to
    $\ilbracks{x\vtp 1_I,y\vtp 1_J\csep 1_ux=y}$ by $\Delta$, where $1_u:1_I\to
1_J$ is
    the unique map over $u$ of this type in $\tot{\fibc}$.
\item We define the fibration $\whfc$ of presheaves\index{fibration!of
presheaves} on $\fibc$ by
  $\whfc=\gl_\Delta(\srel{\catr}{\fibc})$.
\begin{equation*}
 \vcenter{\xymatrix{ 
\tot{\whfc}\pullbackcorner\ar[r]\ar[d]_{\whfc}
&
 \commacat{\srel{\catr}{\fibc}}{\srel{\catr}{\fibc}}
\ar[d]^{\fund{\srel{\catr}{\fibc}}}
\\ 
\catr\ar[r]^\Delta&\srel{\catr}{\fibc} }}
\end{equation*}
  \end{enumerate}
\end{definition}
\begin{lemma}
$\Delta:\catr\to\srel{\catr}{\fibc}$ is regular, full, and faithful and reflects
regular epimorphisms.
\end{lemma}
\begin{proof}
It is clear that $\Delta$ preserves finite products. To see that it preserves
equalizers, consider $u,v:I\to J$ in $\catr$.

The equalizer of $\Delta u$, $\Delta v$ in $\srel{\catr}{\fibc}$ is represented
by the
predicate $\ilbracks{x\vtp 1_I\csep 1_ux=1_vx}$ in $\siev(\fibc)$, which as a
morphism
in $\tot{\fibc}$ is exactly the equalizer of $u$ and $v$. Thus it remains to
check that the functor $1:\catr\to\tot{\fibc}$, $I\mapsto 1_I$ preserves
equalizers. This is easy to see.

To show that $\Delta$ is preserves and reflects regular epimorphisms, let
$u:I\to J$ in
$\catr$. From Lemma~\ref{lem:whfc-exact} we know that $\Delta u$ is a regular
epimorphism iff \eqref{eq:judg-surj} holds, and substituting definitions we see
that this judgment is concretely given by
$
y\vtp 1_J\csep\ent \exists x\vtp 1_I\qdot 1_u x=y
$.
Now the interpretation of the formula $\ilbracks{y\vtp 1_J\csep \exists x\vtp
1_I\qdot
1_u x=y}$ is $1_u$ viewed as a predicate in $\siev(\fibc)_{1_J}$, and by
the definition of the fibration of sieves, $1_u$ is
equivalent to $\id_J$ (the true predicate) iff there exists a cover-cartesian
$e:\xymatrix@1{X \epicart[r] & 1_J}$ which factors through $1_u$. This is in
turn equivalent to the existence of an epimorphism $p$ into $J$ which factors
though $u$, which is equivalent to $u$ being epic.

To see that $\Delta$ is faithful, let $u,f:I\to J$ in $\catr$. We have $\Delta
u=\Delta v$ iff $\ent 1_ux=1_vx$ holds iff the equalizer of $1_u$ and $1_j$ is
equivalent to $\id_I$ in $\siev(\fibc)_{1_I}$. By Remark~\ref{rem:ord-mono},
this implies that the equalizer is an isomorphism.

To see that $\Delta$ is full, let $\phi:\Delta I\to\Delta J$.
By
Lemma~\ref{lem:tracking-fam}-\ref{lem:tracking-fam-strict}, there exists a map
$f:1_I\to 1_J$ in $\tot{\fibc}$ such that $x\vtp 1_I\csep\ent \phi(x,fx)$. The
inverse image of $\phi$ is given by $\fibc f$.
\end{proof}
We can express the fibers of $\whfc$ directly in terms of $\siev(\fibc)$, 
without using the gluing construction:
\begin{lemma}\label{lem:whfci-by-localization}
\begin{enumerate}
 \item \label{lem:whfci-by-localization-loc}
  For $I\in \catr$, we have an equivalence
\[
 \whfc_I\stackrel{\mathrm{def}}{=}\srel{\catr}{\fibc}/\Delta
I\simeq\srel{(\catr/I)}{\fibc/I},
\]
of categories, where $\fibc/I$ is the localization of $\fibc$ to
$I$, introduced in Definition~\ref{def:localization}.
\item\label{lem:whfci-by-localization-slice}
More generally, given $C\in\fibc_I$, we have
\[
 \srel{\catr}{\fibc}/(C,\predeq)\simeq\srel{(\catr/I)}{\fibc/C},
\]
where $\fibc/C:\tot{\fibc/C}\to\catr/I$ is the slice of $\fibc$ over
$C$ (see Lemma~\ref{lem:slice-fibration}).
\end{enumerate}
\end{lemma}
\begin{proof}
We only show the second claim, the first one is a special case by
Lemma \ref{lem:local-terminal-slice}. Objects in $\srel{\catr}{\fibc}/\Delta
C$ are equivalence relations $\sigma\in\siev(\fibc)_{D\times D}$ together with
morphisms $\phi:\dsi\to(C,\predeq)$,  which by
Lemma~\ref{lem:tracking-fam}-\ref{lem:tracking-fam-strict} correspond to
morphisms $f:D\to C$ such that $f\sigma_1=f\sigma_2$. Objects in
$\srel{(\catr/I)}{\fibc/C}$, on the other hand, are equivalence relations in
$\siev(\fibc/C)$, where concretely the underlying object is a map $f:D\to C$ and
the relation is a map $\sigma:S\to D\times_C D$ satisfying the axioms. The key
observation now is that an equivalence relation in $\siev(\fibc)$ on $D$ whose
components are equalized by $f$ is the as an equivalence relation in
$\siev(\fibc/C)$ on $f$. We leave it to the reader to extend the correspondence
to morphisms.
\end{proof}

The preceding lemma gives a streamlined representation of the fibration
$\whfc$, which we will use from now on. Let us spell it out for reference.
\begin{itemize}
 \item Objects in $\whfc_I$ are given by triples $(u:J\to I, C\in\fibc_J,
\rho:R\to C\times C)$ such that $\rho$ is an equivalence relation in
$\siev(\fibc)$ and factors through $C\times_I C$.
\item A morphism from $(u:J\to I,C,\rho)$ to $(v:L\to K,D,\sigma)$ over $w:
I\to K$ is a span $\xymatrix@1@C-1.5mm{C & \epicart[l]_e\bullet\ar[r]^f & D}$
such that $w\,u\,\fibc(e)=v\,\fibc(f)$ and $\rho(ex,ey)\ent\sigma(fx,fy)$ in
$\siev(\fibc)$.
\end{itemize}

This description allows us to give an easy definition of the Yoneda embedding.
\begin{definition}\label{def:fib-yoneda}
  The \iemph{fibered Yoneda embedding} $Y:\fibc\to\whfc$ is given by
\[
 \fibc_I\ni C\mapsto (\id_I, C,\predeq)\in\whfc_I.
\]
\end{definition}
\begin{lemma}\label{lem:yoneda}
  \begin{enumerate}
  \item\label{lem:yoneda-fff} $Y$ is full and faithful.
  \item\label{lem:yoneda-univ-prop} Given a fibered pretopos
$\fibx:\tot{\fibx}\to\catr$, precomposition by
    $Y$ induces an equivalence
    \[
    \catlex(\catr)(\fibc,\fibx)\simeq\catpretop(\catr)(\whfc,\fibx),
    \]
    where $\catlex(\fibc,\fibx)$ is the category of fibered finite limit
    preserving functors, and $\catpretop(\whfc,\fibx)$ is the category of
    fibered regular functors preserving internal sums.
  \end{enumerate}
\end{lemma}
\begin{proof}
\emph{Ad \ref{lem:yoneda-fff}.} It is sufficient to show that $Y$ is fiberwise
full and faithful. For $I\in \catr$, $Y_I$ can be decomposed as follows.
\begin{equation*}
 \xymatrix@1{
\fibc_I\ar[r]^-\simeq&(\fibc/I)_1\ar[rr]^-{C\mapsto
(C,\predeq)}&&\srel{(\catr/I)}{\fibc/I}
}
\end{equation*}
The first part is an equivalence, thus it suffices to show that for arbitrary
finite limit pre-stacks $\fibd:\tot{\fibd}\to\cats$, the embedding $\fibd_1\to
\srel{\cats}{\fibd}$ is full and faithful. It is easy to see that the embedding
is faithful; fullness is a consequence of
Lemma~\ref{lem:tracking-fam}-\ref{lem:tracking-fam-strict}, since every
morphism between two objects in the terminal fiber is vertical.

\emph{Ad \ref{lem:yoneda-univ-prop}.}
To see that precomposition with $Y$ is faithful, consider fibered functors $F,
G:
\whfc\to \fibx$ which are fiberwise regular and preserve cocartesian
morphisms. Assume that $\eta, \theta:F\to G$ are fibered natural transformations
such that $\eta\circ Y=\theta\circ Y$. Let $I\in\catr$ and $(u,C,\rho)$ in
$\widehat{\fibc}_I$.

We can cover $(u,C,\rho)$ with an object in the image of $Y$ as,
\begin{equation*}
\xymatrix@1{YC=(\id_J, C, \predeq)\coca[r]^-s & (u,C, \predeq) \depi[r]^e &
(u,C, \rho)},
\end{equation*}
where $s$ is cocartesian and $e$ is regular epic in $\widehat{\fibc}_I$.
Applying $\eta$ and $\theta$ to this sequence, we obtain
\begin{equation*}
\vcenter{\xymatrix@C+2mm{
FYC\coca[r]^{Fs}\ar[d]_{\eta_{YC}=\theta_{YC}} &
F(u, C,  \predeq)\ppair{d}{\theta_{(u,C,\predeq)}}{\eta_{(u,C,\predeq)}}
\depi[r]^{Fe} &
F(u,C,\rho)\ppair{d}{\theta_{(u,C,\rho)}}{\eta_{(u,C,\rho)}} \\
GYC\coca[r]_{Gs} & G(u,C,\predeq) \depi[r]_{Ge} & G(u,C,\rho)
}}.
\end{equation*}
Now the facts that $Fs$ is
cocartesian and the components
of $\eta$ and $\theta$ are
vertical imply that ${\eta_{(C,u,\predeq)}}={\theta_{(C,u,\predeq)}}$, 
and since $Fe$ is regular epic in $\fibx_I$, this furthermore implies that
${\eta_{(C,u,\rho)}}={\theta_{(C,u,\rho)}}$.

To see that precomposition by $Y$ is full, consider $F,G:\whfc\to\fibx$ and
$\eta:FY\to GY$. We have to construct $\tilde{\eta}:F\to G$ such that
$\tilde{\eta}Y=\eta$. Let $(u,C,\rho)\in\whfc_I$ as before. Consider the
following diagram in ${\whfc}$.
\begin{equation}\label{eq:dia-coeq-ucro}
\vcenter{\xymatrix@R-3.5mm{
YR\coca[rrr]\ar[rdd]_{Y\rho} & & & (w,R,\predeq)\depi[d]^p\\
& & & \bullet\mono[d]^r\\
& Y(C\times_I C)\coca[rr]\ppair{dr}{\pil}{\pir}\ & & (u,
C,\predeq)^2\ppair{d}{\pil}{\pir} \\
K\ar[rdd]_v\ar[rrrdd]^w & & YC\coca[r] & (u,C,\predeq)\depi[d]^e\\
& & & (u,C,\rho)\\
& J\times_I\ppair{rd}{\pil}{\pir} J\ar[rr]|{u\times_Iu} & & I \\
& & J\ar[ru]|u
}}
\end{equation}
Applying $F$, $G$, and $\eta$, we get:
\begin{equation*}
 \xymatrix@R-6.5mm@C-2mm{
&GYR\coca[rr]\ar[dddd]|(.265){\hole\hole}_{GY\rho} &&
G(w,R,\predeq)\depi[dd]^{Gp}\\
FYR\ar[dddd]_{FY\rho}\coca[rr]\ar[ur]^{\eta_R} && F(w,R,\predeq)\depi[dd]^{Fp}
\dotted[ru]_{\tilde{\eta}_{(w,R,\predeq)}}\\
&& & G\bullet\mono[dd]^{Gm}\\
& & F\bullet\mono[dd]^(.25){Fm}\dotted[ru]_h\\
&G(C\times_I
C)\coca[rr]|{\hole\hole}{\ar@<3pt>[dd]|{\phantom{\big[}}\ar@<-3pt>[dd]|{
\phantom{\big[}} } && G(u,C,\predeq)^2\ppair{dd}{}{}\\
F(C\times_I C)\ar[ur]^{\eta_{C\times_I C}}\coca[rr]\ppair{dd}{}{} &&
F(u,C,\predeq)^2\ppair{dd}{}{}\dotted[ru]_{\tilde{\eta}_{(u,C,\predeq)^2}}\\
&GY C\coca[rr]|{\hole\hole\hole} && G(u,C,\predeq)\depi[dd]^{Ge}\\
FYC\ar[ur]^{\eta_C}\coca[rr] &&
F(u,C,\predeq)\depi[dd]^{Fe}\dotted[ru]_{\tilde{\eta}_{(u,C,\predeq)}}\\
&&& G(u,C,\rho)\\
&& F(u,C,\rho)\dotted[ru]_{\tilde{\eta}_{(u,C,\rho)}}\\
}
\end{equation*}
Here, the arrows $\tilde{\eta}_{(w,R,\predeq)}$,
$\tilde{\eta}_{(u,C,\predeq)^2}$, and $\tilde{\eta}_{(u,C,\predeq)}$ are
cocartesian liftings, $h$ exists since the regular epimorphism $Fp$ is left
orthogonal to the monomorphism $Gm$ in $\fibx_I$, and $\eta_{(u,C,\rho)}$ exists
since $\eta_{(u,C,\predeq)}$ is compatible with the kernels $Fm$ and $Gm$  of
$Fe$ and $Ge$, which is expressed by the existence of $h$.

To see that $\widetilde{\eta}$ is natural, consider a morphism
$\phi:(u,C,\rho)\to (v,D,\sigma)$ over $t:I\to K$ in ${\widehat{\fibc}}$. By
\ref{lem:tracking-fam}-\ref{lem:tracking-fam-exist}, $\phi$ has a tracking
family
$\xymatrix@1@C-1.5mm{C &
    \epicart[l]_eC^*\ar[r]^f & D}$, and since $\phi$ is
over $w$, we have $t\,u\,\fibc(e)=v\,\fibc(f)$. Set $w \,:=\, u\,\fibc(e)$.
\[
\xymatrix@R-3mm{
C & \epicart[l]_e C^*\ar[r]^f & D \\
J\ar[d]_u & \depi[l]J^*\ar[r]\ar[dl]_w & L\ar[d]^v \\
I\ar[rr]^t & & K
}
\]
Defining $\rho^*\in\siev(\fibc)_{C^*\times C^*}$ to be the predicate
$\rho^*=[x,y\csep \rho(ex,ey)]$, we have $(u,C,\rho)\cong(w,C^*,\rho^*)$.
Thus, to show commutativity of the outer rectangle in the diagram
\[
 \xymatrix@R-3mm{
F(u,C,\rho)\ar[d]_{\tilde{\eta}} &\ar[l]_\cong
F(w,C^*,\rho^*)\ar[d]_{\tilde{\eta}}\ar[r]& F(v,D,\sigma)\ar[d]^{\tilde{\eta}}\\
G(u,C,\rho) &\ar[l]_\cong G(w,C^*,\rho^*)\ar[r]& G(v,D,\sigma)
}
\]
in $\tot{\fibx}$, it suffices to show that the left and right squares commutes,
i.e.\ we have reduced the problem of showing naturality to checking the
condition for morphisms which have tracking families with identity
cover-cartesian
part.
Specifically, we show that the right square in the preceding diagram commutes,
and the left one is analogous. In the diagram
\begin{equation}
 \vcenter{\xymatrix@-7mm{
FYC^*\coca[dd]_c\ar[rd]^\eta \ar[rr] && FYD\coca[dd]\ar[rd]^\eta \\
& GYC^*\coca[dd] \ar[rr] && GYD\coca[dd] \\
F(w,C^*,\predeq)\depi[dd]_p\ar[rr]|(.515){\hole}\ar[rd] &&
F(v,D,\predeq)\depi[dd]|(.5)\hole\ar[rd]\\
&G(w,C^*,\predeq)\depi[dd]\ar[rr] && G(v,D,\predeq)\depi[dd]\\
F(w,C^*,\rho^*)\ar[rr]|(.515)\hole\ar[rd]_{\tilde{\eta}}&&F(v,D,\sigma)\ar[rd]^{
\tilde { \eta }
}\\
&G(w,C^*,\rho^*)\ar[rr]&&G(v,D,\sigma)
}}, 
\end{equation}

we know that the top face and the sides commute, and we have to show
commutativity of the bottom face. Because $p$ is epic, the two paths around the
bottom face are equal iff they are equal precomposed with $p$, and since these
two composites have the same image under $\fibx$, it suffices to check their
equality when moreover precomposed with $c$ (by the universal property of
cocartesian arrows). This equality follows from the commutativity of the top
and side faces. 

This finishes the proof that precomposition with $Y$ is full. 

\medskip

Finally, we show that precomposition by $Y$ is essentially surjective. Let
$F:\fibc\to\fibx$ be a finite limit preserving fibered functor. The idea of how
to extend $F$ to $\widetilde{F}$ along $Y:C\to\whfc$ 
\[
\xymatrix{
{\fibc}\ar[d]_Y\ar[rd]^F & \\
{\whfc}\ar[r]_{\widetilde{F}}\emar[ur]|(.3)\cong & \fibx
} 
\]
is contained in
diagram~\eqref{eq:dia-coeq-ucro}, which presents an object $(u,C,\rho)$ in
$\fibc_I$ as a quotient of an internal sum of a representable,
where the associated equivalence relation is itself the image (in the sense of
image factorization) of the sum of a representable. Accordingly, we construct
$\widetilde{F}(u,C,\rho)$ as in
\[
 \xymatrix@R-3mm{
FR\coca[rr]\ppair{rdd}{\rho_1}{\rho_2} & & {\Sigma_w} FR\depi[d] \\
&&{\bullet}\ppair{d}{r_1}{r_2}\\
 & FC \coca[r] & \Sigma_u FC\depi[d]^e  \\
K\ppair{rd}{}{v_1,v_2}\ar[rrd]^w & & {}\widetilde{F}(u,C,\rho)\\
 & J\ar[r]_u & I
}
\]
where $r=\langle r_1,r_2\rangle$ is an equivalence relation obtained by
cocartesian lifting and image factorization, and $e$ is its
quotient.
\end{proof}

In the non-fibered case, categories of presheaves on small categories can be
characterized as $\infty$-pretoposes having a small generating family of
indecomposable projectives. We will prove the fibered analogue of this
statement after defining fibrational versions of projectivity and
indecomposability.
\begin{definition}\label{def:proj-indec}
 Let $\fibp:\tot{\fibp}\to\catr$ be a positive pre-stack.
\begin{enumerate}
 \item \label{def:proj-indec-proj}
 We call $P\in\tot{\fibp}$ \emph{projective}\index{projective object in a
fibration} (with
respect to $\fibp$), if 
given $c,e,f$ as in the diagram
\[
\xymatrix@R-2mm{
&\bullet\epicart[d]_d\dotted[r]_g & Y\depi[d]^e \\
P & \cart[l]_c P^*\ar[r]^f & X
}
\]
where $c$ is cartesian and $e$ is vertical and a regular epimorphism in its
fiber, we can fill in $d,g$ with $d$ cover-cartesian such that the square
commutes.
\item\label{def:proj-indec-indec}
    We call $X\in\tot{\fibp}$ \emph{indecomposable}\index{indecomposable object in a fibration}, if for every
  diagram
\[
\vcenter{\xymatrix{
X^*\cart[r]^c\ar[dr]\dotted[d]_m & X\\
Y\coca[r]_s & S
}}
\]
where $c$ is cartesian and $s$ is cocartesian, there exists a \emph{unique}
mediating arrow $m$. 
\end{enumerate}
\end{definition}
Observe that the definitions are given in a way which assures that
indecomposables and projectives are closed under reindexing. 
\begin{lemma}$ $\label{lem:chr-psh}
  \begin{enumerate}
\item\label{lem:chr-psh-yip}
    Let $\fibc:\tot{\fibc}\to\catr$ be a finite limit pre-stack.
    \begin{enumerate}
    \item\label{lem:chr-psh-yip-a}
 The objects in the image of
$Y:\fibc\to\whfc$ are
      indecomposable projectives in $\whfc$.
    \item\label{lem:chr-psh-yip-b} Given an indecomposable
projective $A\in\tot{\whfc}$, there exists a
      $C\in\tot{\fibc}$ and a cover-cartesian map $e:\xymatrix@1{YC\epicart[r]&
        A}$. 

In other words, the co-restriction of $Y$ to the subfibration of
$\whfc$ on indecomposable projectives is a weak equivalence
(Definition~\ref{def:weak-equivalences}).
    \end{enumerate}
  \item \label{lem:chr-psh-characterization}
    Let $\fibx:\tot{\fibx}\to\catr$ be a fibered pretopos.

    $\fibx$ is equivalent to the fibration of presheaves on its subfibration
of indecomposable
    projectives iff the latter is closed under finite limits, and  every
$X\in\tot{\fibx}$ 
    can be covered as
\[
\xymatrix{ A\coca[r]^-s&S\depi[r]^-e&X}
\]
where $A$ is indecomposable and projective, $s$ is cocartesian, and $e$ is
vertical and a regular epimorphism in its
fiber.
  \end{enumerate}
\end{lemma}
\begin{proof}
\emph{Ad \ref{lem:chr-psh-yip-a}.} Since indecomposability is
about cocartesian maps, we first have to understand what those look like in
${\widehat{\fibc}}$. It turns out that this is very easy -- given $u:J\to
I$, $(u,C,\rho)\in \widehat{\fibc}_I$, and $v:I\to H$, the internal sum of 
$(u,C,\rho)$ along $v$ is simply $(vu,C,\rho)$.
Now consider $D\in \fibc_L$ and a morphism of type
$(\id_L,D,\predeq)\to(vu,C,\rho)$, given by a map $w:L\to H$ and a span
$\xymatrix@1@C-1.5mm{D & \epicart[l]_eD^*\ar[r]^f & C}$
such that $vuh=wp$ and $ex=ey\ent \rho(fx,fy)$ in $\siev(\fibc)$. 
\[
 \xymatrix@R-2mm{
D & \epicart[l]_e D^*\ar[r]^f & C & \ppair{l}{\rho_2}{\rho_1} R \\
L\ar[rrdd]_w & \depi[l]_p L^*\ar[r]^h & J\ar[d]^u & \ppair{l}{r_2}{r_1} K \\
& & I\ar[d]^v\\
& & H
}
\]
Lifting
$(e,f)$ along $\xymatrix@1{(u,C,\rho)\coca[r]&(vu,C,\rho)}$ amounts to lifting
$w$ along $v$, or equivalently extending $uh$ along $p$. This is possible if
and only if the kernel pair of $p$ is contained in the kernel pair of $uh$,
which follows from $ex=ey\ent \rho(fx,fy)$ and $ur_1=ur_2$.

For projectivity, recall that
by~\ref{lem:decompo}-\ref{lem:decompo-coverrepr}, any regular epimorphism
in $\widehat{\fibc}_I$ can up to isomorphism be represented as
$(u,C,\rho)\eepi(u,C,\sigma)$ where $u:J\to I$, $C\in \fibc_J$, and
$\rho(x,y)\ent\sigma(x,y)$. Consider $D\in\fibc_L$ and a morphism of type
$(\id,C,\predeq)\to(u,C,\sigma)$ given by $w:L\to I$ and a span $(e,f)$
\[
 \vcenter{\xymatrix@R-2mm{
D & \epicart[l]_e D^*\ar[r]^f & C\\
L\ar[rrd]_w & \depi[l]_p L^*\ar[r]^h & J\ar[d]^u\\
& & I\\
}},
\]
where $ex=ey\ent\sigma(x,y)$. The desired lifting in $\widehat{\fibc}$ is given
in the following square.
\[
 \xymatrix@R-2mm{
(\id,D^*,\predeq)\ar[r]\epicart[d] & (u,C,\rho)\depi[d]\\
(\id,D,\predeq)\ar[r] & (u,C,\sigma)
}
\]

\emph{Ad~\ref{lem:chr-psh-yip-b}.}
Assume that $X\in\whfc_I$ is indecomposable and projective. Just like any
object in $\tot{\whfc}$, we can cover $X$ as
$\xymatrix@1{YC\coca[r]^c&\bullet\depi[r]^p&X}$, and since $X$ is indecomposable
and projective, there exists $e:\xymatrix@1{X^*\epicart[r]& X}$ and $f:X^*\to
YC$ such that $pcf=e$. Pulling $pc$ back along $e$, we obtain
\[
 \vcenter{\xymatrix{
YC^*\pullbackcorner\ar[d] \epicart[r] & YC\ar[d]_{pc}\\
X^* \epicart[r]^e\ar[ur]^f\ar@/^3mm/[u]^g & X
}}
\]
where $g$ is induced by $f$ and the pullback property and exhibits $X^*$ as a
retract of $YC^*$. Since $\fibc$ has finite limits and $Y$ preserves them, the
essential image of $Y$ is closed under retracts in $\tot{\whfc}$, which proves
the claim.

\medskip

\emph{Ad~\ref{lem:chr-psh-characterization}.} 
Clearly, the condition is necessary. Conversely, let $\fibc$ be the
subfibration of $\fibx$ on indecomposable
projectives and $F:\fibc\to\fibx$ the inclusion. We have to show that the
canonical map $\widetilde{F}:\whfc\to\fibx$ 
is an 
equivalence. By Moens' theorem, it is sufficient to show this for the functor
$\widetilde{F}_1:\srel{\catr}{\fibc}\to \fibx_1$ between the terminal fibers.
Consider the following diagram.
\begin{equation}\label{eq:two-pullbacks}
 \xymatrix{
\Siev(\fibc)
\pullbackcorner
\ar[d]_{\siev(\fibc)}
\ar[r]^H 
&
\Sub(\fibx)
\pullbackcorner
\ar[d]^{\sub(\fibx)}
\ar[r]^K 
&
\Sub(\fibx_1)
\ar[d]^{\sub(\fibx_1)}
\\
\tot{\fibc}
\ar[dr]_\fibc
\ar[r]^F
&
\tot{\fibx}\ar[r]^\Sigma
\ar[d]^\fibx
 &
\fibx_1
\\
& 
\catr
}
\end{equation}
Here, $H$ maps sieves on $C\in\fibc_I$ onto the image of their cocartesian
lifting
\[
 \vcenter{\xymatrix@R-4mm{
D\coca[r]\ar[ddr]_f & \bullet\depi[d]\\
& \bullet\mono[d]^{Hf}\\
& C\\
J\ar[r] & I
}},
\]
and $K$ is the action on vertical monomorphisms of the sum functor $\Sigma$.
$K$ is well defined since vertical monomorphisms are monomorphisms in the total
category, and $\Sigma$ preserves finite limits (and thus monos)
by~\cite[Lemma~15.7]{streicherfib}.

It follows from indecomposability and projectivity that $H$ is fiberwise
order-reflecting, and the assumption that every object in $\fibx$ can be covered
by
an indecomposable projective implies that $H$ is fiberwise essentially
surjective.
Thus, the left square in
diagram~\eqref{eq:two-pullbacks} is a pullback. The right square  is a pullback
by extensivity of $\fibx$. Furthermore, $F$ trivially preserves finite limits,
and $\Sigma$ preserves finite limits by~\cite[Lemma~15.7]{streicherfib}.

Now the functor $\widetilde{F}_1:\srel{\catr}{\fibc}\to \fibx_1$ can be
expressed as
applying $KH$ to an equivalence relation in $\siev(\fibc)$, and then forming
the quotient in the exact category $\fibx_1$. This functor is full and
faithful, since $H$ and $K$ are parts of pullbacks as established previously,
and thus preserve all logic, and furthermore $\fibx_1\simeq\per(\sub(\fibx_1))$.
$\widetilde{F}_1$ is essentially surjective since objects in $\fibx$ can be
covered by indecomposable projectives.
\end{proof}
\begin{remarks}
\begin{itemize}
 \item
As mentioned at the beginning of Section~\ref{sec:fibered-presheaves}, the
fibered presheaf construction is motivated by the
desire to make certain constructions involving exact completions independent of
the axiom of choice. Essentially the same question motivated
Hofstra's~\cite{hofstra2004relative}, but he uses a different approach:
instead of working in a fibrational framework, he constructs a left biadjoint to
the forgetful functor
\[
\catr/\catex\to\catr/\catlex,
\]
where the objects of the 2-categories are regular and finite limit preserving
functors, respectively, and in both cases the 1-cells are the respective
commutative triangles.
Via Moens' theorem, this can of course be understood as a biadjunction between
2-categories of lextensive fibrations and fibered pretoposes (apart from the
detail that Hofstra considers \emph{strict}\index{strict!slice category}, not \emph{pseudo-}slice
categories), but taking this point of view there is still a discrepancy since in
the present work we locate the transition from partitioned assemblies to
realizability toposes in the left biadjoint to
$\catpretop(\catr)\to\catlex(\catr)$ and not in the left biadjoint to
$\catpretop(\catr)\to\catlxv(\catr)$ as a reading of Hofstra's work through the
glasses of Moens would suggest.
\item
In recent work~\cite{shulman2012exact}, Michael Shulman presents a notion of
\emph{exact completion of unary sites} (a unary site is a small category
equipped with a Grothendieck topology which is generated by singleton
covering families) that is related to the fibered presheaf construction: given a
finite-limit pre-stack $\fibc:\tot{\fibc}\to\catr$, the cover-cartesian maps in
$\tot{\fibc}$ generate a unary topology, and the exact completion of the
corresponding unary site is equivalent to $\srelrc$.

This was observed by Wouter Stekelenburg in the case of realizability over pcas.
\end{itemize}
\end{remarks}

\subsection{The fibered geometric cocompletion}

This section is about the left biadjoint to
$
 \catgeo(\catr)\to\catlex(\catr)
$. In the non-fibered case, the geometric cocompletion of a small finite limit
category $\catc$ can be described as the full subcategory of $\widehat{\catc}$
on sub-representables. In the fibered case, we can use essentially the same
construction.

\begin{definition}\label{def:d-fib}
  Let $\fibc:\tot{\fibc}\to\catr$ be a finite limit pre-stack on a regular
  category. The fibration $D\fibc$ is defined to be the full subfibration of
$\whfc$ on (vertical) subobjects of representables. 
\[
\vcenter{\xymatrix@-4mm{
\fibc\ar[dr]^Y\ar[d]_y\\
D\fibc\ar[r] & {}\whfc
}}
\]
We use a lowecase `$y$' (read as `little Yoneda') to denote the embedding of
$\fibc$ into $D\fibc$, a
convention that appears particularly natural in the posetal case which we will
examine in more detail in Section~\ref{sec:preordered-case}.
\end{definition}

By Lemma~\ref{lem:decompo}-\ref{lem:decompo-monorepr}, subobjects of
$Y(A)$ for $A\in\fibc_I$ correspond to predicates in $\siev(\fibc)_A$, i.e. to
morphisms
$h:B\to A$ in $\tot{\fibc}$. This leads us to the following concrete description
of the fibration
$D\fibc$.
  \begin{itemize}
  \item Objects in $(D\fibc)_I$ are morphisms $h:B\to A$ in $\tot{\fibc}$
with $A\in\fibc_I$.
  \item A morphism from $h:B\to A$ to $k:D\to C$ over
    $u:I\to J$ is an equivalence class of spans $\xymatrix@1{B &
B^*\epicart[l]_e\ar[r]^f & D}$ with
$e$ cover-cartesian, such
    that $\fibc(kf)=u\fibc(he)$ and 
\begin{equation}\label{eq:judg-map-dc}
x,\! y\vtp B^* \csep h(ex)=h(ey)\ent
k(fx)=k(fy) 
\end{equation}
in $\siev(\fibc)$, where
  \item 
two spans
    $\xymatrix@C-3.5mm@1{B & B^*\epicart[l]_e\ar[r]^f & D}$ and
$\xymatrix@C-3.5mm@1{B &
      {B^*}'\epicart[l]_{e'}\ar[r]^{f'} & D}$ represent the same morphism  over
$u$, if
\[x\vtp B^*, y\vtp {B^*}' \csep h(ex)=h(e'y)\ent
k(fx)=k(f'y) 
\]
holds in $\siev(\fibc)$.
\item Using this representation, the embedding $y:\fibc\to D\fibc$ is given by
\[
 \fibc_I\ni A\mapsto \id_A\in (D\fibc)_I.
\]
  \end{itemize}

\begin{lemma}[Localization and slicing]\label{lem:d-loc-slice}
Let $\fibc$ be a finite limit pre-stack on $\catr$.
\begin{itemize}
 \item Given $I\in\catr$, we have an equivalence
\[
 D(\fibc/I)\simeq(D\fibc)/I
\]
of fibrations on $\catr/I$.
\item For $I\in\catr$ and $A\in\fibc_I$, we have an equivalence
\[
D(\fibc/A)\simeq  (D\fibc)/(yA)
\]
of fibrations on $\catr/I$.
\end{itemize}
\end{lemma}
\begin{proof}
 The claim about localization is straightforward. The proof of the claim about
slicing involves one argument which is not purely formal. Let us just consider
the objects. Let $u:J\to I$ in $\catr$. An object in $D(\fibc/A)_u$ is a
configuration
\begin{equation}\label{eq:obj-dcau}
 \vcenter{\xymatrix@R-6mm{
C\ar[rd]_g \\
& B\ar[r]^f & A\\
K\ar[rd]^v \\
& J\ar[r]^u & I
}},
\end{equation}
where $g$ is viewed as a sieve on $f\in(\fibc/A)_u$. An object in 
$((D\fibc)/(yA))_u$, on the other hand, is given by a configuration
\[
 \vcenter{\xymatrix@R-8.5mm{
C^*\ar[rrd]^f\epicart[dd]_e\\ 
&& A\ar[ddd]^\id\\
C\ar[rdd]_g \\ \\
& B & A\\
K^*\depi[dd]_p\ar[rrd]^w\\
& & I\ar[ddd]^\id\\
K\ar[ddr]^v\\ \\
& J\ar[r]^u & I
}} 
\]
such that $g(ex)=g(ey)\ent fx=fy$ in $\siev(\fibc)$
(see~\eqref{eq:judg-map-dc}) -- $g$ is represents a sieve on $A$, and $(e,f)$
represents a morphism from the corresponding sub-representable into $yA$. Now
for reasons similar to
Lemma~\ref{lem:tracking-fam}-\ref{lem:tracking-fam-strict} and since $yA$ is
represented by $\id_A$, we can represent the morphism given by $(e,f)$ without
cover-cartesian part, yielding a simplified configuration
\begin{equation}\label{eq:obj-dcyau}
 \vcenter{\xymatrix@R-6mm{
C\ar[rd]_g\ar[rrd]^h \\
& B & A\\
K\ar[rd]^v \\
& J\ar[r]^u & I
}},
\end{equation}
where $g$ is a sieve on $B$, and $h$ represents a map from the corresponding
sub-representable into $yA$.
To establish the equivalence, observe first that every
configuration of the form~\eqref{eq:obj-dcau} induces a configuration of the
form~\eqref{eq:obj-dcyau}. To transform data of the latter into the former
form, however, we have to make use of finite limit structure -- the object in
$D(\fibc/A)_u$ corresponding to~\eqref{eq:obj-dcyau} is given by
\begin{equation}
 \vcenter{\xymatrix@R-6mm{
C\ar[rd]_-{\langle g,h\rangle} \\
& B\times_I A\ar[r]^-{\pi_A} & A\\
K\ar[rd]^v \\
& J\ar[r]^u & I
}}.
\end{equation}
The verification that this construction induces the desired equivalence is
purely technical.
\end{proof}

\begin{lemma}\label{lem:characterize-d}
  Let $\fibc:\tot{\fibc}\to\catr$ be a finite limit pre-stack on a regular
  category. 
  \begin{itemize}
  \item $D\fibc$ is a \geostack{}.
  \item Given a \geostack{} $\fibs$, we have
    $\catlex(\fibc,\fibs)\simeq\catgeo(D\fibc,\fibs)$.
  \end{itemize}
\end{lemma}
\begin{proof}
By
Lemma~\ref{lem:positive-prestack-geometric}-\ref{
lem:positive-prestack-geometric-geo},
 $\whfc$ is a \geostack{}. The
subfibration on sub-representables is closed under finite limits (since $\fibc$
is), image factorization, and internal unions of subobjects (i.e., existential
quantification in the fibered subobject fibration), hence it is geometric as
well.

Given a second \geostack{} $\fibs$, precomposition with
$y:\fibc\to D\fibc$ induces a functor of type
$\catgeo(D\fibc,\fibs)\to\catlex(\fibc,\fibs)$, since $y$ preserves
finite limits. This functor is faithful since every object in $D\fibc$ is a
vertical subobject of an object in the image of $y$, and the functors in
$\catgeo(D\fibc,\fibs)$ preserve monomorphisms. To see that precomposition is
full, let $F,G\in\catgeo(D\fibc,\fibs)$, and $\eta:F\circ y\to G\circ y$. 
We have to extend $\eta$ to a natural transformation $\eta_0:F\to G$.
Let $h:B\to A$ in $\tot{\fibc}$ with $A\in\fibc_I$ represent an object $h\in
(D\fibc)_I$. Then we have a decomposition of $y(h):y(B)\to y(A)$ in
$\tot{D\fibc}$ as
\[
 \xymatrix{
y(B)\ccov[r]^e& h \mono[r]^m & y(A)
}
\]
with $e$ collectively covering and $m$ vertical monic. Applying $F$ and $G$, we
get 
\[
 \xymatrix{
F(y(B))\ccov[r]^{Fe}\ar[d]_{\eta_B}& F(h)\dotted[d]_{\eta_{0,h}} \mono[r]^{Fm} &
F(y(A))\ar[d]_{\eta_A}\\
G(y(B))\ccov[r]^{Ge}& G(h) \mono[r]^{Gm} & G(y(A))
}
\]
Since $F$ and $G$ preserve vertical monomorphisms and collective covers, the
lifting property for $Fe$ and $Gm$ gives us a unique candidate for
${\eta_{0,h}}$. We leave it to the reader to verify that $\eta_0$ is well
defined.

It remains to show that precomposition by $y:\fibs\to D\fibs$ is essentially
surjective. For this, let $F:\fibc\to\fibs$ be a finite limit preserving
fibered functor. We want to construct a geometric functor
$\widetilde{F}:D\fibc\to\fibs$ such that $\widetilde{F}\circ y\cong F$. Let
$h:B\to A$ represent an object $h\in(D\fibc)_I$ as before. We construct
$\widetilde{F}h\in\fibs_I$ by taking the collective-cover/vertical-mono
factorization.
\[
 \xymatrix@1{
y(B)\ar[r] & {\widetilde{F}h}\mono[r] & y(A)
}
\]
We leave it to the reader to figure out how to extend this operation to
morphisms, and to verify that $\widetilde{F}$ is a geometric extension of $F$.
\end{proof}

\subsection{The fibered positive
cocompletion}\label{sec:pos-pstack-from-fl-pstack}

We now come to the left adjoint of $\catpos(\catr)\to\catlex(\catr)$. Just as
for the geometric cocompletion, we define the positive cocompletion of a
finite limit pre-stack as a subfibration of the fibration of presheaves. Since
positive pre-stacks are in the realm of Moens' theorem, we define the fibration
by giving the terminal fiber and gluing.
\begin{definition}
  Let $\fibc:\tot{\fibc}\to\catr$ be a finite limit-pre-stack. The category
  $\flpos{\fibc}_1$ is the full subcategory of $\srel{\catr}{\fibc}$ on
subobjects of objects of the form $(X,\predeq)$.
\end{definition}
Similar as for the geometric cocompletion, subobjects of
of objects $(X,\predeq)$ correspond to predicates in $\siev(\fibc)_A$ by
Lemma~\ref{lem:decompo}-\ref{lem:decompo-monorepr}, which allows us to
give the following more concrete description of the category
$\flpos{\fibc}_1$:
  \begin{itemize}
  \item Objects are morphisms $h:Y\to X$ in $\tot{\fibc}$.
  \item A morphism from $h:Y\to X$ to $k:W\to Z$ is a span $\xymatrix@1{Y &
      Y^*\epicart[l]_e\ar[r]^f & W}$ such that $\rho(ex,ey)\ent\sigma(fx,fy)$ in
    $\siev(\catc)$, where $\rho$ and $\sigma$ are the kernel pairs of $h$ and
    $k$ viewed as equivalence relations in $\siev(\fibc)$.
  \item Two morphisms from $h$ to $k$ given by spans
    $\xymatrix@1{Y & Y^*\epicart[l]_e\ar[r]^f & W}$ and $\xymatrix@1{Y &
      {Y^*}'\epicart[l]_{e'}\ar[r]^{f'} & W}$ are identified as morphisms in
    $\flpos{\fibc}_1$, if
\[
y\vtp Y,w\vtp W\csep \exists y^*\vtp Y^*\!\qdot ey^*=y\wedge
\sigma(fy^*,w)\adj\ent \exists y^*\vtp {Y^*}'\!\qdot e'y^*=y\wedge
\sigma(f'y^*,w)
\]
holds in $\siev(\fibc)$.
  \end{itemize}
Since $\srel{\catr}{\fibc}$ is an exact category and $\flpos{\fibc}_1$ is a
subcategory which is closed under products and subobjects, it is regular.
Furthermore, the functor $\Delta:\catr\to\srel{\catr}{\fibc}$ factors through
$\flpos{\fibc}_1$
\[
\xymatrix@R-3mm{
\catr\ar[r]^\Delta\ar[dr]_\Delta & \flpos{\fibc}_1\incl[d]\\
& \srel{\catr}{\fibc}
},
\]
allowing us to make the follwing definition.
\begin{definition}
 The fibration $\flpos{\fibc}:\tot{\flpos{\fibc}}\to\catr$ is defined by
\[\flpos{\fibc}=\gl_\Delta(\flpos{\fibc}_1\].
\end{definition}
Since $\flpos{\fibc}_1$ is a regular category, and
$\Delta:\catr\to\flpos{\fibc}_1$ is a regular functor, $\flpos{\fibc}$ is a
positive pre-stack. We state its universal property without proof.
\begin{lemma}
 Let $\fibc$ be a finite limit pre-stack, and $\fibp$ a positive pre-stack.
Then we have an equivalence
\[
 \catlex(\catr)(\fibc,\fibp)\simeq\catpos(\catr)(\flpos{\fibc},\fibp)
\]
of categories of fibered functors.
\qed
\end{lemma}

\section{The fibered sheaf construction}\label{sec:fibered-sheaf-construction}
\index{fibered!sheaf construction}
Having treated cocompletions of finite limit pre-stacks in the
preceding section, we now come to cocompletions of \emph{geometric pre-stacks}.
In this context, the fibered subobject fibration (Definition~\ref{def:fisufi})
plays a similar role as the fibration of sieves for cocompletions of
finite limit pre-stacks.

First, we describe the free fibered pretopos on a geometric
fibration. We give a construction via gluing, in close analogy to what we did
for the fibered \emph{presheaf} construction.

In analogy to Definition~\ref{def:rc}, we define:
\begin{definition}\label{def:cat-from-geo-fib}
  Let $\fibs:\tot{\fibs}\to\catr$ be a \geostack{}. The category
  $\catr[\fibs]$ is the category of (total) equivalence relations and functional
  relations in $\sub(\fibs)$ -- in other words the full subcategory of
$\per(\sub(\fibs))$ on total equivalence relations.
\end{definition}
If we view objects in $\tot{\fibs}$ as families of objects, and predicates in
$\sub(\fibs)$ as families of subobjects, we realize that the equivalence
relations in $\catr[\fibs]$ can be viewed as heterogeneous equivalence
relations, just as we explained for the fibered presheaf construction in
Section~\ref{sec:fpc}. In particular, given a geometric pre-stack $\fibs$, we
can form the categories $\catr\{\fibs\}$ and $\catr[\fibs]$, both of which can
be
viewed as categories of heterogeneous equivalence relations -- the first one
with respect to $\siev(\fibs)$, and the second one with respect to
$\sub(\fibs)$.

\medskip

Let me make a slightly deviating remark at this point.
If $\cats$ is a small geometric category, then $\catset\{\famf(\cats)\}$ is
equivalent to the category $\widehat{\cats}$ of presheaves on $\cats$, and
$\catset[\famf(\cats)]$ is equivalent to
the category $\sheaf{\cats}$ of sheaves on $\cats$ for the canonical topology.
Furthermore, both categories are toposes and the latter is a subtopos of the
former. For general geometric pre-stacks, the categories do not have to be
toposes, nor does it generally seem to be the case that 
$\catr[\fibs]$ is a localization\footnote{in the sense of `reflective
subcategory with finite limit preserving left adjoint'}\index{localization!of a finite-limit category} of $\catr\{\fibs\}$.
Interestingly, what is lacking for the localization is not the inverse but the
\emph{direct} image
part -- while the embedding $\sheaf{\cats}\to\widehat{\cats}$ is just the
identity in the usual presentation, there is no generic method to construct a
functor of type $\catr[\fibs]\to\catr\{\fibs\}$. An intuition as to what goes
wrong is given by thinking about \emph{small} sheaves and presheaves on large
geometric categories. Already here the embedding of sheaves into presheaves does
not have to be well defined since if a functor $F:\cats^\op\to\catset$ is a
small
colimit of representables in the category of sheaves, there is no reason why it
should also be such a small colimit in the category of presheaves. On the
positive side, if the
geometric pre-stack is a tripos, then the embedding can be constructed using
impredicativity/weakly complete objects -- see also the remark after
Corollary~\ref{cor:tripos-subtripos-free-tripos}.

\medskip

Back to the main line of thoughts, we state a lemma analogous to Lemma~\ref{lem:whfc-exact}.
\begin{lemma}\label{lem:rs-exact}
If $\fibs:\tot{\fibs}\to\catr$ is a geometric fibration, then 
  $\catr[\fibs]$ is equivalent to $\per(\sub(\fibs))$ and thus in particular
exact.
\end{lemma}
\begin{proof}
Again, we have to show that every partial equivalence relation is equivalent to
a total one. This is evident for $\sub(\fibs)$ since predicates are vertical
monomorphisms, and we can restrict any partial equivalence relation to its
support.
\end{proof}
\begin{definition}\label{def:delta-r-rs}
  The functor $\Delta:\catr\to\catr[\fibs]$ maps $I\in\catr$ to $(1_I,\predeq)$
  and $u:I\to J$ to the predicate $\ilbracks{x\vtp 1_I,y\vtp 1_J\csep
1_u(x)=y}$, where
  $1_u:1_I\to 1_J$ is the unique function of this type over $u$.
\end{definition}
It is easy to see that $\Delta$ is regular, which allows us to give the
expected definition of the fibration of sheaves on $\fibs$.
\begin{definition}
For a geometric pre-stack $\fibs:\tot{\fibs}\to\catr$, the \emph{fibration
$\sheaf{\fibs}:\tot{\fibs}\to\catr$ of sheaves on}\index{fibration!of sheaves} $\fibs$ is defined as
$\sheaf{\fibs}=\gl_\Delta(\catr[\fibs])$.
\end{definition}
In Lemma~\ref{lem:sh-uprop}, we will show that $\sheaf{\fibs}$ is the free
fibered-pretopos completion of a \geostack{} $\fibs$, but before that we need a
representation of the slices of $\sheaf{\fibs}$ analogous to
Lemma~\ref{lem:whfci-by-localization}.
\newcommand{\shs}{\sheaf{\fibs}}
\newcommand{\fibsar}{\fibs/(A,R)}
\begin{lemma}\label{lem:sheaf-slice}
Let $\fibs$ be geometric pre-stack and $A\in\fibs_I$.
\begin{enumerate}
\item
$\fibs/A$ is a geometric fibration.
\item 
 We have
$\catr[\fibs]/(A,\predeq)\simeq(\catr/I)[\fibs/A]$.
\end{enumerate}
\end{lemma}
\begin{proof}
We already know that $\fibs/A$ has finite limits. Image factorizations and
internal unions are inherited from $\fibs$ in the straightforward way,
preserving all stability properties. Thus, $\fibs/A$  is a \geostack.

To show the claimed equivalence, we sketch constructions of functors in both
directions. An object in $(\catr/I)[\fibs/A]$ is given by a map $f:B\to A$ in
$\tot{\fibs}$ and a vertical subobject $\rho:R\emono B\times_A B$ which is
reflexive and transitive as a predicate in $\sub(\fibs/A)$. We can transform
$\rho$ into a predicate on $B\times B$ by quantifying existentially along
$m:B\times_A B\emono B\times B$, and the resulting $\exists_m\rho$ is an
equivalence relation in $\sub(\fibs)$, whence $(B,\exists_m\rho)$ is an object
in $\catr[\fibs]$. We claim that $f$ induces a functional relation of type
$(B,\exists_m\rho)\to(A,\predeq_A)$ -- to show this, we have to verify that
$f$ is compatible with the two relations, i.e. that we can fill in the dashed
arrow in
\[
\vcenter{\xymatrix@R-5mm{
\vmono[dd]_\rho\ccov[rd]\ar[rr] & & \vmono[dd]_(.35)\top\ccov[rd]\\
& \vmono[dd]_(.25){\exists_m\rho}\dashed@/_5pt/[rr] & &
\vmono[dd]^{\exists_{\delta_A}\top}\\
B\times_A B\ar[rr]|(.56)\hole\mono[rd]_m& & A\mono[rd]^{\delta_A}\\
& B\times B\ar[rr]_{f\times f} & & A\times A\\
J\times_I J\ar[rr]\mono[rd]& & I\mono[rd]^{\delta_I}\\
& J\times J\ar[rr]^{u\times u} & & I\times I\\
}},
\]
where the two lower horizontal squares (in the spatial sense) are pullbacks in
the total category and the base, respectively, and the predicate
$\exists_{\delta_A}\top$ is (by definition) the equality predicate $\predeq_A$
in $\sub(\fibs)$. The validity of the claim can be seen by chasing along the
backside of the cube, i.e., over $B\times_A B\to A\to A\times A$; the fact that
the squares are pullbacks is not essential. This shows how an object in
$(\catr/I)[\fibs/A]$ can be transformed into an object in
$\catr[\fibs]/(A,\predeq)$.

In the other direction, take an object in $\catr[\fibs]/(A,\predeq)$, i.e.\ a
morphism $\phi:(B,\rho)\to(A,\predeq)$. Let $m:U\emono A\times B$ be the
vertical subobject corresponding to the image of $\langle
\phi,\id\rangle:(B,\rho)\to(A\times B,\predeq\brprod\rho)$. Then
$(U,(\predeq\brprod\rho)|_U)$ is isomorphic to $(B,\rho)$, 
and composing this isomorphism with $\phi$, we
obtain an isomorphic representation of $\phi$ which is tracked by the
projection $p:U\emono A\times B\to A$. Since the projection is compatible with
the equivalence relations, $\predeq\brprod\rho$ factors through $U\times_A U$,
which induces an equivalence relation on $p$ in $\fibs/A$. This shows how to go
from $\catr[\fibs]/(A,\predeq)$ to $(\catr/I)/[\fibs/A]$.

We leave it to the reader to verify that the sketched constructions are
functorial and constitute an equivalence of categories.
\end{proof}

The preceding lemma does in particular 
give us a representation of $\sheaf{\fibs}$ as
\[\sheaf{\fibs}_I\simeq(\catr/I)[\fibs/I].\]
 In the following, we will identify
$\sheaf{\fibs}_I$ with $(\catr/I)[\fibs/I]$, as this is easiest to work
with.
Explicitely, 
\begin{itemize}
 \item 
an object in $\sheaf{\fibs}_I$ is thus given by a triple $(u,A,\rho)$ with
$u:J\to I$, $A\in\fibs_J$, and $\rho$ a vertical subobject of $A\times_I A$
which is an equivalence relation in $\sub(\fibs/I)$, and
\item
a morphism from $(u,A,\rho)\in\sheaf{\fibs}_I$ to
$(v,B,\sigma)\in\sheaf{\fibs}_K$ over $w:I\to K$ is a vertical subobject of
$A\times_K B$ which is functional with respect to the extension of $\rho$ 
to $A\times_K A$ and
$\sigma$ in $\sub(\fibs/K)$.
\end{itemize}
We also give explicit constructions of cartesian and cocartesian lifting and
internal unions in $\sheaf{\fibs}$ relative to the above representation, since
we will need them later. 
\begin{itemize}
 \item 
Given
$(u,A,\rho)$ in $\sheaf{\fibs}_I$ and $v:K\to I$, the cartesian lifting of
$(u,A,\rho)$ along $v$ is given by $(v^*u, v^*A, v^*\rho)$, where $v^*u$,
$v^*A$, and $v^*\rho$ are defined as in the following diagrams.
\[
\vcenter{\xymatrix@R-2mm{
v^* A\cart[r] & A\\
v^*J\pullbackcorner\ar[r]\ar[d]_{v^*u} &  J\ar[d]^u\\
K\ar[r]^v & I
}
}\quad
\vcenter{\xymatrix@R-2mm{
\vmono[d]_{v^*\rho}\pullbackcorner\cart[r] & \vmono[d]^\rho\\
v^*A\times_K v^*A\cart[r]& A\times_I A\\
v^*J\times_K v^*J\pullbackcorner\ar[d]\ar[r] & J\times_I J\ar[d]\\
K\ar[r]^{v} & I
}}
\]
\item
In the other direction, the \emph{co}cartesian lifting of $(u,A,\rho)$ along
$w:I\to L$
is given by $(wu, A, \exists_m\rho)$, where 
$m:A\times_I A\to A\times_L A$ is the canonical embedding.
\item For internal unions, recall that subobjects of $(u,A,\rho)$ in
$\sheaf{\fibs}_I$ correspond compatible predicates in
$\sub(\fibs/I)_{(u,A)}$, i.e., vertical subobjects of $A$ which are
compatible with $\rho$. Now given $(u,A,\rho)\in\sheaf{\fibs}_I$, $v:K\to I$,
and a compatible subobject $m$ of $v^*A$ (as above), the subobject $n$ of $A$
corresponding to the internal union along $v$ is given by internal union along
$u^*v$ in $\fibs$,
\[
\vcenter{\xymatrix@R-2mm{
\vmono[d]_m\ccov[r] & \vmono[d]^n\\
v^* A\cart[r] & A\\
v^*J\pullbackcorner\ar[r]^{u^*v}\ar[d]_{v^*u} &  J\ar[d]^u\\
K\ar[r]^v & I
}},
\]
which is automatically strict with respect to $\rho$ (intuitively because sets
which are compatible with an equivalence relation are closed under arbitrary
unions).
\end{itemize}

Using this representation, we can define the embedding of $\fibs$
into $\sheaf{\fibs}$, which we call $Z$, since it is close to the fibered
Yoneda embedding $Y:\fibc\to\whfc$ defined in Lemma~\ref{def:fib-yoneda}.
\begin{definition}\label{def:z-embedding}
For a \geostack{} $\fibs:\tot{\fibs}\to \catr$, we define
\begin{align*}
Z:\fibs&\to\sheaf{\fibs}\\
 \fibs_I\ni A&\mapsto (\id_I,A,\predeq)\in\sheaf{\fibs}_I.
\end{align*}
\end{definition}
\begin{lemma}
The previously defined $Z$ is full, faithful and geometric.
\end{lemma}
\begin{proof}
Fullness, faithfulness, and regularity are fiberwise properties, thus it
suffices to show them for the functors $Z_I$, $I\in\catr$. It suffices even to
verify them for $Z_1:\fibs_1\to\catr[\fibs]$, since
$\fibs_I\simeq(\fibs/I)_{\id_I}$, and we can apply the same argument to
$\fibs/I$. For $A,B\in\fibs_1$, a functional relation from $(A,\predeq)$ to
$(B,\predeq)$ in $\sub(\fibs)$ is just a functional relation between $A$ and
$B$ in the (ordinary) subobject fibration of $\fibs_1$, and we know that
morphisms are in bijection with functional relations in regular categories,
whence $Z_1$ is full and faithful. A morphism $\phi:(A,\rho)\to(B,\sigma)$ in
$\catr[\fibs]$ is regular epic iff $b\csep\vdash\exists A\qdot \phi(a,b)$
holds. Since $\exists$ is given by collective-cover/vertical-mono
factorization, which restricts to ordinary cover/mono factorization in
$\fibs_1$, it is easy to see that $Z_1$ preserves regular epimorphisms.

Finally, the preservation of internal unions follows directly from the
description of internal unions in $\sheaf{\fibs}$ given before
Definition~\ref{def:z-embedding}.
\end{proof}

\begin{lemma}\label{lem:sh-uprop}
  $\sheaf{\fibs}$ is characterized by the equivalence 
\[
\catgeo(\catr)(\fibs,\fibx)\simeq\catpretop(\catr)(\sh(\fibs),\fibx)
\]
for fibered pretoposes $\fibx$.
\end{lemma}
\begin{proof}
 Since $Z:\fibs\to\sheaf{\fibs}$ is a geometric fibered functor, precomposition
induces a functor of type
$\catpretop(\sh(\fibs),\fibx)\to\catgeo(\fibs,\fibx)$. We show that this
functor is full, faithful, and essentially surjective.

For faithfulness, assume that $F,G\in\catpretop(\sh(\fibs),\fibx)$ and
$\eta,\theta:F\to G$ such that $\eta Z=\theta Z$. Let
$(u,A,\rho)\in\sheaf{\fibs}_I$ with $u:J\to I$. This object can be covered by
an object in the image of $Z$ as
$ZA=(\id_J,A,\predeq)\to(u,A,\predeq)\to(u,A,\rho)$ where the first map is
cocartesian, and the second is a vertical regular epi, whence both are
collective covers, which are preserved by fibered geometric functors, and in
particular fibered pretopos morphisms. Thus, applying $F, G, \eta$, and
$\theta$, we obtain
\[
 \xymatrix{
\ar[d]_{\eta_{(\id_J,A,\predeq)}=
\theta_ { (\id_J , A ,
\predeq)}}F(\id_J,A,\predeq)\ccov[r]&F(u,A,\rho)\ppair{d}{\theta_{(u,A,\rho)}}{
\eta_ { (u , A , \rho)}}\\
G(\id_J,A,\predeq)\ccov[r]&G(u,A,\rho)
}
\]
in $\fibx$, and we deduce that $\eta_{(u,A,\rho)}=\theta_{(u,A,\rho)}$ since
collective covers are collectively epic, and the components of $\eta$,
$\theta$ are vertical.

To see that precomposition by $Z$ is full, let $\mu:FZ\to GZ$ and consider
again $(u,A,\rho)\in\sheaf{\fibs}_I$ covered by
$(\id_J,A,\predeq)\to(u,A,\predeq)\to(u,A,\rho)$. Applying $F$ and $G$, we
obtain 
\[
 \xymatrix{
\ar[d]_{\mu_A}FZA\coca[r]&F(u,A,\predeq)\vepi[r]\dashed[d]&F(u,A,\rho)\dashed[d]
\\
GZA\coca[r]&G(u,A,\predeq)\vepi[r]&G(u,A,\rho)
}
\]
and we have to show that we can fill in the dashed arrows. The existence of the
fist one follows from the universal property of cocartesian liftings, for the
second one we have to compare the kernels of the two vertical (horizontal in
the diagram) epimorphisms. Now these two kernels can be collectively covered by
the images of $\rho:R\emono A\times_I A$ under $FZ$ and $GZ$, respectively, and
the desired inclusion of kernels can be deduced by considering the arrow $\mu_R$
between their respective coverings. We leave it to the reader to verify that
this construction give rise to a natural transformation between $F$ and $G$.

It remains to check that precomposition by $Z$ is essentially surjective. Let
$F:\fibs\to\fibx$ be a geometric fibered functor. We have to construct an
extension $\widetilde{F}:\sheaf{\fibs}\to\fibx$ of $F$. Consider
$(u,A,\rho)\in\sheaf{\fibs}_I$. If we apply $F$ to $A$ and $\rho$, we obtain
$FA\in\fibs_J$ and $F\rho:FR\emono FA\times_I FA$. To construct
$\widetilde{F}(u,A,\rho)$, we take the internal sum of $FA$ along $u$, and then
push $F\rho$ from $FA\times_I FA$ to $\Sigma_uFA\times_I\Sigma_uFA$.
\[
\vcenter{\xymatrix{
FA\coca[r] & \Sigma_uFA\\
J \ar[r]^u & I
}}\quad
\vcenter{\xymatrix@R-6mm@C-9mm{
\vmono[dd]_{F\rho}\coca[rr] & & \vmono[dd]^{\rho'}\\
\\
  FA\times_I FA\coca[rr]\cart[rd]& &
\Sigma_uFA\times_I\Sigma_uFA\cart[rd]\\
& FA\times FA\coca[rr] & & \Sigma_uFA\times\Sigma_uFA\\
J\times_I J\ar[rr]\mono[rd]& & I\mono[rd]^{\delta_I}\\
& J\times J\ar[rr]^{u\times u} & & I\times I\\
}}
\]
$\rho'$ defined like this is an equivalence relation in $\fibx_I$, and we
take $\widetilde{F}(U,A,\rho)$ to be its quotient.
\end{proof}

\begin{corollary}\label{cor-hat-sh-d}
 For a finite limit pre-stack $\fibc:\tot{\fibc}\to\catr$, on a regular
 category, we have an equivalence
\[
\widehat{\fibc}\simeq\sheaf{D\fibc}
\]
of fibrations. In particular, we have $\srel{\catr}{\fibc}\simeq\catr[D\fibc]$.
\end{corollary}
\begin{proof}
  The universal property of the fibered presheaf construction says that it
  constitutes a left biadjoint to the forgetful functor
  $\catpretop(\catr)\to\catlex(\catr)$ from fibered pretoposes on $\catr$ to
  finite limit pre-stacks on $\catr$. The universal properties of $\sheaf{-}$
  and $D$ characterize the two constructions as left biadjoints to the
  forgetful functors of type $\catpretop(\catr)\to\catgeo(\catr)$ and
  $\catgeo(\catr)\to\catlex(\catr)$, respectively. The claim follows from the
  fact that biadjunctions compose.
\end{proof}

\subsection{Discrete sheaves}\label{sec:dis-sheaf}

In this section, we consider the left biadjoint to
\[
\catpos(\catr)\to\catgeo(\catr).
\]
We do this without going into details and without proofs. Abstractly, the
positive cocompletion of a geometric fibration $\fibs$ can be understood as the
subfibration of $\sheaf{\fibs}$ which is generated by the image of
$Z:\fibs\to\sheaf{\fibs}$ by closing under sums. Concretely, we give a
construction analogous to the positive completion of a finite limit pre-stack
in Section~\ref{sec:pos-pstack-from-fl-pstack} -- by identifying an
appropriate subcategory of $\catr[\fibs]$ and gluing.
 It turns out that this subcategory has a particularly
simple description.

\begin{definition}\label{def:geopos}
Let $\fibs:\tot{\fibs}\to\catr$ be a geometric pre-stack.
  The category $\geopos{\fibs}_1$ is defined as the full subcategory of
$\catr[\fibs]$ on \emph{discrete} equivalence relations. In other words, the
objects of $\geopos{\fibs}_1$ are the objects of $\tot{\fibs}$, and the
morphisms are functional relations in $\sub(\fibs)$.
 \end{definition}
Since discrete equivalence relations in $\catr[\fibs]$ are closed under finite
products and subobjects, $\geopos{\fibs}_1$ is a regular category. Moreover,
the functor $\Delta:\catr\to\catr[\fibs]$ (Definition~\ref{def:delta-r-rs})
factors through $\geopos{\fibs}_1$
\[
 \xymatrix@R-3mm{
\catr\ar[r]^\Delta\ar[rd]_\Delta & \geopos{\fibs}_1\incl[d]\\
& \catr[\fibs]
}
\]
and we can define $\geopos{\fibs}$ by
\[
 \geopos{\fibs}=\gl_\Delta({\geopos{\fibs}_1}).
\]
Since ${\geopos{\fibs}_1}$ and $\Delta$ are both regular, $\geopos{\fibs}$ is
indeed a positive pre-stack. Furthermore, the characterizing equivalence holds.
\begin{lemma}
 Let $\fibs$ be a geometric pre-stack, and $\fibp$ a positive pre-stack on
$\catr$. Then we have
\[
 \catgeo(\catr)(\fibs,\fibp)\simeq\catpos(\catr)(\geopos{\fibs},\fibp)
\]
\qed
\end{lemma}

\section{Fibered pretoposes from positive
fibrations}\label{sec:fibered-pretop-pos}

To finish our treatment of left biadjoints to the forgetful functors
\eqref{eq:sequence-inclusions}, it remains to construct the left biadjoint to
$\catpretop(\catr)\to\catpos(\catr)$, i.e., the construction of the free
fibered pretopos on a positive pre-stack. This is simply done by fiberwise
ex/reg completion.

One way to understand this is to consider instead the forgetful functor
\[
\pscommacat{\catr}{\mathbf{U}}\to \catr\pslice\catreg 
\]
between the pseudo-co-slice 2-categories which are biequivalent to
$\catpretop(\catr)$ and $\catpos(\catr)$ via the correspondence of Moens
theorem (see Lemmas~\ref{lem:moens-positive} and \ref{lem:moens-pretop}). It is
easy to see that the ordinary ex/reg completion, which is left biadjoint to
$\mathbf{U}:\catex\to\catreg$, lifts to a biadjunction between the
pseudo-co-slice 2-categories.

\section{The preordered case}\label{sec:preordered-case}

The free constructions that we have developed in the preceding sections
become easier when we consider fibered preorders instead of general fibrations.

In this section, we explore some of the consequences, with the goal of
giving a characterization of when the category $\srel{\catr}{\ffrma}$ for a
pre-stack $\ffrma:\tot{\ffrma}\to \catr$ of meet-semilattices is locally
cartesian closed.

\subsection{Fibered frames}

Fibered frames are posetal geometric fibrations. Since in the posetal case in
the presence of greatest elements the existence of internal unions already
implies the existence of internal sums, we can give the following equivalent
definition.
\begin{definition}\label{def:fibered-frame}
Let $\catr$ be a regular category.
\begin{enumerate}
\item
A \emph{fibered frame}\index{fibered!frame}\index{frame!fibered} is
an existential fibration $\fifx:\tot{\fifx}\to\catr$
(Definition~\ref{def:existential-fibration}) which is a pre-stack.
 \item 
$\ffrm(\catr)$ is the locally ordered category of fibered frames on $\catr$. Its
morphisms are fibered monotone maps preserving finite meets and existential
quantification.
\end{enumerate}
\end{definition}
Given a fibered frame $\fifx:\tot{\fifx}\to\catr$, we can use the existential
quantification to `embed' the subobject fibration of $\catr$ into $\fifx$ via
the fibered monotone map
\begin{equation}\label{eq:def-delta}
 \delta:\sub(\catr)\to\fifx,\qquad (m:U\emono I)\;\mapsto
\;(\exists_m\top\in\fifx_I).
\end{equation}
We write `embed' in quotes since $\delta$ is not necessarily
order-reflecting (a counterexample is the terminal fibration). However,we can
show the following statements.
\begin{lemma}\label{lem:preds-incl-exists}
 Let $\fifx$ be a fibered frame on $\catr$.
\begin{itemize}
\item 
$\delta$ preserves conjunction and existential quantification.
 \item 
Let $m:U\emono I$ be a
monomorphism in $\catr$. Then we have 
\[\fifx_U\simeq\{\varphi\in\fifx_I\msep \varphi\leq \delta m\}.\]
\end{itemize}
\end{lemma}
\begin{proof}
The preservation of conjunction follows from the Beck-Chevalley condition and
Frobenius reciprocity. Preservation of $\exists$ follows from the fact that
$\fifx$ is a pre-stack.

For the second claim, use reasoning in the internal logic, making use of the
facts that $\exists_mm^* \varphi = \ilbracks{i\csep \exists u\qdot mu=i\wedge
\varphi(mu)}$ and $m^*\exists_m\psi=\ilbracks{u\csep\exists v\qdot
mu=mv\wedge\psi v}$
for $\varphi\in\fifx_I,\psi\in\fifx_U$, and that $mu=mv\ent u=v$ holds in
$\fifx$ since $m$ is a monomorphism.
\end{proof}

\subsubsection{Totally connected fibered frames}

\begin{definition}\label{def:totally-connected-fibered-frame}
 We call a fibered frame $\fifx:\tot{\fifx}\to\catr$ \emph{totally connected}\index{totally connected!fibered frame}\index{fibered!frame!totally connected},
if $\delta:\sub(\catr)\to\fifx$ has a finite meet preserving left adjoint
$\pi:\fifx\to\sub(\catr)$.
\end{definition}
The term `totally connected' comes from topos theory -- 
a geometric morphism $\Delta\adj\Gamma:\tope\to\tops$ is called totally
connected if the fibered functor
$\Delta:\fund{\tops}\to\gl_\Delta(\tope)$ (see~\eqref{eq:def-fibered-delta}) has
a finite limit preserving left adjoint (which is automatically positive
since it is a left adjoint). In this case, the adjunction is necessarily a
reflection, since $\fund{\tops}$ is bi-initial among fibered pretoposes on
$\tops$.

In the case of totally connected fibered frames, we can deduce that the
adjunction is a reflection in a similar way.

\begin{lemma}\label{lem:totally-connected-reflection}
 Let $\fifx:\tot{\fifx}\to\catr$ be a totally connected fibered frame. Then
$\pi\adj\delta$ is a reflection and $\delta$ is order reflecting.
\end{lemma}
\begin{proof}
Let $m:U\emono I$ in $\catr$. To show that $\pi\adj\delta$ is a reflection, it
suffices to show that $U\subseteq\pi\delta U$, which is equivalent to
$m^*(\pi\delta U)\cong\top$. This follows from the fact that $m^*U\cong\top$,
and that $\pi$ and $\delta$ commute with reindexing and preserve $\top$.

The second claim follows from the first one.
\end{proof}
In Lemma~\ref{lem:da-totally-connected} we will see that fibered frames of the
form $D\fifa$ for $\fifa$ with finite meets are always totally connected; in
Section~\ref{sec:assemblies} we will show that sheaves over totally connected
fibered frames have well behaved subcategories of assemblies.

\subsection{Cocompletions}\label{sec:pos-cocomp}

In the following we revisit the constructions of fibered presheaves and
sheaves, and the geometric completion $D$ for pre-stacks of preorders.
We will do this in the converse order, since it appears more natural in the
posetal case.

\subsubsection{The \texorpdfstring{$D$}{D}-construction}\label{sec:pos-d}

Since fibered frames are posetal \geostack{}s, we can expect to obtain a 
fibered frame when applying the $D$-construction (Definition~\ref{def:d-fib})
to a pre-stack of
meet-semilattices.

Recall that for a finite limit pre-stack $\fibc$, the fibration $D\fibc$ is
defined as the full subfibration of $\widehat{\fibc}$ on sub-representables, and
can be characterized as the completion of $\fibc$ to a \geostack{}
(Lemma~\ref{lem:characterize-d}). Now for pre-stacks of
meet-semilattices, $D$ can alternatively be characterized as freely adjoining
existential quantification.
This point of
view allows us to give a direct
construction of $D$ without detour over fibered presheaves, which allows us to
go in the opposite direction and define fibered presheaves in terms of $D$
(following Corollary~\ref{cor-hat-sh-d}) without being circular.

Furthermore, the point of view on $D$ as adjoining existential quantification
gives a construction that even makes sense in the absence of finite
limits/meets:
\begin{definition}
 Let $\ffrma:\tot{\ffrma}\to\catr$ be a pre-stack of preorders. The fibration
$D{\ffrma}$ on $\catr$ is defined as follows.
\begin{itemize}
 \item A predicate on $I\in \catr$ is a pair $(u:J\to I, \varphi\in\ffrma_J)$
\item $(u:J\to I,\varphi)\leq(v:K\to I,\psi)$ if there exists
a span $\xymatrix@1@C-1.9mm{J & \depi[l]_eL\ar[r]^f & K}$ with $e$
regular epic, $ue=vf$, and $e^*\varphi\leq f^*\psi$.
\end{itemize}
The embedding $y:\fifa\to D\fifa$ is defined by
\[
 \fifa_I\ni\varphi\mapsto(\id_I,\varphi)\in D\fifa_I.
\]
\end{definition}
\begin{lemma}
For pre-stacks $\fifa$, the ordering in $D\fifa$ can equivalently
be defined in terms of total relations: we have $(u:J\to I,\varphi)\leq(v:K\to
I,\psi)$ iff there exists a 
total relation $t:J\etrel K$ such that $vt=u$, and $\varphi(j),(\delta
t)(j,k)\ent \psi(k)$ holds in $\fifa$.
\end{lemma}
\begin{proof}
 Any total relation gives a span via its two projections, the left leg is epic
since the relation is total. Conversely, given a span $\xymatrix@1@C-1.9mm{J &
\depi[l]_eL^*\ar[r]^f & K}$ with left leg epic, we get a total relation by
taking the image of $\langle e,f\rangle$ in $J\times K$.
\end{proof}

\begin{lemma}\label{lem:da}
Let $\ffrma:\tot{\ffrma}\to\catr$ be a pre-stack of preorders.
 \begin{enumerate}
\item $y:\fifa\to D\fifa$ is (monotone and) order-reflecting.
\item $D\fifa$ has existential quantification.
\item\label{lem:da-generated} $D\fifa$ is generated under existential
quantification by the image of
$y:\fifa\to D\fifa$. More precisely, for every $\varphi\in D\fifa_I$ there
exists $\psi\in\fifa_J$ and $u:J\to I$ with $\exists_uy(\psi)\cong\varphi$.
\item\label{lem:da-equiv} 
Given a posetal pre-stack $\fifb$ with existential quantification,
pre-composition with $y:\fifa\to D\fifa$ induces an equivalence
\[
\pfib(\catr)(\fifa,\fifb)\simeq\epfib(\catr)(D\fifa,\fifb)
\]
between preorders of fibered monotone maps, and
fibered monotone maps commuting with
existential quantification. 
 \end{enumerate}
\end{lemma}

\begin{proof}
It is easy to see that $y$ is order-reflecting (the fact that $\fifa$ is a
pre-stack is necessary).

The existential quantification of $(u,\varphi)$
along $v:I\to K$ is given by $(vu, \varphi)$, and the Beck Chevalley condition
follows from
the pullback lemma. The fact that $D\fifa$ is generated by the image of $\fifa$
follows immediately from the description of existential quantification.

For the fourth claim, we have to show that 
precomposition by $y:\fifa\to D\fifa$ is
order-reflecting and essentially surjective. 
Let $F, G:D\fifa\to\fifb$ be
fibered monotone maps commuting with $\exists$, such
that $F\circ y \leq F\circ y$. 
Then $F\leq G$ follows from the facts that $D\fifa$ is generated by the image
of $y$ under existential quantification, and that $F$ preserves existential
quantification.
To show that precomposition is essentially surjective, let $F:A\to B$ be a
fibered monotone map. Then we define $\widetilde{F}:D\fifa \to \fifb$ by
$\widetilde{F}(u,\varphi)=\exists_u F\varphi$.
\end{proof}

\begin{definition}\label{def:exprime}
Let $\fifb:\tot{\fifb}\to\catr$ be a posetal pre-stack with existential
quantification subject to the Beck-Chevalley condition. We call $\pi\in\fifb_I$
\emph{$\exists$-prime}\index{$\exists$-prime}, if whenever we are given maps
$I\xleftarrow{u}J\xleftarrow{v}K$ and $\theta\in\fifb_K$ such that
$u^*\pi\leq\exists_v\theta$, there exists a span 
$\xymatrix@1@C-1.5mm{J &\depi[l]_eL\ar[r]^w & K}$ such that
$vw=e$ and $(ue)^*\pi\leq w^*\theta$.
\[
 \xymatrix@R-3mm{
\pi\dline[d] & L\ar@{-->>}[d]_e\dashed[dr]^w & \theta\dline[d]\\
I & J\ar[l]_u & \ar[l]_vK
}
\]
\end{definition}
\begin{remarks}
\begin{itemize}
 \item 
Again, we can replace the span in the definition by a total relation since we
are in a posetal framework.

In terms of relations, $\pi$ is prime iff for all
$I\xleftarrow{u}J\xleftarrow{v}K$ and $\theta\in\fifb_K$ such that
$u^*\pi\leq\exists_v\theta$, there exists a total relation $t:J\etrel K$ such
that $v\circ t = \id_J$ and
\[
\pi(uj),(\delta t)(j,k)\ent \theta(k) 
\]
holds in $\fifb$.
\item Let $(L,\leq)$ be a complete lattice. Then a family $\varphi:I\to L$ is
$\exists$-prime in $\famf(L):\Famf(L)\to\catset$ iff all components
$\varphi(i)$ for $i\in I$ are \emph{completely join prime}\index{completely join prime} in the usual order
theoretic sense~\cite{wiki:lattice}.
\end{itemize}
\end{remarks}
\begin{lemma}\label{lem:epstack-prime}
Let $\fifa:\tot{\fifa}\to\catr$ be a posetal pre-stack.
\begin{enumerate}
 \item\label{lem:epstack-prime-d-image} Given $\varphi\in\fifa_I$, $y(\varphi)$
is $\exists$-prime in $D\fifa$.
\item\label{lem:epstack-prime-in-d-image} Given an $\exists$-prime $\pi\in
(D\fifa)_I$, there exists $\psi\in\fifa_K$ and $e:K\eepi I$ regular epic
such that $y(\psi)\cong e^*\pi$.
\item\label{lem:epstack-prime-equiv} Let $\fifb$ be a posetal pre-stack
with existential quantification, and let $\fifa$ be its subfibration on
$\exists$-prime predicates. The fibered monotone map 
$\widetilde{H}:D\fifa\to\fifb$ induced by the inclusion $H:\fifa\to\fifb$ is
order-reflecting. It is essentially surjective (and thus an
equivalence) iff for every $\psi\in\fifb_I$ there exists an $\exists$-prime
predicate $\pi\in\fifb_J$ and a map $u:J\to I$ such that
$\psi\cong\exists_u\pi$.
\end{enumerate}
\end{lemma}
\begin{proof}
 \emph{Ad~\ref{lem:epstack-prime-d-image}.} Assume that
$y(\varphi)\leq\exists_u(v,\theta)$ for $(v,\theta)\in(D\fifa)_J$, i.e., $v:K\to
J$ and $\theta\in\fifa_K$ (since the image of $y$ is closed under reindexing,
we can omit the reindexing in the hypothesis of the definition of primality).
The definition of the ordering in $D\fifa$ provides us with the required span.

\emph{Ad~\ref{lem:epstack-prime-in-d-image}.} Assume that $\pi\in(D\fifa)_I$ is
$\exists$-prime. By Lemma~\ref{lem:da}-\ref{lem:da-generated}, there exist
$\varphi\in (D\fifa)_J$ and $u:J\to I$ such that $\varphi$ is in the image of
$y$ and $\exists_u\varphi\cong\pi$. Since $\pi$ is prime, there exists a span
$\xymatrix@1@C-1.5mm{J &\epicart[l]_eK^*\ar[r]^w & I}$ with $wu=e$ and
$e^*\pi\leq w^*\varphi$. If we can show that $e^*\pi\cong w^*\varphi$ we are
done, since the essential image is closed
under reindexing. We already know that the left hand side is $\leq$ than the
right hand side; the other inequality is derived as follows.
\[
 \exists_u\varphi\leq\pi\quad\imp\quad \varphi\leq u^*\pi\quad\imp\quad
w^*\varphi\leq w^*u^*\pi\cong e^*\pi
\]

\emph{Ad~\ref{lem:epstack-prime-equiv}.} 
Consider $(u:J\to I,\pi),(v:K\to I,\xi)\in (D\fifa)_I$ such that
$\pi$ and $\xi$ are $\exists$-prime in
$\fifa$ and
$\widetilde{H}(u,\pi)=\exists_u\pi\leq\exists_v\xi=\widetilde{H}(v,\xi)$. Form
the pullback of $u,v$.
\[
 \xymatrix{
& L\ar[r]^w\ar[d]^x\pullbackcorner & K\ar[d]^v \\
M\dashed[ur]^y\ar@{-->>}[r]^e & J\ar[r]^u & I
}
\]
From the Beck-Chevalley condition it follows that $\pi\leq\exists_x w^*\xi$,
and since $\pi$ is prime, we can find a span $(e,y)$ such that $xy=e$ and
$e^*\pi\leq y^*w^*\xi$. The span $(e,wy)$ witnesses the inequality
$(u,\pi)\leq(v,\xi)$, which shows that the canonical functor $D\fifa\to\fifb$
is order reflecting. The condition on essential surjectivity is clear from the
construction of $\widetilde{H}$ given in the proof of Lemma~\ref{lem:da}.
\end{proof}
\begin{lemma}\label{lem:msl-to-frame}
 If $\fifa$ is a pre-stack of meet-semilattices, then $D\fifa$ is its
completion to a fibered frame. More precisely, we have
\begin{itemize}
  \item $D\fifa$ has finite meets, compatible with $\exists$ in the sense of
Frobenius reciprocity.
 \item $y:\fifa\to D\fifa$ preserves finite meets.
\item For any fibered frame $\fifx$, pre-composition with $y:\fifa\to D\fifa$
induces an equivalence
\[
\mpfib(\catr)(\fifa,\fifx)\simeq\ffrm(\catr)(D\fifa,\fifx)
\]
between preorders of fibered monotone maps commuting with $\wedge$ and fibered
monotone maps commuting with $\wedge$ and $\exists$.
\end{itemize}
Moreover, $D\fifa$ is equivalent to the completion of $\fifa$ to a geometric
category from Defintion~\ref{def:d-fib}, which justifies the use of the same
notation.
\end{lemma}
\begin{proof}
 Given predicates $(u:J\to I,\varphi),(v:K\to I,\psi)$, their meet is given by
$(y, x^*\varphi\wedge w^*\psi)$ with $w,x,y$ as in the following pullback
diagram.
\[
\xymatrix{
 L\pullbackcorner\ar[r]^w\ar[d]_x\ar[dr]^y|(.25)\hole & K\ar[d]^v \\
J\ar[r]_u & I
}
\]
This
description makes it clear that finite meets are preserved by $y$; thus, in
particular, precomposition by $y$ induces a monotone map of type
$\ffrm(D\fifa,\fifx)\to\mpfib(\fifa,\fifx)$ which is order-reflecting as we
showed in Lemma~\ref{lem:da}-\ref{lem:da-generated}. It remains to check that
if $F:\fifa\to\fifx$ preserves finite meets, then so does
$\widetilde{F}:D\fifa\to\fifx$. For this it suffices to show (for the binary
case) that for predicates $\gamma\in\fifx_J,\delta\in\fifx_J$ we have
$(\exists_u\gamma)\wedge(\exists_v\delta)\cong\exists_yx^*\gamma\wedge
w^*\delta$, which is an easy exercise.

\medskip

For the claim about the coincidence of the $D$-construction
from this section and the geometric completion from
Section~\ref{suse-geo-fib}, note that fibered frames are the same thing as
posetal \geostack{}s, and fibered geometric functors in the posetal case
coincide with fibered monotone maps commuting with $\wedge$ and $\exists$.
Moreover, as we defined the geometric completion of a finite limit pre-stack
$\fibc$ as the subfibration of $\whfc$ on the sub-representables, 
it is easy to see that the geometric completion of a fibered frame is posetal,
and thus a fibered frame. From this, it follows that the geometric completion
in the posetal case has the same universal property that we just established
for $D\fifa$.
\end{proof}

The fact that $\exists$ in $D\fifa$ is free has an interesting consequence,
described in the following lemma.
\begin{lemma}\label{lem:da-totally-connected}
If $\fifa$ is a pre-stack of meet-semilattices. Then the fibered frame $D\fifa$
is totally connected in the sense of
Defintion~\ref{def:totally-connected-fibered-frame}.
\end{lemma}
\begin{proof}
For a predicate
$(u:J\to I,\varphi)$, $\pi(u,\varphi)$ is the image of $u$ as subobject of $I$.
It is easy to see that this is well defined, left adjoint to $\delta$, and
preserves finite meets.
\end{proof}
\begin{example}\label{ex:dufamia}
 If we apply the $D$-construction to $\ufam\pcaia$ for a typed
pca $\pcaia$, we obtain a fibration that is equivalent to the
{realizability hyperdoctrine} $\hyph\pcaia$
(Definition~\ref{def:realizability-hyperdoctrine}-
\ref{def:realizability-hyperdoctrine-real}). This is not difficult to show by
hand, but the nicest way to understand this is via uniform preorders -- see
Example~\ref{ex:pca-d-hyp-trip}.

In the same way, if $\pcaa$ is an \emph{untyped} pca, the fibration
$D(\ufam(\pcaa))$ is equivalent to the realizability tripos $\rtr{\pcaa}$
(Definition~\ref{def:realizability-tripos}-\ref{def:realizability-tripos-real}).
\end{example}

\subsubsection{Fibered sheaves and assemblies}\label{sec:fibered-sheaves}

\comment{this is phrased too complicated}

In \ref{def:cat-from-geo-fib}, we defined the category $\catr[\fibs]$ for a
\geostack{} $\fibs$ on $\catr$ as the category of equivalence relations
and functional relations in $\sub(\fibs)$. For a fibered frame $\fifx$, given
$\varphi\in\fifx_I$, a predicate in $\sub(\fifx)_\varphi$ is simply a predicate
$\psi\in\fifx_I$ such that $\psi\leq\varphi$, and it seems easier to express
$\catr[\fifx]$ directly in terms of $\fifx$ without making the additional layer
in $\sub(\fifx)$ explicit. When taking this point of view, an equivalence
relation in $\sub(\fifx)$ becomes simply a \emph{partial} equivalence relation
in $\fifx$, and we can easily show the following lemma.
\begin{lemma}\label{lem:sheaf-per}
Given a fibered frame $\fifx:\tot{\fifx}\to\catr$, $\catr[\fifx]$ is equivalent
to the category $\per(\fifx)$. \qed
\end{lemma}
Let us recall that for geomeric fibrations $\fibs$, we defined $\catr[\fibs]$
as the category of total equivalence relations in $\sub(\fibs)$, and we showed
in Lemma~\ref{lem:rs-exact} that $\catr[\fibs]$ is equivalent to
$\per(\siev(\fibs))$. The lemma says that in the posetal case we can drop the
$\siev(-)$.

In the following, we will always view $\catr[\fifx]$ for $\fifx$ a fibered
frame as category of partial equivalence relations. This is the point of view
that is familiar from the tripos-to-topos construction. In particular, in
this representation the diagonal functor $\Delta:\catr\to\catr[\fifx]$ from
Definition~\ref{def:delta-r-rs} has the familiar description
\[
 I\mapsto (I,\predeq),
\]
which is known as the `constant objects functor'. The following lemma is a
direct consequence of Lemma~\ref{lem:decompo}-\ref{lem:decompo-monorepr}.
\begin{lemma}\label{lem:pullback-sub}
Let $\fifx$ be a fibered frame on $\catr$. Then we can reconstruct $\fifx$ from
$\catr[\fifx]$ and $\Delta$ as pullback
\[
 \vcenter{\xymatrix@R-3mm{
\tot{\fifx}\pullbackcorner\ar[r]\ar[d]_{\fifx} & \Sub(\catr[\fifx])
\ar[d]^{\sub(\catr[\fifx])}\\
\catr\ar[r]_\Delta & \catr[\fifx]
}}.
\]
In other words, $\fifx$ is equivalent to the subfibration of
$\gl_\Delta(\catr[\fifx])$ on monomorphisms.
\qed
\end{lemma}

In Section~\ref{sec:dis-sheaf}, we briefly considered he subcategory
$\geopos{\fibs}_1\subseteq \catr[\fibs]$ of \emph{discrete sheaves} that arises
as terminal fiber of the positive cocompletion of a geometric pre-stack $\fibs$.

In the posetal case, say for a fibered frame $\fifx$, this category is called
the category $\asm(\fifx)$ of $\fifx$-assemblies in
\cite[after Propositon~2.9]{vanoosten2008realizability}. We give an explicit
definition for future reference.
\begin{definition}\label{def:assemblies}
 Let $\fifx:\tot{\fifx}\to\catr$ be a fibered frame. The category $\asm(\fifx)$
of \emph{$\fifx$-assemblies}\index{assemblies} is the subcategory of $\catr[\fifx]$ on objects of
the form $(I,\predeq_\varphi)$, where $\varphi\in\fifx_I$, and $(x =_\varphi
y)\;\equiv\; (\varphi(x)\wedge x=y)$.
\end{definition}
Assemblies have particularly good properties if $\fifx$ is {totally
connected}, as we will explain in Section~\ref{sec:assemblies}.

\begin{remark}\label{rem:rosmai}
There are interesting analogies between the description and universal
characterization of the categories $\asm(\fifx)$ and $\catr[\fifx]$ for a
fibered frame as given here, and work of Maietti and Rosolini~\cite{rosmai08}.

Slightly paraphrasing, Maietti and Rosolini describe how to get the categories
$\asm(\fifx)$
and $\catr[\fifx]$ from a fibered frame (or rather an existential fibration) by
a sequence of completion operations on fibrations. In
particular, they introduce what we call the fibered subobject fibration, albeit
with a different viewpoint: in our notation, they view
$\sub(\fifx):\tot{\sub(\fifx)}\to\tot{\fifx}$ as the result of freely adding
Lawvere comprehension to $\fifx$, the coincidence of this with the fibered
subobject fibration occurs only in the posetal case.

The major difference between the approach of \cite{rosmai08} and the present
work is that whereas both are about fibrational (co)completions, in the present
work we adjoin categorical structure to fibrations on a constant base category
whereas Maietti and Rosolini complete the base category with respect to the
logical structure of the fibration. In particular, the categories $\asm(\fifx)$
and $\catr[\fifx]$ occur as terminal fibers here, whereas they occur as base
categories in \cite{rosmai08}.

A nice feature of Maietti and Rosolini's approach is that it highlights the fact
that the base category $\catr$ is actually not essential for the construction of
$\asm(\fifx)$ and $\catr[\fifx]$, both of which only depend on the 
\emph{cartesian bicategory}~\cite{carboni1987cartesian}\index{cartesian 
bicategory} of relations in $\fifx$.
\end{remark}

\subsubsection{Fibered presheaves}

We will now study the fibration of presheaves on a fibered
meet-semilattice $\fifa$, and give a criterion for $\srel{\catr}{\fifa}$ to be
locally cartesian closed. Since the direct manipulation of the fibration of
sieves feels a bit unhandy and cumbersome in the posetal case, we will make use
of the equivalence $\srel{\catr}{\fifa}\simeq\catr[D\fifa]$ from
Corollary~\ref{cor-hat-sh-d}, and work explicitly only in the second category.
An advantage of the first representation would be that we have tracking
families (Definition~\ref{def:tracking-family}, Lemma~\ref{lem:tracking-fam})
available, but with a bit of care we can access them also in $\catr[D\fifa]$,
as we will explain now. 
\begin{definition}\label{def:tracking-rel}
 Let $\fifx:\tot{\fifx}\to\catr$ be a fibered frame, and let
$\phi:(I,\rho)\to(J,\sigma)$ be a morphism in $\catr[\fifx]$, where $\rho$ and
$\sigma$ are partial equivalence relations as in Lemma~\ref{lem:sheaf-per}. A
\emph{tracking relation}\index{tracking!relation}\index{relation!tracking} for $\phi$ is a total relation $t:I\etrel J$ such that
\[\rho(i),(\delta t)(i,j)\ent\phi(i,j)\] holds in $\fifx$.
\end{definition}
\begin{lemma}\label{lem:tracking-rel}
 \begin{enumerate}
  \item\label{lem:tracking-rel-unique}
If $t:I\etrel J$ is a tracking relation of $\phi:(I,\rho)\to(J,\sigma)$,
then 
\[\phi(i,j)\adj\ent\rho(i)\wedge\exists j'\qdot (\delta t)(i,j')\wedge
\sigma(i,j'),\]
 thus $\phi$ is uniquely determined by the tracking relation and
can be reconstructed from it.
\item\label{lem:tracking-rel-characterize}
 A total relation $t:I\etrel J$ is a tracking relation of a morphism of
type $(I,\rho)\to(J,\sigma)$ iff
\[
\rho(i,i'),(\delta t)(i,j), (\delta t)(i',j')\ent \sigma(j,j') 
\]
holds in $\fifx$.
\item\label{lem:tracking-rel-equiv} Two total relations $t,u:I\etrel J$ are
tracking relations of the
\emph{same} morphism of type $(I,\rho)\to(J,\sigma)$, iff
\[
\rho(i,i'),(\delta t)(i,j), (\delta u)(i',j')\ent \sigma(j,j') 
\]
holds in $\fifx$.
 \end{enumerate}
\end{lemma}
\begin{proof}
This is proved in a similar way as Lemma~\ref{lem:tracking-fam}. 
\end{proof}
\begin{lemma}\label{lem:track-prime}
\begin{enumerate}
 \item\label{lem:track-prime-track}
Let $\fifx$ be a fibered frame, and $(I,\rho),(J,\sigma)$ two objects in
$\catr[\fifx]$. If the \emph{existence part}\index{existence part!of partial equivalence relation} $(x\csep \rho(x,x))\in (D\fifa)_I$
of
$\rho$ is $\exists$-prime, then every morphism
$\phi:(I,\rho)\to(J,\sigma)$ has a tracking relation.
\item\label{lem:track-prime-prime}
For a fibered meet-semilattice $\fifa$, every object in $\catr[D\fifa]$ is
isomorphic to an object with $\exists$-prime existence predicate.
\end{enumerate}
\end{lemma}
\begin{proof}
\emph{Ad~\ref{lem:track-prime-track}.}
Given $\phi:(I,\rho)\to(J,\sigma)$, the judgment $\rho(i)\ent\exists
j\qdot\phi(i,j)$ holds since $\phi$ is total. Primality of
$\ilbracks{i\csep\rho(i)}$
implies that there exists a total relation $t:I\etrel I\times J$ such that
$\rho(i),(\delta t)(i,i',j)\ent\phi(i',j)$ and $\pi_I\circ t=\id_I$.
$\pi_I\circ t =\id_I$ means that $t(i,i',j)\ent i=i'$ holds in $\sub(\catr)$,
and thus that $(\delta t)(i,i',j)\ent i=i'$ holds in $\fifx$ since $\delta$
commutes with $\exists$. We claim that $\pi_J\circ t$ is a tracking relation
for $\phi$. Since $\pi_J\circ t$ is just the relation $\ilbracks{i,j\csep
t(i,i,j)}$
(and in particular total since total relations are stable under composition),
it suffices thus to show that $\rho(i),(\delta t)(i,i,j)\ent\phi(i,j)$, which
is just a substitution instance of the judgment following from primality.

\medskip
\emph{Ad~\ref{lem:track-prime-prime}.} We could deduce this directly from
Corollary~\ref{cor-hat-sh-d}, since the objects with $\exists$-prime existence
predicate in $\catr[D\fifa]$ correspond to the objects in $\srel{\catr}{\fifa}$.
Instead, we give an explicit construction since it will be useful later. Let
$(I,\rho)\in\catr[D\fifa]$. By Lemma~\ref{lem:epstack-prime} there exists a
moprphism $u:J\to I$ and a $\exists$-prime predicate $\pi\in (D\fifa)_J$ such
that $\exists_u\pi\cong\delta_I^*\rho$. Define a partial equivalence relation
$\sigma\in (D\fifa)_{J\times J}$ by
$\sigma(i,i')\equiv\pi(i)\wedge\pi(i')\wedge\rho(ui, ui')$. It is then easy to
see that $(J,\sigma)$ has $\pi$ as existence predicate and is isomorphic to
$(I,\rho)$.
\end{proof}

\begin{lemma}\label{lem:eda-ccc}
Let $\tope$ be a topos, and let $\ffrma:\tot{\ffrma}\to\tope$ be a pre-stack of
meet-semilattices. If
$D\fifa$ has implication and universal quantification, then
$\tope[D\fifa]$ is cartesian closed.
\end{lemma}
\begin{proof}
Let $(J,\sigma),(K,\tau)\in\catr[D\fifa]$, and assume without loss of
generality that $\sigma$ has $\exists$-prime existence part.
Denote by $\trel(J,K)\subseteq P(I\times J)$ the object of total
relations from $J$ to $K$ in the sense internal to $\tope$. We claim that an
exponential $(K,\tau)^{(J,\sigma)}$ is given by $(\trel(J,K),\tau^\sigma)$,
where the partial equivalence relation $\tau^\sigma$ is defined by
\[
(\tau^\sigma)(t,u)\equiv \forall jj'kk'\qdot \sigma(j,j')\wedge\delta((j,k)\in
t)\wedge\delta((j',k')\in u)\imp\tau(k,k').
\]
We leave it to the reader to verify that this is indeed a partial equivalence
relation; the fact that the variables $t,u$ range over \emph{total} relations
is important. 

To obtain the evaluation map associated to the exponential object, consider the
relation
$e:\trel(J,K)\times J\etrel K$ which is just the appropriate transposition of
the membership relation $(\in)\subseteq J\times K\times \trel(J, K)$, i.e.,
$e(t,j,k)\Leftrightarrow (j,k)\in t$. The relation $e$ is total precisely
because $\trel(J,K)$ contains only total relations; to verify that it tracks a
morphism
\[
 \ve:(\trel(J,K),\tau^\sigma)\times(J,\sigma)\to(K,\tau),
\]
we have to show (following
Lemma~\ref{lem:tracking-rel}-\ref{lem:tracking-rel-characterize}) that
\[
(\tau^\sigma)(t,u),\sigma(j,j'),\delta((j,k)\in t),\delta((j',k')\in
u)\ent\tau(k,k'),
\]
holds in $D\fifa$, which is immediate from the definition of $\tau^\sigma$.

It remains to show how to construct exponential transposes. Consider $\phi :
(I,\rho)\times (J,\sigma)\to (K,\tau)$, where we assume without loss of
generality that $\rho$ has $\exists$-prime existence part. Since
$\exists$-prime predicates in $D\fifa$ are closed under finite meets (being the
stack-completion of the $y$-image of $y:\fifa\to D\fifa$), $\rho\brprod\sigma$
has $\exists$-prime existence part as well, whence $\phi$ has a tracking
relation
$s:I\times J\etrel K$. This relation corresponds to a unique function $s':I\to
P(J\times K)$, which factors through $\trel(J,K)\hookrightarrow P(J\times K)$ as
$\tilde{s}:I\to\trel(J,K)$ since $s$ itself is total.
\[
\xymatrix@R-3mm{
I\ar[r]^-{\tilde{s}}\ar[dr]_-{s'} & \trel(J,K)\mono[d]\\
& P(J\times K)
} 
\]
It is easy to see that $\tilde{s}$ is a (functional) tracking relation of a
morphism of type $(I,\rho)\to(\trel(J,K),\tau^\sigma)$ which is an exponential
transpose of $\phi$.
\end{proof}
\begin{remarks}\label{rem:eda-ccc}
\begin{enumerate}
 \item 
It is \emph{not} possible to carry out the preceding proof
entirely in the (non-replete) full subcategory of $\tope[D\fifa]$ on objects
with $\exists$-prime existence predicate, since the existence part of
$\tau^\sigma$ might not be $\exists$-prime.

\item\label{rem:eda-ccc-full}
 The proof of the preceding lemma is \emph{impredicative} -- we
make use of the totality $\trel(J,K)$ of total relations between two objects
which we define as a subobject of $P(I\times J)$. In the presence of
appropriate choice principles, however, the use of impredicativity can be
avoided -- for example, if regular epimorphisms split in $\tope$, we can use
tracking \emph{functions}\index{tracking!function} instead of tracking \emph{relations}, which allows us
to take $K^J$ instead of $\trel(J,K)$ as underlying object of the exponential.

In general, all we need to make the proof work is a sufficient supply of
total relations in $\tope$ -- more precisely we need $\tope$ to be a regular
category with universal quantification such that for every pair $J,K\in\tope$ of
objects there exists an object $T(J,K)$ parameterizing a family of \emph{total}
relations \[\ve:E\emono T(J, K)\times J\times K\] from $J$ to $K$ (`total'
meaning that $t,j\csep\ent\exists k\qdot \ve(t,j,k)$ holds) such that for
every $I$-indexed family $r:R\emono I\times J \times K$ of total relations from
$J$ to $K$ we have
\[
 \forall i\;\exists t\;\forall j,k\qdot \ve(t,j,k)\imp r(i,j,k)\footnote{As I
learned from Benno van den Berg, this property is called \emph{fullness}\index{fullness}
in constructive set theory (see~\cite[Definition~4.11]{aczel2001notes}).
}.
\]
This property holds for example in regular locally cartesian closed categories
satisfying the \emph{internal} axiom of choice; here $T(J,K)$ is simply given
by $K^J$. This example is interesting because the internal axiom of choice is
\emph{not} sufficient to allow us to work entirely with tracking functions
instead of tracking relations (we need epi-splitting for that), but in the
construction of the exponential we can still use the function-object. To make
this work, we have to modify the construction of the exponential transposes from
the above proof a bit. Given $s:I\times J\etrel K$ as in the proof, the
associated function $s':I\to P(J\times K)$ does not generally factor through
$K^J\emono P(J\times K)$ (it does only if $s$ is already functional); instead we
define the tracking relation $\tilde{s}:I\etrel K^J$ of the exponential
transpose by 
\[
 \tilde{s}(i,f)\quad\defequi\quad\forall j,k\qdot fj=k\imp s(i,j,k)
\]
and the totality follows from internal choice.
\end{enumerate}
\end{remarks}
In Theorem~\ref{theo:rda-lccc}, we will show that in the situation of the
preceding lemma, $\tope[D\fifa]$ is even \emph{locally} cartesian closed, but
since the proof is a bit subtle we need two more lemmas.
\begin{lemma}\label{lem:excat-lccc-cover}
Let $\catx$ be an exact category, and $e:J\eepi I$ a regular epimorphism in
$\catx$. If $\catx/J$ is cartesian closed, then $\catx/I$ is cartesian closed.

In particular, to show that $\catx$ is \emph{locally} cartesian closed it is
sufficient to show that every object $I$ can be covered by an object $J$ with
cartesian closed slice $\catx/J$.
\end{lemma}
\begin{proof}
By transitivity of slicing it is sufficient to show that if $\catx/I$ is
cartesian closed for an object $I$ with global support, then $\catx$ is already
cartesian closed. 

The idea of the proof is as follows:
\begin{enumerate}
 \item Given $X\in\catx$, the object $\pi_I:I\times X\to I$ in $\catx/I$ is the
$I$-indexed family of objects which has value constant $X$.
\item Constant families in $\catx/I$ can be identified with objects in $\catx$.
\item Given two constant families in $\catx$, their exponential is constant as
well.
\end{enumerate}
To make this precise, we identify the concept of `constant family' by 'object
with descent data' (Definition~\ref{def:cat-descent-data}) which leaves us to
show that the category $\ddesc(\fund{\catx},I\to 1)$ is cartesian closed, which
is sufficient since the fundamental fibration $\lfund{\catx}$ of an exact
category is a stack for the regular topology (this is the formalization of
(ii)).

For the proof we use the representation of $\ddesc(\fund{\catx},I\to 1)$ using
an explicit cleavage from Lemma~\ref{lem:desc-cleavage}.
Concretely, an object in $\ddesc(\fund{\catx},I\to 1)$ is a pair $(b:B\to
I,\beta:\partial_1^*b\to\partial_0^*b)$ subject to the coherence axiom in the
Lemma.
Given two such 
objects, $(b,\beta)$, $(c,\gamma)$, we 
construct a structure map on $c^b$.
The important step here is to realize that (as can easily be verified) pullback
in finite limit categories
preserves exponentials, in particular $\partial_i^*(c^b)$ is an exponential of
$\partial_i^*b$ and $\partial_i^*c$ in $\catx/(I\!\times\! I)$, independently of
the question
whether $\catx/(I\!\times\! I)$ has all exponentials. This observation
allows us to define a structure map on $c^b$ via the following derivation.
\[
\def\fCenter{ \to }
\AX$\pi_1^*(c^b)\fCenter\pi_1^*(c^b)$
\UI$\pi_1^*(c^b)\times \pi_1^*(b)\fCenter\pi_1^*(c)$
\UI$\pi_1^*(c^b)\times \pi_2^*(b)\fCenter\pi_2^*(c)$
\UI$\pi_1^*(c^b)\fCenter\pi_2^*(c^b)$
\DP
\]
It remains to check that this structure map satisfies the coherence axiom, and
that
$c^b$ with this structure map is an exponential of $(b,\beta)$, $(c,\gamma)$ in
$\mathbf{Desc}(\catx,I)$. This follows again from pullback stability of
exponentials, and from calculations $\catx/I$, $\catx/(I\times I)$ and
$\catx/(I\times I\times I)$ which are most easily carried out in
$\lambda$-calculus.
\end{proof}
\begin{lemma}\label{lem:forall-imp-slice}
 Let $\fifa:\tot{\fifa}\to\catc$ be a fibered meet-semilattice and
$\varphi\in\fifa_I$. If $\fifa$ has implication and universal quantification,
then so has $\fifa/\varphi$.
\end{lemma}
\begin{proof}
Recall that a predicate in $(\fifa/\varphi)_u$ for $u:J\to I$ is a predicate
$\psi\in\fifa_J$ such that $\psi\leq u^*\varphi$. To obtain implication and
universal quantification in $\fifa/\varphi$, it is sufficient to perform the
corresponding construction in $\fifa$ and then take the conjunction with
the appropriate reindexing of $\varphi$.
\end{proof}

Now we can prove the theorem.

\begin{theorem}\label{theo:rda-lccc}
 Let $\tope$ be a topos\footnote{Actually all we need is a locally cartesian
closed regular category where the fullness principle from
Remark~\ref{rem:eda-ccc}-\ref{rem:eda-ccc-full} holds in all slices.}, and let
$\ffrma:\tot{\ffrma}\to\tope$ be a pre-stack of
meet-semilattices. The
following are equivalent.
\begin{itemize}
\item $D\ffrma$ has implication and universal quantification.
\item $\catr[D\fifa]$ is locally cartesian closed.
\end{itemize}
\end{theorem}
\begin{proof}
Since $D\fifa$ can be represented by the pullback
\[
\vcenter{ \xymatrix@R-3mm{
{}\pullbackcorner\ar[r]\ar[d]_{D\fifa} &
\ar[d]^{\sub(\tope[D\fifa])}\\
\tope\ar[r]_\Delta & \tope[D\fifa]
}},
\]
it is clear that existence of implication and universal quantification in
$D\fifa$ follow from local cartesian closedness of $\tope[D\fifa]$.

In the converse direction, by Lemma~\ref{lem:excat-lccc-cover} it is
sufficient to show that for all $I\in\tope$, $\varphi\in\fifa_I$, the slice
categories $(\tope[D\fifa])/(I,\predeq_{y\varphi})$ are cartesian closed, since
every object is isomorphic to one with $\exists$-prime existence predicate
(Lemma~\ref{lem:track-prime}-\ref{lem:track-prime-prime}), which in turn can be
covered by an object with the same existence predicate whose partial equivalence
relation is sub-diagonal.

To see that $(\tope[D\fifa])/(I,\predeq_{y\varphi})$ is cartesian closed,
consider the chain
\[
 (\tope[D\fifa])/(I,\predeq_{y\varphi})\simeq (\tope/I)/[D\fifa/y\varphi]\simeq
(\tope/I)/[D(\fifa/\varphi)]
\]
of equivalences, where the first equivalence was proved in
Lemma~\ref{lem:sheaf-slice}, and the second one follows from
Lemma~\ref{lem:d-loc-slice}. To conclude cartesian closure,
it suffices by Lemma~\ref{lem:eda-ccc} to show that $D(\fifa/\varphi)$, or
equivalently $D\fifa/y\varphi$, has implication and universal quantification.
This follows from Lemma~\ref{lem:forall-imp-slice}.
\end{proof}
\begin{remark}
The preceding result, and in particular its proof via
Lemma \ref{lem:excat-lccc-cover}, is based on Carboni and Rosolini's
characterization of categories with locally cartesian closed exact
completion~\cite[Theorem~3.3]{carboni2000locally}.

It would be nice to have a criterion for $\srelrc$ to be locally cartesian
closed for not necessarily posetal finite-limit pre-stacks $\fibc$. If regular
epis split in the base, then by
Carboni and Rosolini's  result $\srelrc$ is locally cartesian closed iff
$\tot{\fibc}$ is weakly cartesian closed in their sense, since then $\srelrc$ is
the exact completion of $\tot{\fibc}$. It seems nontrivial, however, to rephrase
this in a way such that it works without any choice principles, since
projectivity in the fibrational sense (Definition~\ref{def:proj-indec}) is more
difficult to handle than projectivity of objects in regular
categories\footnote{This is related to the fact that given a regular category
$\catr$, $f:P\to I$ is projective in $\catr/I$ iff $P$ is projective in $\catr$
whereas the same is not true for internal projectives. The link between internal
projectives and projectives in fibrations is that $f$ is internal projective in
$\catr/I$ iff it is fibrationally projective in $\fund{\catr}$}.

A careful study of \cite{rosicky1999cartesian} might give new clues.
\end{remark}

\subsection{Assemblies}\label{sec:assemblies}

We introduced assemblies over fibered frames in Definition~\ref{def:assemblies}.
Originally, assemblies were introduced
in the case of realizability over pcas~\cite{carboni1988categorical}, and in
this case they have particularly good properties --
if
$\rtr{\pcaa}:\tot{\pcaa}\to\catset$ is a realizability tripos
(Definition~\ref{def:realizability-tripos}-\ref{def:realizability-tripos-real})
over a pca $\pcaa$, then:
\begin{itemize}
 \item $\asm(\rtr{\pcaa})$ can naturally be represented as a `concrete
category', by
which we
mean that the objects are sets with additional structure, morphisms are
functions which are compatible in some sense with this structure, and equality
of morphism is equality of functions (we emphasize this last condition since it
does \emph{not} hold e.g., for strict tracking families in the sense of
Definition~\ref{def:tracking-family}).
\item $\asm(\rtr{\pcaa})$ is the category of separated objects for the
$\neg\neg$-topology on $\catset[\rtr{\pcaa}]=\catrt(\pcaa)$, and furthermore
$\catset$ itself is
the category of sheaves.
\end{itemize}
The aim of this section is to work out which requirements on a fibered
frame $\fifx$ we need in
order to have these properties. 

The $\neg\neg$-topology is only definable if we have
sufficient logical structure in $\fifx$ and $\catr[\fifx]$, and we can only
expect the base
category to coincide with the $\neg\neg$-sheaves if its internal logic validates
classical logic. However, it will turn out that rather weak conditions are
sufficient to ensure that the assemblies coincide with separated objects for
\emph{some} topology -- all we need for that is that the base category is an
exact category $\catx$, and $\fifx$ is \emph{totally
connected}. This entails that $\catx$ is a localization of $\catx[\fifx]$, and
with this localization comes a category of separated objects, which we can
identify as the assemblies. If $\fifx$ is a tripos and the base is boolean, then
the topology corresponding to the localization is precisely the
$\neg\neg$-topology. Finally, the `concrete' representation of $\asm(\fifx)$
follows formally from facts about the localization.

\begin{lemma}\label{lem:delta-left-adj}
  If $\fifx:\tot{\fifx}\to\catx$ is a totally connected fibered
frame on an \emph{exact category} $\catx$, then $\Delta:\catx[\fifx]\to\catx$
has a finite limit preserving left adjoint $\Pi\adj\Delta$ such that the counit
$\ve:\Pi\Delta\to\id_\catx$ is a natural isomorphism.
\end{lemma}
\begin{proof}
  $\pi:\fifx\to\sub(\catx)$ preserves $\exists$ since it is a left
adjoint. Since $\pi$ furthermore preserves finite meets by assumption, it
preserves partial equivalence relations and functional relations, which allows
us to construct a functor $\catx[\fifx]\to\catx[\sub(\catx)]$, which is
automatically regular. Moreover, since $\catx$ is exact, we have
$\catx[\sub(\catx)]\simeq\catx$ -- composing the former functor with the
equivalence we obtain $\Pi:\catx[\fifx]\to\catx$. It is easy to see that $\Pi$
and $\Delta$ do indeed form a reflection $\Pi\adj\Delta$.
\end{proof}

The previous lemma says in particular that $\catx$ is a \emph{localization} of
$\catx[\fifx]$. For toposes, localizations correspond
precisely to local operators (also known as Lawvere-Tierney topologies) on the
subobject classifier. Since $\catx[\fifx]$ does not necessarily have a subobject
classifier, we can't work with local operators here -- however, we still have
the corresponding \emph{universal closure operation}
$j:\sub(\catx[\fifx])\to\sub(\catx[\fifx])$. Let us recall the relevant
concepts from~\cite[A4.3]{elephant1}.

\begin{itemize}
 \item Let $L:\catc\to\catc$ be a \emph{cartesian reflector}\footnote{That is
an idempotent monad which preserves finite limits.}\index{cartesian reflector}
on a category $\catc$ with finite limits, and denote by $\catl$ the
corresponding replete reflective subcategory.
We can define a fibered functor
\[j:\sub(\catc)\to\sub(\catc)\]
by the construction given in the following diagram.
\begin{equation}\label{eq:closop}
\vcenter{\xymatrix{
U \dashed[r]\mono[rd]_m& V\pullbackcorner\ar[r]^{\eta_V}\mono[d]^{jm} & L
U\mono[d]^{L m} \\
 & X\ar[r]^{\eta_X} & LX
}}
\end{equation}
$j$ is a \emph{universal closure operation}, meaning that it is isotone and
idempotent.
\item
Given a universal closure operation
$j:\sub(\catc)\to\sub(\catc)$,  call $m\in\sub(\catc)_C$
\emph{dense}, if $jm\cong\top$, and \emph{closed} if $jm\cong m$. Call
$S\in\catc$ \emph{$j$-separated}, if given $m$ and $f$ as in
\[
 \xymatrix@-3mm{
U\mono[d]_m\ar[r]^f & S\\
A\dashed[ur]_h
}
\]
with $m$ dense in $\sub(\catc)_A$, there exists at most one mediating $h$, and
call $S$ a \emph{$j$-sheaf}, if for any such $f$ and $m$ there exists
\emph{exactly one} $h$.
\end{itemize}
Given a cartesian reflector $L:\catc\to\catc$ with induced universal
closure operation $j:\sub(\catc)\to\sub(\catc)$, we have by
\cite[Lemma~A4.3.6]{elephant1} that
\begin{itemize}
 \item $A\in\catc$ is $j$-separated \emph{iff} $\eta_A:A\to LA$ is monic
\emph{iff} $A$ is a subobject of an object in $\catl$ \emph{iff} the diagonal
map $A\emono A\times A$ is closed.
\item
$A\in\catc$ is a $j$-sheaf \emph{iff} $\eta_A:A\to LA$ is an isomorphism
\emph{iff} $A\in\catl$.
\end{itemize}
For universal closure operations on \emph{exact} categories, we can furthermore
show the following.

\begin{lemma}\label{lem-j-ex}
  Let $j:\sub(\catx)\to\sub(\catx)$ be a universal closure operation on an exact
  category $\catx$.
  \begin{enumerate}
  \item\label{lem-j-ex-j-sep-reg} $j$-separated objects in $\catx$ are closed
under finite products and
    subobjects. Thus, the full subcategory of separated objects is closed under
    finite limits, regular, and an arrow is a regular epi in the subcategory
    iff it is one in $\catx$.
    \item\label{lem-j-ex-j-sep-refl} The subcategory $\sep_j(\catx)$ of
$j$-separated objects is
reflective in $\catx$.
  \end{enumerate}
\end{lemma}
\begin{proof}
  \emph{Ad \ref{lem-j-ex-j-sep-reg}.} The terminal object is separated since its
diagonal predicate is
  already maximal in the lattice of subobjects. To see that the product of
  separated objects $X,Y\in\catx$ is separated, we have to verify the validity
  of \[x,x'\vtp X,\;y,y'\vtp E\csep j(x=x'\wedge y=y')\ent x=x'\wedge
  y=y'.\] This holds since $j$ commutes with conjunction.  For subobjects, let
  $m:U\emono X$ be a monomorphism into a separated object. We have
  \[u,v\vtp U\csep\;j(u=v)\adj\ent j(mu=mv)\adj\ent mu=mv\adj\ent u=v.\] 

  \emph{Ad \ref{lem-j-ex-j-sep-refl}.} Let $X\in\catx$. The closure $j\delta_X$
of the diagonal
  predicate is an equivalence relation since $j$ commutes with finite meets. We
  claim that the quotient $p:X\eepi SX$ of $X$ by $j\delta_X$ is the separated
  reflection of $X$. To see that $SX$ is separated, observe that $x,x'\vtp
  X\csep px=px'\adj\ent
  j(x=x')$, and thus $j(px=px')\ent px=px'$. This implies $y,y'\vtp SD\csep
  j(y=y')\ent y=y'$ since $e$ is an epimorphism. Now an arbitrary map
  $f:X\to Y$ with $Y$ separated lifts along $p$ since its kernel is $j$-closed
  and thus contains $j\delta_X$.
\end{proof}
In our case, the cartesian reflector is given by
$\Delta\Pi:\catx[\fifx]\to\catx[\fifx]$, and the corresponding reflective
subcategory is equivalent to $\catr$. 
The associated universal closure operator $j$ is best understood as an extension
of $\delta\pi:\fifx\to\fifx$ to $\sub(\catx[\fifx])$, an explicit description
will be given in Lemma~\ref{lem:tc-dense}-\ref{lem:tc-dense-repr-j}. We will now
give a nice, `assembly style' style representation of $\sep_j(\catx[\fifx])$.
\begin{definition}\label{def-subfib-dense}
  Let $\fifx$ be a totally connected fibered frame. We denote by
  $\fifx_d$ (the \emph{dense} part of $\fifx$) the subfibration of
$\fifx$
  on the predicates $\varphi$ such that $\ent_\catc\pi\varphi$, or equivalently
  $\ent_\fifx\delta\pi\varphi$.
\end{definition}
$\fifx_d$ is closed under finite meets since $\pi$ preserves them, and
we obtain an
assembly-style presentation of the separated objects in $\catx[\fifx]$.

\begin{lemma}\label{lem:tc-dense}
  Let $\fifx:\tot{\fifx}\to\catx$ be a totally connected fibered frame on an
exact
  category.
\begin{enumerate}
\item\label{lem:tc-dense-repr-j} In terms of the representation of subobjects by
strict predicates, 
the
universal closure
operation $j$ on
  $\catx[\fifx]$ (see Diagram~\eqref{eq:closop}) is
  given by $\varphi\mapsto\ilbracks{c\csep \delta\pi\varphi c\wedge \rho c}$,
where
  $\varphi\in\fifx_C$ represents a subobject of $\cro$.
\item\label{lem:tc-dense-total} $\sep_j(\catx[\fifx])\simeq\tot{\fifx_d}$ via
the embedding
  $(\varphi\in\fifx_{d,C})\mapsto(C,\predeq|_\varphi)$.
\item\label{lem:tc-dense-asm} $\sep_j(\catx[\fifx])$ coincides up to equivalence
with the terminal fiber $\geopos{\fifx}_1$ of the positive completion of $\fifx$
from
Section~\ref{sec:dis-sheaf}.
\end{enumerate}
\end{lemma}
\begin{proof}
  \emph{Ad \ref{lem:tc-dense-repr-j}.}
Boring calculation.

\emph{Ad \ref{lem:tc-dense-total}.} To see that the embedding is faithful,
assume that $\varphi\in\fifx_C,\psi\in\fifx_D$ are dense predicates.
Since $\fifx$ is faithful, morphisms of type $\varphi\to\psi$ in $\tot{\fifx}$
can be identified with morphisms $f:C\to D$ in $\catx$ such that
$\varphi(x)\ent\psi(fx)$. 
Take two such maps $f,g:C\to D$ such that the
induced  maps of type
$(C,\predeq|_\varphi)\to(D,\predeq|_\psi)$ that are tracked by $f$ and
$g$ are equal.
Then we have $\varphi c,fc=d\ent_\fifx gc=d$, and applying $\pi$ gives
$fc=d\ent_\catx gc=d$ which implies that $f=g$. 

For fullness, observe that
$\Pi(C,\predeq|_\varphi)$ is
isomorphic to $C$ if $\varphi$ is dense (by construction of $\Pi$ in the proof
of~\ref{lem:delta-left-adj}). Given
$\phi:(C,\predeq|_\varphi)\to(D,\predeq|_\psi)$, a preimage of $\phi$ can be
obtained by composing
 $\Pi\phi$ with these isomorphisms.

It remains to check that the essential image coincides with the separated
objects. Let $A\in\sep_j(\catx[\fifx])$. Then the monomorphism
$\eta_A:A\to\Delta\Pi A\cong (I,\predeq)$ corresponds to a predicate
$\varphi\in\fifx_I$ which is dense in $\fifx$ since $\eta_A$ is dense and by
\ref{lem:tc-dense-repr-j}.

\emph{Ad \ref{lem:tc-dense-asm}.} In
Definition~\ref{def:geopos}, $\geopos{\fifx}_1$ is defined as full subcategory
of $\catx[\fifx]$ on `discrete equivalence relations'; in the posetal case this
means subobjects of constant objects $(I,\predeq)$. $j$-separated objects, on
the other hand, are precisely the subobjects of $j$-sheaves, and the
proposition follows since constant objects coincide with sheaves.
\end{proof}

\subsubsection{Total connectedness and double negation}

\begin{lemma}
Let $\trip:\tot{\trip}\to\tope$ be a tripos on a topos. Then the embedding 
$
    \delta:\sub(\tope)\to\trip
$
preserves $\bot$.
\end{lemma}
\begin{proof}
Let $?:0\to I$ in $\tope$. We have to show that $\exists_?\top \leq\bot$ in
$\trip_I$, 
which is equivalent to $\top\leq\bot$ in $\trip_0$. But this follows from the
fact that the predicates over $0$ can be parametrized by $\tope(0,\prop)$, and
there exists only one such map.
\end{proof}
\begin{lemma}
Let $\trip:\tot{\trip}\to\tope$ be a totally connected tripos on a
topos. Then 
$\delta:\sub(\tope)\to\trip$ preserves implication.
\end{lemma}
\begin{proof}
Let $U,V\subseteq I$ in $\tope$, $\phi\in\trip_I$. We
have
\begin{align*}
\phi&\leq\delta (U\imp V) & \text{iff}\\
\pi \phi&\leq U\imp V & \text{iff}\\
\pi \phi\wedge U&\leq V & \text{iff}\\
\pi \phi\wedge \pi\delta U&\leq V & \text{iff}\\
\pi( \phi\wedge\delta U)&\leq V & \text{iff}\\
\phi\wedge\delta U&\leq \delta V & \text{iff}\\
\phi&\leq\delta U\imp \delta V \\
\end{align*}
\end{proof}
The preceding proof works in general for strong monoidal reflections
between monoidal closed categories. The lemma can also be seen as an analogue of
the
fact that for locally connected geometric morphisms
$\Delta\adj\Gamma:\tope\to\tops$ between toposes, the fibered functor
$\Delta:\fund{\tops}\to\gl_\Delta(\tope)$ preserves fiberwise cartesian closed
structure~\cite[Proposition~C3.3.1]{elephant2} (and actually the lemma can also
be proved in the same way, without relying on the fact that $\pi$ preserves
finite meets).
\begin{corollary}
    Let $\trip:\tot{\trip}\to\tops$ be a totally connected tripos on an
arbitrary
topos. Then
$\delsep$ preserves negation.
\end{corollary}
\begin{proof}
 This follows from the preservation of falsity and implication.
\end{proof}
Let us recapitulate. If $\trip$ is a totally connected regular tripos, then
$\delta:\sub(\tope)\to\trip$ preserves $\exists, \wedge, \top$ for general
reasons; furthermore it preserves $\forall$ since it is a right adjoint, and we
just showed that it also preserves $\imp$ and $\bot$. This means that the only
connective that is \emph{not} preserved is disjunction $\vee$. Similarly,
$\pi:\trip\to\sub(\tope)$ preserves $\wedge, \top$ by assumption, and $\exists,
\bot, \vee$ since it is a left adjoint.

We can now show that the closure operation $\delta\pi$ on $\trip$ coincides
with double negation whenever the base is boolean.
\begin{theorem}
    Let $\trip:\tot{\trip}\to\tope$ be a totally connected tripos on a
\emph{boolean}
topos $\tope$.
Then for any $A\in\tope$ and $\varphi\in\trip_A$, we have
$\delta\pi\varphi\cong\neg\neg\varphi$.
\end{theorem}
\begin{proof}
Let $\varphi\in\trip_A$ for $A\in\tope$.
The first implication is shown as follows
\[
\def\fCenter{ \leq }
\AX$\varphi\fCenter\delta\pi\varphi$
\UI$\neg\neg\varphi\fCenter\neg\neg\delta\pi\varphi$
\AX$\neg\neg\pi\varphi\fCenter\pi\varphi$
\UI$\delta\neg\neg\pi\varphi\fCenter\delta\pi\varphi$
\UI$\neg\neg\delta\pi\varphi\fCenter\delta\pi\varphi$
\BI$\neg\neg\varphi\fCenter\delta\pi\varphi$
\DisplayProof,
\]
and here is the proof of the second implication
\[
\def\fCenter{ \leq }
\AX$\neg\varphi\wedge\varphi\fCenter\bot$
\UI$\pi\neg\varphi\wedge\pi\varphi\fCenter\bot$
\UI$\pi\neg\varphi\fCenter\neg\pi\varphi$
\UI$\neg\varphi\fCenter\delta\neg\pi\varphi$
\UI$\neg\varphi\fCenter\neg\delta\pi\varphi$
\UI$\neg\varphi\wedge\delta\pi\varphi\fCenter\bot$
\UI$\delta\pi\varphi\fCenter\neg\neg\varphi$
\DP.
\]
\end{proof}
\begin{corollary}
     Let $\trip:\tot{\trip}\to\tope$ be a totally connected tripos on a
\emph{boolean}
topos $\tope$. Then we have
\[
 j(m)\cong\neg\neg m\qquad\text{for any }A\in\tope[\trip]\text{ and
}m\in\sub(\tope[\trip])_A,
\]
where $j$ is the universal closure operation on $\sub(\tope[\trip])$
corresponding to the reflection $\Pi\adj\Delta$.
\end{corollary}
\begin{proof}
 Assume that $A\in\tope[\trip]$ given by $(I,\rho)$. 
Relative to the representation of predicates in $\sub(\tope[\trip])_{(I,\rho)}$
by predicates in $\trip_I$ which are strict with respect to $\rho$, we know by
Lemma~\ref{lem:tc-dense}-\ref{lem:tc-dense-repr-j} that $j$ is given by
\[\varphi\mapsto \ilbracks{i\csep (\delta\pi\varphi)( i)\wedge \rho(i)}.\] In
the same way, $\neg\neg$ on $\sub(\tope[\trip])$ can be expressed in terms of
$\neg\neg$ on $\trip$ by 
\[\varphi\mapsto \ilbracks{i\csep (\neg\neg\varphi)(
i)\wedge \rho(i)}.\]
The claim then follows from the theorem.
 \end{proof}

\chapter{Uniform preorders}\label{chap:ufp}

Uniform preorders are representations of fibered preorders. More precisely, the
category $\catufp$ of uniform preorders can be identified with a full
subcategory of the locally ordered category $\catpfib(\catset)$ of fibered
preorders on $\catset$. The definition of $\catufp$ is essentially a combination
of ideas
from Hofstra's~\cite{hofstra2006all} work on \emph{basic combinatory objects}\index{basic combinatory object}
(BCOs)\index{BCO}, and Longley's~\cite{longley2011computability} work on
\emph{computability structures}\index{computability structure} (C-structures)\index{C-structure}.

Hofstra and Longley both introduce locally ordered categories of combinatory
structures as frameworks for an abstract study of concepts from realizability
(most importantly pcas, ordered pcas and typed pcas). $\catufp$ can be
viewed as having Longley's objects and Hofstra's morphisms. We will see
later that we can in fact recover Hofstra's BCOs as a full subcategory, and
Longley's C-structures as a Kleisli-category of $\catufp$.

Retrospectively, Hofstra's and Longley's approach can be contrasted by saying
that Longley works with relations, and Hofstra with (partial) functions. Our
approach is to take relations as structuring data of the objects, and functions
as morphisms. This choice is justified by the fact that $\catufp$ is equivalent
to a subcategory of $\catpfib(\catset)$ that can be characterized in a concrete
way (see Lemma~\ref{lem:ufp-embed}).

\medskip

We compose relations like functions and denote their composition by $\circ$ or
by juxtaposition, i.e.\ if $(x,y)\in r$ and $(y,z)\in s$, then $(x,z)\in s
r$. In particular, we allow composition of relations with functions. For a
relation $r$, $r^\circ$ denotes its opposite relation.

\section{Definitions}

\begin{definition}\label{def:uniform-preorder}
  \begin{enumerate}
  \item A \emph{uniform preorder}\index{uniform preorder} (or
C-structure)\index{C-structure} $\upa$ is a triple $\upa=(I,
A, R)$,
    where $A=(A_i)_{i\in I}$ is a family of sets, and $R=(R_{ij})_{i,j\in I}$,
    $R_{ij}\subseteq P(A_i\times A_j)$ is a family of sets of relations,
    subject to the following axioms.
    \begin{enumerate}
    \item $i,j\in I,\;r\in R_{ij},\; s\subseteq r\implies s\in R_{ij}$
    \item $i\in I\implies \id\in R_{ii}$
    \item $i,j,k\in I,\;r\in R_{ij},\;s\in R_{jk}\implies sr\in R_{ij}$
    \end{enumerate}
  \item A \emph{monotone map}\index{monotone map!of uniform preorders} between uniform preorders $\upiar$, $\upjbs$ is
    a pair $(u:I\to K,(f_i:A_i\to B_{ui})_{i\in I})$ such that $r\in R_{ij}$
    implies $f_jrf_i^\op\in S_{ui,uj}$\footnote{$f_jrf_i^\op$ is just another
      way of writing $(f_i\times f_j)(r)$. We chose this representation since it
      is most natural in some calculations in Section~\ref{suse-calc-dist}.}.
\item
    Given $(u,f),(v,g):\upiar\to\upjbs$, we define $(u,f)\leq(v,g)$ iff for all
    $i\in I$ we have $\{(f_ia,g_ia)\msep a\in A_i\}\in S_{ui,vi}$.
  \end{enumerate}
Uniform preorders and monotone maps form an order-enriched category
$\catufp$.
\end{definition}
We call the set $I$ of a uniform preorder $\upiar$ its set of
\emph{sorts}\index{sort!of uniform preorder}. If $I$ has exactly one element,
we simply write $\brar$ for the uniform preorder. We sometimes refer to
one-sorted uniform preorders as \emph{basic relational objects}\index{basic
relational object} (BROs)\index{BRO}.

It is often convenient to describe a uniform preorder by giving only a
generating system of relations. To this end, we introduce the concept of
\emph{base}.
\begin{definition}\label{def:base-ufp}
  Let $A=(A_i)_{i\in I}$ be a family of sets. A \emph{base} for a uniform
  preorder \index{base!of a uniform preorder} structure on $A$ is a family $(R_{ij})_{i,j\in I}$ with
  $(R_{ij}\subseteq P(A_i\times A_j))$ of sets of binary relations such that
  \begin{itemize}
  \item $i\in I\implies \exists r\in R_{ii}\qdot \id_{A_i}\subseteq r$
  \item $i,j,k\in I,\; r\in R_{ij},\; s\in R_{jk}\implies \exists t\in
    R_{ik}\qdot sr\subseteq t$
  \end{itemize}
  Given such a base, the family $\adcl R=(\adcl R_{ij})_{ij}$ with
  $\adcl R_{ij} = \{r\subseteq A_i\times A_j\msep \exists s\in R_{ij}\qdot
  r\subseteq s\}$ is a uniform preorder structure on $A$, and we call
  $(I,A,\adcl R)$ \emph{the uniform preorder generated by $(I,A,R)$}.
\end{definition}
Longley defines C-structures directly in terms of bases, without imposing the
downward closedness condition. This has the advantage that it generalizes
to predicative contexts, since predicatively it makes sense to talk about
families of subsets, but not about downward closed such families.
Products (Lemma~\ref{lem-closure-uord}-\ref{lem-closure-uord-prod}) are also
most easily defined in terms of bases, since we can avoid an additional
downward closure operation. 
In the reconstruction of uniform preorders from fibered preorders in the proof
of~\ref{lem:reconstuct-ufp}, on the other hand, the downward closure condition
comes for free.

\begin{examples}\label{ex:uords}
 \begin{enumerate}
  \item \label{ex:uords-preords}
 Ordinary preorders can be viewed as uniform preorders. More precisely, given a
preorder $(D,\leq)$, the singleton set $\{\leq\}$ is a base of a
uniform preorder structure
on $D$ in the sense of Definition~\ref{def:base-ufp} (this follows directly
from reflexivity and transitivity), whence $(D,\adcl\{\leq\})$
is a uniform preorder. Moreover, the assignment
$(D,\leq)\mapsto(D,\adcl\{\leq\})$ extends
to a 2-functor of type
\begin{equation}\label{eq:preord-to-uord}
 \catord\to\catuord
\end{equation}
which is a {local equivalence}.
\item\label{ex:uords-pca}
\begin{enumerate}
\item\label{ex:uords-pca-pca}
 Given a pca $\pcaa$, the partial functions of the form $(a\appca\,
-):\pcaa\pto\pcaa$ for $a\in \pcaa$ form a base of a uniform preorder structure
on $\pcaa$
(the closure under composition follows from functional completeness).
We denote the generated uniform preorder by $(\pcaa, R(\pcaa))$.
\item\label{ex:uords-pca-t}
Given a typed pca $\pcaia$, the partial functions of the
form $(a\appca\,
-):\pcaa_i\pto\pcaa_j$ for $a\in \pcaa_{i\imp j}$ form a base of a
uniform preorder structure on $\pcaia$. We
denote the generated uniform preorder by
$\ufpcaia$.
\item\label{ex:uords-pca-i}
Given an \emph{inclusion} $\pcaas\subseteq\pcaa$ of pcas, the partial functions
of the form $(a\appca\,
-):\pcaa\pto\pcaa$ for $a\in \pcaas$ form a base of a uniform preorder structure
on $\pcaa$.
We denote the generated uniform preorder by $(\pcaa,R(\pcaas))$.
\item\label{ex:rel-typed-pca}\label{ex:uords-tpca-it}
In the same way, an inclusion $(I,(\pcaasi\subseteq\pcaai)_{i\in I})$ of
\emph{typed} pcas induces a uniform preorder 
$\ufpcaias$
\end{enumerate}
\item\label{ex:uords-prim}
 Denote by $\prim$ the set of unary primitive recursive functions. Since $\prim$
is closed under composition, it is a base of a uniform preorder structure on
$\N$. We call
the induced uniform preorder $(\N,\adcl\prim)$ the \emph{primitive recursive
uniform preorder}\index{primitive recursive uniform preorder}.

We will see later that $(\N,\adcl\prim)$ has finite meets
(Section~\ref{sec:finitely-complete-ufps}), but is not relationally complete
(Section~\ref{sec:reational-completeness}). We can perform a construction
analogous to the construction of the effective topos using only primitive
recursive functions, but the resulting category will only be a
pretopos\footnote{More precisely, it seems to be a \emph{list-arithmetic
pretopos} in the sense of~\cite{maietti2010joyal}.}.
\end{enumerate}
\end{examples}

\section{Uniform preorders and fibered preorders}

\begin{definition}
\label{def:ufp-fib}
  For any uniform preorder $\upa=\upiar$, we define a fibered preorder
  \[\ufam(\upa):\Ufam(\upa)\to\catset\]
 as follows.
  \begin{itemize}
  \item a predicate on a set $M$ is a pair $(i\in I,\phi:M\to A_i)$
  \item given $(i,\phi),(j,\gamma)\in\ufam(\upa)_M$, we define 
\[(i,\phi)\leq(j,\gamma)\defequi{}\{(\phi m,\gamma m)\msep m\in M\}\in R_{ij}\]
\item reindexing is given by precomposition
  \end{itemize}
\end{definition}
We will often omit the indices when talking about predicates in
$\ufam(\upa)$ and just write $\phi$ instead of $(i,\phi)$.

\begin{lemma}\label{lem:ufp-embed}
  The assignment $\upa\mapsto\ufam(\upa)$ gives rise to a 2-functor
  \[\ufam(-):\catufp\to\catpfib(\catset)\]
 into posetal fibrations on $\catset$
  which is a local equivalence.
This 2-functor fits into a commutative (up to isomorphism) triangle
\[
\vcenter{\xymatrix{
\catord \ar[r]\ar[dr]_{\famf(-)} & \catuord\ar[d]^{\ufam(-)} \\
\emar[ru]|(.7)\cong& \catpfib(\catset) }}
\]
of local equivalences, 
where the horizontal arrow is defined in
Example~\ref{ex:uords}-\ref{ex:uords-preords}, and the diagonal arrow is the
posetal version of the family construction for categories
from Definition~\ref{def:fam-fibration}.

\end{lemma}
\begin{proof}
  It is easy to see that $\ufam(-)$ is functorial. To see that it is locally
  order-reflecting, consider monotone maps $(u,f),(v,g):\upiar\to\upjbs$ such
  that $\ufam(u,f)\leq\ufam(v,g)$. Then we have in particular that for any
  $i\in I$, $\ufam(u,f)(\id_{A_i})\leq\ufam(v,g)(\id_{A_i})$, which
  implies $(u,f)\leq (u,g)$.

  To see that $\ufam(-)$ is essentially full, let
  $ F:\ufam\upiar\to\ufam\upjbs$.  For each $i\in I$,
  $ F(i,\id_{A_i})$ is a pair of an element $u(i)\in J$ and a
  $u(i)$-valued predicate on $A_i$, i.e.\ a map $f_i:A_i\to B_{ui}$. We claim
  that the assignment $i\mapsto u(i)$ together with the maps $(f_i)_{i\in I}$
  gives the desired monotone map. Indeed, for $h:M\to A_i$, we have
  $ F(i,h)= F((i,\id_{A_i})\circ h)\cong F(i,\id_{A_i})\circ
  h=(ui,f_i)\circ h$, thus the action of $ F$ is given by postcomposition
  with $ F(\id_{A_i})$.

  It remains to show that $(u,f)$ is monotone. Let $r\in R_{ii'}$. Then we
  have projection mappings $\pil:r\to A_i$, $\pir:r\to A_{i'}$ such that
  $\pil\leq\pir$ as predicates in $\ufam\upiar$. By monotonicity of $ F$ we
  deduce that $f_i\pil\cong F(\pil)\leq F(\pir)\cong f_{i'}\pir$, which
  means that $\{(f_ia,f_{i'}b)\msep (a,b)\in r\}\in S_{ui,ui'}$ as required.

The commutativity of the triangle of 2-functors is straightforward.
\end{proof}
\begin{example}
Given a typed pca $\pcaia$, the fibration $\ufam\ufpcaia$ obtained by
applying the family construction to the uniform preorder $\ufpcaia$ defined in
Example~\ref{ex:uords}-\ref{ex:uords-pca-t} is \emph{precisely} the fibration
$\ufam\pcaia$ from Definition~\ref{def:fam-pca}.

In the same way, for an untyped pca $\pcaa$, $\ufam\ufpcaa$
is equal to
$\ufam(\pcaa)$.
\end{example}

\begin{definition}\label{def:gen-fam-pred}
  Let $\ffrma:\tot{\ffrma}\to\catset$ be a fibered preorder. We way that a
  family $(\iota_i\in\ffrma_{A_i})_{i\in I}$ is a \emph{generic family of
    predicates}\index{generic family of predicates}, if for every set $M$ and
every predicate $\phi\in \ffrma_M$ there
  exists $i\in I$ and $f:M\to A_i$ such that $\phi\cong f^*\iota_i$.

If a generic family comprises exactly one predicate, we call it a \emph{generic
predicate}\index{generic predicate}.
\end{definition}

\begin{lemma}\label{lem:reconstuct-ufp}
  A fibered preorder $\fpa$ can up to equivalence be represented by a uniform
  preorder iff it has a generic family $(\iota_i\in\ffrma_{A_i})_{i\in I}$ of
predicates and is a pre-stack with
  respect to the regular topology, which means that $e^*\phi\leq e^*\psi$
implies
  $\phi\leq \psi$ for surjective $e:J\eepi I$\footnote{
The pre-stack condition is redundant in the presence of choice.
}.
\end{lemma}
\begin{proof}
Given a uniform preorder $\upiar$, a generic family of predicates for
$\ufam\upiar$ is given by the family of identity maps $(\id_{A_i})_{i\in I}$.

Conversely, let $\fpa$ be a 
  posetal pre-stack with generic family of predicates $\iota_i\in P_{A_i}$. We
  define a uniform preorder structure on the family $(A_i)_i$ by setting
  $R_{ij}=\{r\subseteq A_i\times
  A_j\msep\pil^*\iota_i\leq\pir^*\iota_j\stext{in}\fpa_r\}$, where $\pil:r\to
A_i,\pir:r\to A_j$ are the projections. To see that
  $\ufam\upiar\simeq \fifa$, it remains to show that for $f:M\to A_i$, $g:M\to
  A_j$ we have $gf^\circ\in R_{ij}$ iff $f^*\iota_i\ent g^*\iota_j$. This can
  be seen by considering the epi-mono factorization $M\eepi
  gf^\circ\hookrightarrow A_i\times A_j$ of $\langle f,g\rangle$ and making use
  of the fact that $\fpa$ is a pre-stack.
\end{proof}
The preceding lemma implies in particular that any regular tripos can be
represented by a uniform preorder.

\section{Functional uniform preorders and modesty}\label{sec:functional-ufps}

\begin{definition}\label{def:functional-ufp}
  We call a uniform preorder $\upiar$ \emph{functional}\index{functional uniform preorder}\index{uniform preorder!functional} if all relations $r\in
  R_{ij}$ ($i,j\in I$) are functional.
\end{definition}
A uniform preorder given by a base (Definition~\ref{def:base-ufp}) is of course
functional whenever all relations in the base are functional. In particular,
the uniform preorders induced by pcas
(Example~\ref{ex:uords}-\ref{ex:uords-pca}) and the primitive recursive uniform
preorder (Example~\ref{ex:uords}-\ref{ex:uords-prim}) are all
functional.

Functionality is an `evil' property, in that it is not closed under equivalence
in the locally ordered category $\catuord$. There is, however, a stronger
condition which is closed under equivalence.
\begin{lemma}
Let $(\bbtwo, \adcl\{\leq\})$ be the uniform preorder associated to the
preorder $\bbtwo=\{0\leq 1\}$
, and let $\brar$ be a uniform preorder. Then every monotone map
$f:(\bbtwo, \adcl\{\leq\})\to(B,S)$ factors (up to isomorphism\footnote{The `up
to isomorphism' is only necessary to make the property stable under equivalence
as promised.}) through the terminal uniform preorder.
\end{lemma}
\begin{proof}
This follows from the fact that the image $(f\times f)(\leq)$ of the order
relation is functional since it is in $R$.
\end{proof}
The previous lemma shows how different functional uniform preorders are from
ordinary preorders -- one can explore the structure of a preorder by looking at
maps from $\bbtwo$ (or any other linear order) into the preorder. If we do the
same thing with a functional uniform preorder, we can't see anything. The idea
of studying preorders by sampling them with simple shapes (chains) is
known as the `nerve construction' which associates a simplicial set to
a given preorder (or category). Functional uniform preorders have trivial
nerves, and it does not seem to be possible to think about them in a
`geometric' way.

\medskip

Lemma~\ref{lem:reconstuct-ufp} gives a criterion for a fibered preorder to be
representable by a uniform preorder. If we want the induced uniform preorder to
be functional, we have to add an additional condition.
\begin{definition}\label{def:modest}
Let $\fifa:\tot{\fifa}\to\catr$ be a posetal pre-stack. We call a predicate $\mu\in\fifa_I$
\emph{modest}\index{modest predicate}, if for any span
$\xymatrix@1{J & K\depi[l]_{e}\ar[r]^{f} & I}$ where $e$ is a regular epi, and
any $\varphi\in\fifa_J$ such that $e^*\varphi\leq f^*\mu$, there exists a
(necessarily unique) $h:J\to I$ such that $he=f$ and $\varphi\leq h^*\mu$.
\end{definition}
\begin{remark}\label{rem:modest-discrete}
There is a concrete and an abstract intuition about modesty. The concrete one
is that in a realizability tripos over a pca, a predicate is modest iff
distinct elements have disjoint sets of realizers (in particular, modest
predicates are \emph{not} stable under reindexing!). Abstractly modesty is
about functionality -- the previous definition states that the relation induced
by the span $(e,f)$ is functional.

\medskip

Retrospectively, the use of the word `modest' for the above concept doesn't seem to be
such a good idea after all -- the reason is that in realizability (say over a
pca $\pcaa$) it only coincides with the intended meaning when applied to the
fibration $\ufam(\pcaa)$, not for the tripos $\rtr{\pcaa}$ -- there the empty
truth value causes problems.

Hyland, Robinson and Rosolini \cite{hrr90} define a discrete object in the
effective topos to be an object whose terminal projection is right orthogonal
(Definition~\ref{def:facsys}-\ref{def:facsys-orth}) to 
$\Omega\to 1$ (which is equivalent to being orthogonal to $\Delta 2\to 1$),
and define a modest object to be a discrete object which is separated for the
$\neg\neg$-topology. We give a generalization of their concept of discreteness
in Definition~\ref{def:discrete}.

Summarizing, while the concepts of modest and discrete used in this work are
related, their relation differs from their relation in \cite{hrr90}, which is
mainly due to the fact that our use of `modest' is a bit unfortunate.
\end{remark}
In analogy to Lemma~\ref{lem:reconstuct-ufp}, we can now show the following.
\begin{lemma}\label{lem:modest-functional}
 A posetal pre-stack $\fifa:\tot{\fifa}\to\catset$ is induced by a functional
uniform preorder iff it has a generic family $(\iota_i\in\ffrma_{A_i})_{i\in
I}$ of \emph{modest} predicates.
\end{lemma}
 \begin{proof}
  Let $\upiar$ be a functional uniform preorders. We have to show that the
predicates $\id_{A_i}$ for $i\in I$ are modest. Take a span $\xymatrix@1{M &
N\depi[l]_{e}\ar[r]^{f} & I}$, and $\varphi:M\to A_j$ such that $e^*\varphi\leq
f^*\id_{A_i}$. Then $\{(g(en),fn)\msep n\in N\}\subseteq A_j\times A_i$ is in
$R_{ji}$ and thus functional, which implies that the relation $\{(gen,fn)\msep
n\in N\}\subseteq M\times A_i$ is functional as well. The second relation is
furthermore total since $e$ is surjective, and gives the desired mediator.

Conversely, assume that $\fifa$ has a generic family
$(\iota_i\in\ffrma_{A_i})_{i\in
I}$ of modest predicates. The uniform preorders structure on $(A_i)_{i\in I}$
is then given by $R_{ij}=\{r\subseteq A_i\times A_j\msep
\pil^*\iota_i\ent\pir^*\iota_j\stext{in}\fpa_r\}$. To see that modesty of
$\iota_j$ implies that all $r\in R_{ij}$ are functional, consider the following
diagram.
\[
\xymatrix@R-4mm{
r\ar[rd]^f\depi[d]_e \\
\dashed[r]_h\mono[d]_m & A_j\\
A_i
} 
\]
The existence of $h$ follows from the facts that $\iota_j$ is modest and that
$(me^*)\iota_i\leq f^*\iota_j$. We can deduce that $e$ is an iso since $me$ and
$f$ are jointly monic; and thus $r$ is functional.
 \end{proof}

\section{Uniform preorders and BCOs}

We know how to compare uniform preorders with Hofstra's
BCOs~\cite{hofstra2006all}, since $\bco$ as well as $\catufp$
can
be identified with full subcategories of $\pfib(\catset)$. Given a BCO
$(A,\leq,\mcf)$,
the relations $(\leq)\circ f$ with $f\in\mcf$ generate a uniform preorder
structure on $A$ which induces the same fibered preorder, thus $\bco$ is a
subcategory of $\catufp$. In the following, we lay out some deliberations that
try
to clarify the status of BCOs among uniform preorders.
\begin{lemma}
Given a one sorted uniform preorder $\brar$, the relations $r\in R$ with
$\id\subseteq r$ are directed
with respect to inclusion and their union is a preorder.
\end{lemma}
\begin{proof}
  Given $r,s\supseteq\id$ in $R$, an upper bound is given by $sr$. It is
  evident that the union is reflexive, for transitivity we have
\[
\bigcup_{s\supseteq\id}s\circ\bigcup_{r\supseteq\id}r=
\bigcup_{r,s\supseteq\id}sr\subseteq\bigcup_{r\supseteq\id}r.
\]
\end{proof}
\begin{definition}
If the preorder $\leq$ from the previous lemma is contained in $R$, we
call $(A,R)$ \emph{condensable}\index{condensable uniform preorder}\index{uniform preorder!condensable}, and we call
$\leq$ its \emph{condensate}\index{condensate!of uniform preorder}.
\end{definition}
Discrete BCOs are condensable; their condensate is the identity. In general,
however, the ordering on a BCO is not necessarily its condensate.  For example,
let $(A,\leq,\mcf)$ be a BCO with least element $\bot$ for $\leq$, where $\mcf$
consists of \emph{all} total monotone functions. Then the constant $\bot$
function witnesses arbitrary inequalities, thus the BCO is biterminal, and its
condensate is the indiscrete preorder.

If a uniform preorder $\brar$ contains a preorder $(\leq)\in R$, all relations
$r\in R$ can
be completed to order theoretic distributors
$(\leq)r(\leq):(A,\leq)\edist(A,\leq)$ which are still in $R$. Since
$r\subseteq(\leq)r(\leq)$, $R$ is generated by
distributors. We call such a distributor $\phi$
\emph{partially functional}\index{partially functional distributor}\index{distributor!partially functional}, if for every $a\in A$, the upper set
$\{b\msep(a,b)\in\phi\}$ is either empty or representable. Partially functional
distributors are precisely the distributors of the form $(\leq)f$ for partial
monotone functions $f$ with downward closed domain, as explained in
\cite{benabou2000distributors} for partial functors between categories. Using
these techniques, we can characterize uniform preorders arising form BCOs.
\begin{lemma}
  A one sorted uniform preorder $\brar$ is induced by a BCO iff there exists a
preorder $(\leq)\in R$
  such that $R$ is generated by partially functional $\leq$-distributors.
\end{lemma}
\begin{proof}
  Given a BCO $(A,\leq,\mcf)$, the generators $(\leq)f$ for $f\in\mcf$ of the
  uniform preorder structure are partially functional by the remarks preceding
the
  lemma. Conversely, if a one sorted uniform preorder $\brar$ contains a
preorder and is generated by
  partially functional distributors, then it is easy to see that the
  corresponding monotone partial functions form a BCO structure equivalent to
  $\brar$.
\end{proof}

\section{Closure properties of \texorpdfstring{$\catuord$}{UOrd}}

\begin{lemma}\label{lem-closure-uord}
  \begin{enumerate}
  \item\label{lem-closure-uord-prod} $\catufp$ has small products.
  \item\label{lem-closure-uord-coprod} $\catufp$ has small coproducts.
  \item\label{lem-closure-uord-ccc} $\catufp$ is cartesian closed.
  \item\label{lem-closure-uord-invol} We can define an involution operation
\[
(-)^\op:\catufp^\co\to\catufp,
\]
that corresponds to taking the fiberwise opposite on the level of fibered
preorders.
  \end{enumerate}
\end{lemma}
\begin{proof}
  \emph{Ad \ref{lem-closure-uord-prod}.} This is also true for one sorted
uniform preorders, and to understand the proof it might be helpful to consider
first the one sorted case to reduce the amount of indices.

Let  $(I^l,A^l,R^l)_{l\in L}$ be a family of uniform preorders, 
and assume that for each $l\in L$, we are given $i_l,j_l\in I^l$ and $r_l\in
R^l_{i_l,j_l}$. Then the set

\[
\left\{\bigl((a_l)_l,(b_l)_l\bigr)\in \prod_l A^l_{i_l}\times\prod_l
A^l_{j_l}\;\bigg|\;\;\forall l\qdot (a_l,b_l)\in r_l\right\}
\]
is a binary relation on $\prod_l A^l_{i_l}\times\prod_l A^l_{j_l}$, and the set
of all relations defined in this way is a basis for a uniform preorder
structure on the family
\[
\left(\prod_l I^l,\Bigl(\prod_l A^l_{i_l}\;\Big|\;\;{(i_l)_l\in\prod_l
    I^l}\Bigr)\right).
\]
The product of the family $(I^l,A^l,R^l)_{l\in L}$ is the uniform preorder
generated by this basis.

\medskip

\emph{Ad \ref{lem-closure-uord-coprod}.}
The coproduct of a family $(I^l,A^l,R^l)_{l\in L}$ of uniform preorders is
given by
\[
\left(\coprod_l I^l,\,\Bigl( A^l_i\;\Big|\;\;(l,i)\in\coprod_l
I^l\Bigr),\,R\right),
\]
where
\[
R_{(l,i),(m,j)}=
\begin{cases}
  R^l_{ij} & \text{if } l = j\\
 \varnothing & \text{otherwise}
\end{cases}
\]

\medskip

\emph{Ad \ref{lem-closure-uord-ccc}.} This is nontrivial and originally due to
Longley~\cite{longley2011computability}, who showed that
the locally ordered $\catcstruct$ of \emph{computability structures} is in his
words `almost cartesian closed' (it is not really
cartesian closed, since it is a Kleisli category of $\catufp$ as we show in
Section~\ref{sec:cstruct}, and cartesian
closure is generally not preserved by Kleisli constructions). Our proof is an
adaption of Longley's to uniform preorders, expressed in terms of fibrations.

We show that the image of $\catufp$ is an exponential ideal in
$\catpfib(\catset)$. For general fibrations $\fibc,\fibd\in\catfib(\catset)$,
their exponential is given by
$(\fibd^\fibc)_I=\catfib(\catset/I)(\fibc/I,\fibd/I)$
(see~\cite[Section~4]{streicherfib}). Instantiating by
$\ufam(\upa),\ufam(\upb)$
for uniform preorders
$\upa=\upiar,\upb=\upjbs$, we get the following description of the fibration
$\ufam(\upb)^{\ufam(\upa)}$.
\begin{itemize}
\item Predicates on $K$ are pairs $(u,f)=(u,(f_i)_i)$ with $u:I\to J$ and
  $f_i:K\times A_i\to B_{ui}$ for every $i\in I$ such that
\[
\forall i,i'\in I,\;r\in R_{ii'}\qdot
\{(f_i(k,a),f_{i'}(k,a'))\msep (a,a')\in r,k\in K\}\in S_{ui,ui'}.
\]
\item
$(u,f)\leq (v,g)$ if 
\[
\forall i\in I\qdot\{(f_i(k,a),g_i(k,a))\msep k\in K,a\in A_i\}\in S_{ui,vi}.
\]
\end{itemize}
By Lemma~\ref{lem:reconstuct-ufp}, it suffices to show that this fibration has a
generic family of predicates (the
pre-stack condition is always preserved by exponentiation).

Let $u:I\to J$. An $L$-indexed family $((f_i^l)_{i\in I})_{l\in L}$ of
monotone maps from $\upiar$ to $\upjbs$ over $u$ is called
\emph{equimonotone}\index{equimonotone family}, if
\[
\forall i,i'\in I\;\forall r\in R_{ii'}\qdot\{(f_i^l(a),f_{i'}^l(b))\msep
(a,b)\in R,\, l\in L\}\in S_{ui,ui'}
\]
For each equimonotone family $(L,((f_i^l)_i)_l)$ over $u$, we can define a
predicate \[(u,g)\in \left(\ufam(\upb)^{\ufam(\upa)}\right)_L\] by
$g_i(l,a)=f_i^l(a)$,
and it is easy to see that the collection of these predicates where $(u,g)$
ranges over all equimonotone families is jointly generic.

\medskip

\emph{Ad \ref{lem-closure-uord-invol}.} The opposite of a uniform preorder
$\upiar$ is given by $(I,A,R^\op)$ where $R^\op_{ij}=\{r^\circ\msep r\in
R_{ji}\}$
\end{proof}
We observe that the construction of products restricts to the one sorted case,
but not the construction of coproducts and exponentials. The construction of
exponentials is particularly interesting, and an obvious question is which
properties of uniform preorders are stable under exponentiation  What is the
exponential of two pcas? What about
relational completeness\footnote{see
Section~\ref{sec:reational-completeness}}?). I haven't examined this at all yet.

\section{Finitely complete uniform preorders}\label{sec:finitely-complete-ufps}

Hofstra observed that the 2-categorical approach gives a well behaved and
useful notion of finite completeness for BCOs. We can do the same for
uniform preorders.
\begin{definition}
  A uniform preorder $\upa=\upiar$ is called \emph{finitely complete}\index{uniform preorder!finitely complete}, or a
\emph{uniform (meet-)semilattice}\index{uniform semilattice}\index{semilattice!uniform}, if the
maps
  \[\delta:\upa\to\upa\times\upa \quad\qtext{and}\quad !:\upa\to 1\]
 have right adjoints
\begin{equation*}
\delta\adj( *, \wedge):\upa\times \upa\to \upa \qtext{and}
  !\adj(1,\top):1\to \upa.\hfill
\end{equation*}
\end{definition}
Let us spell this out. Concretely, the existence of the right adjoints means
that there exist maps and elements
\begin{align*}
 1 &\in I & (-\ptype -)&:I\times I\to I\\
\top &\in A_1 & (-\wedge_{ij} -)&: A_i\times A_j\to A_{i\ptype j}\quad\text{for
}i,j\in I 
\end{align*}
such that
\begin{enumerate}
 \item $(\ptype,\wedge)$ is monotone:
\begin{equation}
\forall i,j,k,l\,\forall r\in R_{ij}\,\forall s\in R_{kl}\qdot
(\wedge\times \wedge)(r\brprod s)\in R_{i\ptype k,j\ptype l}
\end{equation}
\item $\delta\circ(\ptype,\wedge)\leq \id_{\upa\times \upa}$ (which is
equivalent to $(\ptype,\wedge)\leq\pil$ and $(\ptype,\wedge)\leq\pir$):
\begin{align}
 \forall i,j\qdot\{(a\wedge_{ij} b, a)\msep a\in A_i,\,b\in A_j\}&\in R_{i\ptype
j,i}\quad\text{and}\label{eq:proj-rel-l}\\
\forall i,j\qdot\{(a\wedge_{ij} b, b)\msep a\in A_i,\,b\in A_j\}&\in R_{i\ptype
j,j}\label{eq:proj-rel-r}
\end{align}
\item $\id_\upa\leq(\ptype,\wedge)\circ\delta$:
\begin{equation}
\forall i\qdot \{(a,a\wedge_{ii} a)\msep a\in A_i\}\in R_{i,i\ptype i} 
\end{equation}
\item $\id_\upa\leq(1,\top)\,\circ \,!$:
\begin{equation}
 \forall i\qdot\{(a,\top)\msep a\in A_i\}\in R_{i,1}
\end{equation}
\end{enumerate}
The conditions for the monotonicity of $(1,\top)$ and for $!\,\circ
(1,\top)\leq\id_1$ are vacuous.

\medskip

Since the embedding $\ufam(-):\catuord\to\catpfib$ of uniform preorders into
fibrations is a local
equivalence (Lemma~\ref{lem:ufp-embed}) and preserves finite products, $\upiar$
is finitely complete iff $\ufam\upiar$ is so, i.e.\
the fibers have chosen finite meets which are preserved up to isomorphism by
reindexing. Hence, in particular uniform preorders induced by preorders
(Example~\ref{ex:uords}-\ref{ex:uords-preords}) are finitely complete iff the
preorder has finite meets, as one would expect.

Moreover, the fibered monotone map $\ufam(u,f):\ufam(\upa)\to\ufam(\upb)$
corresponding to a monotone map $(u,f):\upa\to\upb$ preserves finite meets iff
$(u,f)$ is compatible with the finite limit structure in the sense that the
diagrams
\[
 \xymatrix@C+5mm{
\upa\times\upa\ar[r]^{\uf\times\uf}\ar[d]_{\umeet} &
\upb\times\upb\ar[d]^{\umeet} & 1\ar[d]_{\uterm}\ar[rd]^{\uterm}\\
\upa\ar[r]^{\uf} & \upb & \upa\ar[r]^(.4){\uf}& \upb
}
\]
commute up to isomorphism, in which case we will also say that \emph{$(u,f)$}
preserves finite meets. This leads us to the following definition.
\begin{definition}\label{def:catmuord}
 $\catmuord$ is the locally ordered category of finitely complete uniform
preorders, and finite meet preserving monotone maps.
\end{definition}

\medskip

For \emph{functional} uniform preorders, the functions $\wedge_{ij}:A_i\times
A_j\to A_{i * j}$ can be viewed
as a `recursive pairing functions' --- in particular we have:
\begin{lemma}
  If $\upa=\upiar$ is a finitely complete functional uniform preorder, then the
  functions $\wedge_{ij}:A_i\times A_j\to A_{i * j}$ are injective.
\end{lemma}
\begin{proof}
From \eqref{eq:proj-rel-l} and \eqref{eq:proj-rel-r} we know that
$\{(a\wedge_{ij} b,a)\msep
  a\in A_i,b\in A_j\}\in R_{i * j,i}$ and $\{(a\wedge_{ij} b,b)\msep a\in
A_i,b\in  A_j\}\in R_{i * j,j}$. By assumption these relations are functional,
which lets us deduce that $a=a'$ and $b=b'$ whenever
$a\wedge_{ij}b=a'\wedge_{ij}b'$.
\end{proof}

\begin{examples}
 \begin{enumerate}
\item As remarked above, the uniform preorder $(D,\adcl\{\leq\})$ associated to
a
preorder
$(D,\leq)$ is finitely
complete iff $(D,\leq)$ has finite meets.
 \item Given a pca $\pcaa$, the associated uniform preorder $(\pcaa, R(\pcaa))$
(Example~\ref{ex:uords}-\ref{ex:uords-pca-pca}) is finitely complete. This can
be
deduced from the fact that the associated fibered preorder $\ufam(\pcaa,
R(\pcaa))$ can be identified with the subfibration of the realizability tripos
$\rtr{\pcaa}$ on singleton valued predicates, and those are closed under finite
meets which are given by pairing.

The same argument generalizes to inclusions of pcas
(Example~\ref{ex:uords}-\ref{ex:uords-pca-i}) and (inclusions of) typed pcas
(Example~\ref{ex:uords}-\ref{ex:uords-pca-t},\ref{ex:uords-tpca-it}).
\item The uniform preorder $(\N,\adcl\prim)$ from
Example~\ref{ex:uords}-\ref{ex:uords-prim}
is finitely complete, the meet map $\wedge:\N\times\N\to\N$ is given by any
primitive recursive pairing function.
 \end{enumerate}
\end{examples}

Since $\catuord$ has small products, we can also talk about infinite meets and
joins in a uniform preorder. In particular, we call a uniform preorder $\upa$
\emph{small (co)complete}\index{uniform preorder!small (co)complete}, if for all sets $I$, the diagonal embedding
$\upa\to\upa^I$ has a right (left) adjoint. For preorders, small meets and
joins allow to define quantification in the associated fibrations, but for
uniform preorders
the two concepts diverge. We explain how to handle quantification in
Section~\ref{sec:quantification}.

\medskip

\subsection{Relational clones}\label{sec:clone}

By viewing morphisms $f:A_1\times\dots\times A_n\to B$ in a finite product
category $\catc$ as `multi'-morphisms with $n$ inputs and one output, any
cartesian category may be viewed as a
(cartesian~\cite{nlab-cartesian-multicategory}) multicategory. In the same
way, any meet-semilattice may be viewed as a `multi'-ordering, allowing
comparisons like $a_1,\dots,a_n\leq b$, and it is not difficult to define a
notion of `fibered multi-ordering' which is induced by finitely complete
uniform preorders. On the relational level, the corresponding structure is that
of a \emph{(many-sorted) relational clone}, which is most intuitive in the
functional case, and will be helpful in
Section~\ref{sec:reational-completeness}.
\begin{definition}
 A \emph{relational clone}\index{relational clone} on a family of sets $(A_i)_{i\in I}$
is a family
\[
 \left(C_{i_1,\dots ,i_n;j}\subseteq P((A_{i_1}\times\dots\times A_{i_n})\times
A_j)\;\big|\;n\in\N,\, i_1,\dots, i_n,j\in I\right)
\]
of sets of $(n+1)$-ary relations which 
\begin{itemize}
 \item 
contains all projections
$\pi_l:A_{i_1}\times A_{i_n}\to A_{i_l}$ for $1\leq l\leq n\in\N$ viewed as
relations $\pi_l\in C_{i_1,\dots ,i_n;i_l}$, and
\item is closed under composition in the sense that whenever $s\in C_{j_1,\dots,
j_n;k}$ and $r_l\in C_{i_1,\dots, i_m;j_l}$ for $1\leq l\leq n$, then the
relation
\begin{multline*}
s\circ (r_1,\dots,r_n)\stackrel{\mathrm{def}}{=}\\
\big\{(a_1\dots
a_m,c)\msep\exists b_1\dots b_n\qdot s(b_1\dots b_n,c)\wedge\bigwedge_{1\leq
l\leq n} r_1(a_1\dots a_m,b_l)\big\}
\end{multline*}
is contained in $C_{i_1,\dots, i_m;k}$.
\end{itemize}
\end{definition}
Given a finitely complete uniform preorder $\upiar$, the natural way to define
a clone-like structure on $(A_i)_{i\in I}$ is by taking $C_{i_1,\dots, i_n;j}$
to
be the set of relations which is generated via downward closure by the relations
\begin{equation}\label{eq:clone-gen}
 \{(a_1\dots a_n,b)\msep r(a_1\wedge\dots\wedge a_n,b)\} \qtext{for} r\in
R_{i_1 *\dots * i_n,j}.
\end{equation}
The idea underlying this construction is due to
Hofstra~\cite[Section~6]{hofstra2006all}, who uses finite limit structure on
BCOs to talk about `computable' partial functions in several
variables.

The following lemma is purely technical.
\begin{lemma}\label{lem:clone-from-ufp}
 For a finitely complete uniform preorder $\upiar$, the previously defined
system $(C_{i_1,\dots, i_n;j})_{i_1\dots i_n,j}$ of sets of relations is a
relational clone, and does not depend on the choice of
n-ary meet maps 
\[-\wedge\dots\wedge-:A_{i_1}\times\dots\times A_{i_n}\to
A_{i_1 *\dots * i_n}.\]
used to define the generators in~\eqref{eq:clone-gen}.
\qed
\end{lemma}
\begin{examples}
\begin{itemize}
 \item The relational clone associated to (the uniform preorder associated to)
the first Kleene algebra $\pcakone$ (Example~\ref{ex:kone}) consists of all
partial subfunctions of $n$-ary partial recursive functions.
\item The relational clone associated to the primitive recursive uniform
preorder $(\N,\adcl\prim)$ (Example~\ref{ex:uords}-\ref{ex:uords-prim}) contains
all partial subfunctions of $n$-ary primitive recursive functions.
\item The relational clone associated to a meet-semilattice $A$ consists of all
$n+1$-ary relations $r\subseteq A^{n}\times A$ satisfying $a_1\wedge\dots\wedge
a_n\leq b$ for all $(a_1\dots a_n,b)\in r$.
\end{itemize}
\end{examples}

\section{The quantification
monads}\label{suse-quant-monads}\label{sec:quantification}

This section is  about the representation of the posetal $D$-construction
from Section~\ref{sec:pos-d} in terms of uniform preorders. 

This `uniform-preorder version' was actually my starting point, it is based on
the construction with the same name and characteristics used by
Hofstra~\cite{hofstra2006all} in the context of BCOs.

\begin{definition}\label{def:ufp-monad}
  Let $\upiar$ be a uniform preorder. For $i,j\in I$ and $r\in R_{ij}$, we
define $\ebracks{r}\subseteq PA_i\times PA_j$ by
 \[ \ebracks{r}=\{(M,N)\msep \forall m\in M\;\exists n\in N\qdot r(m,n)\}.\] 
The family
$(\{\ebracks{r}\msep r\in R_{ij}\})_{ij\in I}
$
is a base for a uniform preorder structure on $(I,(PA_i)_{i\in I})$, and we
denote
the generated uniform preorder by $D\upiar$.
\end{definition}
Let us clarify the connection between the $D$-construction on uniform preorders
and the $D$-construction on fibered preorders:
\begin{lemma}\label{lem:character-d-ufp}
 For any uniform preorder $\upa=\upiar$, we have
\[
 \ufam(D\upa)\simeq D(\ufam(\upa)).
\]
In particular, $\ufam(D\upa)$ has existential quantification, and if $\upb$ is a
second uniform preorder
such that $\ufam(\upb)$ has existential quantification, then we have
\[
 \catuord(\upa,\upb)\simeq\cateuord(D\upa,\upb),
\]
where $\cateuord(D\upa,\upb)$ is the preorder of monotone maps whose associated
fibered monotone maps preserve existential quantification.
\end{lemma}
\begin{proof}
Given a set $M$, a predicate in $\ufam(D\upa)_M$ is a function $\varphi:M\to
PA_I$, whereas a predicate in $D(\ufam(\upa))_M$ is span
$M\xleftarrow{u}N\xrightarrow{\psi}A_i$. Given a predicate
$\varphi\in\ufam(D\upa)_M$ of the first form, we get a predicate in
$(u,\psi)\in D(\ufam(\upa))_M$ by taking the pullback
\[
 \vcenter{\xymatrix@R-3mm{
& N\pullbackcorner\ar[r]^u\ar[d]\ar[ld]_\psi & M\ar[d]^\varphi\\
A_i & E\ar[l]^{\in_1}\ar[r]_{\in_2}&PA_i
}},
\]
where ${\in}=\langle\in_1,\in_2\rangle:E\hookrightarrow A_i\times PA_i$ is the
membership predicate. Conversely, given a predicate $(u,\psi)\in
D(\ufam(\upa))_M$, we can
define a predicate $\varphi:M\to PA_i$ in $\ufam(D\upa)_M$ by $\varphi(m)=\{
\psi(n)\msep n\in u^{-1}(m)\}$. It is easy to see that these two
operations give the desired equivalence.
\end{proof}

The previous lemma implies that $D$ is a left biadjoint of the forgetful
functor $\cateuord\to\catuord$, and composing the two adjoints, we  get  a
2-monad \[
 D:\catuord\to\catuord
\]
for which we use the same name.  Let us describe the morphism part, and unit
and multiplication of $D$ explicitly.
A monotone map $(u,f):\upiar\to\upjbs$ is
mapped to $(u,Df):D\upiar\to D\upjbs$ with $Df_i(M)=\{f_im\msep m\in M\}$.  
Unit $\upiar\to D\upiar$ and multiplication $DD\upiar\to D\upiar$ are given by
indexwise singleton-map and union, respectively. Observe that these definitions
make $D$ into a \emph{2-monad}, not merely a pseudo-monad. Given a uniform
preorder $\upa$, it is easily seen that we have
\[
 \mu_\upa\adj\eta_{D\upa}:D\upa\to DD\upa
\](one only has to check that $\id_{DD\upa}\leq \eta_{D\upa}\circ\mu_D$), which
means that the monad is a \emph{KZ-monad}~\cite[Definition~B1.1.11]{elephant1}.
Recall that for KZ-monads, Eilenberg-Moore algebra structures coincide
with left adjoints to the
unit, whence algebras are not objects with supplementary structure, but rather
objects with a property. In the case of $D$, algebras give us an `internal' way
to talk about existential quantification, without mentioning the fibered
preorder.
\begin{lemma}\label{lem:quant-alg}
A uniform preorder $\upb=\upjbs$ is a $D$-algebra iff $\ufam(\upb)$ has
existential quantification.
\end{lemma}
\begin{proof}
This is not completely trivial, since we do not a priori know whether the
adjunction between $\catuord$ and $\cateuord$ is monadic (at least I don't know
of a generic argument to show this).

In the easy direction, if $\upb$ has $\exists$, we have
$\catuord(\upb,\upb)\simeq\cateuord(D\upb,\upb)$ and the algebra map is the
transpose of $\id_\upb$.

Conversely, a $D$-algebra structure $\bigvee:D\upb\to \upb$ induces a fibered
monotone map
\[
\bigvee:D(\ufam(\upb))\simeq\ufam(D\upb)\xrightarrow{\ufam(\bigvee)}
\ufam(\upb),
\]
and the quantification of $\varphi:N\to B_j$ along $u:N\to M$ is given by
$\bigvee_M(u,\varphi)$.
\end{proof}
\begin{example}\label{ex:pca-d-hyp-trip}
 Given a typed pca $\pcaia$, the fibration $\ufam(D\ufpcaia)$ associated
to the uniform preorder $D\ufpcaia$ precisely the realizability hyperdoctrine
$\hyph\pcaia$ from
Definition~\ref{def:realizability-hyperdoctrine}-
\ref{def:realizability-hyperdoctrine-real}.

Given an untyped pca $\pcaa$, $\ufam(D\ufpcaa)$ is the realizability tripos
$\rtr{\pcaa}$
(Definition~\ref{def:realizability-tripos}-\ref{def:realizability-tripos-real}).
\end{example}
\begin{remark}\label{rem:sym-mon-mon}
 While not essential for this work, it is interesting to know that $D$ is even
a \emph{symmetric monoidal} monad\index{symmetric monoidal monad}. This can most easily seen by remarking that
$D$ is given by the power set monad on underlying sets, which is known to be
monoidal from~\cite{kock1972strong}. Concretely, the monoidal structure is
given by
\begin{align*}
\phi_{\upa,\upb}=(\id_{I\times J},p)&:D\upa\times D\upb\to D(\upa\times\upb)
&\text{with}\\
p_{i,j}&:PA_i\times PA_j\to P(A_i\times A_j)&\text{given by}\\
&(M,N)\mapsto M\times N ,
\end{align*}
where $\upa=\upiar$ and $\upb=\upjbs$ are uniform preorders.
\end{remark}

\begin{remark}\label{rem:forall-in-ufps}
Since a fibered preorder has universal quantification iff its opposite fibered
preorder has existential quantification, and $\catufp$ is closed
under $(-)^\op$, we can define a \emph{universal quantification monad}\index{universal quantification monad} $U$ by
simply dualizing the definition of $D$.

Explicitly, given a uniform preorder $\upiar$, $U\upiar$ is defined by setting 
\[
\abracks{r}=\{(M,N)\in PA_i\times PA_j\msep \forall n\in N\;\exists m\in M\qdot
r(m,n)\}\quad\text{for $i,j\in I$, $r\in R_{ij}$,}\\
\]
and by setting $U\upiar$ the uniform preorder generated by the basis
$(I,(PA_i)_i,(\{\abracks{r}\msep r\in R_{ij}\})_{ij})$.
\end{remark}

\subsection{Uniform frames}\label{sec:uniform-frames}

\begin{definition}\label{def:uniform-frame}
\begin{itemize}
 \item A \emph{uniform frame}\index{uniform frame}\index{frame!uniform} is a uniform preorder $\upa$
such that $\ufam(\upa)$ is a fibered frame
(Definition~\ref{def:fibered-frame}).
\item $\catufrm$ is the locally ordered category of uniform frames and
monotone maps that preserve finite meets and existential quantification.
\end{itemize}
\end{definition}
\begin{remark}
 If $\upa$ is a uniform preorder, then $\ufam(\upa)$ has finite meets iff
$\upa$ is finitely complete, and $\ufam(\upa)$ has existential quantification
iff $\upa$ is a $D$-algebra. If we want to completely `internalize' the
definition of uniform frame (i.e.\ express it without referring to the family
fibration), then we still have to internalize the Frobenius law. This can be
done using the monoidal structure of $D$ from Remark~\ref{rem:sym-mon-mon}.
Concretely, if $\upa$ is a finitely complete uniform preorder with $D$-algebra
structure $\bigvee\adj\eta_{D\upa}$, one can show that the Frobenius law holds
in $\ufam(\upa)$ iff the two paths around the rectangle
\[
    \vcenter{\xymatrix@R-2mm@C+1mm{
      D\upa\times D\upa \ar[r]^{\phi_{\upa,\upa}} &
D(\upa\times\upa)\ar[r]^-{D(\ptype,\wedge)} & D\upa\ar[d]^{\bigvee}\\
      \upa\times D\upa\ar[u]_{\eta\times\id}\ar[r]^{\id\times{\bigvee}} &
\upa\times\upa \ar[r]^-{(\ptype,\wedge)} & \upa }}
    \]
are isomorphic.
\end{remark}
By specializing Lemma~\ref{lem:msl-to-frame} from fibered to uniform
preorders,
we can deduce that $D$ is well behaved in relation to conjunction, which gives
us a way to construct uniform frames.
\begin{lemma}\label{lem:ufp-d-meet}
If $\upa$ is a finitely complete uniform preorder, then
\begin{itemize}
 \item $D\upa$ is a uniform frame,
 \item $y:\upa\to D\upa$ preserves finite meets, and
 \item for any uniform frame $\upb$, precomposition with $y$ induces an
equivalence
\[
 \catmuord(\upa,\upb)\simeq\catufrm(D\upa,\upb)
\]
of preorders.
\end{itemize}
\end{lemma}
\begin{proof}
 This follows from Lemma~\ref{lem:msl-to-frame}, since the occuring categories
of uniform preorders can be identified with full subcategories of the
corresponding categories of posetal pre-stacks on $\catset$ in a way which
preserves all relevant structure.
\end{proof}
\begin{convention}
When applying the constructions of Section~\ref{sec:preordered-case} to
the family fibrations of finitely complete uniform preorders $\upa$ and uniform
frames $\upb$, we usually leave the $\ufam(-)$ implicit. Thus,
\begin{itemize}
\item $\catset[\upb]$ is the category of partial equivalence relations in
$\ufam(\upb)$, and $\sheaf{\upb}$ is the corresponding gluing fibration
 \item $\srel{\catset}{\upa}$ is the category of equivalence relations in
$\siev(\ufam(\upa))$, with corresponding gluing fibration $\widehat{\upa}$
\end{itemize}
In particular, we have $\srel{\catset}{\upa}\simeq\catset[D\upa]$ by
Corollary~\ref{cor-hat-sh-d}.
\end{convention}

\subsection{Preservation of logical structure by \texorpdfstring{$D$}{D}}

In Lemma~\ref{lem:ufp-d-meet}, we showed that if a uniform preorder $\upa$ has
finite meets, then so does $D\upa$, and furthermore $y:\upa\to D\upa$ preserves
them. Analogous statements are true for implication and universal
quantification, as we show now.
\begin{lemma}
 Let $\upa=\upiar$ be a uniform preorder such that $\ufam(\upa)$ has universal
quantification. Then $\ufam(D\upa)$ has universal quantification as well, and
$y:\upa\to D\upa$ preserves it.
\end{lemma}
 \begin{proof}
  We show that if $\upa$ has an algebra structure for the universal
quantification monad $U$ from Remark~\ref{rem:forall-in-ufps}, then so does
$D\upa$. Assume that $(\pi,\bigwedge):U\upiar\to\upiar$ is such an algebra
structure. The algebra structure on $D\upiar$ is given by
$(\pi,\widetilde{\bigwedge}):UD\upiar\to D\upiar$, where
$\widetilde{\bigwedge}_i:PPA_i\to PA_{\pi i}$ is given by 
\[
 \widetilde{\bigwedge}_i\mcm = \{{\bigwedge}_i U\msep
U\subseteq\bigcup\mcm,\forall M\in\mcm\;\exists a\in M\qdot a\in
U\}\qquad\text{for }\mcm\subseteq PA_i.
\]
It is not difficult to see that $(\pi,\widetilde{\bigwedge})$ is monotone and
right adjoint to $\eta:D\upa\to UD\upa$ (algebra structures are right adjoint
to the unit in the case of $U$, since it is a dualized KZ-monad). Furthermore,
$y:\upa\to D\upa$ is a strict $U$-algebra morphism with respect to the algebra
structures  $(\pi,\bigwedge)$ and $(\pi,\widetilde{\bigwedge})$, which follows
from the construction of $\widetilde{\bigwedge}$.
\end{proof}
\begin{lemma}\label{lem:d-imp-forall}
 Let $\upa=\upiar$ be a uniform preorder such that $\ufam(\upa)$ has
implication \emph{and} universal
quantification. Then $\ufam(D\upa)$ has implication and universal quantification
as well, and
$y:\upa\to D\upa$ preserves both connectives.
\end{lemma}
\begin{proof}
Assume that $\upiar$ has $\imp$ and $\forall$. It remains to show that
$D\upiar$ has $\imp$. For
$i,j\in I$, the projections
$
 A_i \xleftarrow{\pil}A_i\times A_j\xrightarrow{\pir}A_j
$
are predicates on $A_i\times A_j$. We denote the sort of $\pil\imp\pir$ by
$i\simp j$, and thus we have a map $(\imp_{ij}):=(\pil\imp\pir):A_i\times A_j\to
A_{i\imp j}$.

Given predicates $\varphi:M\to PA_i$, $\psi:M\to PA_j$, we construct the
implication of $\varphi$ and $\psi$ by
\[
(\varphi\imp\psi)(m)=\left\{{\bigwedge}_{i\imp j}\left\{a\imp_{ij}b\msep
(a,b)\in r\right\}\msep r\in\trel(\varphi(m),\psi(m)) \right\} \subseteq
A_{\pi(i\imp j)},
\]
where $\trel(\varphi(m),\psi(m))$ is the set of total relations from
$\varphi(m)$ to $\psi(m)$, and $(\pi,\bigwedge):U\upa\to\upa$ is the $U$-algebra
structure representing universal quantification.

It is not difficult to see that this predicate has the desired property, and
preservation of $\imp$ follows again from the definition.
\end{proof}
Since regular triposes correspond
to one-sorted uniform preorders $\upa$ such that $\ufam(\upa)$ has finite
meets, implication, and universal quantification, the previous lemma implies in
particular that $D\trip$ is a tripos for every regular tripos
$\trip:\tot{\trip}\to\catset$. Moreover, we can show the following.
\begin{corollary}\label{cor:tripos-subtripos-free-tripos}
 Every regular tripos is a subtripos of a regular tripos with freely generated
existential quantification.
\end{corollary}
\begin{proof}
 Let $\upa$ be a uniform preorder representing a tripos. Since $\ufam(\upa)$ has
existential quantification, by Lemma~\ref{lem:quant-alg} there exists 
$D$-algebra structure $\bigvee:D\upa\to\upa$ on $\upa$ which is given by
transposing  $\id_\upa$ as in Lemma~\ref{lem:character-d-ufp}, and is left
adjoint to $y:\upa\to D\upa$ since $D$ is a KZ-monad. Furthermore, it follows
from Lemma~\ref{lem:ufp-d-meet} that $\bigvee$ preserves finite meets.
Since algebra maps are left pseudoinverse to units by definition, we thus
have a geometric inclusion
\[
 \bigvee\adj y:\upa\to D\upa,
\]
where $D\upa$ is a tripos by the previous lemma.
\end{proof}
Since inclusions of triposes induce inclusions of toposes via the
tripos-to-topos construction, we thus obtain a subtopos inclusion
\[
\catset[\trip] \hookrightarrow \catset[D\trip]\simeq\catset\{\trip\}
\]
which can be viewed as  tripos-theoretic
analogue of the statement that any Grothendieck topos is a subtopos of a
presheaf topos. Every subtopos inclusion induces a Lawvere-Tierney topology on
the larger topos, and a \emph{quasitopos}\index{quasitopos} of separated objects for this
topology. In the case of the above inclusion, it can be shown that this
quasitopos is equivalent to the \emph{q-topos} associated to $\trip$ as defined
in \cite[Definition~5.1]{frey2011}. Thus, q-toposes constructed from regular
triposes on $\catset$ are always quasitoposes.

\subsection{Computability structures and the monad
\texorpdfstring{$\dplus$}{D+}}\label{sec:cstruct}

Hofstra~\cite{hofstra2006all} defines, apart from the monad
$D$, a monad
$\dplus$ -- the `nonempty downset monad'\footnote{The definition can be traced
back to Hofstra's thesis~\cite{hofstra2003completions}\index{nonempty downset monad} where he defines an
analogous monad $T$ on ordered pcas.}. This monad corresponds to existential
quantification along \emph{surjective functions} in the same way that $D$
corresponds to general existential quantification.

We can define an analogous monad for uniform preorders, and it turns out
that its Kleisli category coincides almost precisely with
Longley's~\cite{longley2011computability} category
$\catcstruct$.

Given a uniform preorder $\upiar$, $\dplus\upiar$ is given by $(I, \pplus
A,\dplus R)$ where $\pplus $ is the non-empty power set and $\dplus R$ is
obtained by restricting the relations in $DR$ to $\pplus A$. 

Using $D_+$, we obtain the following description of  $\catcstruct$.
\begin{lemma}
  $\catcstruct$ is equivalent to the Kleisli 2-category of $D_+$ on those
  uniform preorders $\upiar$ where all $A_i$ are inhabited.
\qed
\end{lemma}
The monad $D_+$ can also be used to define Longley's applicative morphisms
between pcas. Let us recall the definition.
\begin{definition}[Longley]
Let $\pcaa,\pcab$ be pcas. An \emph{applicative morphism}\index{applicative morphism} from $\pcaa$ to
$\pcab$ is a \emph{total} relation $\gamma:\pcaa\etrel \pcab$ such
that there exists an $e\in\pcab$ (the \emph{realizer} of $\gamma$) such that
\[
\gamma(a,b),\gamma(a',b'),a a'=a'' \ent \gamma(fa'',e b
b')
\]
\end{definition}
The following lemma describes applicative morphisms in terms of monotone
maps between uniform preorders and
$D_+$. Similar characterizations have been given by Hofstra and van
Oosten~\cite{hofstra2003ordered} in terms of order-pcas (see
also~\cite[Proposition~1.8.10]{vanoosten2008realizability}), and by
Longley~\cite[Proposition~5.24]{longley2011computability} in terms of
C-structures.
\begin{lemma}
  Let $\pcaa,\pcab$ be pcas. The following concepts are equivalent.
\begin{enumerate}
 \item applicative morphisms $\gamma:\pcaa\etrel\pcab$
\item finite meet preserving monotone maps $f:(\pcaa,R(\pcaa))\to
\dplus(\pcab,R(\pcab))$
\item finite meet preserving monotone maps $f:(\pcaa,R(\pcaa))\to
D(\pcab,R(\pcab))$
\end{enumerate}
\end{lemma}
\begin{proof}
We start by showing the equivalence between the first two concepts. Total
relations $\gamma:\pcaa\etrel\pcab$ are in bijection with functions
$f:\pcaa\to\pplus\pcab$ and $\pplus\pcab$ is the underlying set of
$\dplus(\pcab,R(\pcab))$, thus we have to check that the condition for
applicative morphisms is equivalent to monotonicity and preservation of finite
meets.

  Let $\gamma:\pcaa\etrel \pcab$ be an applicative morphism with realizer
  $e\in\pcab$. To show that the corresponding $f:\pcaa\to\pplus\pcab$
constitutes a monotone map of type
  $(\pcaa,R(\pcaa))\to \dplus (\pcab,R(\pcab))$, we have to show that 
\[\forall
  r\in \pcaa\;\exists s\in\pcab\qdot\forall a,a'\in\pcaa\qdot ra=a'\imp\forall
  b\in fa\qdot sb\in f(b').\] 
For given $r\in\pcaa$, let $s_0\in fr$. Without
  loss of generality, we can assume that $es_0$ is defined\footnote{By
    functional completeness, we can replace $e$ by an $e'$ for which it is
    everywhere defined, and which still realizes $f$ as an applicative morphism,
see
    \cite[Remark~2.1.2~(i)]{longley1995realizability}.}, and setting $s=es_0$
  the claim is immediate.

  To see that $f$ commutes with binary meets, we have to find an $s\in\pcab$
  such that for every $a,a'\in\pcaa$, we have 
\[b\in fa,b'\in fa'\ent s(b\wedge
  b')\in f(a\wedge a').\]
 Let $p\in\pcaa$ such that $\forall a,a'\qdot a\wedge
  a'=paa'$ and let $q\in fp$. For $a,a'\in\pcaa$, $b\in fa$, and  $b'\in fa'$,
we
  have $eqb\in f(pa)$ and $e(eqb)b'\in f(paa')=f(a\wedge a')$. The existence of
an $s$ such that $s(b\wedge b')=e(eqb)b'$ follows from functional completeness.

  Conversely, assume that $f:(\pcaa,R_\pcaa)\to D(\pcab,R_\pcab)$ is a meet
  preserving monotone map. Let $r\in \pcaa$ such that $\forall
  a,a',a''\qdot aa'=a''\imp r(a\wedge a')=a''$. Since $f$ is
  monotone, there exists an $s\in\pcab$ such that $ra=a'',b\in fa\ent sb\in
  fa''$, which together with the first statement implies $aa'=a'',b\in
  f(a\wedge a')\ent sb\in fa''$. Now since $f$ preserves finite meets, there
exists
  $t\in\pcab$ such that $b\in fa,b'\in fa'\ent t(b\wedge b')\in f(a\wedge a')$,
  and we can deduce $aa'=a'',b\in fa,b'\in fa'\ent s(t(b\wedge b'))\in fa''$.
By functional completeness there exists $e\in\pcab$ such that $\forall
b,b'\in\pcab\qdot s(t(b\wedge b'))\les ebb'$, and this $e$ is the realizer of
the applicative morphism $\gamma:\pcaa\etrel\pcab$ corresponding to $f$.

It remains to show that meet preserving maps 
$(\pcaa,R(\pcaa))\to \dplus (\pcab,R(\pcab))$ are equivalent to meet preserving
maps $(\pcaa,R(\pcaa))\to D(\pcab,R(\pcab))$. In one
direction, we can compose with the embedding  $\dplus(\pcab,R(\pcab))\to 
D(\pcab,R(\pcab))$ which preserves finite meets. For the other direction,
let $f:(\pcaa,R(\pcaa))\to D(\pcab,R(\pcab))$ be finite meet preserving
monotone, and consider the component
\[
 \ufam(f)_1:\ufam(\pcaa,R(\pcaa))_1\to \ufam(D(\pcab,R(\pcab)))_1
\]
of the corresponding fibered monotone map between the terminal fibers. Since
all elements of $\ufam(\pcaa,R(\pcaa))_1$ are equivalent, and $\ufam(f)_1$
preserves $\top$ (as a particular finite meet), all $fa$ are equivalent to
$\{\top\}$ in $\ufam(D(\pcab,R(\pcab)))_1$ and thus in particular inhabited,
meaning that $f$ factors through $\dplus(\pcab,R(\pcab))\to 
D(\pcab,R(\pcab))$.
\end{proof}
\begin{remark}
It might be argued that from a presheaf theoretic point of view and for general
uniform preorders, the relevant objects are the meet-preserving maps of type
$\upa\to D\upb$, since they correspond to positive fibered functors between the
associated fibrations of presheaves. The fact that such maps factor through
$\dplus\upb\to D\upb$ is specific to pcas.
\end{remark}

\section{Global sections}\label{sec:glob-secs}

This section is not a priori about uniform preorders, but I decided to put it
in this chapter nevertheless, since it is strongly related to
Section~\ref{sec:arbitrary-bases}.

\medskip

Given a Grothendieck topos $\tope$, the diagonal functor
$\Delta:\catset\to\tope$ is left adjoint to the global sections functor
$\Gamma=\tope(1,-):\tope\to\catset$. 

For a (say bounded) geometric morphism
$\Delta\adj\Gamma:\tope\to\tops$, the fact that $\Delta$ has a right adjoint is
equivalent to the statement that the fibration $\gl_\Delta(\tope)$ has
\emph{small global sections}\index{fibration!with small global sections}\index{small global sections in a fibration}, which is a special case of the fibrational
property of being locally small (see~\cite[Sections 10,
16]{streicherfib})\footnote{Having small global sections is also equivalent to
having \emph{comprehension}\index{fibration!with comprehension}\index{comprehension in a fibration} in the sense of
Lawvere~\cite{lawvere1970equality}.}. Thus, it still makes sense to view
$\Gamma$ as a global sections functor relative to the base topos $\tops$.

In general, $\Gamma$ does not fit into the framework of Moens' theorem as stated
in Theorem~\ref{theo:moens} since for $\Gamma$ to correspond to a 1-cell in
$(\tops\pslice\catlex)(\Delta, \id_\tops)$ we need
$\Gamma\circ\Delta\cong\id_\tops$. However, we can still use the generalized
version of the correspondence presented in Remark~\ref{rem:generalized-moens},
allowing us to view $\Gamma$ as a fibered functor of type
$\gl_\Delta(\tope)\to\fund{\tops}$ which does not necessarily preserve
internal sums.

\medskip

If $\Delta:\catr\to\catx$ is an arbitrary regular functor into an exact
category, there is no reason for the existence of a right adjoint -- in other
words the fibration $\gl_\Delta(\catx)$ does not have to have small global
sections in the fibrational sense. However, if
$\catr=\catset$ and $\catx$ is locally small, the ordinary global sections
functor $\Gamma=\catx(1,-)$ in the inverse direction of $\Delta$ is always
definable. Moreover, we have half of the structure of the adjunction, namely
a natural transformation $\eta:\id_\catset\to\Gamma\Delta$, given by
\[
M\cong\catset(1,M)\to\catx(\Delta 1,\Delta M)\cong\catx(1,\Delta
M)=\Gamma\Delta M.
\]

\medskip

If the regular category $\catx$ is of the form $\catset[\fifx]$ for a fibered
frame $\fifx$, the same story can be told directly on the level of fibered
preorders. We can define $\gamma:\fifx\to\sub(\catset)$ by
\[
\fifx_M\ni\varphi\quad\mapsto\quad\{m\msep \top\ent m^*\varphi\}\subseteq M.
\]
Then it is straightforward that $\gamma$ preserves finite meets, 
and by a similar argument as the one above one can show $U\subseteq
\gamma\delta U$ for $U\subseteq M$. Moreover, $\gamma$ can be reconstructed
from $\Gamma$ and $\eta$ as the following lemma shows.
\begin{lemma}\label{lem:gamma-gamma}
Let $\fifx:\tot{\fifx}\to\catset$ be a fibered frame, and $\varphi\in\fifx_M$.
Then $\gamma(\varphi)\cong\eta_M^*\Gamma\tilde{\varphi}$, where
$\tilde{\varphi}$ is the subobject of $\Delta M$ represented by $\varphi$.
\[
\vcenter{\xymatrix{
{}\phantom{X}\pullbackcorner\mono[d]_{\gamma(\varphi)}\ar[r] &
\Gamma(M,\predeq_\varphi)\mono[d]^{\Gamma\tilde{\varphi}} \\
M\ar[r]_{\eta_M} & \Gamma\Delta M
}}
\]
\end{lemma}
\begin{proof}
$m_0\in M$ is contained in  $\eta_M^*\Gamma\tilde{\varphi}$ iff the global
element of $\Delta M$ given by the singleton predicate $\ilbracks{m\csep
m=m_0}$ factors through $\varphi$, which is equivalent to $\ent\varphi(m_0)$.
\end{proof}

We remark that $\gamma$ is definable not only for fibered frames, but for
arbitrary fibered posets with $\top$. For \emph{uniform preorders} $\upiar$ with
greatest element $\top\in A_1$, $\gamma$ is defined by
\[
\ufam\upiar_M\ni(i,\varphi)\quad\mapsto\quad\bigl\{m\msep
\{(\top,\varphi(m))\}\in
R_{1,i}\bigr\}\subseteq M,
\]
the fact that we are talking about the pointwise and not the uniform ordering
corresponds to the use of \emph{singleton} relations
$\{(\top,\varphi(m))\}$.

\section{Calculus of distributors}\label{suse-calc-dist}

\begin{definition}\label{def:posetal-dist}
Given preorders $D$ and $E$, a \emph{(posetal)
distributor}\index{distributor!posetal}\index{posetal!distributor} $\phi:D\edist E$ is
a monotone function of type $E^\op\times D\to\bbtwo$, or equivalently a
relation $\phi\subseteq D\times E$ which is upward closed in $D$ and downward
closed
in $E$. 
\end{definition}
Preorders and posetal distributors form a compact closed locally
ordered
category $\catpdist$\footnote{
For a presentation of $\catpdist$ from a domain theory and linear logic
perspective, see~\cite{hyland2010some} -- because of its links
to linear logic,
the category is called $\mathbf{Lin}$ there.
}, and we have an embedding
\[
\catord{}\hookrightarrow\catpdist
\]
which identifies the category $\catord$ of preorders with the subcategory of
$\catpdist$\label{page:catpdist} on left
adjoints. We can do something completely analogous for uniform preorders.

\begin{definition}\label{def:udist}
\begin{enumerate}
\item A \emph{uniform distributor}\index{uniform!distributor}\index{distributor!uniform} $H:\upiar\edist\upjbs$ between
  uniform preorders $\upiar$ and $\upjbs$ is a family of sets \[H_{ji}\subseteq
  P(B_j\times A_i)\]
of binary relations such that
\begin{enumerate}
\item $h\in H_{ji},k\subseteq h\implies k\in H_{ji}$
\item $h\in H_{ji},r\in R_{ii'}\implies rh\in H_{ji'}$
\item $s\in S_{j'j},h\in H_{ji}\implies hs\in H_{j'i}$
\end{enumerate}
\item The \emph{composition} of uniform distributors
\[
\xymatrix@1{
\upiar\dist[r]^G&\upjbs\dist[r]^H&\upkct
},
\]
  is defined by $(H\circ G)_{ki}=\adcl\{gh\msep j\in J,\,g\in
  G_{ji},\,h\in H_{kj}\}$.
\item The \emph{product} of uniform distributors 
\[
\xymatrix@R-5mm{
\upiar\dist[r]^G& \upkct\\
\upjbs\dist[r]^H&\upldu
}
,\] 
is the uniform preorder
\[\xymatrix@C+2mm{
  \upiar\times\upjbs\dist[r]^{G\times H}&\upkct\times\upldu}\] defined by
$(G\times
  H)_{klij}=\adcl\{g\times h\msep g\in G_{ki},\,h\in H_{lj}\}$.
\end{enumerate}
\end{definition}
These definitions give rise to a locally ordered monoidal category
$\catudist$,
where the order on the morphisms is componentwise inclusion\footnote{Observe
that since the ordering of morphisms is defined in terms of inclusions,
$\catudist$ is enriched in \emph{posets}, and not just in preorders as most of
our locally ordered categories.} and the identity on
$\upiar$ is $R$.

Given a monotone map $(u,f):\upiar\to\upjbs$, we define uniform distributors
\[
\ldist{\uf}:\upiar\edist\upjbs\qtext{and}\rdist{\uf}:\upjbs\edist\upiar
\]
by
\begin{align*} 
  (\ldist{\uf})_{ji}&=\{g\subseteq B_j\times A_i\msep f_ig\in S_{j,ui}\} \quad
  \text{and} \\
  (\rdist{\uf})_{ij}&=\{h\subseteq A_i\times B_j\msep h f^\circ_i\in S_{ui,j}\}.
\end{align*}
\begin{lemma}\label{lem-updist}
  \begin{enumerate}
  \item\label{lem-updist-dual} $\upiar^\op$ is dual to $\upiar$ in
$\catudist$.
  \item\label{lem-updist-adj} For $(u,f):\upiar\to\upjbs$, we have
$\ldist{\uf}\adj\rdist{\uf}$.
  \item\label{lem-updist-map} Let $ G\adj H:\upjbs\to\upiar$ be an adjunction
in $\catudist$.  Then the axiom of choice implies that there exists a
monotone $(u,f):\upiar\to\upjbs$ such that $\ldist{\uf}= G$ and $\rdist{\uf}=
H$.
  \end{enumerate}
\end{lemma}
\begin{proof}
  The counit $\ve_{\upiar}:\upiar^\op\times\upiar\edist 1$ has to be a family of
  sets of relations of type $1\erel A_i\times A_j$ which are closed under
  composition with structuring relations of $\upiar^\op\times\upiar$ on the
  left. It is given by $\ve_{\upiar}= R$. Dually, $\eta_{\upiar}:1\edist
  \upiar\times\upiar^\op$ is a family of sets of relations of type $A_i\times
  A_j\erel 1$ which is closed under composition with structuring relations of
  $\upiar\times\upiar^\op$ on the right, and is given by $R$ as well. The
  verifications of the triangle equalities are straightforward.

  To show that $\ldist{\uf}\adj\rdist{\uf}$, we have to verify
  $R\subseteq\rdist{\uf}\ldist{\uf}$ and $\ldist{\uf}\rdist{\uf}\subseteq S$
  componentwise. Let $r\in R_{ii'}$. We have to find $j\in J$ and
  $g\in(\ldist{\uf})_{ji},h\in(\rdist{\uf})_{ij}$ such that $r\subseteq
  gh$. This is the case for $j=ui$, $g=rf_i^\circ$, and $h=f_i$. For the second
  inclusion, we have to show that
  $g\in(\ldist{\uf})_{ji},h\in(\rdist{\uf})_{ij'}$ implies $hg\in
  S_{jj'}$. This follows from $f_ig\in S_{j,ui}, hf_i^\circ\in S_{ui,j'}$ and
  $\id_{A_i} \subseteq f_i^\circ f_i$.

  For the third claim, it follows from $R\subseteq H G$ that for given $i\in I$
  there exist $j\in J$, $h_i\in H_{ij}, g_i\in G_{ji}$ such that
  $\id_{A_i}\subseteq g_ih_i$, which means that $\forall a\vtp A_i\,\exists
  b\vtp B_j\qdot h_iab\wedge g_iba$. Let $u:I\to J$ be a choice function for
  the association of $j$ to $i$, and for each $i\in I$, let $f_i:A_i\to B_{ui}$
  be a choice function for the second part of the statement. Then in particular
  $f_i\subseteq h_i$ and $f_i^\circ\subseteq g_i$. We claim that $(u,f)$ is the
  desired monotone map.

  To see that $(u,f)$ is monotone, let $r\in R_{ii'}$. Since $R\subseteq H G$,
  there exist $j\in J$, $h'\in H_{ij}$, and $g'\in G_{ji'}$ such that
  $r\subseteq g'h'$. Therefore we can argue $f_{i'}rf_i^\circ\subseteq
  h_{i'}g'h'g_i\in (G H G H)_{ui,ui'}\subseteq (SS)_{ui,ui'} = S_{ui,ui'}$. To
  verify that $\uf$ induces the adjunction $ G\adj H$, we show
  $G\subseteq\ldist{\uf}$ and $H\subseteq\rdist{\uf}$. Let $g\in G_{ji}$. We
  have to show that $f_ig\in S_{j,ui}$, which follows because $f_ig\subseteq
  h_ig\in (G H)_{j,ui}\subseteq S_{j,ui}$. The verification of
  $H\subseteq\rdist{\uf}$ is similar.
\end{proof}

In Lemma~\ref{lem:ufp-embed}, we showed that we can identify the 2-category
$\catuord$ with a full subcategory of $\catpfib(\catset)$. In the following, we
will do something similar for $\catudist$, but in this case it turns out to be
more convenient to work with indexed preorders instead of fibered ones.

Recall that an indexed preorder is a pseudo-functor of type
$\catset^\op\to\catord$, and that -- if we ignore size issues -- the locally
ordered categories $\catiord(\catset)$\label{page:catiord} and
$\catpfib(\catset)$ of indexed
and fibered preorders on $\catset$ are biequivalent. Given a uniform preorder
$\upa$, we do not distinguish the associated indexed and fibered preorders
notationally -- we denote both by  $\ufam(\upa)$.
\begin{definition}
  \begin{enumerate}
  \item Let $\ffrma,\ffrmb$ be indexed preorders. An \emph{indexed
distributor}\index{indexed!distributor}\index{distributor!indexed}
$\Phi:\ffrma\edist \ffrmb$ is a family $\Phi_M:
    \ffrma_M\edist \ffrmb_M$ of posetal distributors such that for $u:N\to M$,
    $\psi\in \ffrma_M$, and $\theta\in \ffrmb_M$, we have
    $\Phi_M(\theta,\psi)\implies\Phi_N(u^*\theta,u^*\psi)$.
  \item 
If $\fifa,\fifb$ are indexed preorders satisfying the pre-stack condition, we 
say that an indexed distributor $\Phi:\ffrma\edist \ffrmb$  is
\emph{separated}\index{separated!indexed distributor}\index{indexed!distributor!separated}, if for any epimorphism
$e:N\eepi M$ and predicates $\psi\in \ffrma_M$ and $\theta\in \ffrmb_M$,
    $\Phi_N(e^*\theta,e^*\psi)\implies \Phi_N(\theta,\psi)$.

In the presence of choice, this condition is always satisfied.
  \end{enumerate}
\end{definition}
Indexed distributors can be composed componentwise, but the
componentwise composition of two separated indexed distributors does not need to
be separated. If we want to define a category of separated indexed
distributors, we therefore have to compose them differently.
\begin{definition}\label{def:compose-separated-indexed-dist}
Let $\fifa,\fifb$, and $\fifc$ be indexed preorders satisfying the pre-stack
condition, and let $\Phi:\fifa\edist\fifb$ and $\Psi:\fifb\edist\fifc$ be two
separated indexed distributors. Their composition 
\[
 \Psi\circ\Phi:\fifa\edist\fifc
\]
is defined by 
\[
 (\Psi\circ\Phi)_M(\gamma,\alpha)\;\defequi\;\exists u\vtp N\!\!\eepi\!\!
M,\,\beta\in\fifb_N\qdot\Psi_N(u^*\gamma,\beta)\wedge\Phi_N(\beta,u^*\alpha)
\]
where $\alpha\in\fifa_M$ and $\gamma\in\fifc_M$.
\end{definition}
Using this definition, it is easy to see that the composition of separated
indexed distributors is separated, which allows us to make the following
definition.
\begin{definition}\label{def:idist}
$\catidist$ is the locally ordered category of indexed preorders satisfying the
pre-stack condition, and separated indexed distributors, where the ordering is
given by componentwise inclusion. 
\end{definition}
Given a uniform distributor $G:\upiar\edist\upjbs$, we can define an indexed
distributor
\[
\ufam(G):\ufam\upiar\edist\ufam\upjbs
\]
by
\[
\ufam(G)_M((j,\psi),(i,\phi))\defequi\{(\psi m,\phi m)\msep m\in M\}\in
G_{ji}.
\]
\begin{lemma}
The previously defined operation gives rise to a 2-functor
\[
 \ufam(-):\catudist\to\catidist.
\]
In particular, given uniform distributors
$\upa\stackrel{G}{\edist}\upb\stackrel{H}{\edist}\upc$ the indexed
distributors $\ufam(G)$ and $\ufam(H)$ are separated, and we have $\ufam(G\circ
H)=\ufam(G)\circ \ufam(H)$ with the composition of separated indexed
distributors defined in~\ref{def:compose-separated-indexed-dist}. 

Moreover, $\ufam(-):\catudist\to\catidist$ is a local equivalence.
\end{lemma}
\begin{proof}
Consider uniform distributors
$\upa\stackrel{G}{\edist}\upb\stackrel{H}{\edist}\upc$, where
$\upa=\upiar,\,\upb=\upjbs,\,\upc=\upkct$. It follows directly from the
definition that $\ufam(G)$ and $\ufam(H)$ are separated. For compatibility
with composition, let $M\in \catset$, $\alpha:M\to A_i$ and $\gamma:C_k$ for
some $i\in I$ and $k\in K$. Instantiating definitions, we get
\begin{multline*}
(\ufam(H)\circ\ufam(G))_M(\gamma,\alpha)\Leftrightarrow\\
\exists
e\vtp N\!\!\eepi\!\! M,\,j,\,\beta\vtp N\!\!\to\!\! B_j\qdot
\setof{(\gamma(en),\beta n)}{n\in N}\in H_{kj}\wedge 
\setof{(\beta n,\alpha(en))}{n\in N}\in G_{ji}
\end{multline*}
and
\begin{multline*}
 \ufam(H\circ G)_M(\gamma,\alpha) \Leftrightarrow\\
 \exists j,\,h\in
H_{kj},\,g\in G_{ji}\;\forall m\vtp M\;\exists v\vtp B_j\qdot
 h(\gamma
m,b)\wedge g(b,\alpha m)
\end{multline*}
It is easy to see that the first condition implies the second one. For the
converse direction, we can set $N=\setof{(m,b)}{h(\gamma m,b)\wedge
g(b,\alpha m)}$ for fixed $h$ and $g$.

Clearly, $\ufam(-)$ is monotone on morphisms. To see that it is also order
reflecting, let $G,H:\upa\edist\upb$ such that $\ufam(G)\leq\ufam(H)$.
Let $B_j\times A_i\supseteq g\in G_{ji}$, and let $\alpha:g\to A_i$ and
$\beta:g\to B_j$ be the projections. Then we can reason
\[
 \ufam(G)_g(\beta,\alpha)\quad\imp\quad\ufam(H)_g(\beta,\alpha)\quad\imp\quad
g\in H_{ji},
\]
whence $G\leq H$.

Finally, let $\Phi:\ufam(\upa)\to\ufam(\upb)$ be a separated indexed
distributor. We can construct a pre-image $G:\upa\to\upb$ under $\ufam(-)$ by
setting
\[
G_{ji}=\setof{g\subseteq B_j\times A_i}{\Phi_g(\beta,\alpha)} ,
\]
where $\beta:g\to B_j$ and $\alpha:g\to A_i$ are again the projections.
\end{proof}
\begin{remarks}\label{rem:hoshino}
\begin{enumerate}
\item\label{rem:hoshino-hoshino}
It is possible to define the subcategory of $\catuord$ on \emph{1-sorted} uniform 
preorders as a category of monads and modules on a more 
primitive locally ordered category whose objects are sets, and whose morphisms 
are downward closed sets of binary relations. I learned this from Naohiko 
Hoshino, who, in unpublished work, defined a locally ordered 
category of combinatory objects similar to (but more general than) 1-sorted uniform preorders using this approach.

A similar suggestion has been made earlier by Paul-André Melliès, who strongly 
promoted the importance of the construction of monads and modules on several 
occasions and in different contexts (see e.g.~\cite{mellies2008groupoides}). 
However, I failed to see
the relevance back then.

Tom Hirschowitz pointed out that it is possible to define \emph{many-sorted} 
uniform preorders as quantaloid-enriched {categories}, which is a direct 
generalization of
the monads-and-modules approach. We describe this in more detail in 
Section~\ref{sec:hirschowitz}.
 \item 
So far, we have observed a remarkable similarity between the theory of ordinary
and of uniform preorders -- ordinary preorders often serve as a guiding
principle to come up with constructions for the uniform case.
In the context of distributors, there is however also a major difference:
In the case of ordinary preorders, the locally ordered category of distributors
is the Kleisli category of the downset monad. In the uniform case, however,
$\catudist$ is not the Kleisli category of the $D$-monad, and it does not seem
to be representable by a Kleisli construction at all.
\end{enumerate}
\end{remarks}

\section{Relationally complete uniform
preorders}\label{sec:reational-completeness}

Given a (say finitely complete) preorder $A$, its preorder $DA$ of downsets has
small meets and Heyting implication, and $\widehat{A}$ is a topos and thus in
particular locally cartesian closed. 

For a finitely complete \emph{uniform} preorder $\upa$, $D\upa$ does not in
general have universal quantification and implication, and neither does
$\widehat{\upa}$ need to be locally cartesian closed -- however, as we showed in
Theorem~\ref{theo:rda-lccc}, these two properties are equivalent. In this
section, we study an equivalent combinatorial criterion for this to be the
case, which we call \emph{relational completeness}.

The development of this concept was initially motivated by Hofstra's
characterization of BCOs $A$ such that $\ufam(DA)$ is a tripos as coming from
`ordered pcas with a filter'~\cite[Theorem~6.9]{hofstra2006all}, but curiously
enough, the end result is quite close to Longley's concept of `higher order
C-structure'~\cite{longley2011computability}.

  \begin{definition}\label{def:rel-compl-ufp}
A finitely complete uniform preorder $\upiar$ is called \emph{relationally
  complete}\index{relationally complete}, if 
for each pair $j,k\in I$ there exists $j\imp k\in I$ and $@^j_k\in
      R_{(j\imp k) * j,k}$ such that for all $i\in I$ and $r\in
      R_{i * j,k}$ there exists $\tilde{r}\in R_{i,j\imp k}$ such that
\begin{equation}\label{eq:rel-compl}
\forall a\in A_i\;\exists h\in A_{j\imp k}\qdot \tilde{r}(a,h)\wedge
r(a\wedge -,-)\subseteq @^j_k(h\wedge -,-).
\end{equation}
  \end{definition}

\begin{remarks}\label{rem:rel-compl}
\begin{enumerate}
\item
Precision: to conform with our convention regarding the axiom of choice
(Section~\ref{sec:choice}), in particular to be able to construct a
\emph{choice} of universal quantification in Lemma~\ref{theo:rel-compl} below,
we require a relationally complete uniform preorder
$\upiar$ to come with a \emph{choice} of sorts $j\simp k$ and relations $@^j_k$
for all $j,k\in I$. For the relations $\tilde{r}$, on the other hand, mere
existence is sufficient.
 \item
In the one sorted case, we call $@$ the \emph{generic relation}\index{generic relation}.
In the case of number realizability, the generic relation is the
universal Turing machine, for a meet-semilattice, the ordering relation is
the generic relation. 

\item\label{rem:rel-compl-prim}
Relational completeness is not tautological. The archetypal example of a
finitely complete uniform preorder which is \emph{not} relationally complete is
the primitive recursive uniform preorder
$(\N,\adcl\prim)$ from Example~\ref{ex:uords}\ref{ex:uords-prim} --
here, relational completeness would imply the existence of a primitive
recursive interpreter for primitive recursive functions, which is impossible
because of diagonalization.
\item
It is instructive to compare condition~\eqref{eq:rel-compl} to the statement
that 
\[r\subseteq @^j_k\circ\wedge\circ(\tilde{r}\times\id_{A_j}),\] 
a variant
of which occurs in Longley's definition of `higher order
C-structure' \cite[Definition~5.22]{longley2011computability}\footnote{It is
not literally the same condition, since Longley's `cartesian
C-structures'~\cite[Definition~5.19-(ii)]{longley2011computability} are defined
differently from our finitely complete uniform preorders.}. Spelled out, the
former is equivalent to
\begin{equation*}
 \forall a\,\exists h\qdot\tilde{r}(a,h)\wedge\forall b,c\qdot r(a,b,c)\imp
@^j_k(h\wedge b,c)
\end{equation*}
whereas the latter is equivalent to
\begin{equation*}
 \forall a,b,c\qdot r(a,b,c)\imp \exists h\qdot \tilde{r}(a,h)\wedge
@^j_k(h\wedge b,c).
\end{equation*}
We see that relational completeness is stronger since the
$\exists h$ is further on the left. 
Nevertheless, the concepts of `relationally
complete uniform preorder' and `higher order C-structure' are remarkably
similar. 
\end{enumerate}
\end{remarks}

\begin{theorem}\label{theo:rel-compl}
  Let $\upa=\upiar$ be a finitely complete uniform preorder.  The following are
  equivalent.
\begin{enumerate}
\item\label{theo:rel-compl-rc} $\upa$ is relationally complete.
\item\label{theo:rel-compl-iuq} $D\upa$ has implication and universal
quantification.
\item\label{theo:rel-compl-lccc} $\srel{\catset}{\upa}$ is locally cartesian
closed.
\end{enumerate}
\end{theorem}
\begin{proof}

We already showed in Theorem~\ref{theo:rda-lccc} that \ref{theo:rel-compl-iuq}
and \ref{theo:rel-compl-lccc} are equivalent;
now we will show the equivalence of \ref{theo:rel-compl-rc} and
\ref{theo:rel-compl-iuq}.

Assume first that $D\upa$ has implication and universal quantification. Given
$j,k\in I$, let $E\emono A_j\times A_k\times P(A_j\times A_k)$ be the membership
predicate. Let $u:E\to P(A_j\times A_k)$ be third projection, and let
$\varphi:E\to PA_j, \psi:E\to PA_k$ be the first and second projections composed
with the singleton maps $A_j\to PA_j$ and $A_k\to PA_k$. We choose $j\imp k\in
I$ to be the sort of the predicate $\forall_u\varphi\imp\psi$ and we
choose $@^j_k\in R_{i * j,k}$ to be a relation realizing the valid
judgment
\begin{equation}\label{eq:def-at}
 (u^*\forall_u\varphi\imp\psi)\wedge\varphi\ent\psi
\end{equation}
(To get a functional choice of $@^j_k$, we can just choose the minimal such
relation, which is unique since the predicate on the right of $\ent$ is
singleton valued).
Now let $i\in I$ and $r\in R_{i * j,k}$. Define 
\[A_i\times A_j\times
A_k\supseteq M =\{(a,b,c)\msep (a\wedge b,c)\in r\},\] 
and let $v:M\to A_i$ be
the first projection. Let $\beta:M\to PA_j$, $\gamma:M\to PA_k$ be the second
and third projections postcomposed with the singleton maps, and let
$\iota:A_i\to PA_i$ be the singleton map on $A_i$. Then $u^*\iota$ is the
predicate on $M$ corresponding to the first projection, and we have
$u^*\iota\wedge \beta\ent \gamma$ realized by $r$. Doing $\imp$ and $\forall$
introduction, we obtain the valid judgment $\iota\ent \forall_v\beta\imp\gamma$
on $A_i$. We define $w:A_i\to P(A_j\times A_k)$ by $A_i\ni a\mapsto
\{(b,c)\msep (a\wedge b,c)\in r\}$, which gives us a pullback square
\[
\vcenter{\xymatrix@-2mm{
M\ar[r]_-v\ar[d]_-x\pullbackcorner & A_i\ar[d]^-w\\
E\ar[r]_-u & P(A_j\times A_k)
} },
\]
and moreover $x^*\varphi=\beta$ and $x^*\psi=\gamma$, which implies using the
Beck-Chevalley condition that $(\forall_v\beta\imp\gamma)\cong
(w^*\forall_u\varphi\imp\psi)$. Choose $\tilde{r}\in r_{i,j\imp k}$ to be a
realizer of 
\begin{equation}\label{eq:judg-tilder}
\iota\ent w^*\forall_u\varphi\imp\psi.
\end{equation} 
To show relational
completeness, it remains to verify~\eqref{eq:rel-compl}. Let $a\in A_i$. Since
$\tilde{r}$ realizes~\eqref{eq:judg-tilder}, there exists
$h\in (\forall_u\varphi\imp\psi)(wa)$ such that $\tilde{r}(a,h)$, and it
remains to show that for $b\in A_j$, $c\in A_k$ we have $r(a\wedge b,c)\imp
@^j_k(h\wedge b,c)$. But if $r(a\wedge b,c)$ then $(a,b,c)\in M$, and
$@^j_k(h\wedge b,c)$ follows from the validity of~\eqref{eq:def-at} at
$x(a,b,c)$.

\medskip

Conversely, assume that $\upa$ is relationally complete. Instead of
constructing implication and universal quantification separately, we show how
to define the `synthetic' connective $\forall_u\varphi\imp\psi$ for
$u:X\to Y$ and $\varphi,\psi\in \ufam(D\upa)_X$. Implication and universal
quantification can then be recovered by either replacing $u$ by the identity, or
$\varphi$ by the true predicate. For $\varphi:X\to PA_j$, $\psi:Y\to
PA_k$, define $(\forall_u\varphi\imp\psi):Y\to PA_{j\imp k}$ by 
\begin{equation*}
 (\forall_u\varphi\imp\psi)(y)=\bigcap_{ux=y}\{h\in A_{j\imp k}\msep\forall
b\in\varphi(x)\,\exists c\in\psi(x)\qdot @^j_k(h\wedge b,c)\}.
\end{equation*}
It is then easy to see that $@^j_k$ realizes
$u^*\forall_u\varphi\imp\psi,\varphi\ent\psi$; and if $\zeta:Y\to
PA_i$ such that the judgment $u^*\xi,\varphi\ent \psi$ is realized by $r\in
A_{i * j,k}$, then $\tilde{r}$ realizes $\xi\ent\forall_u\varphi\imp\psi$.
\end{proof}
\begin{examples}
 \begin{itemize}
\item
If $\upa$ is a uniform preorder such that $\ufam(\upa)$ has $\wedge, \imp$, and 
$\forall$, then $\upa$ is relationally complete. This follows from
Lemma~\ref{lem:d-imp-forall}.
  \item 
Uniform preorders that come from meet-semilattices are always relationally
complete, since downset lattices are complete Heyting algebras.
 \end{itemize}
\end{examples}

 Relational completeness is particularly interesting for \emph{functional}
uniform preorders. Let us rephrase the definition in this case, using the
language of \emph{relational clones} from Section~\ref{sec:clone}.

\begin{lemma}\label{lem:clone-functional-ufp}
 A finitely complete \emph{functional} uniform preorder $\upa=\upiar$ is
relationally
complete iff for every pair $j,k\in I$ there
exists $j\imp k\in I$ and a partial function 
\[
( -\,\appca_{jk}\,-): A_{j\imp k}\times A_j\pto A_k
\]
such that for all $i_1\dots i_n\in I$ and every partial function
$f\in C_{i_1,\dots, i_n
j;k}$
there exists a \emph{total} function $\Lambda f\in C_{i_1,\dots, i_n;j\imp k}$
such that 
\[
 \Lambda{f}(a_1,\dots,a_n)\appca b\ges f(a_1,\dots,a_n,b)
\]
for appropriately typed $a_1,\dots,a_n,b$, where $(C_{i_1,\dots,
i_n;j})_{i_1\dots i_n j\in I}$ is the relational clone associated to $\upa$
defined before Lemma~\ref{lem:clone-from-ufp}.
\qed
\end{lemma}
The notation $\Lambda f$ is inspired by the apparent analogy to abstraction in
$\lambda$-calculus; we denote iterated abstraction by $\Lambda^jf$. If we
abstract an $n$-ary function $f\in C_{i_1,\dots, i_n;j}$ $n$ times, we obtain a
`total $0$-ary function', which is a singleton $\{\Lambda^n f\}\in
C_{\varnothing;i_1\imp\dots\imp
i_n\imp j}$ (i.e.\ $\Lambda^nf\in A_{i_1\imp\dots\imp
i_n\imp j}$) such that 
\[
 (\Lambda^nf)\appca a_1\appca\dots\appca a_n
\simeq \dots
\simeq (\Lambda f)(a_1,\dots,a_{n-1})\appca
a_n\ges f(a_1,\dots,a_n),
\]
where we have strong equality $\simeq$ between all but the last two terms since
the functions $\Lambda^i(f)$ are total.

\begin{lemma}\label{lem:rel-compl-func-tpca}
Every relationally complete functional uniform preorder $\upa = \upiar$ is
induced (in the sense of Definition~\ref{ex:uords}-\ref{ex:rel-typed-pca}) by an
inclusion of typed pcas.
\end{lemma}
\begin{proof}
Relational completeness provides us with the type constructors on $I$. First,
let us show that the partial binary application maps $(-\,\appca\,-):A_{i\imp
j}\times A_{i}\pto A_j$ make $(A_i)_{i\in I}$ a typed pca. For this we have to
construct the combinators postulated in
Definition~\ref{def:typed-pca} and verify the axioms.

For $\coms_{ijk}$, consider the partial function $f_\coms:A_{i\imp j\imp
k}\times A_{i\imp j}\times A_i\pto A_k$ given by $(x,y,z)\mapsto xz(yz)$, which
is contained in the clone $(C_{i_1,\dots, i_n;j})_{i_1\dots i_n j\in I}$, since
the application maps are and the clone is closed under composition. Three-fold
abstraction gives us a singleton $\Lambda^3 f_\coms=\{\coms\}\subseteq A_{(i\imp
j\imp k)\imp (i\imp j)\imp i \imp k}$, and it follows from the remarks
preceding the lemma
that $\coms$ has the desired properties. In the same way, the $\comk$
combinator is obtained by double abstracting the function $(x,y)\mapsto x$.

The combinator $\compair$ is given by abstracting $(x,y)\mapsto x\wedge y$
twice. 
For the (say) first
projection, observe that for given $i,j\in I$ we have $\pil\wedge\pir \leq
\pil$  in $\ufam(\upa)_{A_i\times A_j}$ ($\pil$ and $\pir$ being the first and
second projection), which implies that the set $p=\{(x\wedge y,x)\msep x\in
A_i, y\in A_j\}$ is contained
in $R_{i\ptype j,i}$, and therefore in $C_{i\ptype j;i}$. We take $\comfst$ to
be $\Lambda p\in C_{\varnothing;i\ptype j\imp i}$. Then
the verification of the corresponding axiom is not difficult:
\[
 \comfst\appca(\compair\appca x\appca y)=\comfst\appca(x\wedge y)=p(x\wedge
y)=x,
\]
where the last equation holds since $(x\wedge y,x)\in p$ by definition.
The construction for $\comsnd$ is analogous, which finishes the construction a
typed pca structure on $(A_i)_{i\in I}$.

To get a typed sub-pca, we set for each $i\in I$
\begin{equation}\label{eq:def-asharp}
 A_{\#,i}=\bigl\{a\in A_i\msep \{a\}\in C_{\varnothing;i}\bigr\}.
\end{equation}
Then the combinators $\comk,\coms,\comfst,\comsnd,\compair$ are contained in
the substructure by construction, closure under application follows from
closure under composition of the clone.

It remains to check that the inclusion $(A_{\#,i}\subseteq A_i)_{i\in I}$ of
typed pcas does indeed induce the uniform preorder $\upa$. This is equivalent
to the statement that the partial functions $(a\appca\, -):A_i\pto A_j$ for
$a\in A_{\#,i\imp j}$ generate the sets $C_{i;j}=R_{i,j}$ under down closure.
In one direction, given $a\in A_{\#,i\imp j}$, the function $(a\appca\,-)$ can
be expressed as a composition of $(-\,\appca\,-)$ and $\{a\}$ which are both in
the clone. In the other direction, every element of $f\in C_{i;j}$ is contained
in $(\Lambda f\appca\,-)$. 
\end{proof}
In the same way, we can show the one sorted case.
\begin{corollary}
 Every one-sorted relationally complete functional uniform preorder $\brar$ is
induced by an inclusion
$\pcaas\subseteq\pcaa$ of pcas.
\qed
\end{corollary}
If we want non-relative versions of the previous two lemmas, we only have to
add one condition. We will do this using terminology of
Pitts~\cite[page~11]{pitts81}.
\begin{definition}
 Let $\upiar$ be a finitely complete uniform preorder, and let $(C_{i_1,\dots,
i_n;j})_{i_1\dots i_n,j}$ be the associated relational clone. A
\emph{designated truth value}\index{designated truth value} is an $a\in A_i$ such that $\{a\}\in
C_{\varnothing;i}$.
\end{definition}
\begin{lemma}\label{lem:pca-designated}
 \begin{itemize}
  \item A one sorted uniform preorder $\brar$ is induced by a pca via the
construction from
Example~\ref{ex:uords}-\ref{ex:uords-pca}, iff $\brar$ is relationally complete,
functional, and all elements are designated truth values.
\item A uniform preorder $\upiar$ is induced by a typed pca iff $\upiar$ is
relationally complete, functional, and all elements are designated truth values.
 \end{itemize}
\end{lemma}
\begin{proof}
 The construction of the (typed) sub-pca in~\eqref{eq:def-asharp} gives the
entire (typed) pca iff all truth values are designated.
\end{proof}

\section{Characterizations of realizability triposes and toposes}

In this section, we assemble all our technology to obtain characterizations of
realizability triposes and hyperdoctrines
(Definitions~\ref{def:realizability-tripos},
\ref{def:realizability-hyperdoctrine}), and of the associated realizability
categories (Definition~\ref{def:realizability-cats}) together with their
constant objects functors.

\subsection{Realizability hyperdoctrines and triposes}

\begin{theorem}\label{theo:character-relrealhyper}
A posetal pre-stack $\fifx:\tot{\fifx}\to\catset$ is equivalent to a
\emph{relative realizability hyperdoctrine}
(Definition~\ref{def:realizability-hyperdoctrine}-
\ref{def:realizability-hyperdoctrine-relreal})
 iff
\begin{enumerate}
  \item $\fifx$ models the logical connectives
$\top,\wedge,\imp,\exists,\forall$\footnote{with
quantification along \emph{all} maps, not just along projections}, and
\item there exists a family $\pcaaii$ of sets, and a family $(\pi_i\in
\fifx_{\pcaai})_{i\in I}$ of \emph{$\exists$-prime} predicates such that
\begin{itemize}
\item
the subfibration of $\fifa\subseteq\fifx$ generated by the $(\pi_i)_{i\in I}$ is
closed under
finite meets, 
\item
all $\pi_i$ are modest in $\fifa$, and
\item
$\fifa$ generates $\fifx$ under existential quantification.
\end{itemize}
\end{enumerate}
\end{theorem}
\begin{proof}
First assume that $\fifx=\hyph(I,\pcaa,\pcaas)$ is a relative realizability
hyperdoctrine. We claim that the family of identities $(\id_{\pcaai})_i$
(more precisely the corresponding singleton maps) has the
desired properties. By Example~\ref{ex:pca-d-hyp-trip} we have
$\fifx=\ufam(D(I,\pcaa,R(\pcaa)))$, and $\fifa=\ufam(I,\pcaa,R(\pcaa))$. 
$\fifa$ generates $\fifx$ under existential quantification by
Lemma~\ref{lem:da}-\ref{lem:da-generated}, and predicates in $\fifa$ are prime
in $\fifx$ by Lemma~\ref{lem:epstack-prime}. We explained in
Remark~\ref{rem:flfibs} that $\ufam(I,\pcaa,R(\pcaas))$ is finitely complete,
and while we didn't explicitly prove it, it can be deduced from
Lemma~\ref{lem:clone-from-ufp} 
that $\hyph(I,\pcaa,\pcaas)$ models the claimed connectives.
Finally, the predicates $\id_{\pcaa_i}$ are modest in $\fifa$ by
Lemma~\ref{lem:modest-functional}.

In the other direction, let $(I,\pcaa, R)$ be the uniform preorder presentation
of $\fifa$. $\fifa$ is functional by Lemma~\ref{lem:modest-functional}, and
$\fifx\simeq D\fifa$
by Lemma~\ref{lem:epstack-prime}-\ref{lem:epstack-prime-equiv}. This allows us
to deduce that $(I,\pcaa, R)$ is relationally complete by
Theorem~\ref{theo:rel-compl}, and thus induced by an inclusion of typed pcas by
Lemma~\ref{lem:rel-compl-func-tpca}.
\end{proof}
It is straightforward to derive a characterization of relative realizability
triposes from the theorem.
\begin{corollary}
A posetal pre-stack $\fifx:\tot{\fifx}\to\catset$ is equivalent to a
\emph{relative realizability tripos}
(Definition~\ref{def:realizability-tripos}-
\ref{def:realizability-tripos-relreal})
 iff it satisfies the conditions of the theorem in such a way that the
family $(\pi_i)_{i\in I}$ can be chosen to comprise a single predicate.
\qed
\end{corollary}

To characterize \emph{non-relative} realizability, we have to add one more
condition. Remark that since relative realizability hyperdoctrines $\fifx$ can
be constructed by freely adding existential quantification to fibered
meet-semilattices, they are totally connected, i.e.\
$\delta:\sub(\catset)\to\fifx$ has a finite meet preserving left adjoint
$\pi:\fifx\to\sub(\catset)$
\begin{corollary}\label{cor:character-realhyper}
 \begin{enumerate}
  \item A posetal pre-stack $\fifx:\tot{\fifx}\to\catset$ is equivalent to a
\emph{realizability hyperdoctrine}
(Definition~\ref{def:realizability-hyperdoctrine}-
\ref{def:realizability-hyperdoctrine-real})
 iff it satisfies the conditions of the theorem, and
$\gamma\footnote{See Section~\ref{sec:glob-secs}}
\cong\pi:\fifx\to\sub(\catset)$.
\item $\fifx$ is equivalent to a \emph{realizability tripos}
Definition~\ref{def:realizability-tripos}-
\ref{def:realizability-tripos-real}), if the family of
predicates $(\pi_i)_{i\in I}$ can moreover be chosen to comprise a single
predicate.
 \end{enumerate}
\end{corollary}
\begin{proof}
It is easy to see that for realizability hyperdoctrines, $\pi$ does indeed
coincide with $\gamma$.

Conversely, it is sufficient by Lemma~\ref{lem:pca-designated} to show that all
truth values in a uniform meet-semilattice $\upa$ are designated iff
$\gamma\cong\pi:\ufam(D\upa)\to\sub(\catset)$, which is easy to see as well.
\end{proof}

\subsection{Realizability categories and toposes}

The `modesty' condition in Theorem~\ref{theo:character-relrealhyper} is a bit
awkward since it is not a property of the $\pi_i$ as predicates in the
fibration $\fifx$, but in the subfibration $\fifa$.
If we want to characterize the fibered pretoposes associated to realizability
categories (Definition~\ref{def:realizability-cats}), this condition becomes
replaced by a somewhat nicer condition involving the concept of
\emph{discreteness} that we introduce now (see also
Remark~\ref{rem:modest-discrete} for a comparison of the concepts of modest and
discrete).
\begin{definition}\label{def:discrete}
 Let $\fibp:\tot{\fibp}\to\catr$ be a positive pre-stack. We call $D\in\fibp_I$
\emph{discrete}\index{discrete object in a fibration}, if in any configuration
\[
\xymatrix@R-6mm@C-3mm{
V\epicart[dr]_e\ar[rrrd]^f\\
& U\dashed[rr]_h && D\\
K\depi[rd]_p\\
{} & J\ar[rr]^u && I
} 
\]
where $U$ is subterminal in its fiber, $e$ is cartesian over the regular epi
$p$, and $f$ is over $u\circ p$, there exists
an $h$ over $u$ such that $he=f$. (In this case, $h$ is necessarily unique since
cover-cartesian maps are collectively epic in pre-stacks.)
\end{definition}

\begin{theorem}\label{theo:char-cat}
Let 
$\Delta:\catset\to\catx$ be a regular functor into an exact category, and
$\fibx=\gl_\Delta(\catx)$ the associated fibered pretopos obtained by gluing. 
$\Delta$ is up to equivalence of the
form $\Delta:\catset\to\catrc(I,\pcaa,\pcaas)$
(Definition~\ref{def:realizability-cats}) for an inclusion $\iipcaasa$ of typed
pcas, iff
\begin{enumerate}
 \item\label{theo:char-cat-lccc} $\catx$ is locally cartesian closed
 \item\label{theo:char-cat-fam} there exists a family $(\pi_i:D_i\emono \Delta
A)_{i\in I}$ of
monomorphisms in $\catx$ such that
\begin{enumerate}
 \item\label{theo:char-cat-fam-ip} all $\pi_i$ are indecomposable and projective
in $\fibx$
\item\label{theo:char-cat-fam-disc} all $D_i\to\Delta 1$ are discrete in
$\fibx$
 \item\label{theo:char-cat-fam-gen} the posetal subfibration
$\fifa\subseteq\fibx$ generated by the $\pi_i$ is closed under finite meets and 
every $X\in\tot{\fibx}$ can be covered by a $\varphi\in\tot{\fifa}$ as in 
\[
\xymatrix{ \varphi\coca[r]^-s&S\vepi[r]^-e&X},
\]
where $s$ is cocartesian and $e$ is a vertical epimorphism.
\end{enumerate}
\end{enumerate}
\end{theorem}
\begin{proof}
First, assume that we are dealing with a functor
$\Delta:\catset\to\hyph(I,\pcaa,\pcaas)$, and let $\upa=(I,\pcaa,R(\pcaas))$ be
the uniform preorder associated to the inclusion $\iipcaasa$ of typed pcas.
Then $\upa$ is relationally complete, and $\catx=\srel{\catset}{\upa}$ is
locally cartesian closed by Theorem~\ref{theo:rel-compl}. Since
$\Delta^*\subf(\fifx)=\ufam(D\upa)$, we can define the $\pi_i$ as predicates
in the latter fibration, and we set $\pi_i\in \ufam(D\upa)_{\pcaai}$ to be the
the singleton map corresponding to $\id_{\pcaai}$, keeping the indexing set $I$
from $\iipcaasa$. With this choice of generators we have $\fifa=\ufam(\upa)$,
and \ref{theo:char-cat-fam-ip} and \ref{theo:char-cat-fam-gen} follow from
Lemma~\ref{lem:chr-psh}.

It remains to verify the condition about discreteness. $D_i$ is given by
$(\pcaai,\predeq|_{\pi_i})$, and by instantiating Definition~\ref{def:discrete}
with $\gl_\Delta(\catx)$, we see that we have to verify the existence of a
mediator $h$ in diagrams of the form
\[
\vcenter{\xymatrix@R-3mm{
N\depi[d]& \Delta N\depi[d]& V\pullbackcorner[ld]\depi[d]\ar[rd]\mono[l]_n
\\
M& \Delta M & U\dashed[r]_h\mono[l]^m & D_i
}}
\]
which live in the subcategory
$\asm(D\upa)\subseteq\catset[D\upa]\simeq\srel{\catset}{\upa}$ of
assemblies.
Now by Lemma~\ref{lem:da-totally-connected} $D\upa$ is totally connected,
which implies with Lemma~\ref{lem:tc-dense}-\ref{lem:tc-dense-total} and
-\ref{lem:tc-dense-asm} that $\asm(D\upa)\simeq\tot{\ufam(D\upa)_d}$ -- the
total category of the subfibration of $\ufam(D\upa)$ on \emph{dense} predicates
(Definition~\ref{def-subfib-dense}). By restricting to the support if
necessary, we can assume without loss of generality that $m$ and $n$ are dense
in the diagram, and since denseness of a predicate $\varphi:M\to P\pcaai$ in
$\ufam(D\upa)$ means that it factors through $P_+(\pcaai)$, it remains to show
that $\varphi e\leq \pi_i f$ for dense $\varphi$ and epic $e$ implies that
there exists $h:M\to \pcaai$ such that $he=f$ and $\varphi\leq\pi_i f$
(observe that compared with Definition~\ref{def:discrete}, the triangle now
lives in the \emph{base}, not the total category). This follows from the fact
that any relation realizing $\varphi e\leq \pi_i f$ is functional, which forces
$f$ to be constant on the fibers of $e$ since distinct elements of $\pcaai$
have disjoint images under $\pi_i$.

\medskip

Conversely, assume that $\Delta:\catset\to\catx$ is an exact functor having the
specified properties. Since the fibered poset $\fifa$ is generated by the
predicates $\pi_i$, it comes from a uniform
preorder structure $\upa=(I,\pcaa,R)$ whose underlying family of sets are the
underlying sets of the $\pi_i$ (Lemma~\ref{lem:reconstuct-ufp}). Since $\fifa$
has finite meets, the same is true for $\upa$.
Conditions \ref{theo:char-cat-fam-ip} and
\ref{theo:char-cat-fam-gen} imply together with
Lemma~\ref{lem:chr-psh}-\ref{lem:chr-psh-characterization} that
$\fibx=\widehat{\upa}\simeq\sheaf{D\upa}$, in particular
$\catset[D\upa]\simeq\srel{\catset}{\upa}\simeq\catx$
which implies by
Theorem~\ref{theo:rel-compl} that $\upa$ is relationally complete.
Since $D\upa$ is totally connected by Lemma~\ref{lem:da-totally-connected},
$\catset[D\upa]\simeq\catx$ has well behaved assemblies, which allows us to
derive 
modesty of the $\pi_i$ in $\fifa$ from discreteness of the maps $D_i\to \Delta
1$ using arguments similar to the ones above. Lemma~\ref{lem:modest-functional}
allows us then to derive that $\upa$ is functional, and finally
Lemma~\ref{lem:rel-compl-func-tpca} allows us to deduce that $\upa$ is induced
by an inclusion of typed pcas.
\end{proof}
\begin{remarks}
By Lemma~\ref{lem:chr-psh}-\ref{lem:chr-psh-yip-b}, $\fifa$
is weakly equivalent to the subfibration of $\fibx$ on indecomposable
projectives. To deduce that $\upa$ is finitely complete in $\catuord$, it is
however important
to assume that $\fifa$ (and not only the fibration of indecomposable
projectives) is closed under finite meets. From closure of indecomposable
projectives under finite meets we can only deduce that $\upa$ is finitely
complete in the locally ordered category of left adjoints in $\catudist$ --
i.e.\ the meet map is given by a distributor which has a left adjoint, but is
not necessarily induced by a family of functions.
 \end{remarks}
As for the characterization of hyperdoctrines and triposes, we can deduce
untyped and non-relative versions of the theorem as corollaries.
\begin{corollary}\label{cor:char-cat}
 Let $\Delta:\catset\to\catx$ be a regular functor into an exact category.
\begin{enumerate}
 \item $\Delta$ is up to equivalence of the form
$\Delta:\catset\to\catrt(\pcaa,\pcaas)$ for an inclusion $\pcaas\subseteq\pcaa$
of pcas, iff $\Delta$ satisfies the conditions of Theorem~\ref{theo:char-cat}
in such a way that the family $(\pi_i:D_i\to\Delta A_i)_{i\in I}$ can be chosen
to comprise a single mono.
\item\label{cor:char-cat-nonrel} $\Delta$ is up to equivalence of the form
$\Delta:\catset\to\catrc(I,\pcaa)$ for a typed pca $(I,\pcaa)$, iff $\Delta$
satisfies the conditions of the theorem, and moreover is
right adjoint to the global sections functor
$\Gamma=\catx(1,-):\catx\to\catset$.
\item\label{cor:char-top-nonrel} $\Delta$ is up to equivalence of the form
$\Delta:\catset\to\catrt(\pcaa)$ for a pca $\pcaa$, iff $\Delta$
satisfies the conditions of the theorem in such a way that
the family $(\pi_i:D_i\to\Delta A_i)_{i\in I}$ can be chosen
to comprise a single mono, and moreover $\Delta$ is right
adjoint to the global sections functor $\Gamma=\catx(1,-):\catx\to\catset$.
\end{enumerate}
\end{corollary}
\begin{proof}
 The only non-obvious part is \ref{cor:char-cat-nonrel}. 
The fact that $\Gamma\adj\Delta$ is well known for realizability over untyped
pcas, and the relevant parts of the theory carry over to the typed case
without change.

It remains to show that $\Gamma\adj\Delta:\catset\to\catrt(I,\pcaa,\pcaas)$
implies $\pcaas=\pcaa$. From
Lemma~\ref{lem:delta-left-adj} we know that $\Delta:\catset\to\catset[\fifx]$
has a finite limit preserving left adjoint $\Pi$ whenever $\fifx$ is totally
connected, thus we only have to show that $\Gamma\cong\Pi$ implies
$\pcaas=\pcaa$.

In the proof of Corollary~\ref{cor:character-realhyper} showed that
$\gamma\cong\pi$ on the level of fibered posets implies $\pcaas=\pcaa$, and it
follows from Lemma~\ref{lem:gamma-gamma} that $\gamma\cong\pi$ whenever
$\Gamma\cong\Pi$.
\end{proof}
\begin{remark}
 Although our general approach is to characterize the categories together with
their constant objects functors, we see that in the {non-relative}
cases~\ref{cor:char-cat-nonrel} and \ref{cor:char-top-nonrel}, the corollary
gives us characterizations of the \emph{bare categories}, since the constant
objects functor is already determined by the fact that it is right adjoint to
$\Gamma$ in this case.
\end{remark}

\section{And on arbitrary bases?}\label{sec:arbitrary-bases}

The definition of uniform preorder can be internalized in any topos $\tops$
(and with a bit of care even in predicative metatheories). However, on base
categories other than $\catset$, uniform preorders most naturally do not embed
into
posetal fibrations on $\tops$ as one might naively expect, but rather into
so-called \emph{fibered fibrations}. The concept of `fibered fibration' can be
attributed to Bénabou, who realized that fibrations compose and more
importantly that a fibration on a total category of another fibration can be
viewed as `fibered fibration' in the sense of `generalized category internal to
another generalized category' (the precise technical statement can be
found in~\cite[Theorem~4.1]{streicherfib}). In the following we are
interested in the case where the `base fibration' is a fundamental fibration.
\begin{definition}
Let $\catc$ be a category with finite limits. A \emph{fibered
fibration}\index{fibered!fibration}\index{fibration!fibered} on $\catc$ is a
fibration on $\commacat{\catc}{\catc}$. 
\end{definition}
The following example is paradigmatic of our use of fibered fibrations.
\begin{definition}\label{def:pointwise-fam}
To any fibration $\fibc:\tot{\fibc}\to\catset$ on $\catset$ we
can associate a fibered fibration
$\tilde{\fibc}:\tot{\tilde{\fibc}}\to\commacat{\catset}{\catset}$ by setting
\[
\tilde{\fibc}_{(m:M\to L)} = \prod_{l\in L}\fibc_{M_l}.
\]
(Note that this is \emph{not} the fibered family construction from
\cite[Definition~6.2]{streicherfib}).
\end{definition}
\begin{lemma}\label{lem:quantification-in-tilde-c}
Let $\fibc:\tot{\fibc}\to\catset$ be a fibration on $\catset$.
\begin{enumerate}
\item
$\tilde{\fibc}$ has left/right adjoints to reindexing along \emph{vertical} maps
in $\fund{\catset}$ iff $\fibc$ has left/right adjoints to reindexing along
arbitrary maps in $\catset$. 

In this case the adjoints to reindexing in $\fibc$ satisfy the Beck Chevalley
condition iff 
the adjoints to reindexing along vertical maps in $\tilde{\fibc}$ satisfy the
Beck Chevalley condition for pullbacks along vertical maps.

\item
$\tilde{\fibc}$ has left/right adjoints to reindexing along \emph{cartesian}
maps in $\fund{\catset}$ iff the fibers of $\fibc$ have small (co)products.

In this case the adjoints to reindexing along cartesian maps in $\tilde{\fibc}$
satisfy the Beck Chevalley condition for pullbacks along vertical maps iff the
small (co)products in the fibers of $\fibc$ are stable under pullback along
arbitrary maps.
\end{enumerate}
The Beck Chevalley condition in $\tilde{\fibc}$ for pullbacks along cartesian
maps is always satisfied.
\qed
\end{lemma}
\begin{definition}\label{def:fibered-ufam}
Let $\upa=\upiar$
be a uniform preorder internal to a topos $\tops$. The fibered fibration
$\ufam(\upa):\tot{\upa}\to\commacat{\tops}{\tops}$ is defined as follows.
\begin{itemize}
 \item predicates on $(m:M\to L)\in\commacat{\tops}{\tops}$ are commutative
squares
\[
\xymatrix@R-3.5mm{
M\ar[d]_m\ar[r]^\varphi& A\ar[d]^a\\
L\ar[r]_u& I
} 
\]
\item $(u,\varphi)\leq(v,\psi)$ over $(m:M\to L)$ iff
\[
 \forall l\vtp L\qdot \{(\varphi m,\psi m)\csep m\in M_l\}\in R_{ul,vl}
\]
in the internal logic.
\end{itemize}
\end{definition}
\begin{remark}
If we apply the construction from Definition~\ref{def:pointwise-fam} to the
(ordinary) uniform family fibration $\ufam(\upa):\tot{\ufam(\upa)}\to\catset$
of a uniform preorder in $\catset$, we obtain the uniform family fibration in
the sense of Definition~\ref{def:fibered-ufam}.

The intuition about
$\ufam(\upa):\tot{\ufam(\upa)}\to\commacat{\tops}{\tops}$ for a uniform
preorder $\upa$ in $\tops$ is that the order on predicates over
$m:M\to L$ is pointwise in $L$, but uniform in the fibers of $m$. In the case of
realizability this means that for each $l$ there exists a realizer that works
uniformly over $M_l$.
\end{remark}

\begin{definition}
Let $\fifa:\tot{\fifa}\to\commacat{\tops}{\tops}$ a fibered posetal fibration on
$\tops$.
\begin{itemize}
\item We say that $\leq$ is \emph{horizontally definable}\index{horizontally
definable} in $\fifa$, if for
  every $\varphi,\psi\in\fifa_{M\xrightarrow{m}L}$ there exists a greatest
  subobject $m:U\emono L$ of $L$ such that $\varphi|_U\leq\psi|_U$.
\[
\xymatrix@R-4mm{
\psi|_U\cart[r]\emar[d]|{\rotatebox[origin=c]{90}{$\leq$}} & \psi \\
\varphi|_U\cart[r] & \varphi \\
{}\phantom{\bullet}\pullbackcorner\ar[d]\mono[r] & M\ar[d]^{m} \\
U\mono[r] & L \\
} 
\]

\item We call $\fifa$ a \emph{fibered posetal pre-stack}\index{fibered!posetal pre-stack}\index{pre-stack!fibered posetal}, if $e^*\varphi\leq
e^*\psi$ implies $\varphi\leq \psi$ for every \emph{vertical} epimorphism $e$ in
  $\commacat{\tops}{\tops}$.
\end{itemize}
\end{definition}
\begin{lemma}\label{lem:uord-tops-fibfib}
The locally ordered category $\catuord(\tops)$ of uniform preorders internal to
$\tops$
 is biequivalent to the locally ordered category of
fibered posetal pre-stacks on $\tops$ 
with horizontally definable $\leq$ and a
generic predicate.
\end{lemma}
\begin{proof}
First of all we have to show that fibered fibrations of the form $\ufam(\upa)$
for internal uniform preorders $\upa=\upiar$ are fibered pre-stacks with
horizontally definable $\leq$. The pre-stack condition is immediate from the
definition of $\ufam(\upa)$. For the horizontal definability of $\leq$ let
$(u,\varphi),(v,\psi)\in\ufam(\upa)_{M\xrightarrow{m} L}$.
\[
\xymatrix@R-3mm{
M\ar[d]_m\ppair{r}{\varphi}{\psi}& A\ar[d]^a\\
L\ppair{r}{u}{v}& I
} 
\]
Then the greatest subobject $M\subseteq L$ such that $\varphi|_U\leq\psi|_U$ is
given by 
\[
L\supseteq M =  
\{l\msep\{(\varphi m,\psi m)\csep m\in M_l\}\in
R_{ul,vl} \}.
\]

The proof that $\upa\mapsto\ufam(\upa)$ is a local equivalence is analogous
to the proof of Lemma~\ref{lem:ufp-embed}: Let $\upa=\upiar$ and $\upb=\upjbs$
be uniform preorders internal to $\tops$.
To see that the construction is
locally order reflecting, assume that $\uf,\vg:\upa\to\upb$ such that
$\ufam\uf\leq\ufam\vg:\ufam(\upa)\to\ufam(\upb)$. Then we have in particular
$\ufam\uf_{A\to I}(\id_I,\id_A)\leq \ufam\vg_{A\to I}(\id_I,\id_A)$, and since 
$\ufam\uf$ and $\ufam\vg$ act on predicates by postcomposition, and taking into
account the definition of the ordering on monotone maps, this implies
$\uf\leq\vg$.
To see that $\ufam(-)$ is essentially full, let $F:\ufam(\upa)\to\ufam(\upb)$.
An essential pre-image is given by $F_{A\to I}(\id_A,\id_I)$.

\medskip

Finally, we have to show that $\ufam(-)$ is bi-essentially surjective.
Let $\fifa:\tot{\fifa}\to\scs$ be a fibered posetal pre-stack with horizontally
definable $\leq$ and generic predicate $\iota\in\fifa_{A\xrightarrow{a}I}$.
Let $P(a\times a)\to I\times I$ be the power object of $A\times
A\xrightarrow{a\times a}I\times I$ in $\tops/(I\times I)$, and let 
$E\emono (A\times A) \times_{I\times I} P(a\times a) $
be the associated membership predicate. Consider the diagram
\[
\vcenter{\xymatrix{
\pil^*\iota|_R\leq\pir^*\iota|_R & \pil^*\iota,\pir^*\iota & \iota \\
\bullet\pullbackcorner\mono[r]\ar[d] &
E\ar@<.6ex>[r]\ar@<-.6ex>[r]\ar[d] & A\ar[d]\\
R\mono[r] & P(a\times a)\ppair{r}{p}{q} & I
}},
\]
where the parallel horizontal pairs are the evident projections, and $R$ is the
maximal subobject of $P(a\times a)$ such that
$p^*\iota|_R\leq q^*\iota|_R$. Then it is straightforward to check that
$(A,R)$ is an internal uniform preorder and that the associated fibered
fibration is equivalent to
$\fifa$. 
\end{proof}

Given an uniform preorder $\upa$ in $\catset$, we saw earlier that the 
logical structure of $\ufam(\upa)$ can be characterized directly in $\catuord$
-- $\upa$ has finite meets iff the diagonal and terminal projection maps of
$\upa$ have right adjoints, and it has quantification iff $\upa$ is an algebra
for the corresponding monad\footnote{Actually we didn't treat implication -- to
get a fibration-free treatment here we need $\catudist$.
But in the following we
are mainly concerned about quantification anyway.}. The treatment of the
propositional connectives generalizes straightforwardly to internal uniform
preorders and fibered fibrations, but the quantifiers require attention: 
analogous to Lemma~\ref{lem:quantification-in-tilde-c}, the existence of a
$D$-algebra structure on an internal uniform preorder $\upa$ in a topos $\tops$
corresponds to 
$\ufam(\upa):\tot{\ufam(\upa)}\to\catset$ having existential quantification 
along \emph{vertical} maps in $\fund{\tops}$ -- quantification along cartesian
 maps would correspond to cocomplete fibers in the simply fibered case over  
$\catset$, a property that is not important in the present work.

\medskip

The emergence of the additional layer in the fibrations seems a bit
frightening technically, especially since one can imagine situations where by
iteration the layers get stacked up even further. Fortunately this can be
avoided in an important special case as we will see now.

Let $\fifa:\tot{\fifa}\to\commacat{\tops}{\tops}$ be a fibered posetal
fibration where $\leq$ is horizontally definable and where the fibers have
greatest elements $\top$ which are stable under reindexing. The fibration
$\fifa^{(1)}:\tot{\fifa^{(1)}}\to\tops$ is the pullback of $\fifa$ along the
functor
\[
\bigl(M\mapsto (M\to 1) \bigr)\;:\;\tops\to\commacat{\tops}{\tops}.
\]
We can define a fibered monotone map $\gamma:\fifa^{(1)}\to\sub(\fibs)$.
The image of a predicate $\varphi\in\fifa_M^{(1)}=\fifa_{(M\to 1)}$ under
$\gamma$ is given by first reindexing $\varphi$ in $\fifa$ onto $\id_M$, and
then taking the greatest subobject of $U\emono M$ such that the restriction of
the reindexing to $U$ is entailed by $\top$.
\[
\begin{matrix}
\top\leq \overline{\varphi}|_U\\
\\
\gamma(\varphi)=U
\end{matrix}
\qquad
\vcenter{\xymatrix@R-4mm{
\overline{\varphi}|_{U}\cart[r] & \overline{\varphi}\cart[r] & \varphi \\
U\mono[r]\ar[d]\pullbackcorner & I\ar[r]\ar[d] & I\ar[d]\\
U\mono[r] & I\ar[r] & 1
}} 
\]
If $\fifb:\tot{\fifb}\to\catset$ is a fibered poset with pullback stable
greatest elements, then the result of applying the previous construction to the
fibration $\tilde{\fifb}:\tot{\tilde{\fifb}}\to\commacat{\catset}{\catset}$
from Definition~\ref{def:pointwise-fam}, is precisely the transformation
$\gamma$ defined in Section~\ref{sec:glob-secs}. Thus the coincidence of
notation is justified, and moreover we see that the construction from
Section~\ref{sec:glob-secs} can be understood as a kind of comprehension
principle after all.

\medskip

Now let $\brar$ be a one-sorted uniform preorder with
$\top,\wedge,\imp,\forall$ in a topos $\tops$ (over $\catset$, these
are exactly the regular triposes). We observe that since we only have one sort,
the
predicates in $\ufam\brar_{M\to L}$ are in bijection with the predicates in
$\ufam\brar_{M\to 1}$ (both are simply morphisms $\varphi:M\to A$), and
furthermore we have the following lemma.
\begin{lemma}\label{lem:fibered-fib-from-gamma}
 Let $\brar$ be a one-sorted uniform preorder with $\top,\wedge,\imp,\forall$ in
a topos $\tops$. Let $m:M\to L$ and $\varphi,\psi:M\to A$. Then
$\varphi\leq\psi$ in $\ufam\brar_{m}$ iff $\top\leq\gamma(\forall_m(
\varphi\imp\psi))$ in $\ufam\brar^{(1)}$.
\end{lemma}
\begin{proof}
Consider the cube
\[
 \vcenter{\xymatrix@-5mm{
& M\ar[rr]\ar[dd]|\hole && L\ar[dd] \\
M\ar[rr]\ar[dd]\ar[ur] && L\ar[ur]\ar[dd]\\
& 1\ar[rr]|\hole && 1 \\
L\ar[rr]\ar[ur] && L\ar[ur]\\
}},\]
which is a pullback square in $\commacat{\tops}{\tops}$ (since both horizontal
faces are). To avoid confusion, we denote the maps $\varphi,\psi$ by
$\varphi_0,\psi_0$ when regarding them as predicates on $M\to 1$, and by
$\varphi,\psi$ when regarding them as predicates on $M\to L$. Now we have of
course that $(\id_M,!_L)^*\varphi_0=\varphi$ and
$(\id_M,!_L)^*\psi_0=\psi$, and we can argue
\begin{align*}
\varphi&\leq\psi &\text{in }\ufam\brar^{\phantom{(1)}}& \text{ iff}\\
\top&\leq\varphi\imp\psi &\text{in }\ufam\brar^{\phantom{(1)}}& \text{ iff}\\
\top&\leq(\id_M,!_L)^*(\varphi_0\imp\psi_0) &\text{in
}\ufam\brar^{\phantom{(1)}}& \text{ iff}\\
\top&\leq\forall_{(m,\id_L)}(\id_M,!_L)^*(\varphi_0\imp\psi_0) &\text{in
}\ufam\brar^{\phantom{(1)}}& \text{ iff}\\
\top&\leq(\id_L,!_L)^*\forall_{(m,\id_1)}(\varphi_0\imp\psi_0) &\text{in
}\ufam\brar^{\phantom{(1)}}& \text{ iff}\\
\top&\leq\gamma(\forall_{m}(\varphi_0\imp\psi_0)) &\text{in
}\ufam\brar^{(1)}\\
\end{align*}
where the two last steps follow from the Beck-Chevalley condition and the
definition of $\gamma$, respectively.
\end{proof}
Thus, all the ordering structure of $\ufam\brar$ can be encoded in terms of
$\ufam\brar^{(1)}$ and $\gamma$. 

In order to get a characterization of one-sorted internal uniform preorders
with $\top,\wedge,\imp,\forall$ in terms of fibered preorders and $\gamma$, it
remains to characterize the fibered monotone maps $\gamma$ that arise in from
the construction described before the lemma. For the purposes of this section,
we use the following definitions.
\begin{definition}\label{def:fibtrip-tripgamma}
 Let $\tops$ be a topos.
\begin{enumerate}
 \item A \emph{fibered tripos}\index{fibered!tripos}\index{tripos!fibered} on $\tops$ is a fibered
posetal pre-stack $\fifx:\tot{\trip}\to\commacat{\tops}{\tops}$ with
horizontally definable $\leq$ whose fibers are pre-Heyting algebras, which has
universal quantification along vertical maps (subject to BC along arbitrary
maps), and a generic predicate $\triptr\in\fifx_{(\prop\to 1)}$
\item \label{def:fibtrip-tripgamma-gamma}
A \emph{`tripos with $\gamma$'}\index{tripos!with $\gamma$} on $\tops$ is a posetal pre-stack
$\trip:\tot{\trip}\to\tops$ which models $\twif$ and has a generic predicate
$\triptr\in\trip_\prop$, together with a finite meet preserving fibered
monotone map $\gamma:\trip\to\sub(\tops)$ satisfying
\[
 \top\leq\gamma(p)\quad\imp\quad\top\leq p\qquad\text{for }p\in \trip_1.
\]
\end{enumerate}
\end{definition}
Observe that every regular tripos on $\catset$ is
\emph{uniquely} a `tripos with $\gamma$' -- the condition that $\top$ is
reflected over $1$ already forces $\gamma$ with to coincide with the
transformation defined in Section~\ref{sec:glob-secs}.
\begin{theorem}\label{theo:fibered-internal-gamma-tripos}
Let $\tops$ be a topos.
 The following concepts are equivalent.
\begin{itemize}
 \item internal one-sorted uniform preorders with $\twif$
\item fibered triposes on $\tops$
\item triposes with $\gamma$ on $\tops$
\end{itemize}
\end{theorem}
\begin{proof}
 The equivalence of internal uniform preorders with the specified properties
and fibered triposes follows from Lemma~\ref{lem:uord-tops-fibfib} and our
remarks about propositional connectives and quantification in fibered
fibrations.

We described before Lemma~\ref{lem:fibered-fib-from-gamma} how to obtain the
fibration $\fifx^{(1)}$ from a fibered fibration $\fifx$, and how to construct
$\gamma$ using the horizontal definability of $\leq$.

Conversely, to construct a fibered tripos $\fifx$ from a tripos $\trip$ with
$\gamma$, we take predicates in $\fifx_{(M\xrightarrow{m} L)}$ to be 
predicates in $\trip_M$, and define the ordering by 
\[
\varphi\leq\psi\quad\defequi\quad \top\leq\gamma(\forall_m(\varphi\imp\psi))
\]
as in Lemma~\ref{lem:fibered-fib-from-gamma}.

We leave the numerous verification necessary to establish that the
constructions are well defined and mutually inverse to the reader.
\end{proof}
\begin{remarks}\label{rem:geometric-theory-of}
\begin{enumerate}
 \item 
We could have phrased the previous result as a biequivalence of locally
ordered categories, but it doesn't really matter which morphisms we consider --
the statement only depends on the concept of `equivalence' of
fibrations. 
\item\label{rem:geometric-theory-of-triposes}
I find it remarkable that the logical structure that is necessary to make the
theorem work is exactly that of a tripos. For me that is a strong indication
that the above concept of `tripos with $\gamma$' might be a good alternative to
the usual definition of tripos on base toposes other than $\catset$, especially
when envisioning a `geometric theory of triposes' which views constant objects
functors as generalizations of geometric morphisms.
\end{enumerate}
\end{remarks}

\begin{example}
 On base toposes $\tops$ other than $\catset$, it is possible that the same
tripos $\trip$ on $\tops$ can be equipped
with different maps $\gamma:\trip\to\sub(\tops)$ corresponding to different
uniform preorders and giving rise to non-equivalent fibered triposes. We
demonstate this using as tripos the subobject fibration
$\sub(\widehat{\bbtwo})$ of the Sierpinski topos $\widehat{\bbtwo}$ ($\bbtwo=
(0\to 1)$).

On the one hand, $\sub(\widehat{\bbtwo})$ is the externalization of the
internal preorder $\Omega$, which can be viewed as uniform preorder via the
construction from Example~\ref{ex:uords}-\ref{ex:uords-preords}. The
corresponding transformation
$\gamma:\sub(\widehat{\bbtwo})\to\sub(\widehat{\bbtwo})$ is just the identity,
and the order on predicates in the associated fibered fibration is given by
\[
U\leq V \stext{over}A\to I\qtext{iff} U\subseteq V
\]
for subobjects $U,V\subseteq A$, in other words the fibered fibration is just
given by the fibered family construction (\cite[Definition~6.2]{streicherfib})
\[
\vcenter{\xymatrix@-3mm{
{}\phantom{\bullet}\pullbackcorner
\ar[r]\ar[d]
&
\Sub(\sierp)
\ar[d]^{\sub(\sierp)}
\\
{}\commacat{\sierp}{\sierp}\ar[r]_{\partial_1} 
& 
{}\sierp
} },
\]
where $\partial_1$ is the projection on the domain.

\medskip

A different transformation $\gamma':\sub(\sierp)\to\sub(\sierp)$ satisfying the
conditions of
Definition~\ref{def:fibtrip-tripgamma}-\ref{def:fibtrip-tripgamma-gamma} is
given by
\[
\vcenter{\xymatrix@-3mm{
U_0\mono[d]_{m_0} & \ar[l]U_1\mono[d]^{m_1}\\
A_0 & \ar[l]A_1
}}
\quad\stackrel{\gamma'}{\mapsto}\quad
\vcenter{\xymatrix@-3mm{
A_0\ar[d]_{\id} & \ar[l]U_1\mono[d]^{m_1}\\
A_0 & \ar[l]A_1
}}\footnote{Incidentally, $\gamma'$ is a universal closure operation, but this
seems to be a different story since we do not use it to describe a subtopos,
but a self-fibering of $\sierp$.}.
\]
To give an explicit description of the associated internal uniform preorder,
remark that a uniform preorder structure on $\Omega$ is a subobject 
$R\subseteq P(\Omega\times \Omega)$, which amounts to a set of binary
relations on $\Omega$ for $R_1$ and a set of binary relations on $\Omega_0$ for
$R_0$, such that the obvious inclusion holds. In our case, $R_1$ is the set of
subrelations of the
implication relation on $\Omega$, and $R_0$ consists of \emph{all} relations
on $\Omega_0\times\Omega_0$. 

Finally, the fibered fibration associated to $\gamma'$ is given
by $U\leq V$ over $A\xrightarrow{a} I$ iff $(\forall_a(U\imp V))=\top$ for
$U,V\subseteq A$ as in the following diagram,
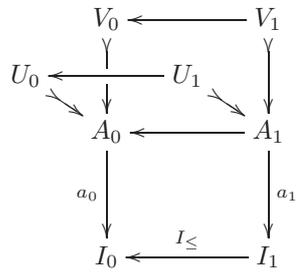
\[
\vcenter{\xymatrix@R-6mm@C-4mm{
& V_0 \mono[dd]|\hole && V_1\mono[dd]\ar[ll]\\
U_0 \mono[dr] && U_1\mono[dr]\ar[ll]\\
& A_0 \ar[ddd]_{a_0} && A_1\ar[ddd]^{a_1}\ar[ll]\\
\\
\\
& I_0  && I_1\ar[ll]_{I_\leq}\\
}}
\]
which concretely means that $U_1\subseteq V_1$, and $U_0\cap
a_0^{-1}(I_\leq(I_1))\subseteq V_0$. As two extreme cases, the ordering of
predicates coincides with the ordering by inclusion whenever $I_\leq$ is
surjective, and the ordering collapses if $I_1$ is empty.
\end{example}

\appendix
\titleformat{\chapter}
  {\normalfont\Huge\bfseries}{}{0em}{}
\titlespacing*{\chapter}{0pt}{0pt}{8.3ex}
\chapter{Appendix}

\section{Partial combinatory algebras and triposes}\label{sec:pcas}

In this appendix, we recall classical concepts of categorical realizability
which don't have a natural place in the main text.

\subsection{Partial combinatory algebras}

\begin{definition}\label{def:pca}
A (weak) partial combinatory algebra\dindex{partial combinatory algebra}
(pca\dindex{pca}) is a set $\pcaa$ together with a
partial
binary
operation $(-\,\cdot\, -):\pcaa\times\pcaa\rightharpoonup\pcaa$ such that there
exist $\comk,\coms\in\pcaa$ satisfying for all $x,y,z\in\pcaa$ the conditions
\begin{align*}
  \comk\appca x\appca y &= x\\
 &\coms\appca x\appca y\defined\\
 \coms\appca x\appca y\appca z&\ges x\appca z\appca (y\appca z).
\end{align*}
\end{definition}
We refer to Section~\ref{sec:conventions} for the notations $t\defined$ and
$s\les t$. We use the usual convention that the binary application
associates to the left; thus, for example, ${\coms}\appca x\appca
y\appca z$ should be
read as $((\coms\appca x)\appca y)\appca z$. We usually omit the dot and
write
application simply by juxtaposition. The notion of the above definition is
usually called \emph{weak} pca, `strong' pcas being those where 
the last condition is replaced by the strong equality $\coms\appca x\appca
y\appca z\simeq x\appca z\appca (y\appca z)$. However, as we deal exclusively
with the weak version in this text, we will simply call them pcas.
\begin{example}\label{ex:kone}
 The archetypal example of a pca is the so-called \emph{first Kleene
algebra}\index{first Kleene algebra} $\pcakone$ which has as underlying set the
set $\N$ of natural numbers, and where the application operation is given by
\[
 n\appca m = \phi_n(m),
\]
where $(\phi_n)_{n\in \N}$ is an effective enumeration of partial recursive
functions (see~\cite[Section~1.4.1]{vanoosten2008realizability}).
\end{example}
Next, we introduce Longley's \emph{typed} pcas~\cite{longley1999unifying},
whose relation to ordinary pcas is analogous to the relation between typed and
untyped $\lambda$-calculus. As for pcas, there is a `weak' and a `strong'
version and we use the weak one.
\begin{definition}\label{def:typed-pca}
 A \emph{typed partial combinatory algebra}\index{partial combinatory algebra!typed}\index{pca!typed} (typed pca) is a pair
$(I,\pcaa)=(I,\pcaaii)$ consisting of 
\begin{enumerate}
 \item 
a set $I$ of `types', equipped with binary operations 
\[(-\ptype -),(- \imp -):I\times I\to I,
 \]
\item 
a family
$(\pcaa_i)_{i\in I}$ of sets, with for each pair $i,j\in I$ of types a partial
`application' map $(-\,\appca_{ij}\,-):\pcaa_{i\imp
j}\times\pcaa_i\pto\pcaa_j$, 
\end{enumerate}
 such that for all $i,j,k\in I$ there exist elements
\begin{align*}
 \comk_{ij} &\in \pcaa_{i\imp j\imp i} & \compair_{ij}& \in\pcaa_{i\imp j\imp
(i\ptype j)}\\
\coms_{ijk} &\in \pcaa_{(i\imp j\imp
k)\imp (i\imp j)\imp i \imp k} & \comfst_{ij}&\in\pcaa_{(i\ptype j)\imp i}\\
 &
 & \comsnd_{ij}&\in\pcaa_{(i\ptype
j)\imp j}\\
\end{align*}
satisfying
\begin{align*}
 \comk x y &= x & \comfst (\compair x y) &= x\\
&\coms x y  \defined& \comsnd (\compair x y) &= y\\
 \coms xyz &\ges x z (y z)
\end{align*}
for all appropriately typed $x,y,z$.
\end{definition}
As usual, the type constructor $(-\imp-)$ for function spaces associates to the
right.

We do not go into the details of the theory of pcas -- for the untyped case we
refer the reader to~\cite{vanoosten2008realizability}, and the typed case is
not much different (certain constructions, such as fixed point combinators,
don't work in the typed case). 

\medskip

\begin{definition}\label{def:sub-pca}
 Let $\pcaa$ be a pca. A \emph{sub-pca}\index{sub-pca}\footnote{called
\emph{elementary} sub-pca in \cite{vanoosten2008realizability}} of $\pcaa$ is a
subset
$\pcaas\subseteq\pcaa$ which is closed under application in the sense that
\[
 a,b\in\pcaas,\; a\appca b\defined \;\imp\; a\appca b\in \pcaas,
\]
such that the combinators $\comk,\coms$ for $\pcaa$ can be chosen in $\pcaas$.

\medskip

In the same way, a \emph{typed sub-pca}\index{sub-pca!typed} of  a typed pca $\pcaia$ is a family
$(\pcaasi\subseteq\pcaai)_{i\in I}$ of subsets which is closed under application
in the
sense that 
\[
 a\in\pcaasi,\; b\in\pcaass{i\imp j},\; a\appca b\defined \;\imp\; a\appca b\in
\pcaass{j}
\]
and such that all the combinators for $\pcaaii$ can be chosen in the subsets.
 \end{definition}
A (typed) pca together with a (typed) sub-pca is also called an \emph{inclusion
of pcas}\index{inclusion of pcas}. Inclusions of pcas are important in relative
realizability and occur naturally in our reconstruction. We denote inclusions
of pcas and of typed pcas by $(\pcaas\subseteq\pcaa)$, and
$(I,\pcaas\subseteq\pcaa)$, respectively.

\subsection{Triposes}

Informally, triposes -- introduced in~\cite{hjp80} and studied further in
\cite{pitts81,pitts2002} -- are fibrational models of higher order
intuitionistic logic.

In one sentence, a tripos on a cartesian closed category $\catc$
is a complete fibered Heyting pre-algebra $\trip$ with a 
generic predicate. In a bit more detail, this means the following.
\begin{definition}\label{def:tripos}
 A \emph{tripos}\index{tripos} on a cartesian \emph{closed}\footnote{It is
possible to define triposes on categories having only finite
limits~\cite{pitts81}, or even finite products\cite{pitts2002,
vanoosten2008realizability,frey2011}, but this is not relevant here.
} 
category $\catc$ is
a fibered preorder
\[
 \trip:\tot{\trip}\to\catc
\]
such that
\begin{enumerate}
\item all fibers $\trip_C$ for $C\in\catc$ are Heyting pre-algebras, and the
Heyting pre-algebra structure is preserved by reindexing.
\item $\trip$ has internal products in the sense of
\cite[Section~7]{streicherfib}, and
\item generic
predicate
$\triptr\in\trip_\prop$.
\end{enumerate}
\end{definition}
Using the postulated structure, we can interpret the $\forall, \imp, \wedge,
\top$ fragment of predicate logic in a tripos, and one can use the generic
predicate and higher order encodings to interpret the remaining connectives
$\exists, \vee, \bot$. If the base category is regular and the tripos is a
pre-stack, it is thus in particular a fibered frame in the sense of
Definition~\ref{def:fibered-frame}. We will refer to a tripos which is a
pre-stack as a \emph{regular tripos}\index{regular
tripos}\index{tripos!regular}.

It seems reasonable to only work with regular triposes,
since all known constructions seem to give rise to triposes satisfying the
pre-stack 
condition, and moreover Pitts'~\cite{pitts81} important \emph{iteration
theorem} depends on it. On the other hand, the condition is not vacuous -- we
will now state an example (due to Streicher) of a tripos which is \emph{not} a
pre-stack.

\begin{example}[Streicher]
Let $\catset^{\bullet\rightrightarrows\bullet}$ be the topos of non-reflexive
graphs. Define the fibered preorder $\trip$ on
$\catset^{\bullet\rightrightarrows\bullet}$ by the pullback
\[
 \xymatrix@R-3mm{
 \tot{\trip}\pullbackcorner\ar[d]_\trip\ar[r] 
&
\Sub(\catset)\ar[d]^{\sub(\catset)}\\
\catset^{\bullet
\rightrightarrows\bullet}_{\phantom{\bullet\rightrightarrows\bullet}}\ar[r]
^\Gamma & \catset
}
\]
where $\Gamma$ is the global sections functor which sends each
graph to its set of loops. $\trip$ interprets full first order logic since
$\sub(\catset)$ does and the relevant structure is stable under change of base
along finite limit preserving functors. Given a graph $G$, a predicate in
$\trip_G$ is a subset of $\Gamma(G)$, i.e.\ a set of loops. This intuition
allows us to construct a generic predicate --  $\prop$ is the graph
with one vertex and two loops, and $\triptr$ singles out one of the two loops.
Thus $\trip$ is a tripos. To see that $\trip$ is not a pre-stack, take $I$ to
be the graph $(\bullet\to\bullet)$ with two vertices connected by one edge. Its
terminal projection $!:I\eepi 1$ is an epimorphism, but
$\top_1\nvdash\exists_!\top_I$, which contradicts the pre-stack property.
\end{example}

\medskip

Having defined pcas and triposes, we will now explain how to construct triposes
from pcas. These so-called \emph{realizability triposes} traditionally belong
to the central concepts of categorical realizability.
\begin{definition}\label{def:realizability-tripos}
\begin{enumerate}
 \item \label{def:realizability-tripos-real}
Let $\pcaa$ be a pca. The \emph{realizability tripos}\index{realizability tripos}\index{tripos!realizability}
\[\rtr{\pcaa}:\tot{\rtr{\pcaa}}\to\catset\] is defined as follows
\begin{itemize}
 \item Predicates on a set $M$ are functions $\varphi:M\to P(\pcaa)$ into
the power set of $\pcaa$.
\item For predicates $\varphi,\psi:M\to P(\pcaa)$ the ordering is defined by
\[
 \varphi\leq\psi\;\defequi\;\exists e\vtp\pcaa\sall m\vtp M\sall
a\!\in\!\varphi(m)\qdot ea\in \psi(m).
\]
\item Reindexing is given by precomposition.
\end{itemize}
\item\label{def:realizability-tripos-relreal} Let $(\ipcaasa)$ be an inclusion
of pcas. The \iemph{relative realizability tripos}\index{tripos!relative realizability}
$\rtr{\pcaa,\pcaas}:\tot{\rtr{\pcaa,\pcaas}}\to\catset$ is defined as follows.
\begin{itemize}
 \item Predicates on a set $M$ are functions $\varphi:M\to P(\pcaa)$.
\item For predicates $\varphi,\psi:M\to P(\pcaa)$, the ordering is defined by
\[
 \varphi\leq\psi\;\defequi\sists e\vtp\pcaas\sall m\vtp M\sall
a\!\in\!\varphi(m)\qdot ea\in \psi(m)
\]
\item Reindexing is given by precomposition.
\end{itemize}
\end{enumerate}
\end{definition}

We can do analogous constructions for \emph{typed} pcas, and this has been done
in \cite[Definition~3.1-(iv)]{lietz2002impredicativity}, but we can not expect
the result to be a tripos anymore.
\begin{definition}\label{def:realizability-hyperdoctrine}
\begin{enumerate}
 \item \label{def:realizability-hyperdoctrine-real}
Let $\pcaia$ be a typed pca. The \iemph{realizability hyperdoctrine}
$\hyph\pcaia:\tot{\hyph\pcaia}\to\catset$ is defined as follows.
\begin{itemize}
 \item Predicates on a set $M$ are pairs $(i\in I,\varphi:M\to P\pcaai)$.
\item For predicates $(i,\varphi),(j,\psi)$ on $M$, the ordering is defined
by
\[
 (i,\varphi)\leq(j,\psi)\;\defequi\;\exists e\vtp\pcaa_{i\imp j}\sall m\vtp
M\sall
a\!\in\!\varphi(m)\qdot ea\in \psi(m).
\]
\item Reindexing is given by precomposition.
\end{itemize}
\item \label{def:realizability-hyperdoctrine-relreal}
Let $\iipcaasa$ be am inclusion of typed pcas. The \iemph{relative realizability
hyperdoctrine}
$\hyph(I,\pcaa,\pcaas):\tot{\hyph(I,\pcaa,\pcaas)}\to\catset$ is defined as
follows.
\begin{itemize}
 \item Predicates on a set $M$ are pairs $(i\in I,\varphi:M\to P\pcaai)$.
\item For predicates $(i,\varphi),(j,\psi)$ on $M$, the ordering is defined
by
\[
 (i,\varphi)\leq(j,\psi)\;\defequi\;\exists e\vtp\pcaa_{\#,i\imp j}\sall m\vtp
M\sall
a\!\in\!\varphi(m)\qdot ea\in \psi(m).
\]
\item Reindexing is given by precomposition.
\end{itemize}
\end{enumerate}
\end{definition}

All (relative) realizability triposes and hyperdoctrines interpret the
$(\top,\wedge,\imp,\exists,\forall)$-fragment of first order logic, in
particular they are fibered frames. Thus, we can construct their
categories of partial equivalence relations, for which we use the following
terminology.
\begin{definition}\label{def:realizability-cats}
 \begin{itemize}
  \item Given a pca $\pcaa$, the \iemph{realizability topos}
$\catrt(\pcaa)$ is
the category $\catset[\rtr{\pcaa}]$ of partial equivalence relations in
$\rtr{\pcaa}$.
\item For an inclusion $\ipcaasa$ of pcas, the \iemph{relative realizability
topos} $\catrt(\pcaa, \pcaas)$ is
the category $\catset[\rtr{\pcaa, \pcaas}]$.
\item Given a typed pca $\pcaia$, the \iemph{realizability category}
$\catrc\pcaia$ is
the category $\catset[\hyph\pcaia]$.
\item Given an \emph{inclusion} $\iipcaasa$ of typed pcas, the
\iemph{relative realizability category} $\catrc(I,\pcaa,\pcaas)$ is
the category $\catset[\hyph(I,\pcaa,\pcaas)]$ .
 \end{itemize}
\end{definition}
\begin{remark}
Given an inclusion $\iipcaasa$ of typed pcas, it can be deduced from
Theorem~\ref{theo:rda-lccc} and the fact that $\hyph(I,\pcaa,\pcaas)\simeq
D(\ufam(I,\pcaa,\pcaas))$ (Example~\ref{ex:dufamia}) that
$\catrc(I,\pcaa,\pcaas)$ is locally cartesian closed. However, in
general $\catrc(I,\pcaa,\pcaas)$ does not seem to be a pretopos since it
doesn't have finite coproducts. Longley~\cite{longley1999unifying} writes:
\begin{quotation}
``[\dots{}] since RC(A) doesn't automatically 
have binary coproducts. But under a mild extra condition that we can 
``simulate'' the booleans within A, we do.''
\end{quotation}
\end{remark}

\section{A first factorization result}\label{sec:decompo}

Remark~\ref{rem:geometric-theory-of}-\ref{rem:geometric-theory-of-triposes}
alluded to a possible `geometric theory of triposes' which views constant
objects functors associated to triposes in analogy to geometric morphisms. From
Pitts' iteration theorem we know that these functors compose (I expect this to
generalize to `triposes with $\gamma$'). Following the analogy to geometric
morphisms, the natural question to ask is whether they can also be
\emph{decomposed}, paralleling the known factorization theorems for geometric
morphisms (see~\cite{elephant1,elephant2}).

Here I present a first such result.

\begin{definition}
Let $(\trip:\tot{\trip}\to\tops,\gamma:\trip\to\sub(\tops))$, be a tripos with
$\gamma$ (Definition~\ref{def:fibtrip-tripgamma}).
\begin{itemize}
 \item $(\trip,\gamma)$ is called
\emph{realizability-like}\index{tripos!realizability-like}\index{realizability-like tripo}\footnote{Suggestions for better terminology
welcome -- maybe `shallow'?}, if $\gamma\adj\delta$
\item $(\trip,\gamma)$ is called \emph{localic}\index{localic tripos}\index{tripos!localic}, if $\delta\adj\gamma$
\end{itemize}
\end{definition}
\begin{remark}
 It is clear that any internal locale $X$ in $\tops$ gives rise to a localic
tripos with $\gamma$. 
Conversely, if $(\trip,\gamma)$ is a localic tripos with $\gamma$ then the
constant objects functor $\Delta:\tops\to\tops[\trip]$ has a right adjoint and
is thus a localic geometric morphism, which shows that the terminology makes
sense.
\end{remark}
In the following we want to show that any tripos on $\catset$ can be
decomposed into a realizability-like part and a localic part. To  this end, we
first recall a result from Birkedal's thesis~\cite{birkedal2000} about
categories of assemblies. In \cite[Definition~3.3.1]{birkedal2000} Birkedal
defines a category $\asm(\fifx)$ (which we will call $\basm(\fifx)$ to
distinguish it from the assemblies of Section~\ref{sec:assemblies}) for any
existential fibration\footnote{`regular fibration' in his terminology}
$\fifx:\tot{\fifx}\to \catc$ on a finite product category $\catc$.
For the moment we only consider the case with base category $\catset$ as it
spares us to spell out some subtleties. For good measure, we shall also assume
that the existential fibration satisfies the pre-stack condition.
\begin{definition}
 Let $\fifx:\tot{\fifx}\to\catset$ be a fibered frame. 
\begin{itemize}
 \item 
The category
$\basm(\fifx)$ of \emph{Birkedal assemblies}\index{Birkedal assemblies} has
\begin{enumerate}
 \item  pairs $(M,\varphi)$ as {objects}, where $\varphi\in\fifx_M$ such
that $\gamma(\varphi)\cong\top$ ($\gamma$ as defined in
Section~\ref{sec:glob-secs}), and
\item functions $f:M\to N$ such that $\varphi\leq f^*\psi$ as {morphisms}
from $(M,\varphi)$ to $(N,\psi)$.
\end{enumerate}
\item
The functor $\Delta:\catset\to\basm(\trip)$ is given by $M\mapsto(M,\top)$.
\end{itemize}
\end{definition}
Birkedal shows that for any existential fibration $\fifx$, $\basm(\fifx)$ is
regular and $\Delta$ is right adjoint to the global sections functor
$\Gamma=\basm(\fifx)(1,-)$ (in particular it preserves finite meets); if
$\fifx$ is a fibered frame then $\Delta$ is moreover regular. In
\cite[Theorem~3.5.1]{birkedal2000}, Birkedal shows that if $\fifx$ models
$\imp,\forall$ in addition to being a fibered frame, then $\basm(\fifx)$ is
locally cartesian closed.

\begin{theorem}
 Let $\trip:\tot{\trip}\to\catset$ be a regular tripos.
Then
$\Delta:\catset\to\catset[\trip]$ can be factorized into two functors, where
the first one is the constant objects functor of a realizability-like tripos
and the second one is the constant objects functor of a localic tripos.
\end{theorem}
\begin{proof}
 We define a fibered poset $\triq:\tot{\triq}\to\catset$ by taking the pullback
\[
\xymatrix@-2mm{
\tot{\triq}\ar[r]\ar[d]_\triq\pullbackcorner &
\Sub(\basm(\trip))\ar[d]^{\sub(\basm(\trip))}\\
\catset\ar[r]_-\Delta &\basm(\trip)
} 
\]
of $\sub(\basm(\trip))$ along $\Delta:\catset\to\basm(\trip)$. Since
$\basm(\trip)$ is regular and locally cartesian closed and $\Delta$ is regular,
$\triq$ is a fibered frame with $\imp,\forall$. Concretely, $\triq$ is given as
follows.
\begin{itemize}
 \item Predicates on $M$ are pairs $(U\subseteq M,\varphi\in\trip_M)$ such that
$\gamma(\varphi)\cong\top$.
\item $(U,\varphi)\leq(V,\psi)$ iff $U\subseteq V$ and $\varphi\leq\psi|_U$.
\item Reindexing is given by pullback and reindexing in $\trip$.
\end{itemize}
To show that $\triq$ is a tripos it remains to construct a generic predicate.
The underlying object is given by
$\prop_\triq=\widetilde{\gamma(\triptr_\trip)}$ -- the partial map classifier of
the subobject $\gamma(\triptr_\trip)\subset\prop_\trip$, and $\triptr_\triq$ is
given by the inclusion
$\gamma(\triptr_\trip)\hookrightarrow\widetilde{\gamma(\triptr_\trip)}$ together
with the predicate $\triptr_\trip|_{\gamma(\triptr_\trip)}$. Thus, $\triq$ is a
tripos, and it is straightforward to verify that $\triq$ is realizability-like.

We now construct a geometric morphism between $\triq$ and $\trip$.
Define $f^*:\triq\to\trip$ by 
\[
 \triq_M\ni(U,\varphi)\quad\mapsto\quad (\exists_U\varphi)\in\trip_M,
\]
where $\exists_U$ is a shorthand for existential quantification along the
inclusion $U\hookrightarrow I$. Define $f_*:\trip\to\triq$ by
\[
 \trip_M\ni\varphi\quad\mapsto\quad
(\gamma(\psi),\varphi|_{\gamma(\psi)})\in\triq_m.
\]
It is easy to see that the fibered monotone maps $f^*$ and $f_*$ constitute a
geometric morphism of triposes and thus by
\cite[Theorem~2.5.8]{vanoosten2008realizability} give rise to a localic
geometric morphism $F^*\adj F_*:\catset[\trip]\to\catset[\triq]$ of toposes with
$F^*$ preserving constant objects, or in other words
$F^*\circ\Delta_\triq\cong\Delta_\trip$.
\end{proof}
\begin{remarks}
\begin{enumerate}
 \item 
The presented decomposition is \emph{not} unique. This can be seen from
relative realizability. Given an inclusion $\pcaas\subseteq\pcaa$, the relative
realizability topos $\catrt(\pcaa,\pcaas)$ is localic over $\catrt(\pcaas)$
(see e.g.\ the introduction of \cite{awodey2002local}), which gives a
decomposition of $\Delta:\catset\to\catrt(\pcaa,\pcaas)$ into a
realizability-like followed by a localic part. However, this is 
\emph{not} the factorization given by the above construction -- the tripos
$\triq$ obtained by applying the construction to the tripos
$\catrt(\pcaa,\pcaas)$ is equivalent to the subfibration of
$\catrt(\pcaa,\pcaas)$ on predicates $\varphi:M\to P\pcaa$ satisfying
\[
 m\vtp M\csep \exists a\vtp\pcaa\qdot a\in \varphi(m)\ent\exists
a\vtp\pcaas\qdot a\in\varphi(m),
\]
which looks quite different from $\rtr{\pcaas}$.

One should search for strengthenings of the concepts `realizability-like' and
`localic', which make the decomposition unique up to equivalence.
\item
Given a tripos $\trip$, we have $\basm(\trip)\cong\asm(\triq)$ ($\triq$ is the
tripos constructed in the proof), which reconciles the definitions of
asssemblies of Birkedal and van Oosten in a certain sense.
\item
Using `triposes with $\gamma$', I hope that the factorization works over
arbitrary bases. The necessity of $\gamma$ is clear from the definition of
$\basm(\trip)$.
\end{enumerate}
\end{remarks}

\section{Bits and pieces, open ends}\label{sec:bits}

In the following, I present some open questions and ideas that have not been
fully explored yet.

\subsection{Uniform preorders as internal
preorders}\label{sec:uniform-as-internal}

A small presheaf on a locally small category $\catc$ is a small colimit of
representable presheaves in $\catset^{\catc^\op}$
\cite{day2007limits,rosicky1999cartesian,freyd-kelly-cats_of_cont_functors}.
Following \cite{day2007limits}, we denote the category of small presheaves on
$\catc$ by $\pp\catc$. $\pp\catc$ is locally small and can abstractly be
characterized as the \emph{small colimit cocompletion} of $\catc$. If $\catc$ is
small, then $\pp\catc=\widehat{\catc}$.

The category $\pp\catset$ of small presheaves on $\catset$ is a \emph{locally
cartesian closed $\infty$-pretopos}, but it is not a topos
(Giraud's theorem
fails since we do not have a small cogenerating family, and the truth value
object $\Omega:\catset^\op\to\catset$ is definable, but not small).
Using the \emph{axiom of choice} one can show that uniform preorders (and thus
fibered preorders with a generic family of predicates) are equivalent to
preorders internal to $\pp\catset$, which means that the theory of uniform
preoders is \emph{order theory in a predicative framework}\footnote{It is not
surprising that fibered preorders are preorders internal to presheaves -- the
interesting part is the relation between smallness and the generic family
of predicates.}. 

This point of view is a rich source of intuitions, and explains
for example why the uniform distributors of Section~\ref{suse-calc-dist} can
not be represented as a Kleisli category in the same way as ordinary
distributors (the monad applied to $1$ would yield the truth value object of
$\pp\catset$ which does not exist). As another example, given a uniform frame
$\upa$ it is possible to view $\catset[\upa]$ as a \emph{sub}category of the
category of \emph{internal presheaves} on $\upa$ in $\pp\catset$ --
intuitively it is only a subcategory since the category of internal presheaves
contains coproducts of elements of $\upa$ over arbitrary indexing objects in
$\pp\catset$ whereas $\catset[\upa]$ only contains the coproducts with respect
to indexing objects in the image of $Y:\catset\to\pp\catset$.

\medskip

If we want to do all this without relying on the axiom of choice, we have to
replace $\pp\catset$ by a category of `small sheaves for the regular topology'.
The problem is that it is not entirely clear how to define this category. One
approach would be to take the subcategory of $\pp\catset$ on sheaves for the
regular topology, but is not clear whether this category is closed under
colimits. Another approach would be to take small colimits of representables in
the subcategory of $\catset^{\catset^\op}$ on sheaves, but here we run into
similar problems. For me, it appears to be the safest option to take a larger
universe $\mathbf{SET}$ of sets with respect to which $\catset$ is small, and
then to take the category of \emph{large sheaves} on $\catset$ with respect to
the regular topology, that is the category of functors
$F:\catset^\op\to\mathbf{SET}$ which are sheaves for the regular topology. This
category is then a topos, and we can define the category of `small regular
sheaves' as small colimits of representables in `large regular sheaves'.

\subsection{Uniform preorders as enriched categories}\label{sec:hirschowitz}

Tom Hirschowitz made the remarkable observation that uniform preorders can be
defined as categories enriched in a quantaloid (see
e.g.~\cite{stubbe2005categorical}). Recall that a quantaloid is a category
which is enriched in cocomplete lattices, thus can be viewed as a 
locally ordered 2-category. Bénabou \cite{benabou1967introduction} defined a
category enriched in a bicategory to be a lax functor from an indiscrete
category into a bicategory, and it is in this sense that 
`quantloid-enriched' has to be read here.

The quantaloid $\mathcal{R}$ of interest is defined as follows.
\begin{itemize}
 \item objects are sets
\item a morphism from $M$ to $N$ is a \emph{downward closed} set
${R}\subseteq P(M\times N)$ of relations from $M$ to $N$
\item the ordering on $\mathcal{R}(M,N)$ is given by inclusion
\item composition is given by ${S}\circ{R}=\adcl\{s\circ r\msep
s\in S,r\in R\}$ for $M\xrightarrow{R}N\xrightarrow{S}O$
\end{itemize}
The reader is invited to verify that a category enriched in $\mathcal{R}$ is
exactly the same thing as a uniform preorder. Furthermore, $\catudist$ is the
category of enriched profunctors, but -- as Isar Stubbe pointed out --
$\catuord$ is more general than the category of enriched functors as defined in
\cite{stubbe2005categorical}.

As pointed out earlier, Hirschowitz's observation can
be viewed as a generalization of Hoshino's approach in the one-sorted case
(Remark~\ref{rem:hoshino}-\ref{rem:hoshino-hoshino}). It remains to be
clarified if and how this approach can be reconciled with the presentation of
uniform preorders as internal preorders
from Section~\ref{sec:uniform-as-internal}.

\subsection{The non-posetal case}\label{sec:non-posetal}

One of the motivating questions of this work was to find a common framework for
Grothendieck toposes and toposes induced by triposes. This goal has not 
been achieved, but the theory developed here can nevertheless shed some light
on the question. A first observation is that one-sorted uniform preorders with
$\twif$ are more `site-like' objects than triposes since they are small objects
`inside a category' instead of fibrations on it, which is already a step in the
right direction. 

The description of uniform preorders as preorders internal to small
regular sheaves from Section~\ref{sec:uniform-as-internal} tells us how to
define `uniform categories', namely as categories internal to small regular
sheaves on $\catset$ (it is possible -- while not quite as nice as for the
preorders -- to give an internal/fibration-free presentation of these objects).
With this in mind, I think a `non-posetal tripos' should be at least a
geometric uniform category $\upc$ such that $\catset[\upc]$ is a topos. 

To find the right conditions, one should also have a
second look at Theorem~\ref{theo:fibered-internal-gamma-tripos}. Since the
conditions on a uniform preorder $\upa$ that are necessary to make the theorem
work are exactly those which make $\ufam(\upa)$ into a tripos, it would be
illuminating to have an analogous result for the tentative `uniform
categories'.

\bibliographystyle{plain}
\nocite{marley1976}
\bibliography{../shared/bib}
\newpage
\include{notation}
\printindex
\end{document}